\documentclass[letterpaper, 11pt]{amsart}
\usepackage[margin=1.25in]{geometry}
\usepackage{mathtools}
\usepackage{amsmath}
\usepackage{amssymb}

\usepackage{amsthm}
\usepackage{graphicx}
\usepackage{mathrsfs}
\usepackage{xcolor}
\usepackage{tikz-cd}
\usepackage{tikz}
\usetikzlibrary{patterns}
\usepackage{hyperref}
\usepackage{cleveref}

\usepackage[utf8]{inputenc}
\usepackage[T1]{fontenc}
\allowdisplaybreaks

\usepackage{verbatim}
\usepackage{float}

\usepackage{enumitem}
\newlist{condenum}{enumerate}{1}
\setlist[condenum]{label=(R\arabic*), ref=(R\arabic*), leftmargin=*}

\DeclareMathAlphabet{\mathpzc}{OT1}{pzc}{m}{it}

\usepackage{thmtools}
\usepackage{thm-restate}

\usepackage{caption}

\newtheorem{theorem}{Theorem}[section]

\newtheorem*{claim*}{Claim}
\newtheorem{step}{Step}

\newtheorem{lemma}[theorem]{Lemma}

\newtheorem{corollary}[theorem]{Corollary}

\newtheorem{proposition}[theorem]{Proposition}

\theoremstyle{definition}
\newtheorem{definition}[theorem]{Definition}
\newtheorem{Def}[theorem]{Definition}

\newtheorem{claim}{Claim}

\theoremstyle{remark}
\newtheorem{remark}[theorem]{Remark}
\newtheorem*{remark*}{Remark}

\numberwithin{equation}{section}

\newcommand{\abs}[1]{\lvert#1\rvert}
\newcommand{\norm}[1]{\lVert#1\rVert}

\newcommand{\op}{\operatorname}

\newcommand{\be}{\begin{equation}}
\newcommand{\ee}{\end{equation}}
\newcommand{\Ga}{\Gamma}

\newcommand{\R}{\mathbb R}

\newcommand{\N}{\mathbb N}
\newcommand{\ga}{\gamma}

\newcommand{\la}{\lambda}
\newcommand{\La}{\Lambda}
\newcommand{\inte}{\op{int}}
\newcommand{\ba}{\backslash}

\newcommand{\cal}{\mathcal}
\newcommand{\br}{\mathbb R}
\newcommand{\SO}{\op{SO}}

\newcommand{\PSL}{\op{PSL}}
\newcommand{\F}{\cal F}

\newcommand{\bH}{\mathbb H}

\newcommand{\diam}{\op{diam}}

\newcommand{\T}{\op{T}}
\renewcommand{\frak}{\mathfrak}

\newcommand{\e}{\varepsilon}
\newcommand{\BR}{\op{BR}}

\newcommand{\BMS}{\op{BMS}}
\renewcommand{\L}{\mathcal L}
\newcommand{\fa}{\mathfrak a}

\renewcommand{\i}{\op{i}}

\newcommand{\C}{\cal C}

\newcommand{\rank}{\op{rank}}

\newcommand{\Leb}{\op{Leb}}
\newcommand{\supp}{\op{supp}}

\renewcommand{\epsilon}{\e}

\newcommand{\SL}{\op{SL}}

\newcommand{\ww}{\mathtt{w}}

\newcommand{\uu}{\mathtt{u}_{\psi}}
\newcommand{\pr}{\mathrm{Pr}}

\newcommand{\ut}{\mathtt{u}}

\newcommand{\Wd}{\fa}

\newcommand{\BP}{\mathrm{BP}}

\newcommand{\sH}{\langle \cal H\rangle_+}

\newcommand{\Fn}{F_2^{\sharp}}
\newcommand{\paral}{\partial^{\alpha}_{\mathfrak{a}}}
\newcommand{\Log}{\mathrm{Log}}
\newcommand{\Arg}{\mathrm{Arg}}

\author{Dongryul M. Kim}
\address{School of Mathematics, Institute for Advanced Study,  Princeton, NJ 08540 (Current)
\newline \indent Simons Laufer Mathematical Sciences Institute, Berkeley, CA 94720
\newline \indent Department of Mathematics, Yale University, New Haven, CT 06511}
\email{dongryul.kim97@gmail.com}

\author{Hee Oh}
\address{Department of Mathematics, Yale University, New Haven, CT 06511 }
\email{hee.oh@yale.edu}

\author{Wenyu Pan}
\address{Department of Mathematics, University of Toronto, Toronto, ON, M5S 2E4}
\email{wenyup.pan@utoronto.ca}

\begin{document}

\title[Higher-rank local mixing]{Higher-rank local mixing for~cusped~positive~representations}

\begin{abstract} 
Let $\Gamma<\operatorname{PSL}_2(\mathbb R)$ be a torsion-free,
non-elementary, finitely generated Fuchsian group with
parabolic elements, and let $\rho:\Gamma\to G$ be a
type-preserving, Zariski dense, positive representation
into a connected split semisimple real algebraic group.
We prove quantitative local mixing for
Bowen--Margulis--Sullivan measures on
$\rho(\Gamma)\backslash G/M$, where $M$ is the centralizer
of a maximal split torus in a compatible maximal compact
subgroup.

For every $u$ in the interior of the limit cone of $\rho(\Gamma)$,
correlations of smooth compactly supported functions
along $\exp(tu)$, normalized by
$t^{(\operatorname{rank}_{\mathbb R}G-1)/2}$, admit a
complete asymptotic expansion in integer powers of $t^{-1}$.
We also obtain a uniform Gaussian asymptotic for transverse
displacements, with an integrable remainder.
We deduce asymptotics for matrix coefficients of the
quasi-regular representation on $L^2(\rho(\Gamma)\backslash G)$
and for orbital counting in Riemannian balls.

More generally, these conclusions hold for Zariski dense,
relatively Borel Anosov representations of $\Gamma$
satisfying a peripheral root-balance condition.
This includes relatively Borel Anosov representations
into split simple groups of rank at most three and
self-joinings of type-preserving geometrically finite
real hyperbolic representations.

To prove these results,
we develop an intrinsic cusp geometry on the limit
set of $\rho(\Gamma)$ in the Furstenberg boundary and use it to construct a countable Markov coding. We establish exponential moment bounds and uniform distortion
estimates for the unbounded vector-valued return cocycle.
Using peripheral Cartan growth and Zariski density, we
establish the quantitative non-integrability needed for
the spectral analysis. 
\end{abstract}

\maketitle

\tableofcontents

\section{Introduction}
We establish local mixing with quantitative error terms for
higher-rank diagonal flows associated with cusped positive
representations and two other classes of representations of finitely generated
Fuchsian groups. Our main theorem gives a complete asymptotic
expansion of correlations along each direction in the interior
of the limit cone, together with a Gaussian asymptotic for
transverse displacements and a uniform, integrable remainder.
It also yields  asymptotics for matrix coefficients of the
quasi-regular representation and orbital counting in
Riemannian balls.

In rank one, the Bowen--Margulis--Sullivan (BMS) measure of a
non-elementary geometrically finite locally symmetric manifold
is finite \cite{Sullivan1984entropy}. Babillot
\cite{Babillot_mixing} proved mixing of the geodesic flow for
this measure, and Winter \cite{Winter_mixing} established
mixing of the frame flow for Zariski dense groups. These
results underlie orbital counting and equidistribution
theorems; see \cite{margulis_anosov,Roblin2003ergodicite,Oh_Shah,OS_inv}.
For finite-volume locally symmetric manifolds, mixing of
regular diagonal flows with respect to Haar measure follows
from the Howe--Moore theorem \cite{HM}.

The higher-rank, infinite-volume setting is different.
For Zariski dense discrete subgroups of simple real algebraic
groups of higher rank, there is an infinite family of
non-proportional BMS measures \cite{Benoist1997proprietes,Quint2002Mesures}, all of which are
infinite \cite{FL_BMS}. Local mixing describes the decay of
correlations after a suitable normalization:
 for a flow
$\{a_t\}$ on $X$ preserving a Radon measure $m$, one seeks
$\mathcal N(t)>0$ and nonzero Radon measures $m_1,m_2$ on $X$ such that, as $t\to +\infty$,
$$
\mathcal N(t)\int_X f_1(xa_t)f_2(x)\,dm(x)
\to m_1(f_1)m_2(f_2)
\qquad(f_1,f_2\in\mathcal C_c(X)),
$$
where $m_i(f):=\int_X f\,dm_i$ for $i=1,2$.
Oh--Pan \cite{OP_local} introduced this terminology and
proved local mixing for geodesic flows on $\mathbb Z^d$-covers
of convex cocompact rank-one locally symmetric manifolds,
with normalization proportional to $t^{d/2}$. For a recent survey of local mixing in infinite volume,
see \cite[Sections~7 and~8]{Oh}.

The main new ingredient in this paper is a geometric
construction of a countable coding that controls the full vector-valued return
cocycle in the presence of arbitrarily long cusp excursions.
We develop an intrinsic cusp geometry on the flag limit
set and derive quantitative non-integrability in every
Cartan direction from peripheral Cartan growth and
Zariski density. The required root-balance condition is
automatic in the three classes below. These estimates
yield complete correlation expansions and uniform
transverse asymptotics for every representation in these
classes.

\subsection*{Setting}
Let $\Gamma<\PSL_2(\mathbb R)$ be a torsion-free,
non-elementary, finitely generated Fuchsian group with
parabolic elements, and let $G$ be a connected semisimple
real linear algebraic group.
Let $\rho:\Gamma\to G$ be a Zariski dense representation.
Assume that one of the following holds:
\begin{itemize}
\item $G$ is $\br$-split, and $\rho$ is
type-preserving and positive in the sense of Fock--Goncharov
\cite{FockGoncharov_higher}; see
\Cref{strongdef,def.positiveframing}.

\item $G$ is $\br$-split and simple with $\operatorname{rank}_{\br} G\leq3$,
and $\rho$ is relatively Borel Anosov.

\item $G=\prod_{i=1}^k\SO(n_i,1)^\circ$, with $n_i\geq2$,
and $\rho$ is type-preserving and each projection
$\rho_i$ is  geometrically finite.
\end{itemize}
In all three cases, $\rho$ is relatively Borel Anosov,
and hence discrete and faithful; see
\Cref{thm.positive_boundary} and \Cref{prop.cusp_adapted_morse}. More generally, our results hold for Zariski dense, relatively
Borel Anosov representations of $\Gamma$ satisfying the
peripheral root-balance condition in
\Cref{def:peripheral-root-balance}; see
\Cref{thm:local-mixing-morse}.

Fix a maximal $\mathbb R$-split torus $A<G$, a compatible
maximal compact subgroup $K<G$, and a closed positive Weyl
chamber $A^+\subset A$, so that the Cartan decomposition $G=KA^+K$ holds. 
Set
$$
\mathfrak a:=\operatorname{Lie}A,\qquad
\mathfrak a^+:=\log A^+,\qquad
r:=\rank_{\br} G=\dim\mathfrak a,\qquad
M:=Z_K(A),
$$
where $Z_K(A)$ denotes the centralizer of $A$ in $K$.

Write $\Gamma_\rho:=\rho(\Gamma)<G$ and let $\mathcal L_\rho\subset\mathfrak a^+$ and $\psi_\rho$
denote the limit cone and growth indicator of $\Gamma_\rho$, respectively.
The limit cone has nonempty interior \cite{Benoist1997proprietes}.
For each $u\in\operatorname{int}\mathcal L_\rho$, there is a
unique linear form $\psi_u\in \frak a^*$ tangent to $\psi_\rho$ at $u$,
meaning that $\psi_u\geq\psi_\rho$ and
$\psi_u(u)=\psi_\rho(u)$ \cite{CZZ_transverse,CZZ_relative}.
Set $V_u:=\ker\psi_u$. Set $X_\rho:=\Gamma_\rho\backslash G/M$, and let
$m_u^{\BMS}$ be the BMS measure on $X_\rho$  associated with $\psi_u$, as defined in \eqref{bmsdef}.
Write
$a_v:=\exp(v)$ for $v\in\mathfrak a$.

\subsection*{BMS local mixing}
The following is our main result. In this higher-rank
cusped setting, the leading local mixing asymptotic
is itself new. We also obtain a complete correlation
expansion and uniform transverse estimates with an
integrable remainder.

\begin{theorem}\label{bmslocal_intro}
Let $u\in\operatorname{int}\mathcal L_\rho$.
\begin{enumerate}
\item There exist $c_u>0$ and continuous bilinear forms $$\mathfrak c_{j,u}: \C_c^\infty(X_\rho)\times
\C_c^\infty(X_\rho)\to \mathbb R,\qquad j\ge 0$$ such that, for every integer $N\geq0$ and all
$f_1,f_2\in \C_c^\infty(X_\rho)$,
$$
\begin{aligned}
&t^{\frac{r-1}2}\int_{X_\rho}
f_1(xa_{tu})f_2(x)\,dm_u^{\BMS}(x)=
\sum_{j=0}^{N}\mathfrak c_{j,u}(f_1,f_2)t^{-j}
+O_{u,f_1,f_2,N}(t^{-N-1})
\end{aligned}
$$
as $t\to +\infty$, where $\mathfrak c_{0,u}(f_1,f_2)
=c_u\,m_u^{\BMS}(f_1)m_u^{\BMS}(f_2)$.

\item There exist a positive definite quadratic form $Q_u$ on $V_u$ and $a>0$ such that, for every $w\in V_u$ and  all
$f_1,f_2\in \C_c^\infty(X_\rho)$,
$$
\begin{aligned}
&\qquad \qquad \qquad t^{\frac{r-1}2}\int_{X_\rho}
f_1(xa_{tu+\sqrt t\,w})f_2(x)\,dm_u^{\BMS}(x)\\
&\qquad=
c_u e^{-Q_u(w)/2}
m_u^{\BMS}(f_1)m_u^{\BMS}(f_2)+O_{u,f_1,f_2}\!\left(
t^{-\frac 12}e^{-a\min\{\|w\|^2,\|w\|\}}\right)
\end{aligned}
$$
as $t\to+\infty$, uniformly for $w\in V_u$, where $\|\cdot\|$ is any fixed norm on $\mathfrak a$.
\end{enumerate}
\end{theorem}

See \Cref{bmslocal} for bounds on the implied constants.
 By smooth approximation, the leading asymptotics in
parts~(1) and~(2) extend to all $f_1,f_2\in\C_c(X_\rho)$,
with an $o(1)$ remainder, uniformly for $w\in V_u$
in part~(2).
\begin{remark}
Deforming principal Fuchsian representations while fixing
the conjugacy classes of all boundary holonomies gives
connected components consisting entirely of type-preserving
positive representations; see \Cref{Ex_postive}.
\Cref{bmslocal_intro} applies to their Zariski dense members.
When $\Gamma$ is a lattice, only the cusp conjugacy classes
need to be fixed. This includes groups uniformizing a
thrice-punctured sphere or a once-punctured torus.
\end{remark}

For Anosov subgroups, the leading local mixing asymptotic
was established by Sambarino
\cite{sambarino_orbit,sambarino2022report} and Chow--Sarkar
\cite{CS_local}, following work of Thirion in special cases
\cite{Thirion_mixing}. More recently, Guedes
Bonthonneau--Lefeuvre--Weich \cite{GLW_spectrum} obtained a
complete asymptotic expansion under a Diophantine hypothesis
on the representation. Our theorem treats representations
with parabolic elements and obtains the quantitative
estimates from the peripheral geometry.

The rank-one theory provides important precedents for these
estimates. Mohammadi--Oh \cite{Mohammadi_Oh} and Edwards--Oh
\cite{Edwards_Oh_Duke} proved matrix-coefficient asymptotics
with exponential errors for large critical exponents.
Sarkar--Winter \cite{SarkarWinter_exp} proved exponential
mixing of BMS frame flows on convex cocompact real hyperbolic
manifolds, and Chow--Sarkar \cite{ChowSarkar_rankone} extended
this to rank-one locally symmetric spaces. In the presence
of cusps, Li--Pan \cite{LP} proved exponential mixing of
geodesic flows on geometrically finite real hyperbolic
manifolds, and Li--Pan--Sarkar \cite{LiPanSarkar_exp}
established the frame-flow version. Khalil \cite{Khalil_exp}
proved exponential mixing of geodesic flows on all geometrically
finite  rank-one locally symmetric spaces.
For infinite invariant measures, Dolgopyat--N\'andori--P\`ene
\cite{DNP} obtained expansions to arbitrary order on
$\mathbb Z^d$-covers of compact negatively curved manifolds.
\subsection*{Orbital counting and Matrix coefficient asymptotics}
The local mixing theorem supplies the input for the counting
and equidistribution arguments of Edwards--Lee--Oh
\cite{ELO_anosov}.
Let $d$ be the Riemannian distance on $G/K$ induced by the
Killing form, and let $\delta_\rho$ be the critical exponent
of $\sum_{\gamma\in\Gamma_\rho}e^{-s d(o,\gamma o)}$,
which is independent of $o\in G/K$.

\begin{corollary}
\label{thm.counting_intro}
For every
$o\in G/K$, there exists  $c(\rho,o)>0$ such that
$$
\#\{x\in\Gamma_\rho o:d(o,x)<T\}
\sim c(\rho,o)e^{\delta_\rho T}
\qquad(T\to\infty).
$$
\end{corollary}

We also obtain bisector counting asymptotics
(\Cref{cor.countingbisector}). 
Canary--Zhang--Zhu--Zimmer \cite{CZZZ_Orbital} established
counting asymptotics for relatively Anosov groups with
cutoffs defined by linear functionals of the Cartan
projection. Linearity reduces counting to equidistribution for a reparametrized flow with finite BMS measure, without requiring
higher-rank local mixing. 
Riemannian balls impose a nonlinear cutoff on the full
Cartan projection, so determining the leading constant
requires information about the transverse distribution
of orbit points. Our uniform transverse local mixing
estimates provide this information.

We also obtain local mixing with respect to Haar measure,
expressed as asymptotics for matrix coefficients of
compactly supported $M$-invariant vectors in
$L^2(\Gamma_\rho\backslash G)$.
Let $\varrho\in\mathfrak a^*$
be the sum of the positive roots, counted with multiplicity,
let $dx$ be induced by a fixed Haar measure on $G$, and
let $m_u^{\BR}$ and $m_u^{\BR_*}$ be the Burger--Roblin
measures defined in \Cref{subsec.BR}. Write $\i$ for the
opposition involution. The transfer argument of
Edwards--Lee--Oh \cite[Proposition~6.8]{ELO_anosov}, extending
Roblin's rank-one result \cite{Roblin2003ergodicite}, gives
the following limit. 
\begin{theorem}\label{m1}
Let $u\in\operatorname{int}\mathcal L_\rho$.
There exists $\kappa_u>0$ such that, for every $w\in V_u$ and all
$f_1,f_2\in \C_c(X_\rho)$,
$$
\begin{aligned}
&\lim_{t\to + \infty}
t^{\frac{r-1}{2}}
e^{(\varrho-\psi_u)(tu+\sqrt t\,w)}
\int_{X_\rho}
f_1(xa_{tu+\sqrt t\,w})f_2(x)\,dx=
\kappa_u c_u e^{-Q_u(w)/2}
m_{\i(u)}^{\BR}(f_1)m_u^{\BR_*}(f_2).
\end{aligned}
$$
 The convergence is locally uniform in $w$.

Moreover, there exists $b>0$ such that, for each compact
$\mathcal K\subset X_\rho$ and all $f_1,f_2\in\mathcal C_c(X_\rho)$
supported in $\mathcal K$, the normalized absolute correlation
on the left is
$$
O_{u,\mathcal K}\!\left(\|f_1\|_\infty\|f_2\|_\infty
e^{-b\min\{Q_u(w),\sqrt{Q_u(w)}\}}\right)
$$
as $t\to\infty$,
uniformly for $w\in V_u$ with $tu+\sqrt t\,w\in\mathfrak a^+$.
\end{theorem}

In rank one, matrix-coefficient asymptotics can also be
obtained from a positive $L^2$ Laplace eigenfunction and
a spectral gap \cite{Mohammadi_Oh,Edwards_Oh_Duke}.
For infinite-volume quotients with no rank-one factors,
such an eigenfunction does not exist
\cite[Theorem~1.1]{EFLO}.
The local mixing approach also gives equidistribution of
expanding horospheres; the precise statement and measure
normalizations appear in
\Cref{cor:horospherical-equidistribution}.

Finally, the Hopf--Tsuji--Sullivan criterion of
Burger--Landesberg--Lee--Oh \cite{BLLO} gives the same
directional ergodicity threshold as for Anosov subgroups.

\begin{corollary}
\label{cor.ergodicdichotomy}
For every $u\in\operatorname{int}\mathcal L_\rho$,
the flow $\{a_{tu}:t\in\mathbb R\}$ on
$(X_\rho,m_u^{\BMS})$ is ergodic if and only if
$\rank_{\br} G\leq3$.
\end{corollary}

\subsection*{Main ideas of the proof}
The support of $m_u^{\BMS}$ is a trivial $V_u$-bundle
over a non-compact flow space \cite{KOW_indicators} and
the base flow carries a finite invariant measure
\cite{KO_entropy,BCZZ}.
For Anosov representations, the corresponding base flow
has a compact phase space and admits a finite Markov coding
\cite{BCLS_gafa}. In our setting, parabolic elements of $\Gamma$ give rise to arbitrarily long
cusp excursions, leading to countably many return branches
and an unbounded $\mathfrak a$-valued return cocycle.
We construct a countable coding for the right translation by
$a_{tu}=\exp(tu)$ on $(X_\rho,m_u^{\BMS})$ and establish
 uniform distortion estimates
and exponential moment bounds for this cocycle.

Our construction builds on the countable coding of
Li--Pan \cite{LP} for geometrically finite real hyperbolic
manifolds. They use Euclidean boundary coordinates,
cusp heights in the upper half-space model of hyperbolic space, and conformality
of the boundary action.

On the limit set of $\Gamma_\rho$ in the Furstenberg boundary,
we develop an intrinsic cusp geometry
using a blown-up premetric and a compatible height theory.
Separation estimates for parabolic limit points and
Patterson--Sullivan mass estimates, together with recurrence
of the reparametrized flow, give exponential decay of the
residual mass in the coding construction.
To control distortion of the full return cocycle, we combine
local Busemann estimates on boundary shadows
\cite{CK_Horospherical} with contraction of the inverse branches.

\subsubsection*{The geometry of the cusp coding} We begin with the unique linear form $\psi\in\mathfrak a^*$
tangent to $\psi_\rho$ at $u$. Rescaling $u$, we may assume that $\psi(u)=1$. Set
$\bar\psi=(\psi+\psi\circ\i)/2$.
The symmetric form $\bar\psi$ defines the cusp geometry, while $\psi$
defines the Patterson--Sullivan measure and the return times.
We use a $\rho$-equivariant continuous
Morse embedding $f:\operatorname{hull}\Lambda(\Gamma)\to G/K$, whose
asymptotic Cartan cone equals $\mathcal L_\rho$.

 Let $\mathcal F := G/P$ be the Furstenberg boundary, where
$P$ is the minimal parabolic subgroup determined by $A^+$. Let
$\Lambda_\rho\subset\mathcal F$ be the limit set of $\Gamma_\rho$, and
let $\nu_\psi$ be the  Patterson--Sullivan measure on $\La_\rho$ associated with $\psi$.  Recall the visual premetric on $\La_\rho$
\cite{LO_invariant,DKO_AR}:
$$
d_{\bar\psi}(\xi,\eta):= e^{-\bar\psi(\mathcal G(\xi,\eta))}
\qquad(\xi\neq\eta\in\Lambda_\rho),
$$
where $\mathcal G$ is the $\mathfrak a$-valued Gromov
product defined in \eqref{eqn.Gromovproductdef}. Fixing a parabolic limit point $p_0\in\Lambda_\rho$,
we introduce the blown-up premetric on
$\Lambda_\rho\smallsetminus\{p_0\}$:
$$
D_{\bar\psi}(\xi,\eta):=
\frac{d_{\bar\psi}(\xi,\eta)}
{d_{\bar\psi}(\xi,p_0)d_{\bar\psi}(\eta,p_0)} \qquad(\xi\neq\eta\in\Lambda_\rho \smallsetminus\{p_0\}).
$$
In the upper half-space model of $\bH^n$, with $p_0=\infty$,
this is, up to scaling, a positive power of the Euclidean
metric on $\partial\mathbb H^n\smallsetminus\{\infty\}$ (\Cref{EucRe}).
We prove a uniform quasi-triangle inequality for $D_{\bar\psi}$ (\Cref{lem.eucdistancetriangle}) and establish
 geometric and measure estimates that allow it to play
the role of the Euclidean metric in the cusp geometry
of $\Lambda_\rho$.

We define heights $h(p)$ of parabolic limit points (\Cref{hd}), relate them
with Cartan projections of suitably chosen group representatives, (\Cref{lem.heightondelta}), and prove the uniform separation estimate
$$
D_{\bar\psi}(p,q)\ge C \sqrt{h(p)h(q)}
\qquad(p\neq q \in\Lambda_{\mathrm p}(\Gamma_\rho)\smallsetminus\{p_0\})
$$
for a uniform $C> 0$; see \Cref{lem.LP6.16}.
Using the global shadow lemma
\cite{KO_globalshadow} and \Cref{lem.regularunipotent_cartan}, we prove doubling
and reverse doubling for $\nu_\psi$
(\Cref{prop.LP5.14}) and obtain estimates for $D_{\bar\psi}$-balls
centered at parabolic limit points
(\Cref{eqn.LP5.16}).

Choose a compact fundamental domain
$\Delta_0\subset\Lambda_\rho\smallsetminus\{p_0\}$
for the stabilizer of $p_0$ in $\Gamma_\rho$.
Starting from $\Delta_0$, we group parabolic limit points
into height layers and successively remove the corresponding
$D_{\bar\psi}$-balls. 
The central geometric step combines annular mass estimates
with recurrence of the reparametrized flow: there exist
$N\geq1$ and $\varepsilon>0$ such that, for every sufficiently large
$n$, stages $n+1,\ldots,n+N$ remove at least an $\varepsilon$-fraction
of the Patterson--Sullivan mass remaining at stage $n$. Thus the residual mass
decays exponentially; see \Cref{sec.expdecayomega}.

The geometric partition defines an expanding map on a
conull invariant subset of the union of the cusp domains.
Taking first returns to $\Delta_0$ gives a full-branch
expanding map, and we prove that the induced branches
retain exponential summability. Finally, we group a fixed
number of successive returns so that the resulting scalar
roof is uniformly bounded away from zero;
see \Cref{sec.expandingmap,sec.exptail}.

The following theorem summarizes the construction and the
estimates needed for the spectral analysis.

\begin{theorem}[Countable Markov coding]
\label{thm:intro_coding}
There exist $\Lambda_+\subset\Delta_0$ with
$\nu_\psi(\Delta_0\smallsetminus\Lambda_+)=0$, a countable family
$\mathcal H\subset\Gamma_\rho$, and a map $T:\Lambda_+\to\Lambda_+$
such that
$$
\Lambda_+=\bigsqcup_{\gamma\in\mathcal H}\gamma\Lambda_+,
\qquad T|_{\gamma\Lambda_+}=\gamma^{-1}.
$$
Moreover, the family $\mathcal H$ freely generates a Zariski dense subsemigroup,
and the coding map
$\Psi:\mathcal H^{\mathbb N}\to\Lambda_+$ is a H\"older
homeomorphism conjugating the one-sided shift to $T$, for the standard
symbolic metric and the premetric $D_{\bar\psi}$.

For $x\in\gamma\Lambda_+$,  define the $\mathfrak a$-valued
return cocycle $K$ and the scalar roof function $R$ by
$$
K(x):=\beta_x(e,\gamma)
+2\mathcal G(Tx,p_0)-2\mathcal G(x,p_0),
\qquad R(x):=\psi(K(x)),
$$
where $\beta$ is the $\mathfrak a$-valued Busemann map.
Then $K\circ\Psi$ is locally H\"older, $R\geq1$ on $\La_+$, and there exists
$\varepsilon_0>0$ such that
$$
\sum_{\gamma\in\mathcal H}\sup_{x\in\Lambda_+}
e^{-(1-\varepsilon_0)R(\gamma x)}<\infty,
\qquad
\int_{\Lambda_+}e^{\varepsilon_0\|K\|}\,d\nu_\psi<\infty.
$$
\end{theorem}
We also prove quantitative non-integrability of the return
cocycle $K$: in every unit direction in $\mathfrak a$, we
obtain quantitative lower bounds on the variation of
differences of projected cocycle sums along suitable inverse
branches (\Cref{thm:non-integrability}).
The proof combines non-concentration of $\Lambda_\rho$
at scales adapted to the contraction rates in the simple-root
directions (\Cref{prop: non-concentration self-joining})
with Zariski density of the semigroup
$\langle\mathcal H\rangle_+$ (\Cref{cor.Zdensesemigroup})
and the peripheral root-balance condition
\eqref{eqn:peripheral-root-balance}.

\subsubsection*{The symbolic model}
Let $\mathsf m_\psi:=m_u^{\BMS}$ be the BMS measure on $X_\rho$ associated with $\psi$.
The local H\"older regularity of $R\circ\Psi$ and
exponential branch summability (\Cref{thm:intro_coding}) allow us to apply
the Ruelle--Perron--Frobenius  theory for countable shifts
\cite{MU03,Sarig_Countable}, obtaining a unique
$T$-invariant Gibbs probability measure $\nu_+$ on $\La_+$ for the
potential $-R$. This measure is equivalent to
$\nu_\psi|_{\Lambda_+}$, with density bounded above and below
by positive constants.  Using the Gibbs property and the ergodic theorem,
we show that 
$$
u=\frac{\int_{\Lambda_+}K\,d\nu_+}
{\int_{\Lambda_+}R\,d\nu_+};
$$
see \Cref{thm.mixingdirectiontangent}.
We use the coding to construct a suspension model
for the flow $\exp(tu)$ on $(X_\rho,\mathsf m_\psi)$
and establish local mixing in this direction.

Let $\Lambda_-\subset\Lambda_\rho$ be a compact neighborhood
of $p_0$, disjoint from $\Delta_0$, chosen so that for every $\gamma\in\mathcal H$,
$\gamma^{-1}$ preserves $\Lambda_-$
and contracts $d_{\bar\psi}$ uniformly; see
\Cref{lem:shrinking fibers}.
Define $\widehat T$ on $\La_+\times \La_-$ by
$$
\widehat T(x,x_-):=(Tx,\gamma^{-1}x_-)
\qquad(x\in\gamma\Lambda_+, \, \ga \in \cal H),
$$
and let $\nu^\diamond$ be its invariant probability measure
projecting onto $\nu_+$ on $\La_+$.  Set 
$$
K^{\circ}(x) := K(x) - R(x)u \in \ker\psi\qquad (x\in \La_+),
$$
so that
$\int_{\Lambda_+}K^\circ\,d\nu_+=0$. The suspension has coordinates $(x,x_-,s,v)$ with
$0\leq s\leq R(x)$ and $v\in\ker\psi$, subject to
the boundary identification
$$
(x,x_-,R(x),v)
\sim
\bigl(\widehat T(x,x_-),0,v-K^\circ(x)\bigr).
$$
The suspension flow increases $s$ at unit speed.
We construct a  factor map $\Phi$ from this
suspension to the $\exp (tu)$-flow on $(X_\rho,\mathsf m_\psi)$ that
intertwines translations in the transverse
coordinates; see \Cref{sec.localmixingsymbolic}.
Although $\Phi$ need not be proper, we prove that for every
compact set $Q\subset X_\rho$, there exists $C_Q>0$ such that
$$\Phi^{-1}(Q)\subset
\bigl\{(x,x_-,s,v):
0\leq s<R(x),\ 
\|su+v\|\leq C_Q(R(x)+1)\bigr\};$$
see \Cref{cor:factor map}. This support bound controls the growth of the Fourier
transforms of pulled-back functions and their frequency
derivatives in terms of $R$. The exponential moment of
$R$ makes these bounds integrable.
Contraction of the stable fibers reduces the analysis
to the expanding system $(\Lambda_+,T)$.

\subsubsection*{The correlation expansion} 
Fix a sufficiently small H\"older exponent $c>0$.
For $\xi\in\mathfrak a$, the normalized twisted transfer operator
$\mathcal L_\xi$ acts on $\mathcal F_c(\Lambda_+,\mathbb C)$, the space of
bounded locally $c$-H\"older functions, by
$$
\mathcal L_\xi H(x)
=\sum_{\gamma\in\mathcal H} e^{-R(\gamma x)}\frac{h(\gamma x)}{h(x)} e^{i\langle\xi,K(\gamma x)\rangle}H(\gamma x),
\qquad x\in\Lambda_+,
$$
where $h>0$ is an eigenfunction with eigenvalue $1$
of the transfer operator associated with the potential $-R$.

 For $\xi$ near zero, $\mathcal L_\xi$ has a simple leading
eigenvalue $\kappa_\xi$. We prove that $\xi \mapsto |\kappa_\xi|$ has a negative-definite Hessian at $\xi=0$; see
\Cref{cor:expansion of complex eigenvalue around 0}.
This Hessian determines the Gaussian dependence of the
leading term on the transverse displacement $w\in \ker\psi$.

To analyze correlations, we adapt the spectral and Fourier methods of
Dolgopyat--N\'andori--P\`ene \cite{DNP}. Fourier inversion expresses the BMS correlation as a sum
over returns $n$ of frequency integrals, which we analyze
using stable projections and $\mathcal L_\xi^n$. Outside a fixed small neighborhood of zero, the decay
of the Fourier transforms in $\xi$ and of
$\mathcal L_\xi^n$ in $n$ controls the contribution;
see \Cref{prop:norm of transfer operator}.
The high-frequency bounds rely on quantitative
non-integrability of $K$ (\Cref{sec.nonintegrability}),
which is therefore needed already in our proof of leading
local mixing asymptotic.
Near $\xi=0$, we normalize the transformed correlation
associated with $n$ returns by $\kappa_\xi^n$.
Spectral estimates show that the Taylor coefficients
at $\xi=0$ of this normalized expression converge
faster than any inverse power of $n$;
see \Cref{lem:ANF1F2}.
Combining these limits with the expansion of
$\kappa_\xi^n$, Fourier inversion gives sums over $n$
of derivatives of a Gaussian determined by the
quadratic term of the centered leading eigenvalue.
Evaluating these sums as in \cite{DNP} gives the
expansion in \Cref{bmslocal_intro}(1).
Uniform estimates for the sums, combined with
exponential-moment bounds for the transverse tails,
give the uniform remainder, integrable in $w\in \ker \psi$, in
\Cref{bmslocal_intro}(2).

\begin{remark}In the Anosov setting, exponential mixing of the translation flow on the compact based space
was established by Delarue--Monclair--Sanders
\cite{DelarueMonclairSanders} for irreducible projective
Anosov subgroups, and by Chow--Sarkar \cite{ChowSarkar_exp}
for general Zariski dense Anosov subgroups outside an
exceptional cone of parameters.
We expect an analogous result for the non-compact base flow considered here.
The symbolic, geometric, and spectral tools developed in this
paper provide key ingredients toward such an extension.
\end{remark}

\subsection*{Organization of the paper}
\Cref{sec.Fuchsian,sec.positive} give the Fuchsian and higher-rank
preliminaries. \Cref{sec.general-hypotheses} states the general
theorem and verifies its hypotheses for the three classes
above.  The remaining sections form four parts.
\begin{itemize}
\item Part~\ref{part.coding} develops the cusp geometry and measure
estimates, constructs the coding, and proves exponential decay and expansion.
\item Part~\ref{part.return} establishes the symbolic conjugacy,
exponential moments, H\"older regularity, non-arithmeticity, and
non-integrability of the return cocycle.
\item 
Part~\ref{part.thermo} develops the spectral theory and suspension
model, proves BMS local mixing, and deduces Haar local mixing,
horospherical equidistribution, and orbital counting.
\end{itemize}

\subsection*{Notation}

For nonnegative functions $f$ and $g$, we write $f\ll g$ if $f\leq Cg$
for some constant $C>0$, and $f\asymp g$ if both $f\ll g$ and $g\ll f$.
For vector-valued functions, $f\approx g$ means that $f-g$ is uniformly
bounded. Dependence of an implied constant is indicated by a subscript
when needed.

\subsection*{Acknowledgments}
We would like to thank Jialun Li, Fran\c{c}oise P\`{e}ne, and Andrew Zimmer for useful discussions. 
Our original proof of \Cref{Zd2} used spherical pairs. We would like to thank Pratyush Sarkar for helpful discussions that led to a simpler proof.
Part of this research was conducted while the authors were visiting SLMath, supported by NSF grant DMS-2424139.  Kim is supported by an anonymous donor for his position at the Institute for Advanced Study. Oh is partially supported by the NSF grant DMS-2450703. Pan is supported by NSERC Discovery Grant RGPIN-2025-06706.

This project began in spring of 2024. During the final preparation
of the manuscript, the authors used AI tools to assist with proof checking, simplification of some arguments, and editing. The main mathematical ideas and methods are due to the authors,
who take full responsibility for the final manuscript.

\section{Fuchsian groups}
\label{sec.Fuchsian}

We use the upper half-plane model
$\mathbb H^2=\{x+iy\in\mathbb C:y>0\}$ with metric
$ds^2=(dx^2+dy^2)/y^2$, and denote the associated distance
by $d$. Its ideal boundary is
$\partial\mathbb H^2=\widehat{\mathbb R}
:=\mathbb R\cup\{\infty\}$.
The orientation-preserving isometry group
$\PSL_2(\mathbb R)$ acts on
$\mathbb H^2\cup\widehat{\mathbb R}$ by M\"obius
transformations.

\subsection{The hyperbolic plane and its isometries}
\label{ss:hyp}

Set $o:=i\in\mathbb H^2$.
The stabilizer of $o$ is $\op{PSO}(2)$, while the
stabilizer of $\infty\in\widehat{\mathbb R}$ is the
upper triangular subgroup $P$. Thus
$\mathbb H^2\simeq\PSL_2(\mathbb R)/\op{PSO}(2)$ and
$\partial\mathbb H^2\simeq\PSL_2(\mathbb R)/P$.
The action on the upward-pointing unit tangent vector
at $o$ identifies $\PSL_2(\mathbb R)$ with the unit tangent bundle
$\mathrm T^1\mathbb H^2$. Under this identification,
the geodesic flow for time $t$ is right multiplication by
$a_t:=
\bigl(\begin{smallmatrix}
e^{t/2}&0\\
0&e^{-t/2}
\end{smallmatrix}\bigr)$.

An element $g\in\PSL_2(\mathbb R)$ is \emph{loxodromic}
if it fixes exactly two points of $\partial\mathbb H^2$.
One fixed point is attracting and the other is repelling;
we denote the attracting fixed point by $y_g$.
An element is \emph{parabolic} if it fixes exactly one
point of $\partial\mathbb H^2$.
Every infinite-order element of a discrete subgroup of
$\PSL_2(\mathbb R)$ is either loxodromic or parabolic.

For $\xi\in\partial\mathbb H^2$ and $x,y\in\mathbb H^2$,
the Busemann function is
$$
\beta_\xi(x,y)
:=
\lim_{z\to\xi}\bigl(d(x,z)-d(y,z)\bigr).
$$
It satisfies the cocycle identity $\beta_\xi(x,z)=\beta_\xi(x,y)+\beta_\xi(y,z)$.
A closed horoball based at $\xi$ is a set of the form
$$
\mathsf H(\xi,r)
:=\{x\in\mathbb H^2:\beta_\xi(o,x)\geq r\},
\qquad r\in\mathbb R.
$$

For distinct $\xi,\eta\in\mathbb H^2\cup\partial\mathbb H^2$,
we write $[\xi,\eta]\subset\mathbb H^2$ for the geodesic
segment, ray, or line joining them.
For distinct $\xi,\eta\in\partial\mathbb H^2$, their
Gromov product based at $o$ is
$$
\mathcal G(\xi,\eta)
:=
\frac12\bigl(\beta_\xi(o,x)+\beta_\eta(o,x)\bigr),
\qquad x\in[\xi,\eta],
$$
which is independent of the choice of $x$.
We set $\mathcal G(\xi,\xi)=+\infty$.
We use a visual metric on $\partial\mathbb H^2$ based
at $o$, which is bi-Lipschitz equivalent to
$(\xi,\eta)\mapsto e^{-\mathcal G(\xi,\eta)}$.

For $y\in\mathbb H^2$ and $r>0$, let
$B(y,r):=\{z\in\mathbb H^2:d(z,y)<r\}$.
For $x,y\in\mathbb H^2$, the shadow of $B(y,r)$
seen from $x$ is
$$
O_r(x,y)
:=
\{\xi\in\partial\mathbb H^2:
[x,\xi]\cap B(y,r)\neq\emptyset\}.
$$
The triangle inequality gives
$$
\sup_{\xi\in O_r(x,y)}
\left|\beta_\xi(x,y)-d(x,y)\right|
\leq 2r.
$$

The Hopf parametrization is the homeomorphism
$$
\PSL_2(\mathbb R)
\to 
\bigl((\partial\mathbb H^2\times\partial\mathbb H^2)
\smallsetminus\operatorname{diag}\bigr)\times\mathbb R,
\qquad
g\mapsto
\bigl(g^+,g^-,\beta_{g^+}(o,go)\bigr),
$$
where $g^+:=g(\infty)$ and $g^-:=g(0)$.
In these coordinates, left multiplication by
$h\in\PSL_2(\mathbb R)$ takes the form
$$
h(\xi,\eta,s)
=
\bigl(h\xi,h\eta,s+\beta_\xi(h^{-1}o,o)\bigr),
$$
and the geodesic flow is translation:
$(\xi,\eta,s)a_t=(\xi,\eta,s+t)$.

A discrete subgroup $\Gamma<\PSL_2(\mathbb R)$ is called
a Fuchsian group. Its limit set
$\Lambda(\Gamma)\subset\partial\mathbb H^2$ is the set
of accumulation points of $\Gamma o$.
In Hopf coordinates, the nonwandering set of the geodesic
flow on $\mathrm T^1(\Gamma\backslash\mathbb H^2)$
is identified with
\begin{equation}
\label{eqn.omegagarankone}
\Omega_\Gamma
:=
\Gamma\backslash\widetilde\Omega_\Gamma,
\qquad
\widetilde\Omega_\Gamma
:=
\bigl((\Lambda(\Gamma)\times\Lambda(\Gamma))
\smallsetminus\operatorname{diag}\bigr)\times\mathbb R \subset \T^1(\bH^2).
\end{equation}
\subsection{Finitely generated Fuchsian groups}

Throughout the paper, all Fuchsian groups are assumed
to be torsion-free.
Let $\Gamma<\PSL_2(\mathbb R)$ be a non-elementary
Fuchsian group.
A point $\xi\in\partial\mathbb H^2$ is a
\emph{conical limit point} of $\Gamma$ if there exist
$r>0$ and a sequence of distinct elements
$\gamma_n\in\Gamma$ such that
$\xi\in O_r(o,\gamma_n o)$ for every $n$.
A point $\xi\in\partial\mathbb H^2$ is a
\emph{parabolic limit point} if it is fixed by a
parabolic element of $\Gamma$.
We denote the sets of conical and parabolic limit
points by $\Lambda_{\mathrm c}(\Gamma)$ and $\Lambda_{\mathrm p}(\Gamma)$,
respectively.

The convex core of $\Gamma\backslash\mathbb H^2$ is
$\operatorname{core}(\Gamma\backslash\mathbb H^2)
:=\Gamma\backslash\operatorname{hull}\Lambda(\Gamma)$,
where $\operatorname{hull}\Lambda(\Gamma)\subset\mathbb H^2$
is the convex hull of $\Lambda(\Ga)$. We call $\Gamma$ \emph{geometrically finite} if
$\Lambda(\Gamma)
=
\Lambda_{\mathrm c}(\Gamma)\sqcup\Lambda_{\mathrm p}(\Gamma)$, or equivalently, its convex core
has finite area.
A Fuchsian group is geometrically finite if and only
if it is finitely generated; see, for example,
\cite{Maskit1988Kleinian,MT_oxford,Bowditch_GF}.

Suppose that $\Gamma$ is geometrically finite and
$\Lambda_{\mathrm p}(\Gamma)\neq\emptyset$.
There exists a $\Gamma$-invariant family $\mathscr H$
of pairwise disjoint closed horoballs, one based at
each parabolic limit point, such that $\Gamma$ acts
cocompactly on $\operatorname{hull}\Lambda(\Gamma)
\smallsetminus
\bigcup_{\mathsf H\in\mathscr H}\operatorname{int}(\mathsf H)$.
Its compact quotient is called the thick part of the
convex core; the complementary components are cusp
neighborhoods.

\section{Relatively Anosov representations and BMS flows}
\label{sec.positive}

\subsection{Symmetric spaces and boundary geometry}
Let $G$ be a connected semisimple real linear algebraic group.
Fix a Cartan involution $\theta$ of $G$, and use the same
notation for its differential on $\mathfrak g:=\operatorname{Lie}(G)$.
Let $K:=G^\theta$ be the corresponding maximal compact subgroup.
The Cartan decomposition is
$\mathfrak g=\mathfrak k\oplus\mathfrak p$, where
$\mathfrak k$ and $\mathfrak p$ are the $+1$ and $-1$
eigenspaces of $\theta$, respectively.

Choose a maximal abelian subspace $\mathfrak a\subset\mathfrak p$,
and set $A:=\exp\mathfrak a$.
Let $\langle\cdot,\cdot\rangle$ and $\|\cdot\|$ denote the
inner product and norm on $\mathfrak a$ induced by the Killing form.
Fix a closed positive Weyl chamber
$\mathfrak a^+\subset\mathfrak a$, and set $A^+:=\exp\mathfrak a^+$.
The Cartan decomposition $G=KA^+K$ defines the
Cartan projection $\mu:G\to\mathfrak a^+$ by
$g\in K(\exp\mu(g))K$.

Let $X:=G/K$ be the associated Riemannian symmetric space,
with basepoint $o:=K$.
We equip $X$ with the $G$-invariant Riemannian metric induced
by the Killing form. Its distance satisfies
$d(go,ho)=\|\mu(g^{-1}h)\|$ for $g,h\in G$.
For $x=go$ and $y=ho$, we also write
$\mu(x):=\mu(g)$ and $\mu(x^{-1}y):=\mu(g^{-1}h)$;
these are independent of the representatives by $K$-bi-invariance.

We will use the following estimate on the Cartan projection
throughout.
\begin{lemma}[{\cite[Lemma~4.6]{Benoist1997proprietes}}]
\label{lem.cptcartan}
For every compact subset $Q\subset G$, there exists $C=C(Q)>0$
such that, for all $g\in G$,
$$
\sup_{q_1,q_2\in Q}\|\mu(q_1gq_2)-\mu(g)\|\le C.
$$
\end{lemma}

For $\alpha\in\mathfrak a^*$, define
$$
\mathfrak g_\alpha
:=
\{Z\in\mathfrak g:[H,Z]=\alpha(H)Z
\text{ for every }H\in\mathfrak a\}.
$$
The nonzero $\alpha$ for which $\mathfrak g_\alpha\neq\{0\}$
form the restricted root system $\Phi$.
Let $\Phi^+$ and $\Pi$ denote the positive
and simple restricted roots determined by $\mathfrak a^+$.
Set $$\mathfrak n:=\bigoplus_{\alpha\in\Phi^+}\mathfrak g_\alpha,\quad
N:=\exp\mathfrak n,\quad  M:=Z_K(A).$$
Then $P:=MAN$ is the corresponding minimal parabolic subgroup,
and $G=KAN$ is the Iwasawa decomposition.
The Furstenberg boundary is $\mathcal F:=G/P$, which we
identify with $K/M$. Let $W:=N_K(A)/M$ be the restricted Weyl group.
Choose a representative $w_0\in N_K(A)$ of its longest element.
Set $\check N:=w_0Nw_0^{-1}$,
$\check{\mathfrak n}:=\operatorname{Lie}(\check N)$, and
$\check P:=w_0Pw_0^{-1}=MA\check N$.
The opposition involution
$\i:=-\operatorname{Ad}_{w_0}|_{\mathfrak a}$ preserves
$\mathfrak a^+$ and satisfies $\mu(g^{-1})=\i\mu(g)$.
We also write $\i$ for the involution on $\mathfrak a^*$
defined by $\i\varphi:=\varphi\circ\i$.

Two flags $\xi,\eta\in\mathcal F$ are 
 transverse if $(\xi,\eta)\in G\cdot(P,w_0P)$, where
$G$ acts diagonally.
We denote the set of transverse pairs by
$\mathcal F^{(2)}:=G\cdot(P,w_0P)$.
For $g\in G$, write $g^+:=gP$ and $g^-:=gw_0P$.

Every $g\in G$ has a unique real Jordan decomposition
$g=g_hg_eg_u$ into commuting hyperbolic, elliptic, and
unipotent components. Here $g_h$, $g_e$, and $g_u$ are
conjugate to elements of $A^+$, $K$, and $N$, respectively.
The Jordan projection $\lambda(g)\in\mathfrak a^+$ is
the unique vector such that $g_h$ is conjugate to
$\exp\lambda(g)$. It satisfies
$\lambda(g)=\lim_{n\to\infty}\frac1n\mu(g^n)$.
In particular, $\lambda(g^{-1})=\i\lambda(g)$.

An element $g\in G$ is called loxodromic if
$\lambda(g)\in\operatorname{int}\mathfrak a^+$.
Equivalently, $g=h(am)h^{-1}$ for some
$a\in\operatorname{int}A^+$, $m\in M$, and $h\in G$.
In this case, $g^{-1}$ is also loxodromic, and the unique
attracting fixed flags of $g$ and $g^{-1}$ are
$y_g:=hP$ and $y_{g^{-1}}:=hw_0P$, respectively.

An element $g$ is unipotent if $g=g_u$. We call $g$ weakly unipotent if $g_h=e$ and $g_u\ne e$.

\begin{definition}
\label{strongdef} For a torsion-free discrete subgroup $\Gamma<\PSL_2(\br)$,
a representation $\rho:\Gamma\to G$ is
\emph{type-preserving} if it maps loxodromic elements
to loxodromic elements and parabolic elements to
weakly unipotent elements.
\end{definition}

For $g\in G$ and $k\in K$, define the Iwasawa cocycle
$\sigma(g,k)\in\mathfrak a$ by
$gk\in K(\exp\sigma(g,k))N$.
For $\xi=kP\in\mathcal F$ and $g,h\in G$, the
$\mathfrak a$-valued Busemann map is
\begin{equation}
\label{eqn.Busemannmapdef}
\beta_\xi(g,h)
:=
\sigma(g^{-1},k)-\sigma(h^{-1},k).
\end{equation}
This is independent of the choice of $k$, and satisfies
the cocycle identity and the equivariance relation
$\beta_{g_0 \xi}(g_0 g, g_0 h)=\beta_\xi(g,h)$ for  $g_0\in G$. It  descends to
$\mathcal F\times X\times X$, where we use the same notation.

\begin{lemma}
\label{lem.zerobuse}
If a weakly unipotent element $g\in G$ fixes
$\xi\in\mathcal F$, then $\beta_\xi(e,g)=0$.
\end{lemma}

\begin{proof}
Choose $k\in K$ with $\xi=kP$ and set $v:=k^{-1}gk\in P$.
The Levi projection $P\to MA$ preserves Jordan components. Since $g$ is weakly unipotent,
 $v\in MN$, so $g^{-1}k=kv^{-1}\in KN$.
It follows that $\sigma(g^{-1},k)=0$ and
$\beta_\xi(e,g)=0$.
\end{proof}

For $(\xi,\eta)\in\mathcal F^{(2)}$, choose $g\in G$ with
$(\xi,\eta)=(g^+,g^-)$ and define the
$\mathfrak a$-valued Gromov product by
\begin{equation}
\label{eqn.Gromovproductdef}
\mathcal G(\xi,\eta)
:=
\frac12\bigl(\beta_\xi(e,g)+\i\beta_\eta(e,g)\bigr).
\end{equation}
This does not depend on the choice of $g$ and satisfies that
for every $g\in G$ and $(\xi,\eta)\in\mathcal F^{(2)}$,
\begin{equation}
\label{GPT}
2\mathcal G(g\xi,g\eta)
=
2\mathcal G(\xi,\eta)
+\beta_{g\xi}(e,g)
+\i\beta_{g\eta}(e,g).
\end{equation}

The Hopf parametrization identifies $G/M$ with
$\mathcal F^{(2)}\times\mathfrak a$ by
$$
gM\mapsto
\bigl(g^+,g^-,\beta_{g^+}(e,g)\bigr).
$$
In these coordinates, left multiplication by $h\in G$
takes the form
$$
h(\zeta,\eta,v)
=
\bigl(h\zeta,h\eta,v+\beta_\zeta(h^{-1},e)\bigr),
$$
while right multiplication by $\exp(w)$, $w\in\mathfrak a$,
sends $(\zeta,\eta,v)$ to $(\zeta,\eta,v+w)$.

\subsection{Limit sets and growth indicators}\label{subsec.GI}

Let $\Gamma<G$ be a discrete subgroup.
We say that a sequence $x_n\in X$ converges to
$\xi\in\mathcal F$ if $x_n=k_na_no$, where
$k_n\in K$, $a_n\in A^+$, $k_nP\to\xi$, and
$\alpha(\log a_n)\to+\infty$ for every $\alpha\in\Pi$.

\begin{definition}
\label{def.limitset}
The limit set $\Lambda(\Gamma)\subset\mathcal F$
consists of the points that arise as limits of sequences
in $\Gamma o$ in the above sense.
\end{definition}

The limit cone of $\Gamma$ is the asymptotic cone of
$\mu(\Gamma)$:
$$
\mathcal L_\Gamma
:=
\left\{
\lim_{n\to\infty}t_n\mu(\gamma_n)\in\mathfrak a^+:
\gamma_n\in\Gamma,\ t_n\to0^+
\right\}.
$$
If $\Gamma$ is Zariski dense, $\mathcal L_\Gamma$ is the
smallest closed cone containing $\lambda(\Gamma)$;
it is convex and has nonempty interior
\cite{Benoist1997proprietes}.

The growth indicator
$\psi_\Gamma:\mathfrak a\to[-\infty,\infty)$
is defined by $\psi_\Gamma(0)=0$ and, for $u\neq0$,
$$
\psi_\Gamma(u)
:=
\|u\|\inf_{\mathcal C\ni u}\tau_{\mathcal C},
$$
where $\mathcal C$ ranges over open cones containing $u$,
and $\tau_{\mathcal C}$ is the abscissa of convergence of
the series
$\sum_{\gamma\in\Gamma,\;\mu(\gamma)\in\mathcal C}
e^{-s\|\mu(\gamma)\|}$. Since $\mu(\ga^{-1})=\i (\mu(\ga))$, we have $\psi_\Ga \circ \i=\psi_\Ga$.
When $\Gamma$ is Zariski dense, Quint
\cite{Quint2002divergence,Quint_indicator} proved that
$\psi_\Gamma$ is upper semicontinuous and concave,
is positive on $\operatorname{int}\mathcal L_\Gamma$,
and satisfies
$\mathcal L_\Gamma
=
\{u\in\mathfrak a:\psi_\Gamma(u)\ge0\}$.

\subsection{Weakly unipotent elements and Cartan growth}
A unipotent element $u\in G$ is called
\emph{regular} if it is contained in a unique minimal
real parabolic subgroup of $G$, and \emph{algebraically regular}
if it is contained in a unique Borel subgroup of $G_{\mathbb C}$.
Every algebraically regular unipotent element is regular:
its unique complex Borel subgroup is defined over $\mathbb R$,
so $G$ is quasi-split and its minimal real parabolic subgroups
complexify to Borel subgroups.
When $G$ is split, the two notions coincide;
see \cite{Steinberg_RegularElements}.

An $\mathfrak{sl}_2$-triple in $\mathfrak g$ is a triple
$(E,H,F)$, with $E\neq0$, satisfying
$[H,E]=2E$, $[H,F]=-2F$, and $[E,F]=H$.

\begin{lemma}[Cartan growth of weakly unipotent elements]
\label{lem.regularunipotent_cartan}
Let $g\in G$ be weakly unipotent.
Then $\log g_u$ belongs to an $\mathfrak{sl}_2$-triple
$(\log g_u,H,F)$ in $\mathfrak g$ whose semisimple element
$H$ is $G$-conjugate to a vector $v_g\in\mathfrak a^+\smallsetminus\{0\}$.
The triple can be chosen so that the associated
homomorphism $\Phi:\SL_2(\mathbb R)\to G$ satisfies
$$
\Phi\begin{pmatrix}1&1\\0&1\end{pmatrix}=g_u,
\qquad
\Phi\begin{pmatrix}e^{t/2}&0\\0&e^{-t/2}\end{pmatrix}=e^{\frac t2H},
\qquad
g_e\Phi(h)=\Phi(h)g_e
\quad(h\in\SL_2(\mathbb R)).
$$
Moreover,
$\mu(\Phi(a_t))=\frac{t}{2} v_g+O_\Phi(1)$ for $t\geq0$,
and
\begin{equation}
\label{latercartan0}
\mu(g^n)=\mu(g_u^n)+O_g(1)
=(\log|n|)v_g+O_g(1)
\qquad(|n|\to\infty).
\end{equation}
The vector $v_g$ depends only on $g_u$ and satisfies
$\i(v_g)=v_g.$
If $g_u$ is regular, then
$v_g\in\operatorname{int}\mathfrak a^+$.
\end{lemma}

\begin{proof}
Set $u:=g_u$ and $E:=\log u$.
Since $g_e$ is semisimple, its centralizer is reductive.
Moreover, $g_e$ commutes with $u$, so
$E\in\operatorname{Lie}Z_G(g_e)$.
Applying the real Jacobson--Morozov theorem in this
reductive Lie algebra
\cite[Chapter~VIII, Section~11, no.~2, Proposition~2
and its corollary]{Bourbaki_Lie78}
gives an $\mathfrak{sl}_2$-triple $(E,H,F)$.
Its associated homomorphism
$\Phi:\SL_2(\mathbb R)\to G$ has image in $Z_G(g_e)$
and satisfies
$\Phi\bigl(\begin{smallmatrix}1&1\\0&1\end{smallmatrix}\bigr)=u$.
Finite-dimensional $\mathfrak{sl}_2$ representation
theory, applied to the adjoint action on $\mathfrak g$,
shows that $\operatorname{ad}H$ is diagonalizable over
$\mathbb R$ with integer eigenvalues.
Thus $H$ is hyperbolic and is $G$-conjugate to a unique
vector $v_g\in\mathfrak a^+$. Since $E\neq0$ and $[H,E]=2E$, we have $H\neq0$,
and hence $v_g\neq0$.

Since
$a_t=\exp(\frac t2\operatorname{diag}(1,-1))$,
we have $\Phi(a_t)=\exp(\frac t2H)$.
Conjugating $H$ to $v_g$ and applying
\Cref{lem.cptcartan} gives
$$
\mu(\Phi(a_t))=\frac t2v_g+O_\Phi(1)
\qquad(t\geq0).
$$
For each integer $n\ne 0$, the Cartan decomposition
in $\SL_2(\mathbb R)$ gives
$$
\begin{pmatrix}1&n\\0&1\end{pmatrix}
=k_n a_{t_n}\ell_n,
\qquad
k_n,\ell_n\in\SO(2),
\qquad
t_n=2\log|n|+O(1).
$$
Since $\Phi(\SO(2))$ is compact,
\Cref{lem.cptcartan} yields
$$
\mu(u^n)
=\mu(\Phi(a_{t_n}))+O_\Phi(1)
=(\log|n|)v_g+O_\Phi(1).
$$
The powers of $g_e$ are relatively compact, and
$g^n=g_e^n u^n$.
Another application of \Cref{lem.cptcartan} therefore
gives \eqref{latercartan0}.
Consequently,
$$
v_g
=\lim_{n\to\infty}\frac{\mu(g^n)}{\log n}
=\lim_{n\to\infty}\frac{\mu(u^n)}{\log n},
$$
so $v_g$ depends only on $u$ and is independent of
the chosen triple.
Applying \eqref{latercartan0} to positive and negative
powers and using $\mu(g^{-n})=\i\mu(g^n)$ gives
$\i(v_g)=v_g$.

Finally, suppose that $u$ is regular.
Let $P_H$ be the real parabolic subgroup whose Lie
algebra is the sum of the nonnegative eigenspaces of
$\operatorname{ad}H$.
Since $[H,E]=2E$, the element $u=\exp E$ belongs to
the unipotent radical of $P_H$.
Every minimal real parabolic subgroup contained in
$P_H$ contains this radical and hence contains $u$.
If $P_H$ were not minimal, there would be more than
one such subgroup, contradicting the regularity of $u$.
Thus $P_H$ is minimal, which is equivalent to
$v_g\in\operatorname{int}\mathfrak a^+$.
\end{proof}

\subsection{Relatively Anosov representations and Morse embeddings}

Let $\Gamma<\PSL_2(\mathbb R)$ be a non-elementary,
finitely generated, Fuchsian group with
parabolic elements. 
For a representation $\rho:\Gamma\to G$, write
$\Gamma_\rho:=\rho(\Gamma)$,
$\Lambda_\rho:=\Lambda(\Gamma_\rho)$,
$\mathcal L_\rho:=\mathcal L_{\Gamma_\rho}$, and
$\psi_\rho:=\psi_{\Gamma_\rho}$.

\begin{definition}
A representation $\rho:\Gamma\to G$ is
\emph{relatively Borel Anosov} if (1)
$\alpha(\mu(\rho(\gamma)))\to+\infty$ for every
$\alpha\in\Pi$ as $\gamma$ leaves every finite subset
of $\Gamma$, (2) any two distinct points of $\Lambda_\rho$
are transverse, and (3) there exists a $\rho$-equivariant
homeomorphism $\Lambda(\Gamma)\to\Lambda_\rho$.
\end{definition}

This is the specialization to the present Fuchsian
setting of the definition in
\cite[Section~4.3]{CZZ_relative}.
The homeomorphism in (3) is unique and is called the
\emph{boundary map of $\rho$}. The condition (1) implies that $\rho$ is discrete
and faithful. 
Moreover, relative Anosovness implies that $\rho(\gamma)$
is weakly unipotent for every parabolic $\gamma\in \Ga$; see \cite[Theorems~1.1(3) and~B.5]{CZZ_cusped}.

For brevity, we refer to relatively Borel Anosov
representations simply as relatively Anosov representations
throughout the paper.

Set $Y:=\operatorname{hull}\Lambda(\Gamma)\subset\mathbb H^2$
and assume, after conjugating $\Gamma$ if necessary, that
$o=i\in Y$. We also write $o=K\in X$.
For $p=gK,q=hK\in X$, write
$\mu(p^{-1}q):=\mu(g^{-1}h)$.
This is well-defined by the $K$-bi-invariance of the Cartan
projection. For a map $f:Y\to X$, let $\mathcal L_f$ be
the asymptotic cone of
$\{\mu(f(x)^{-1}f(y)):x,y\in Y\}$.

\begin{definition}[Morse embedding]
\label{def.relatively_morse_cusp_adapted}
A map $f:Y\to X$ is a \emph{Morse embedding} if it is
a quasi-isometric embedding\footnote{
A map $f:(Z_1,d_1)\to(Z_2,d_2)$ is a
\emph{quasi-isometric embedding} if there exist
$C_1\geq1$ and $C_2\geq0$ such that
$C_1^{-1}d_1(z,z')-C_2
\leq d_2(f(z),f(z'))
\leq C_1d_1(z,z')+C_2
\;\; (z,z'\in Z_1)$.}  and
$\mathcal L_f\subset\operatorname{int}\mathfrak a^+\cup\{0\}$. We call $\mathcal L_f$ its \emph{Morse limit cone}.
\end{definition}
Equivalently, by the Kapovich--Leeb--Porti Morse lemma
\cite[Theorem~5.16 and Corollary~5.28]{KLP_2018},
the images of geodesic segments in $Y$ are
$\Pi$-Morse quasigeodesics with uniform constants;
see also \cite[Section~4]{DKO_AR}.

For $G=\SL_n(\mathbb R)$, the existence of an equivariant Morse embedding
follows from
\cite[Corollary~1.14]{ZZ_relatively2} and
\cite[Proposition~11.3]{ZZ_relatively}. We construct a continuous embedding for general $G$ and identify its Morse limit cone with $\L_\rho$.
\begin{proposition}[Equivariant Morse embedding]
\label{prop.cusp_adapted_morse0}
Let $\rho:\Gamma\to G$ be a Zariski dense, relatively
 Anosov representation. 
Then there exists a continuous $\rho$-equivariant Morse
embedding $f:Y\to X$ such that $f(o)=o$ and
$\mathcal L_f=\mathcal L_\rho$. Moreover, $f$ extends continuously to
$Y\cup\Lambda(\Gamma)$, and its restriction to
$\Lambda(\Gamma)$ is the boundary map of $\rho$.
\end{proposition}

\begin{proof}
Choose representatives $\xi_0,\ldots,\xi_\ell$ of
the $\Gamma$-orbits in $\Lambda_{\mathrm p}(\Gamma)$, and set
$\Gamma_{i}:=\operatorname{Stab}_\Gamma(\xi_i)$.
Choose a $\Gamma$-invariant family $\mathscr H$ of pairwise disjoint
closed horoballs based at $\Lambda_{\mathrm p}(\Ga)$,
with $\mathsf H_i$ based at $\xi_i$, all avoiding $o$.
Let $Y_{\rm th}:= Y-\cup_{\mathsf H\in \mathscr H}  \inte (\mathsf H)$; then $\Gamma\backslash Y_{\rm th}$ is compact.

\noindent\textbf{Step 1: Construction of $f$.}
Write $n_s:=\left(\begin{smallmatrix}1&s\\0&1\end{smallmatrix}\right)$.
For each $i$, choose a primitive generator $h_i\in\Gamma_i$
and $b_i\in\PSL_2(\mathbb R)$ such that
$b_i(\infty)=\xi_i$ and $b_i n_1b_i^{-1}=h_i$.
By relative Anosovness, $\rho(h_i)$
is weakly unipotent. Let $g_i:=\rho(h_i)$, $e_i:= (g_i)_e$ and $u_i:=(g_i)_u$.
By \Cref{lem.regularunipotent_cartan},
choose a homomorphism $\Phi_i:\SL_2(\mathbb R)\to G$
with $\Phi_i(n_1)=u_i$ whose image commutes with $e_i$,
and let $v_i:=v_{g_i}\in \mathfrak a^+\smallsetminus\{0\}$
be the associated vector. For every $\alpha\in\Pi$, relative  Anosovness
and \eqref{latercartan0} give
$$
\alpha(\mu(g_i^n))
=(\log n)\alpha(v_i)+O(1)\to+\infty.
$$
Consequently, $\alpha(v_i)>0$ for every $\alpha\in\Pi$,
so $v_i\in\operatorname{int}\mathfrak a^+$.
Moreover, since $g_i^n\in\Gamma_\rho$,
\eqref{latercartan0} gives $v_i\in\mathcal L_\rho$.

Since the compact subgroups $\overline{\langle e_i\rangle}$
and $\Phi_i(\SO(2))$ commute, we can choose $o_i\in X$
fixed by both subgroups, and choose $q_i\in G$
with $q_i o=o_i$.
Write 
$\mathsf H_i=b_i\{z:\operatorname{Im}z\geq e^{T_i}\}$ for some $T_i\in \br$ and
define
$$
f_i(b_in_sa_t i):=\Phi_i(n_sa_t)o_i
\qquad(s\in\mathbb R,\ t\geq T_i).
$$
Since $e_i$ fixes $o_i$ and commutes with
$\Phi_i(\SL_2(\mathbb R))$, it fixes the image of $f_i$
pointwise. Thus, for $z=b_in_sa_t i$,
$$
f_i(h_i z)
=\Phi_i(n_{s+1}a_t)o_i
=u_i f_i(z)
=e_i u_i f_i(z)
=\rho(h_i)f_i(z).
$$
Since $\Gamma_i=\langle h_i\rangle$, the map $f_i$ is
$\rho|_{\Gamma_i}$-equivariant.
Restricting to $\mathsf H_i\cap Y$, we extend to all
cusp regions in $Y$ by $f(\gamma z):=\rho(\gamma)f_i(z)$.
Precise invariance of the horoballs makes this
well-defined.
We extend $f$ continuously and $\rho$-equivariantly over
$Y_{\rm th}$, retaining its cusp-boundary values and
prescribing $f(\gamma o)=\rho(\gamma)o$.
To do so, lift a finite triangulation of
$\Gamma\backslash Y_{\rm th}$ compatible with the cusp
boundaries and having the image of $o$ as a vertex.
Choose images for representatives of the remaining
vertex orbits and extend equivariantly.
Extend over representatives of the edge orbits using
path connectedness of $X$, and over representatives
of the triangle orbits using simple connectedness.
The action is free, so these extensions propagate
equivariantly and agree on common faces.
The lifted triangulation and the family of cusp regions
are locally finite, so the glueing lemma gives a continuous
map on $Y$.

\noindent\textbf{Step 2: Quasi-isometry and root estimates.}
We first claim that, whenever $x,y\in Y$ lie in a
horoball $\gamma\mathsf{H}_i$, with $\gamma\in\Gamma$,
\be\label{cuspes}
\mu(f(x)^{-1}f(y))
=\frac12d_{\mathbb H^2}(x,y)v_i+O(1),
\ee 
with an error uniform in $x,y,\gamma$, and $i$.
By equivariance, it suffices to consider
$x,y\in\mathsf{H}_i\cap Y$.
Write $x=b_i g o$ and $y=b_i h o$, where
$g=n_sa_t, h=n_{s'}a_{t'} \in \SL_2(\mathbb R)$.
If $r=d_{\mathbb H^2}(x,y)$, then
$g^{-1}h\in\SO(2)a_r\SO(2)$.
By \Cref{lem.cptcartan,lem.regularunipotent_cartan},
$\mu(\Phi_i(g^{-1}h))=\frac r2v_i+O(1)$.
As $f_i(x)=\Phi_i(g)q_i o$ and
$f_i(y)=\Phi_i(h)q_i o$, we have
$\mu(f_i(x)^{-1}f_i(y))
=\mu(q_i^{-1}\Phi_i(g^{-1}h)q_i)
=\frac r2v_i+O(1)$.
Choose a compact fundamental set $D\subset Y_{\rm th}$ for the $\Ga$-action.
Suppose that $d_{\mathbb H^2}(x,y)\leq1$.
If $x,y$ lie in the same horoball in $\mathscr H$, the preceding estimate
bounds $d_X(f(x),f(y))$ uniformly.
Otherwise, $[x,y]$ meets $Y_{\rm th}$.
After translating by an element of $\Gamma$, both
endpoints lie in the closed $1$-neighborhood of $D$.
This neighborhood is compact in $Y$, so continuity and
equivariance again give a uniform bound on
$d_X(f(x),f(y))$.
Subdividing an arbitrary geodesic segment into pieces
of length at most $1$ now gives
\begin{equation}
\label{upper}
d_X(f(x),f(y))\leq A\,d_{\mathbb H^2}(x,y)+B
\qquad(x,y\in Y)
\end{equation}
for some $A,B>0$. 

For the lower bound, we will use the relative Anosov
hypothesis to prove 
\begin{equation}\label{eqn.root_lower_f}
\alpha(\mu(f(x)^{-1}f(y)))
\geq a\,d_{\mathbb H^2}(x,y)-b
\qquad(x,y\in Y,\ \alpha\in\Pi)
\end{equation}
for some $a,b>0$. Fix $\alpha\in\Pi$, and choose an irreducible proximal Tits
representation $\tau_\alpha:G\to\SL(V_\alpha)$
whose highest weight is a positive multiple of
the fundamental weight $\omega_\alpha$
\cite{Tits_representations}.
Equip $V_\alpha$ with a Euclidean inner product for
which $\tau_\alpha(K)$ acts orthogonally and
$\tau_\alpha(A)$ acts by positive self-adjoint operators.
Let $\|\cdot\|_\alpha$ denote the associated norm.
For $x\in Y$, choose $g_x\in G$ with $g_xo=f(x)$
and set
$\|v\|_{\alpha,x}:=\|\tau_\alpha(g_x^{-1})v\|_\alpha$
for $v\in V_\alpha$.
Since $\tau_\alpha(K)$ preserves $\|\cdot\|_\alpha$,
this definition is independent of the choice of $g_x$.
These norms form a continuous,
$\tau_\alpha\circ\rho$-equivariant family.
This family is canonical in the sense of
\cite[Section~3.1]{CZZ_cusped}.
Set
$|v|_{\alpha,i}:=\|\tau_\alpha(q_i^{-1})v\|_\alpha$.
Since $e_i$ and $\Phi_i(\SO(2))$ fix $q_i o$,
this norm is invariant under their images by $\tau_\alpha$.
On $\mathsf{H}_i\cap Y$, our construction gives
$$
\|v\|_{\alpha,x}
=
\bigl|\tau_\alpha(\Phi_i(n_sa_t)^{-1})v\bigr|_{\alpha,i},
\qquad x=b_in_sa_t i\in\mathsf{H}_i\cap Y.
$$
In the coordinates given by $b_i$, this is the
canonical family associated with the cusp homomorphism
$\tau_\alpha\circ\Phi_i$.
Thus evaluating $\|\cdot\|_{\alpha,x}$ at the
basepoints of geodesics gives a canonical family
on the geodesic-flow space.

Since $\rho(\Gamma)$ is Zariski dense,
$\tau_\alpha\circ\rho$ is irreducible, and the $\rho$-boundary map
induces a continuous, transverse, equivariant
line--hyperplane limit map.
By \cite[Corollary~6.3(1)]{CZZ_cusped},
$\tau_\alpha\circ\rho$ is relatively $P_1$-Anosov, where $P_1$ denotes
the stabilizer of a line in $V_\alpha$.
Let $E_\alpha(\xi)$ and $F_\alpha(\xi)$ denote
the line and hyperplane components of its limit map.
By \cite[Corollary~6.2]{CZZ_cusped}, there exist $C_\alpha\geq1$ and $c_\alpha>0$ such that,
for every unit-speed geodesic $\ell:\mathbb R\to Y$
with endpoints $\xi_\pm\in\Lambda(\Gamma)$,
$$
\frac{\|v\|_{\alpha,\ell(t)}}{\|w\|_{\alpha,\ell(t)}}
\leq C_\alpha e^{-c_\alpha(t-s)}
\frac{\|v\|_{\alpha,\ell(s)}}{\|w\|_{\alpha,\ell(s)}}
\qquad(s\leq t)
$$
for all nonzero $v\in E_\alpha(\xi_+)$ and
$w\in F_\alpha(\xi_-)$. Thus, when passing from
the norm at $\ell(t)$ to the norm at $\ell(s)$,
the stretching factor on $E_\alpha(\xi_+)$ exceeds
every stretching factor on $F_\alpha(\xi_-)$ by
at least $C_\alpha^{-1}e^{c_\alpha(t-s)}$.
Writing $\sigma_j$ for the $j$-th singular value, 
we obtain
$$
\alpha(\mu(f(\ell(s))^{-1}f(\ell(t))))= \log
\frac{\sigma_1\bigl(
\tau_\alpha(g_{\ell(s)}^{-1}g_{\ell(t)})\bigr)}
{\sigma_2\bigl(
\tau_\alpha(g_{\ell(s)}^{-1}g_{\ell(t)})\bigr)}
\geq c_\alpha(t-s)-\log C_\alpha.
$$

Any two points $x,y\in Y$ lie within a uniform distance
$D$ of points $\ell(s),\ell(t)$ on a common complete
geodesic with endpoints in $\Lambda(\Gamma)$; see
\cite[Proposition~3.5 and Lemma~B.3]{ZZ_relatively}.
Orienting $\ell$ so that $s\leq t$, we have
$t-s\geq d_{\mathbb H^2}(x,y)-2D$.
By \eqref{upper}, the corresponding image points are
also uniformly close. Applying \Cref{lem.cptcartan}
and taking uniform constants over the finite set $\Pi$
proves \eqref{eqn.root_lower_f}.
Since $\alpha(v)\leq\|\alpha\|_*\|v\|$, this implies
$d_X(f(x),f(y))\geq a'd_{\mathbb H^2}(x,y)-b'$
for some $a',b'>0$.
Together with \eqref{upper}, this proves that $f$
is a quasi-isometric embedding.

\noindent\textbf{Step 3: The Morse property and the limit cone.}
Combining \eqref{eqn.root_lower_f} with
\eqref{upper} gives
$$
\alpha(\mu(f(x)^{-1}f(y)))
\geq \frac{a}{A}\|\mu(f(x)^{-1}f(y))\|
-b-\frac{aB}{A}.
$$
This implies
$\mathcal L_f\subset\operatorname{int}\mathfrak a^+\cup\{0\}$, and hence $f$ is a Morse embedding.

Equivariance and $f(o)=o$ imply that $\Gamma_\rho (o)=f(\Gamma o)\subset f(Y)$,
 and hence
$\mathcal L_\rho\subset\mathcal L_f$.
For the reverse inclusion, let $\mathcal C$ be any closed
convex cone such that $\mathcal L_\rho\smallsetminus\{0\}
\subset\operatorname{int}\mathcal C$ and
$\mathcal C\subset\operatorname{int}\mathfrak a^+\cup\{0\}$.
Then
$\mu(\Gamma_\rho)\smallsetminus\mathcal C$ is bounded.
If $x,y$ lie in a horoball $\gamma\mathsf{H}_i$,
the cusp estimate and $v_i\in\mathcal L_\rho$ show
that $\mu(f(x)^{-1}f(y))$ lies in a uniform
neighborhood of $\mathcal C$.
If $x$ and $y$ do not lie in a common horoball,
then $[x,y]$ meets $Y_{\rm th}$.
Let $x'$ and $y'$ be its first and last intersection
points with $Y_{\rm th}$, in the orientation from
$x$ to $y$.
Since $f$ is Morse, coarse additivity
\cite[Corollary~4.12]{DKO_AR} gives
$$
\begin{aligned}
\mu(f(x)^{-1}f(y))
={}&\mu(f(x)^{-1}f(x'))
+\mu(f(x')^{-1}f(y'))+\mu(f(y')^{-1}f(y))+O(1).
\end{aligned}
$$
The first and third terms each lie within a uniform
distance of one of the rays $\mathbb R_{\geq0}v_i$,
all of which lie in $\mathcal C$. By cocompactness of $\Gamma$ on $Y_{\rm th}$,
\eqref{upper}, equivariance, and \Cref{lem.cptcartan},
the middle term lies within a uniform distance of
$\mu(\Gamma_\rho)$ and hence of $\mathcal C$.

Since $\mathcal C$ is convex, the sum lies in a
uniform neighborhood of $\mathcal C$. This implies $\mathcal L_f\subset\mathcal C$.
As $\mathcal C$ can be chosen arbitrarily close to $\L_\rho$,
we conclude that $\mathcal L_f=\mathcal L_\rho$.
Finally, the boundary extension theorem
\cite[Theorem~1.4]{KLP_2018} gives a continuous extension
of $f$ whose boundary restriction is a
$\rho$-equivariant homeomorphism
$\Lambda(\Gamma)\to\Lambda_\rho$.
By uniqueness, this is the boundary map of $\rho$.
\end{proof}

Conversely, an equivariant Morse embedding implies relative
Anosovness by \cite[Theorem~8.3]{KL_relativizing}.
Thus, for Zariski dense representations of $\Gamma$,
relative Anosovness is equivalent to the existence
of an equivariant Morse embedding.

\subsection{Geometric estimates}
For the remainder of this section, fix a Zariski dense,
relatively Anosov representation $\rho:\Gamma\to G$ and
a continuous $\rho$-equivariant Morse embedding $f:Y\to X$
supplied by \Cref{prop.cusp_adapted_morse0}, so that
$f(o)=o$ and $\mathcal L_f=\mathcal L_\rho$.
We also denote its boundary extension
$\Lambda(\Gamma)\to\Lambda_\rho$ by $f$. Let $\Lambda_{\mathrm p}(\Gamma_\rho):=f(\Lambda_{\mathrm p}(\Gamma))$.
The Morse property and equivariance give
$$
\mathcal L_\rho\subset\mathcal L_f
\subset\operatorname{int}\mathfrak a^+\cup\{0\}.
$$

We repeatedly use the following:
\begin{lemma}
\label{lem.scalar-Cartan-comparison}
For every $\varphi\in\mathfrak a^*$ positive on
$\mathcal L_\rho\smallsetminus\{0\}$, there exist $a\geq1$ and $b\geq0$
such that, for all $x,y\in Y$,
$$
a^{-1}d(x,y)-b
\le \varphi\bigl(\mu(f(x)^{-1}f(y))\bigr)
\le a\,d(x,y)+b.
$$
\end{lemma}
\begin{proof} Since $\mathcal L_f=\mathcal L_\rho$, positivity of
$\varphi$ and compactness of the unit directions in
$\mathcal L_f$ give   $c>1$ and $C\geq0$ such that
$$
c^{-1} \|\mu(f(x)^{-1}f(y))\|-C \leq \varphi\bigl(\mu(f(x)^{-1}f(y))\bigr)\le 
 c \|\mu(f(x)^{-1}f(y))\|
$$ for all $x,y\in Y$.
 Since
$\|\mu(f(x)^{-1}f(y))\|=d_X(f(x),f(y))$, the assertion
now follows from the quasi-isometry bounds for $f$.
\end{proof}

\begin{lemma}[Coarse additivity, {\cite[Corollary~4.12]{DKO_AR}}]
\label{lem.DKO_additive}
There exists $D_0>0$ such that, for all $x,z\in Y$ and
$y\in[x,z]$,
$$
\left\|
\mu(f(x)^{-1}f(z))
-\mu(f(x)^{-1}f(y))
-\mu(f(y)^{-1}f(z))
\right\|\leq D_0.
$$
\end{lemma}

We distinguish the hyperbolic shadows $O_r(x,y)$ from
shadows in the symmetric space $X=G/K$.
Following \cite[Section~5]{LO_invariant}, for $x,y\in X$
and $r>0$, define
$$
\mathcal O_r(x,y)
:=
\left\{
gkP\in\mathcal F:
k\in K,\quad
gk(\operatorname{int}A^+)o\cap B_X(y,r)\neq\emptyset
\right\},
$$
where $x=go$ and $B_X(y,r)$ is the open ball in $X$
centered at $y$. This definition is independent of the
choice of $g$ with $go=x$.
Thus $\mathcal O_r(x,y)$ consists of the flags whose
Weyl chambers based at $x$ meet $B_X(y,r)$.
The shadows are $G$-equivariant:
$h\mathcal O_r(x,y)=\mathcal O_r(hx,hy)$ for $h\in G$.

\begin{lemma}[Busemann--Cartan comparison]
\label{lem.buseandcartan}
There exists $\kappa>0$ such that, for every $r>0$,
$x,y\in X$, and $\xi\in\mathcal O_r(x,y)$,
\begin{equation}
\label{eqn.ambient_shadow_busemann}
\left\|
\beta_\xi(x,y)-\mu(x^{-1}y)
\right\|\leq\kappa r.
\end{equation}
Moreover, for every $r>0$, there exists $C=C(r)>0$
such that, for all $x,y\in Y$ and
$\zeta\in O_r(x,y)\cap\Lambda(\Gamma)$,
$$
\left\|
\beta_{f(\zeta)}(f(x),f(y))
-\mu(f(x)^{-1}f(y))
\right\|\leq C.
$$
\end{lemma}

\begin{proof}
The first assertion is \cite[Lemma~5.7]{LO_invariant}.
For the second, the Morse property implies that
$f([x,\zeta])$ lies in a uniform neighborhood of the
Weyl chamber based at $f(x)$ with endpoint $f(\zeta)$.
If $\zeta\in O_r(x,y)$, then $[x,\zeta]$ meets
$B(y,r)$. Since $f$ is a quasi-isometric embedding,
there is a uniform $c_0\geq1$ such that
$f(\zeta)\in
\mathcal O_{c_0(r+1)}(f(x),f(y))$.
Applying \eqref{eqn.ambient_shadow_busemann} proves
the claim.
\end{proof}

The two types of shadows are comparable after a
uniform change of radius.
\begin{lemma}[Comparison of shadows]
\label{lem.shadowscomparison}
There exists $c\geq1$ such that, for all $x,y\in Y$
and $r\geq c$,
$$
f\bigl(O_{c^{-1}r}(x,y)\cap\Lambda(\Gamma)\bigr)
\subset
\mathcal O_r(f(x),f(y))\cap\Lambda_\rho
\subset
f\bigl(O_{cr}(x,y)\cap\Lambda(\Gamma)\bigr).
$$
\end{lemma}

\begin{proof}
The Morse argument in the proof of
\Cref{lem.buseandcartan} gives the first inclusion for sufficiently large $c$;
compare \cite[Proposition~5.14]{LO_invariant}.
Suppose that
$f(\zeta)\in\mathcal O_r(f(x),f(y))$ for
$\zeta\in\Lambda(\Gamma)$.
Let $z$ be the nearest-point projection of $y$ to
$[x,\zeta]$. Hyperbolicity implies that $z$ lies
within a uniform distance of both $[x,y]$ and
$[y,\zeta]$. Set
$u:=\mu(f(x)^{-1}f(z))$ and $
v:=\mu(f(z)^{-1}f(y))$.
By \Cref{lem.DKO_additive,lem.cptcartan},
$$
\mu(f(x)^{-1}f(y))=u+v+O(1).
$$
By \Cref{lem.buseandcartan} and the cocycle identity,
$$
\begin{aligned}
\beta_{f(\zeta)}(f(x),f(y))
&=\beta_{f(\zeta)}(f(x),f(z))
  -\beta_{f(\zeta)}(f(y),f(z))=u-\i v+O(1).
\end{aligned}
$$
All error terms are uniform in $x,y,\zeta$.
It follows from \eqref{eqn.ambient_shadow_busemann}
that
$\|v+\i v\|\leq\kappa r+O(1)$.
Since $v,\i v\in\mathfrak a^+$, there is a
constant $c_1>0$ such that
$\|v\|\leq c_1\|v+\i v\|$ for all
$v\in\mathfrak a^+$.
The quasi-isometric embedding property of $f$
therefore gives
$d(y,z)\leq c_1(r+1)$, by increasing $c_1$ if necessary.
Increasing $c$ so that $c> 2c_1+1$, we obtain,
for $r\geq c$, $d(y,z)\leq c_1(r+1)<cr$.
Thus $\zeta\in O_{cr}(x,y)$, proving the second inclusion.
\end{proof}

The proof of \cite[Lemma~6.6]{LO_invariant} uses
the Morse property and hyperbolicity of the domain,
and therefore applies to $f$.

\begin{lemma}
\label{lem.gromovproductnearest}
There exists $C>0$ such that, for any distinct
$\zeta,\eta\in\Lambda(\Gamma)$, if $z$ is the
nearest-point projection of $o$ to $[\zeta,\eta]$,
then
$$
\left\|
\mathcal G(f(\zeta),f(\eta))
-\frac12\bigl(
\mu(f(z))+\i\mu(f(z))
\bigr)
\right\|\leq C.
$$
\end{lemma}

The following estimate relates the Gromov product
to the Cartan projection along peripheral orbits.

\begin{lemma}[{\cite[Proposition~4.3]{KO_globalshadow}}]
\label{lem.gromovproductalongparabolic}
Let $p\in\Lambda_{\mathrm p}(\Gamma_\rho)$, and set
$\mathsf P:=\operatorname{Stab}_{\Gamma_\rho}(p)$.
For every compact subset
$Q\subset\Lambda_\rho\smallsetminus\{p\}$,
there exists $C=C(p,Q)>0$ such that
$$
\sup_{\eta\in Q,\;g\in\mathsf P}
\left\|
\mathcal G(p,g\eta)-\frac12\mu(g)
\right\|\leq C.
$$
\end{lemma}

Finally, we record a local regularity estimate for
the $\mathfrak a$-valued Busemann cocycle.
Fix a Riemannian distance $d_{\mathrm{Riem}}$ on
$\mathcal F$.

\begin{lemma}[Local regularity]
\label{prop.holderlocal_Buse_prelim}
For every $r>0$, there exists $C=C(r)>0$ such that,
for all $g\in\Gamma_\rho$ and
$\xi,\eta\in\mathcal O_r(go,o)\cap\Lambda_\rho$,
$$
\left\|
\beta_\xi(e,g)-\beta_\eta(e,g)
\right\|
\leq C\,d_{\mathrm{Riem}}(\xi,\eta).
$$
\end{lemma}

\begin{proof}
Enlarging $r$ if necessary, we may apply
\Cref{lem.shadowscomparison}.  Therefore  the estimates in \cite[Propositions~4.7, 5.3, and 5.11]
{CK_Horospherical} imply the assertion outside a finite subset
of $\Gamma_\rho$.
For each fixed $g$, the map
$\xi\mapsto\beta_\xi(e,g)$ is smooth on the compact
manifold $\mathcal F$, and hence Lipschitz.
Enlarging $C$ includes the finitely many remaining
elements.
\end{proof}

The action of $\Gamma_\rho$ on $\Lambda_\rho$ is a
convergence group action, being conjugate to the action
of $\Gamma$ on $\Lambda(\Gamma)$.
Together with \Cref{lem.shadowscomparison}, this gives
the following convergence properties.

\begin{lemma}[Convergence of shadows]
\label{lem.ambient-shadow-convergence}
Every sequence of pairwise distinct $g_n\in\Gamma_\rho$
has a subsequence and points $\xi_+,\xi_-\in\Lambda_\rho$
with the following properties:
\begin{enumerate}
\item $g_n\eta\to\xi_+$ locally uniformly on
$\Lambda_\rho\smallsetminus\{\xi_-\}$.
\item For every $r>0$, if
$\eta_n\in\mathcal O_r(o,g_n o)\cap\Lambda_\rho$,
then $\eta_n\to\xi_+$.
\item If $r_n\to\infty$ and
$\eta_n\in\Lambda_\rho\smallsetminus\mathcal O_{r_n}(g_n o,o)$,
then $\eta_n\to\xi_+$.
\end{enumerate}
The same assertions hold for $g_n^{-1}$ with $\xi_+$ and
$\xi_-$ interchanged. In particular, for each $r>0$,
$$
\operatorname{diam}_{d_{\mathrm{Riem}}}
\bigl(\mathcal O_r(o,go)\cap\Lambda_\rho\bigr)\to 0\quad\text{($g\to \infty$ in $\Ga_\rho$).}
$$
\end{lemma}

\begin{proof}
Put $\widehat g_n:=\rho^{-1}(g_n)$. After passing to a
subsequence, $\widehat g_n o\to\zeta_+$ and
$\widehat g_n^{-1}o\to\zeta_-$ in $Y\cup\Lambda(\Gamma)$.
Set $\xi_\pm:=f(\zeta_\pm)$. The convergence action on
$\partial\mathbb H^2$ and equivariance of $f$ give (1).
Let $c$ be the constant in \Cref{lem.shadowscomparison}.
For (2), enlarge $r$ to at least $c$; the preimages of $\eta_n$
then lie in hyperbolic shadows of a fixed radius centered
at $\widehat g_n o$, and hence tend to $\zeta_+$.
For (3), the first inclusion of that lemma gives
$f^{-1}(\eta_n)\notin O_{r_n/c}(\widehat g_n o,o)$
for all large $n$. Thus the rays from $\widehat g_n o$
to $f^{-1}(\eta_n)$ avoid balls about $o$ whose radii
tend to infinity, so $f^{-1}(\eta_n)\to\zeta_+$.
Continuity of $f$ proves (2) and (3).
The same argument applies to the inverse sequence.
The diameter assertion follows from (2) by a subsequence argument.
\end{proof}

\subsection{Patterson--Sullivan and BMS measures}
\label{subsection:PSBMS}

For $\psi\in\mathfrak a^*$, a Borel probability measure
$\nu$ on $\mathcal F$ is called $(\Gamma_\rho,\psi)$-conformal
if, for every $\gamma\in\Gamma_\rho$,
$$
\frac{d\gamma_*\nu}{d\nu}(\xi)
=
e^{\psi(\beta_\xi(e,\gamma))}
\qquad\text{for $\nu$-almost every }\xi\in\mathcal F.
$$
If $\nu$ is supported on $\Lambda_\rho$, we call it
a $(\Gamma_\rho,\psi)$-Patterson--Sullivan measure. For $\psi\in\mathfrak a^*$, let
$\delta_\psi=\delta_\psi(\Gamma_\rho)$ be the abscissa
of convergence of
$\sum_{\gamma\in\Gamma_\rho}e^{-s\psi(\mu(\gamma))}$.

\begin{theorem}[{\cite{CZZ_relative}}]
\label{thm.ps atomless 0}
If $\delta_\psi=1$, there exists a unique
$(\Gamma_\rho,\psi)$-Patterson--Sullivan measure
$\nu_\psi$. Moreover, $\nu_\psi$ is
$\Gamma_\rho$-ergodic and atomless.
\end{theorem}

Recall that $\psi$ is tangent to $\psi_\rho$ if
$\psi\geq\psi_\rho$ and equality holds at some nonzero
vector. In the present setting, this is equivalent
to $\delta_\psi=1$; see \cite{KMO_tent,CZZ_relative}.
Moreover, $\psi_\rho$ is strictly concave between
non-collinear vectors in $\L_\rho$ \cite{Reyes-Wang,KOZ_Manhattan}.
The parametrization of tangent forms in
\cite[Corollary~1.3]{KOZ_Manhattan}, together with
the preceding theorem, gives the following.

\begin{corollary}
\label{thm.ps atomless}
For every $u\in\operatorname{int}\mathcal L_\rho$, there exists
a unique $\psi_u\in\mathfrak a^*$ tangent to
$\psi_\rho$ at $u$. Let $\nu_u$ be its Patterson--Sullivan
measure from \Cref{thm.ps atomless 0}.
The assignments $u\mapsto\psi_u\mapsto\nu_u$ give bijections
among unit vectors in $\operatorname{int}\mathcal L_\rho$,
tangent linear forms, and Patterson--Sullivan measures on $\La_\rho$.
\end{corollary}

Using the Hopf coordinates, we set
$$\Lambda_\rho^{(2)}
:=(\Lambda_\rho\times\Lambda_\rho)\cap\mathcal F^{(2)},\qquad 
\Omega_\rho
:=
\Gamma_\rho\backslash
\bigl(\Lambda_\rho^{(2)}\times\mathfrak a\bigr)
\subset X_\rho:=\Gamma_\rho\backslash G/M.$$
Fix a Lebesgue measure $\operatorname{Leb}_{\mathfrak a}$
on $\mathfrak a$.
For a tangent form $\psi$, the form $\psi\circ\i$
is also tangent, since $\psi_\rho$ is invariant under $\i$. The measure
\begin{equation}
\label{bmsdef}
d\widetilde{\mathsf m}_\psi(\zeta,\eta,v)
=
e^{2\psi(\mathcal G(\zeta,\eta))}
\,d\nu_\psi(\zeta)\,d\nu_{\psi\circ\i}(\eta)
\,d\operatorname{Leb}_{\mathfrak a}(v)
\end{equation}
is $\Gamma_\rho$-invariant and induces an $A$-invariant
Radon measure $\mathsf m_\psi$ on $X_\rho$, supported
on $\Omega_\rho$. We call $\mathsf m_\psi$ the
Bowen--Margulis--Sullivan measure associated with $\psi$.

\subsection{The base flow and reparametrization}
\label{ss:rep}
Let $\psi\in\mathfrak a^*$ be tangent to $\psi_\rho$ and
positive on $\mathcal L_f\smallsetminus\{0\}$.
Then $\delta_\psi=1$, and
$\psi$
is proper on $\mu(\Gamma_\rho)$ by \cite[Theorem~10.1]{CZZ_relative}. 
Consider the action of
$\Gamma_\rho$ on  $\widetilde \Omega_\psi
:= \Lambda_\rho^{(2)}\times\mathbb R $ given by
$$
g(\xi,\eta,s)
=
(g\xi,g\eta,s+\psi(\beta_\xi(g^{-1},e))).
$$
This action is properly discontinuous, and hence the quotient
$$
\Omega_\psi
:=
\Gamma_\rho\backslash \widetilde \Omega_\psi
$$
is locally compact and Hausdorff \cite{KOW_indicators}. Translation in the
$\mathbb R$-coordinate defines a flow on $\Omega_\psi$,
which we also denote by $a_t$.
Let $m_\psi$ be the translation-invariant
Radon measure on $\Omega_\psi$ induced by
$$
d\widetilde m_\psi(\xi,\eta,s)
=e^{2\psi(\mathcal G(\xi,\eta))}
\,d\nu_\psi(\xi)\,d\nu_{\psi\circ\i}(\eta)\,ds.
$$

\begin{theorem}[Finiteness and mixing {\cite{KO_entropy,BCZZ}}]
\label{koe}
The measure $m_\psi$ is finite and nonzero, and the
translation flow on $(\Omega_\psi,m_\psi)$ is strongly
mixing. In particular, it is ergodic.
\end{theorem}

Until the end of Part~\ref{part.thermo}, we replace
$m_\psi$ and $\mathsf m_\psi$ simultaneously by
\begin{equation}
\label{eqn:normalizedBMS}
m_\psi(\Omega_\psi)^{-1}m_\psi
\qquad\text{and}\qquad
m_\psi(\Omega_\psi)^{-1}\mathsf m_\psi,
\end{equation}
respectively, retaining the same notation.
The probability measure  $m_\psi$ is the unique invariant
probability measure of maximal entropy \cite{KO_entropy}. The projection
$(\xi,\eta,v)\mapsto(\xi,\eta,\psi(v))$
is $\Gamma_\rho$-equivariant and induces a trivial
$\ker\psi$-bundle $\pi_\psi:\Omega_\rho\to\Omega_\psi$;
see \cite[Section~10]{KOW_indicators}.
Let $\operatorname{Leb}_{\ker\psi}$ be the unique Lebesgue measure on
$\ker\psi$ such that, for one (and hence every) vector
$u\in\mathfrak a$ satisfying $\psi(u)=1$, the map
$\mathbb R\times\ker\psi\to\mathfrak a$, $(s,v)\mapsto su+v$,
sends $ds\,d\operatorname{Leb}_{\ker\psi}$ to $d\operatorname{Leb}_{\mathfrak a}$. 
Under a trivialization respecting translations by
$\ker\psi$, we have
\begin{equation}\label{mpsi}
\mathsf m_\psi
=m_\psi\otimes\operatorname{Leb}_{\ker\psi}.
\end{equation}
When $\psi=\psi_u$, we also write
$m_u^{\BMS}:=\mathsf m_{\psi_u}$, subject to the
temporary normalization in \eqref{eqn:normalizedBMS}.
We use the same notation for the lift to $\Gamma_\rho\backslash G$
obtained using Haar probability measure on $M$.

The following is the hyperbolic-plane analogue of the
reparametrization construction in
\cite[Section~6]{KO_entropy}.
We give a direct proof using the continuous Morse embedding $f$.

\begin{theorem}[Reparametrization]
\label{thm.reparametrization}
There exist a $\rho$-equivariant homeomorphism
$\Phi:\widetilde\Omega_\Gamma\to
\widetilde\Omega_\psi$
and a continuous cocycle
$\mathtt t:\widetilde\Omega_\Gamma\times\mathbb R\to\mathbb R$ such that, for
every $v\in\widetilde\Omega_\Gamma$ and $s\in\mathbb R$,
$$
(\Phi(v)^+,\Phi(v)^-)=(f(v^+),f(v^-)),
\qquad
\Phi(va_s)=\Phi(v)a_{ {\mathtt t} (v,s)}.
$$
For every $v \in \widetilde\Omega_\Gamma $, the map $s\mapsto {\mathtt t}(v,s)$ is an increasing
homeomorphism of $\mathbb R$ fixing the origin. 
Moreover, there exists $c>1$ such that
$$
c^{-1}|s|-c\leq | {\mathtt t} (v,s)|\leq c|s|+c
\qquad(v\in\widetilde\Omega_\Gamma,\ s\in\mathbb R).
$$
\end{theorem}
\begin{proof}
For $v\in\widetilde\Omega_\Gamma$, let $\pi(v)\in Y$ be its
basepoint and set
$q(v):=\psi(\beta_{f(v^+)}(o,f(\pi(v))))$.
By \Cref{lem.buseandcartan}, the quasi-isometry property
of $f$, and positivity of $\psi$ on
$\mathcal L_f\smallsetminus\{0\}$, there exist
$a,A>0$ and $B\geq0$ such that
$$
as-B\leq q(va_s)-q(v)\leq As+B
\qquad(s\geq0),
$$
uniformly in $v$.
Following the averaging argument in
\cite[proof of Lemma~4.3]{BCLS_gafa}, choose $T>B/a$
and set $Q(v):=T^{-1}\int_0^Tq(va_r)\,dr$. Then
$$
\frac{d}{ds}Q(va_s)
=\frac{q(va_{s+T})-q(va_s)}{T}
\in[a-B/T,A+B/T].
$$
Define
$\Phi(v):=(f(v^+),f(v^-),Q(v))$ and
$ {\mathtt t} (v,s):=Q(va_s)-Q(v)$. For $\gamma\in\Gamma$, the Busemann cocycle identity gives
$q(\gamma v)-q(v)
=\psi(\beta_{f(v^+)}(\rho(\gamma)^{-1},e))$.
This difference is constant along the flow orbit of $v$,
so the same identity holds with $Q$ in place of $q$.
Hence $\Phi$ is $\rho$-equivariant.
The flow relation and cocycle identity follow from the definitions.
Since $a-B/T>0$, the derivative bounds show that $Q$
maps each such orbit increasingly and bijectively onto
$\mathbb R$. Hence $\Phi$ is bijective.
Continuity of $Q$ and the uniform positive lower bound
on its derivative along flow orbits give continuity
of the inverse in Hopf coordinates.
The same derivative bounds give the asserted properties
of $\mathtt t$.
\end{proof}

The construction of \Cref{thm.reparametrization} uses only positivity
of the linear form on $\mathcal L_f\smallsetminus\{0\}$, not its tangency
to the growth indicator. We will also use it for such forms.

\section{General local mixing theorems and examples}\label{sec.general-hypotheses}
We prove the local mixing theorems of the introduction
under the following more general hypotheses.

\begin{theorem}[Local mixing for relatively Anosov representations]
\label{thm:local-mixing-morse}
Let $\rho:\Gamma\to G$ be a Zariski dense relatively  Anosov representation
satisfying the
peripheral root-balance condition defined below. Then the conclusions of \Cref{bmslocal_intro,m1} hold
for every $u\in\operatorname{int}\mathcal L_\rho$.
The orbital counting conclusion of
\Cref{thm.counting_intro} also holds.
\end{theorem}

\begin{definition}[Peripheral root-balance condition]
\label{def:peripheral-root-balance}
A relatively Anosov representation $\rho:\Ga \to G$ satisfies the
\emph{peripheral root-balance condition} if, for every
parabolic $\gamma\in\Gamma$,
\begin{equation}
\label{eqn:peripheral-root-balance}
\liminf_{n\to\infty}
\frac{\min_{\alpha\in\Pi}\alpha(\mu(\rho(\gamma)^n))}
     {\max_{\alpha\in\Pi}\alpha(\mu(\rho(\gamma)^n))}
>\frac12.
\end{equation} \end{definition}
By relative Anosovness, the denominator is eventually positive.

We  describe three classes to which the theorem applies. We first establish a Cartan-growth estimate. Let $v_{\rm pr}\in\operatorname{int}\mathfrak a^+$ be
the unique vector satisfying
\begin{equation}
\label{principal}
\alpha(v_{\rm pr})=2
\qquad(\alpha\in\Pi).
\end{equation}
We call $\mathbb R_{>0}v_{\rm pr}$ the principal direction.
Since $\i$ permutes the simple
roots, $\i(v_{\rm pr})=v_{\rm pr}$.
\begin{lemma}[Principal Cartan direction]
\label{latercartan00}
Let $g\in G$ be weakly unipotent.
If $g_u$ is algebraically regular, then  $v_g=v_{\rm pr}$.
In particular,
$$
\mu(g^n)= (\log|n|)v_{\rm pr}+O_g(1)
\qquad(|n|\to\infty).
$$
\end{lemma}

\begin{proof} Let $u:=g_u$. Since $u$ is regular, \Cref{lem.regularunipotent_cartan}
provides an $\mathfrak{sl}_2$-triple $(E,H,F)$ with
$E=\log u$ and, after conjugation,
$H=v_g\in\operatorname{int}\mathfrak a^+$.
The identity $[H,E]=2E$ then implies that
$E\in\mathfrak n$.
Since $u$ is algebraically regular, $G$ is quasi-split, so $P_{\mathbb C}$ is a Borel
subgroup of $G_{\mathbb C}$.
Choose a maximal torus of $P_{\mathbb C}$ containing
$A_{\mathbb C}$.
Algebraic regularity implies that $E$ has a nonzero
component in every absolute simple-root space of the
nilpotent radical; see \cite{Steinberg_RegularElements}.
Projecting $[H,E]=2E$ to these root spaces gives
$\widetilde\alpha(H)=2$ for every absolute simple root
$\widetilde\alpha$.
Since every restricted simple root is the restriction
of such an absolute simple root, $\alpha(H)=2$ for
every $\alpha\in\Pi$.
Thus $H=v_{\rm pr}$.\end{proof}

\subsection{Positive representations}
In this subsection, assume that $G$ is split over $\mathbb R$.
We recall the notion of positivity on $\mathcal F=G/P$.

For each $\alpha\in\Pi$, choose a nonzero vector
$X_\alpha\in\mathfrak g_\alpha$ and set
$x_\alpha(t):=\exp(tX_\alpha)$ for $t\in\mathbb R$.
These choices determine the positivity structure used below.
Let $s_\alpha$ denote the corresponding simple reflection.
For a reduced expression
$w_0=s_{\alpha_{i_1}}\cdots s_{\alpha_{i_k}}$, set
$$
N_{>0}:=
\left\{
x_{\alpha_{i_1}}(t_1)\cdots x_{\alpha_{i_k}}(t_k):
t_1,\ldots,t_k>0
\right\}.
$$
This is Lusztig's totally positive semigroup; it is
independent of the reduced expression
\cite{Lusztig_TotalPositivity}.

A quadruple $(F_1,F_2,F_3,F_4)$ of flags is called
\emph{positive} if there exist $h\in G$ and $u,v\in N_{>0}$
such that
$$
(hF_1,hF_2,hF_3,hF_4)
=
(P, w_0P, u w_0P,uv w_0P).
$$
This agrees with the usual definition of a positive
quadruple; see
\cite[Definition~2.13 and Proposition~2.17]{GRW_PositiveParam}.

\begin{definition}
\label{def.positiveframing}
A representation $\rho:\Gamma\to G$ is called
\emph{positive} if there exists a $\rho$-equivariant map
$\xi_p:\Lambda_{\mathrm p}(\Gamma)\to\mathcal F$
sending every cyclically ordered quadruple of pairwise
distinct points to a positive quadruple of flags.
\end{definition}

\begin{remark}\label{Ex_postive}
If the compact core of $\Gamma\backslash\mathbb H^2$ has
genus $g$ and $b\ge 1$ boundary components,
choose the standard
presentation
$\Gamma=
\langle
a_1,b_1,\ldots,a_g,b_g,c_1,\ldots,c_b:
\prod_{i=1}^g[a_i,b_i]\,c_1\cdots c_b=e \rangle$,
where $c_1,\ldots,c_s$ are cusp generators and
$c_{s+1},\ldots,c_b$ represent the funnel boundaries,
with $1\leq s\leq b$. 
Choose a principal homomorphism
$\iota:\mathrm{SL}_2(\mathbb R)\to G$, meaning an algebraic
homomorphism sending nontrivial unipotent elements to
regular unipotent elements of $G$. For $G=\mathrm{SL}_n(\mathbb R)$, a homomorphism is
principal precisely when it is irreducible.
Composing $\iota$ with a lift of the inclusion
$\Gamma<\mathrm{PSL}_2(\mathbb R)$ to $\mathrm{SL}_2(\mathbb R)$
gives a \emph{principal Fuchsian representation}
$\sigma:\Gamma\to G$.
 The connected component containing $\sigma$
in  $$
\operatorname{Hom}_{\partial,\sigma}(\Gamma,G)
:=
\left\{
\rho\in\operatorname{Hom}(\Gamma,G):
\rho(c_j)\text{ is conjugate to }\sigma(c_j)
\text{ for }1\leq j\leq b
\right\}$$
consists entirely of type-preserving positive representations.
This follows from \cite[Theorem~G, Proposition~4.2,
and Corollary~6.4]{FTWZ_positive}. When $\Gamma$ is a lattice,
$s=b$, so only the cusp conjugacy classes are fixed.
\end{remark}

\begin{theorem}
\label{thm.positive_boundary}
Every type-preserving positive representation
$\rho:\Gamma\to G$ is relatively Anosov.  Moreover, for every parabolic $\gamma\in\Gamma$, the unipotent
part $\rho(\gamma)_u$ is algebraically regular.
In particular, $\rho$ satisfies the peripheral
root-balance condition.
\end{theorem}
\begin{proof}
Relative  Anosovness follows from \cite[Theorem~6.17]{FTWZ_positive};
see also \cite[Theorems~1.4 and~9.2]{CZZ_cusped}
for $G=\SL_d(\mathbb R)$.

Every simple reflection occurs in a reduced expression
for $w_0$, so every simple-root coordinate of an element
of $N_{>0}$ is a nonempty sum of positive parameters. Thus every element of $N_{>0}$ has a nonzero component
in every simple root group of $N/[N,N]$ and is
algebraically regular by \cite[Lemma~3.2 and Theorem~3.3]
{Steinberg_RegularElements}.

Let $\gamma\in\Gamma$ be parabolic with fixed point
$p\in\Lambda_{\mathrm p}(\Gamma)$, and choose
$x\in\Lambda_{\mathrm p}(\Gamma)\smallsetminus\{p\}$.
Algebraic regularity of the unipotent part is preserved
under inversion, so
we may replace $\gamma$ by $\gamma^{-1}$ and assume
that $(p,\gamma^{-1}x,x,\gamma x)$ is cyclically ordered.
Set $g:=\rho(\gamma)$, $F_1:=\xi_p(p)$, and
$F_2:=\xi_p(x)$.
Positivity and equivariance imply that
$(F_1,g^{-1}F_2,F_2,gF_2)$ is a positive quadruple.

By the definition of a positive quadruple,
after conjugation we may assume that
$F_1=P$, $F_2=w_0P$, $g^{-1}w_0P=v^{-1}w_0P$ and
$gw_0P=uw_0P$
for some $u,v\in N_{>0}$.
Since $gP=P$ and $u^{-1} g (w_0P)=w_0P$, we obtain
$u^{-1}g\in P\cap\check P=MA$, so $g=u\ell$
for some $\ell\in MA$.
Since $g$ is weakly unipotent, $\ell$ has trivial
hyperbolic part.
As $G$ is split, every element of $MA$ is semisimple,
and its elliptic elements belong to $M$.
Thus $\ell=m\in M$.

Since $m$ fixes $w_0P$, we deduce
$(m^{-1}um)^{-1}w_0P=v^{-1}w_0P$.
The action of $N$ on $N w_0P$ is free,
so $m^{-1}um=v$.
In $N/[N,N]$, every simple-root coordinate of both
$u$ and $v$ is positive.
Conjugation by $m\in M$ acts on each one-dimensional
simple root space by a sign; hence all these signs
are positive.
Thus $\operatorname{Ad}m$ is the identity on every
simple root space.
Since $m\in K$, it is also the identity on the opposite
simple root spaces.
These root spaces generate $\mathfrak g$, so
$m$ belongs to the center $Z(G)$.
By uniqueness of the real Jordan decomposition,
$g_u=u$ and $g_e=m$. Since $u\in N_{>0}$, the element $g_u$ is algebraically
regular. By \Cref{latercartan00}, the peripheral
root-balance ratio therefore converges to $1$.
\end{proof}

\subsection{Split simple groups of low rank}
For split simple groups of rank at most three,
relative Borel Anosovness alone implies peripheral
root balance.

\begin{proposition}
\label{prop:automatic-root-balance}
Suppose that $G$ is split and simple with
$\operatorname{rank}_{\br}G\leq3$.
Let $\rho:\Gamma\to G$ be relatively  Anosov.
Then, for every parabolic $\gamma\in\Gamma$, there
exists $c_\gamma\in\{1,\frac12\}$ such that
$$
\mu(\rho(\gamma)^n)
=c_\gamma(\log|n|)v_{\rm pr}+O_\gamma(1)
\qquad(|n|\to\infty).
$$
The value $c_\gamma=\frac12$ can occur only when $G$
has type $A_2$.
In particular, $\rho$ satisfies the
peripheral root-balance condition.
\end{proposition}

\begin{proof} Fix a parabolic $\gamma\in\Gamma$ and set
$g:=\rho(\gamma)$.
By relative Borel Anosovness, $g$ is weakly
unipotent.
Let $v_g\in\mathfrak a^+$ be a vector conjugate to the
semisimple element of an $\mathfrak{sl}_2$-triple
containing $\log g_u$. Conjugate the triple so that its semisimple element
is $H=v_g$. By \Cref{lem.regularunipotent_cartan},
$\mu(g^n)=(\log|n|)v_g+O_g(1)$. Since
$\alpha(\mu(g^n))\to+\infty$ for every $\alpha\in\Pi$, we have
$v_g\in\operatorname{int}\mathfrak a^+$.
We determine the possibilities for $H$. For the classical types, consider the standard module
of the complexified Lie algebra.
A Jordan block of size $d$ of $\log g_u$ contributes
the eigenvalues $d-1,d-3,\ldots,1-d$ of $H$.

In type $A_r$, with $r\leq3$, the condition
$H\in\operatorname{int}\mathfrak a^+$ requires these
eigenvalues to be distinct.
Among partitions of $r+1\leq4$, this leaves only the
single-block partition, together with $(2,1)$ when
$r=2$.
A single block gives $\alpha(H)=2$ for every
$\alpha\in\Pi$.
For $(2,1)$, the eigenvalues are $1,0,-1$, so both
simple roots take the value $1$.
Thus $H=v_{\rm pr}$ or, in the latter case,
$H=\frac12v_{\rm pr}$.

In types $B_r$ and $C_r$, with $r=2$ or $3$, interiority
forces the eigenvalues of $H$ on the standard module
to be distinct.
In type $B_r$, blocks of even size are excluded because
they occur with even multiplicity, while two odd blocks
would both contribute the eigenvalue $0$.
In type $C_r$, odd blocks are excluded because interiority
also rules out the eigenvalue $0$, while two even blocks
would both contribute $1$ and $-1$.
Thus there is a single Jordan block, of size $2r+1$ or
$2r$, respectively, and $H=v_{\rm pr}$.

The remaining type is $G_2$.
Its only nilpotent weighted Dynkin diagram with all
labels positive is $(2,2)$; see
\cite[Table~2]{deGraaf_Nilpotent}.
Since these labels are the values $\alpha(H)$,
we obtain $H=v_{\rm pr}$.
\end{proof}

\begin{remark}
This includes the cusped Barbot examples into
$\SL_3(\mathbb R)$ described by Lee--Zhang \cite{LeeZhang_Barbot}.
Their peripheral unipotent parts have Jordan type $(2,1)$,
so $c_\gamma=\frac12$.
\end{remark}

\subsection{Self-joinings of rank-one representations}
For representations $\rho=(\rho_1,\ldots,\rho_k)$ into products of real-rank-one groups
with discrete projections, type preservation is equivalent
to type preservation of each projection. Indeed, if $\rho$ is type-preserving, each $\rho_i$ preserves
loxodromic elements and hence is faithful, since every
nontrivial normal subgroup of $\Gamma$ contains a loxodromic
element. For parabolic $\gamma$, the element $\rho_i(\gamma)$
has trivial hyperbolic part; faithfulness and discreteness
exclude ellipticity, so it is parabolic.
The converse follows from the componentwise Jordan decomposition.

\begin{proposition}
\label{prop.cusp_adapted_morse}\label{QI}
Let $G=\prod_{i=1}^kG_i$, where each $G_i$ has real rank one,
and let $\rho:\Gamma\to G$ be a Zariski dense type-preserving representation whose projections
$\rho_i:\Gamma\to G_i$ have geometrically finite images.
For every parabolic $\gamma\in\Gamma$, conjugate the
unipotent parts $u_i:=\rho_i(\gamma)_u$ into
$N_i:=N\cap G_i$. Suppose that $u_i\in[N_i,N_i]$
either for every $i$ or for no $i$.
Then $\rho$ is relatively  Anosov and satisfies
the peripheral root-balance condition. More precisely,
$$
\mu(\rho(\gamma)^n)
=c_\gamma(\log|n|)v_{\rm pr}+O_\gamma(1)
\qquad(|n|\to\infty),
$$
where $c_\gamma=\frac12$ if $u_i\in[N_i,N_i]$ for every
$i$, and $c_\gamma=1$ otherwise.
\end{proposition}

\begin{proof} 
Since $\rho_i(\Ga)$ is geometrically finite,
the generalization of Tukia's isomorphism theorem in
\cite[Theorems~1.4(i) and~1.5]{DSU_Tukia} therefore
gives a $\rho_i$-equivariant boundary homeomorphism
$\xi_i:\Lambda(\Gamma)\to\Lambda(\rho_i(\Gamma))$. The product map $(\xi_1,\ldots,\xi_k)$ satisfies the
boundary-map criterion of \cite[Corollary~B.8]{CZZ_cusped}, and
hence $\rho$ is relatively Anosov.

Fix a parabolic $\gamma\in\Gamma$ and set $g:=\rho(\gamma)$.
By \Cref{lem.regularunipotent_cartan}, choose an
$\mathfrak{sl}_2$-triple $(E,H,F)$ with $E=\log g_u$.
Since every simple root diverges along $\mu(g^n)$,
the Cartan asymptotic in that lemma gives
$v_g\in\operatorname{int}\mathfrak a^+$.
Conjugate the triple so that $H=v_g$.
The identity $[H,E]=2E$ then implies that $E\in\mathfrak n$.
Write $E=(E_i)$ and $H=(H_i)$.

Each nontrivial $u_i$ lies in a unique minimal real
parabolic subgroup of $G_i$.
Consequently, any two conjugations placing $u_i$ in
$N_i$ differ by an element of $P_i:=P\cap G_i$, which
normalizes $[N_i,N_i]$.
The assumed membership condition therefore still holds
after the preceding conjugation of the triple.

Let $\alpha_i$ be the simple restricted root of $G_i$.
For $\mathfrak n_i:=\operatorname{Lie}(N_i)$, we have
$[\mathfrak n_i,\mathfrak n_i]=\mathfrak g_{i,2\alpha_i}$
and $\mathfrak n_i/[\mathfrak n_i,\mathfrak n_i]
\simeq\mathfrak g_{i,\alpha_i}$, where
$\mathfrak g_{i,2\alpha_i}:=\{0\}$ if $2\alpha_i$ is not a root.
If $u_i\notin[N_i,N_i]$ for every $i$, the image
$\overline E_i$ of $E_i$ in this quotient is nonzero.
Projecting $[H_i,E_i]=2E_i$ to the quotient gives
$\alpha_i(H_i)\overline E_i=2\overline E_i$.
Thus $\alpha_i(H_i)=2$ for every $i$, so $H=v_{\rm pr}$.

If $u_i\in[N_i,N_i]$ for every $i$, then
$E_i\in\mathfrak g_{i,2\alpha_i}\smallsetminus\{0\}$.
Hence $2\alpha_i(H_i)E_i=[H_i,E_i]=2E_i$,
so $\alpha_i(H_i)=1$ for every $i$ and
$H=\frac12v_{\rm pr}$.
The claimed asymptotic follows from
\Cref{lem.regularunipotent_cartan}.
In either case, all simple roots take the same positive
value on $H$, so the root-balance ratio converges to $1$.
\end{proof}

For $G_i=\SO(n_i,1)^\circ$, with $n_i\geq2$, the groups
$N_i$ are abelian. The peripheral condition is then
automatic, and $c_\gamma=1$ for every parabolic $\gamma$.

\part{Cusp geometry and Markov coding} \label{part.coding}

For the rest of the paper, let $\rho:\Gamma\to G$ be a Zariski dense relatively  Anosov representation
satisfying the
peripheral root-balance condition of
\Cref{def:peripheral-root-balance}.

We first remove compact normal factors.
Let $G_c$ be the maximal connected compact normal subgroup
of $G$, and let $q:G\to G/G_c$ be the quotient map.
Since $G_c\subset M=K\cap P$, the spaces $G/P$, $G/K$,
and $G/M$ identify with the corresponding spaces for
$G/G_c$. The Cartan and Jordan projections, restricted
roots, and growth indicator are unchanged.
Thus $q\circ\rho$ is Zariski dense and relatively Anosov,
and satisfies the same peripheral root-balance condition.
It therefore suffices to prove the results for $G/G_c$.
For the remainder of the proof, we work in this quotient
and retain the notation $G$ and $\rho$; in particular,
$G$ has no compact factors.

Fix a continuous $\rho$-equivariant Morse embedding
$f:Y\to X$ constructed in \Cref{prop.cusp_adapted_morse0},
so that $f(o)=o$ and $\mathcal L_f=\mathcal L_\rho$.
We also denote its boundary extension
$\Lambda(\Gamma)\to\Lambda_\rho$ by $f$.
Set $\Lambda_{\mathrm c}(\Gamma_\rho):=f(\Lambda_{\mathrm c}(\Gamma))$ and $\Lambda_{\mathrm p}(\Gamma_\rho):=f(\Lambda_{\mathrm p}(\Gamma))$. Choose representatives $\xi_0,\ldots,\xi_\ell$ of the
$\Gamma$-orbits in $\Lambda_{\mathrm p}(\Gamma)$.  Let $\mathscr H$, $\mathsf H_i$, $0\le i\le \ell$ and $Y_{\rm th}$ be as in the proof of
\Cref{prop.cusp_adapted_morse0}.
Let
$$
p_i:=f(\xi_i),
\quad
\mathsf P_i:=\rho(\Gamma_{\xi_i}), \quad \Gamma_{\xi_i} :=\operatorname{Stab}_\Gamma(\xi_i).$$ After conjugating
$\Gamma$ in $\PSL_2(\mathbb R)$ and $\Gamma_\rho$ in $G$, we assume
$$
\xi_0=\infty,
\quad
f(\infty)=p_0 =P, \quad 0\in\Lambda(\Gamma),\quad  f(0) = w_0 P .$$

Choose compact intervals
$\widehat\Delta_{\xi_i}\subset\partial\mathbb H^2$
($0\le i\le\ell$), each a fundamental set for the action of
$\Gamma_{\xi_i}$ on
$\partial\mathbb H^2\smallsetminus\{\xi_i\}$,
with endpoints in $\Lambda_{\mathrm p}(\Gamma)$, such that
\begin{equation}
\label{eqn.choiceoffunddomain}
\{\xi_i:0\le i\le\ell\}\cap
\bigcup_{i=0}^\ell\widehat\Delta_{\xi_i}=\emptyset,
\qquad
\widehat\Delta_{\xi_i}\cap\widehat\Delta_{\xi_j}
=\emptyset\quad(i\ne j).
\end{equation}
We consider the following compact fundamental domain for the
$\mathsf P_i$-action on
$\Lambda_\rho\smallsetminus\{p_i\}$:
$$
\Delta_i:=f\bigl(\widehat\Delta_{\xi_i}
\cap\Lambda(\Gamma)\bigr)
\subset\Lambda_\rho\smallsetminus\{p_i\},
$$
 Set
$\partial\Delta_i:=f(\partial\widehat\Delta_{\xi_i})$ and
$\operatorname{int}\Delta_i:=
\Delta_i\smallsetminus\partial\Delta_i$. Set
$$
\Delta:=\bigcup_{i=0}^\ell\Delta_i
$$

\section{The visual premetric and its blow-up}
\label{sec.visual}

Throughout this section and \Cref{sec.height}, fix a linear form
$\chi\in\mathfrak a^*$ satisfying
\begin{equation}
\label{eqn:chi-positive}
\chi\circ\i=\chi,
\qquad
\chi>0\quad\text{on }\mathcal L_\rho\smallsetminus\{0\}.
\end{equation}

In this section we introduce the visual premetric $d_\chi$ on
$\Lambda_\rho$, compare it with the spherical metric on
$\partial\mathbb H^2$, and study its blow-up at the parabolic limit point
$p_0$.

\begin{Def}
For $(\xi,\eta)\in\mathcal F^{(2)}$, set
$$
d_\chi(\xi,\eta):=e^{-\chi(\mathcal G(\xi,\eta))}.
$$
\end{Def}

Since $\mathcal G(\eta,\xi)=\i\mathcal G(\xi,\eta)$, the identity $\chi\circ\i=\chi$ makes $d_\chi$ symmetric. The cocycle identity gives, for every
$(\xi,\eta)\in\mathcal F^{(2)}$ and $g\in G$,
\begin{equation}
\label{eqn.metrics_sph}
d_\chi(g\xi,g\eta)
=
d_\chi(\xi,\eta) e^{-\frac12\chi\bigl(\beta_{g\xi}(e,g)+\beta_{g\eta}(e,g)\bigr)}.
\end{equation}

\subsection{Restriction to the limit set}
By \Cref{lem.gromovproductnearest,lem.scalar-Cartan-comparison},
$d_\chi$ extends continuously to
$\Lambda_\rho\times\Lambda_\rho$ by setting
$d_\chi(\xi,\xi)=0$. Thus $d_\chi$ is a continuous premetric\footnote{
We call a symmetric nonnegative function on $S\times S$
a \emph{premetric} if it vanishes precisely on the diagonal.} on $\Lambda_\rho$.
By compactness, it induces the usual topology on
$\Lambda_\rho$.

\begin{lemma}
\label{lem.dpsidiameterfinite}
The diameter of $(\Lambda_\rho,d_\chi)$ is finite.
\end{lemma}

\begin{proof}

Let $\xi,\eta\in\Lambda(\Gamma)$ be distinct, and let $z$ be a nearest
point projection of $o$ to the geodesic $[\xi,\eta]\subset\mathbb H^2$.
By \Cref{lem.gromovproductnearest},
$\chi(\mathcal G(f(\xi),f(\eta)))\approx \chi(\mu(f(z)))$.
Since $\L_f=\L_\rho$, \eqref{eqn:chi-positive} and the definition of $\mathcal L_f$ imply that
$\chi(\mu(f(z)))$ is bounded below uniformly in $z\in Y$.
This proves the assertion.
\end{proof}

\subsection{Shadows and $d_\chi$-balls}

For $r>0$ and $\xi\in\Lambda_\rho$, set
$$
B_\chi(\xi,r):=\{\eta\in\Lambda_\rho:d_\chi(\xi,\eta)<r\}.
$$
For $\xi\in\Lambda_\rho$, let
$\widehat z_\xi:[0,\infty)\to Y$ be the unit-speed geodesic
ray from $o$ to $f^{-1}(\xi)$, and put
$z_\xi(t):=f(\widehat z_\xi(t))$. Thus $z_\xi$ is a uniformly regular quasigeodesic ray
in $X$ with flag endpoint $\xi$, with constants
independent of $\xi$.
Applying \cite[Proposition~4.12 and (8.3)]{KO_globalshadow}
to the hyperbolic rays $\widehat z_\xi$ and then using
\Cref{lem.shadowscomparison}, we obtain the following
comparison between ambient shadows and $d_\chi$-balls.
\begin{proposition}
\label{prop.shadowcompare}
The following hold.
\begin{enumerate}
\item For every sufficiently large $r>0$, there exists $c>1$
such that, for every $\xi\in\Lambda_\rho$ and $t\geq0$,
$$
B_\chi\bigl(\xi,c^{-1}e^{-\chi(\mu(z_\xi(t)))}\bigr)
\subset
\mathcal O_r(o,z_\xi(t))\cap\Lambda_\rho
\subset
B_\chi\bigl(\xi,ce^{-\chi(\mu(z_\xi(t)))}\bigr).
$$
\item For every $c>0$ and every sufficiently large $r>0$,
depending on $c$, there exists $t_0=t_0(c,r)>0$ such that,
for every $\xi\in\Lambda_\rho$, $t\geq0$, and $s\geq t+t_0$,
$$
\mathcal O_r(o,z_\xi(s))\cap\Lambda_\rho
\subset
B_\chi\bigl(\xi,ce^{-\chi(\mu(z_\xi(t)))}\bigr)
\subset
\mathcal O_r(o,z_\xi(t))\cap\Lambda_\rho.
$$
\end{enumerate}
\end{proposition}

\begin{remark}
Without the symmetry assumption on $\chi$, the radial
scale in this proposition must be replaced by
$e^{-\frac12(\chi+\chi\circ\i)(\mu(z_\xi(t)))}.
$
\end{remark}

\subsection{Blow-up of $d_{\chi}$}

Set $\mathsf P_0=\operatorname{Stab}_{\Gamma_\rho}(p_0)$. 
Let
$$
\mathcal U_{p_0}:=\{\xi\in\mathcal F:(\xi,p_0)\in\mathcal F^{(2)}\}
$$
be the open Bruhat cell transverse to $p_0$.

\begin{Def}
For $\xi,\eta\in\mathcal U_{p_0}$ with $(\xi, \eta)\in \F^{(2)}$, define $D_\chi:=D_{\chi,p_0}$ by
\begin{equation}
\label{dedef}
D_{\chi}(\xi,\eta)
:=
d_\chi(\xi,\eta)
e^{\chi(\mathcal G(\xi,p_0))}
e^{\chi(\mathcal G(\eta,p_0))}= \frac{d_\chi(\xi,\eta)}{d_\chi(\xi,p_0)d_\chi(\eta,p_0)}
\end{equation}
and set $D_{\chi}(\xi,\xi)=0$ for
$\xi\in\mathcal U_{p_0}$.
\end{Def}

\begin{remark}\label{EucRe}
For $G=\SO(n,1)^\circ$, use the upper half-space model
of $\mathbb H^n$ with curvature $-1$, basepoint $o=(0,1)$,
and $p_0=\infty$. Let $H\in\mathfrak a^+$ satisfy
$\exp(tH)o=(0,e^t)$.
For $\xi,\eta\in\mathbb R^{n-1}$, we have
$$
d_\chi(\xi,\eta)
=
\left(
\frac{\|\xi-\eta\|}
{\sqrt{1+\|\xi\|^2}\sqrt{1+\|\eta\|^2}}
\right)^{\chi(H)},
\qquad
d_\chi(\xi,\infty)=(1+\|\xi\|^2)^{-\chi(H)/2}.
$$
Consequently, $D_{\chi}$ is a positive power of the Euclidean
metric: $D_{\chi}(\xi,\eta)=\|\xi-\eta\|^{\chi(H)}$. \end{remark}

The cocycle identity \eqref{GPT} and \Cref{lem.zerobuse},
together with $\chi=\chi\circ\i$, give the following
transformation properties. The uniform comparison follows
by continuity and compactness.
\begin{lemma}[Transformation of $D_\chi$]
\label{lem.metrics}
\label{lem.LP5.12}
Let $g\in G$. For pairwise transverse 
$\xi,\eta,g^{-1}p_0\in\mathcal F$, we have
$$
D_\chi(g\xi,g\eta)
=d_\chi(\xi,\eta)
e^{\chi(\mathcal G(\xi,g^{-1}p_0))}
e^{\chi(\mathcal G(\eta,g^{-1}p_0))}
e^{\chi(\beta_{p_0}(e,g))}.
$$
If $p:=g^{-1}p_0\in\Lambda_\rho$, then for
every compact $Q\subset\Lambda_\rho\smallsetminus\{p,p_0\}$,
there exists $C=C(Q,g)\geq1$ such that
$$
C^{-1}D_\chi(\xi,\eta)
\leq D_\chi(g\xi,g\eta)
\leq C D_\chi(\xi,\eta)
\qquad(\xi,\eta\in Q).
$$

Moreover, $D_\chi$ is $\mathsf P_0$-invariant:
\begin{equation}
\label{inv}
D_\chi(g\xi,g\eta)=D_\chi(\xi,\eta)
\qquad
(g\in\mathsf P_0,\ \xi,\eta\in\Lambda_\rho\smallsetminus\{p_0\}).
\end{equation}
\end{lemma}

\subsection{Metric properties of $D_\chi$}
 By \eqref{dedef}, the restriction of $D_\chi$ to
$\Lambda_\rho\smallsetminus\{p_0\}$ is a continuous
premetric inducing the usual topology.
On every compact, pairwise transverse subset of
$\mathcal U_{p_0}$, the premetrics $d_\chi$ and $D_\chi$
are bi-Lipschitz equivalent.
We also observe:
\begin{lemma}[Comparison of the visual and blown-up premetrics]
\label{lem.sphereeuccompare}
For every $\xi_0\in\Lambda_\rho\smallsetminus\{p_0\}$, there exists $c=c(\xi_0)\ge1$
such that, for every $\xi\in\Lambda_\rho\smallsetminus\{p_0\}$,
$$
\frac{c^{-1}}{1+D_\chi(\xi,\xi_0)^2}
\le
d_\chi(\xi,p_0)^2
\le
\frac{c}{1+D_\chi(\xi,\xi_0)^2}.
$$
\end{lemma}
\begin{proof}
By \eqref{dedef},
$\bigl(1+D_\chi(\xi,\xi_0)^2\bigr)d_\chi(\xi,p_0)^2
=
d_\chi(\xi,p_0)^2
+\frac{d_\chi(\xi,\xi_0)^2}{d_\chi(\xi_0,p_0)^2}$.
The right-hand side is continuous and strictly positive
on the compact set $\Lambda_\rho$, since $p_0\neq\xi_0$.
It is therefore bounded above and away from zero,
which proves the claim.
\end{proof}

Via $f$, we regard the Euclidean metric
$d_{\operatorname{Euc}}$ on
$\mathbb R=\partial\mathbb H^2\smallsetminus\{\infty\}$
as a metric on $\Lambda_\rho\smallsetminus\{p_0\}$.
The following shows that $D_\chi$ and
$d_{\operatorname{Euc}}$ are quasisymmetrically equivalent on $\La_\rho\smallsetminus \{p_0\}$.
\begin{lemma}
\label{lem.quasisymmetry}
There exist $a,C\ge1$ such that, for all distinct
$\xi,\eta_1,\eta_2\in\Lambda_\rho\smallsetminus\{p_0\}$,
$$ \frac{D_\chi(\xi,\eta_1)}{D_\chi(\xi,\eta_2)}
\le
C\max\left\{
\left(\frac{d_{\operatorname{Euc}}(\xi,\eta_1)}
           {d_{\operatorname{Euc}}(\xi,\eta_2)}\right)^a,
\left(\frac{d_{\operatorname{Euc}}(\xi,\eta_1)}
           {d_{\operatorname{Euc}}(\xi,\eta_2)}\right)^{1/a}
\right\}.$$
\end{lemma}

\begin{proof}
Let $\sigma:[0,\infty)\to\operatorname{hull}\Lambda(\Gamma)$ be a
geodesic ray ending at $\infty$. For $i=1,2$, let $u_i$ be a nearest
point projection of $\sigma(t)$ to $[\xi,\eta_i]$, where $t$ is
sufficiently large. The equivariant forms of
\Cref{lem.gromovproductnearest,lem.buseandcartan} give
$$
\log D_\chi(\xi,\eta_i)
\approx
-\chi\!\left(\mu(f(\sigma(t))^{-1}f(u_i))\right)
+\chi\!\left(\beta_{p_0}(o,f(\sigma(t)))\right),
$$
with an additive error independent of $\xi,\eta_i,t$. Choose
$v_i\in[\sigma(t),\xi]$ at uniformly bounded distance from $u_i$. By
coarse additivity,
\begin{equation}
\label{eqn.QS1}
\log\frac{D_\chi(\xi,\eta_1)}{D_\chi(\xi,\eta_2)}
\approx
\epsilon\,\chi\!\left(\mu(f(v_1)^{-1}f(v_2))\right),
\end{equation}
where $\epsilon\in\{\pm1\}$ is determined by the order of $v_1,v_2$
on $[\sigma(t),\xi]$. The corresponding calculation in $\mathbb H^2$
gives, with the same sign,
\begin{equation}
\label{eqn.QS2}
\log\frac{d_{\operatorname{Euc}}(\xi,\eta_1)}
          {d_{\operatorname{Euc}}(\xi,\eta_2)}
\approx \epsilon\,d(v_1,v_2).
\end{equation}
Set
$u:=\log\frac{d_{\operatorname{Euc}}(\xi,\eta_1)}
                 {d_{\operatorname{Euc}}(\xi,\eta_2)}$ and $
v:=\log\frac{D_\chi(\xi,\eta_1)}{D_\chi(\xi,\eta_2)}$.
By \Cref{lem.scalar-Cartan-comparison} and
\eqref{eqn.QS1}--\eqref{eqn.QS2}, there exist uniform
constants $a\geq1$ and $c\geq0$ such that
$v\leq au+c$ if $u\geq0$, and
$v\leq a^{-1}u+c$ if $u\leq0$.
Indeed, these bounds follow by considering the two
possible orders of $v_1,v_2$; if $\epsilon u<0$,
\eqref{eqn.QS2} makes $d(v_1,v_2)$ uniformly bounded,
and enlarging $c$ covers this case.
Exponentiating gives the claim with $C=e^c$.
\end{proof}

\begin{corollary} [Quasi-triangle inequality]
\label{lem.eucdistancetriangle}
There exists $N_0\ge1$ such that, for all 
$\xi,\eta,\zeta\in\Lambda_\rho\smallsetminus\{p_0\}$,
$$
D_\chi(\xi,\zeta)
\le
N_0\bigl(D_\chi(\xi,\eta)+D_\chi(\eta,\zeta)\bigr).
$$
\end{corollary}
\begin{proof}
The assertion is immediate if two points coincide.
Otherwise, note
$
d_{\operatorname{Euc}}(\xi,\zeta)
\leq 2\max\{
d_{\operatorname{Euc}}(\xi,\eta),
d_{\operatorname{Euc}}(\eta,\zeta)\}.
$
Apply \Cref{lem.quasisymmetry} at $\xi$ or $\zeta$,
according to which term realizes the maximum.
Using symmetry of $D_\chi$, we obtain
$D_\chi(\xi,\zeta)
\leq 2^aC\max\{D_\chi(\xi,\eta),D_\chi(\eta,\zeta)\}.$
Thus we may take $N_0=2^aC$.
\end{proof}

\section{Cusp height and geometric estimates}
\label{sec.height}
We define cusp heights, compare them to Cartan projections, and prove
the separation and inversion estimates used in the coding.

\subsection{Height}
Each $\mathsf P_0$-orbit $\mathfrak O$ in
$\Lambda_{\mathrm p}(\Gamma_\rho)\smallsetminus\{p_0\}$ meets
$\Delta_0$ in one or two points; fix a representative
$p_{\mathfrak O}\in\mathfrak O\cap\Delta_0$.
Choose
$\gamma_{\mathfrak O}\in\Gamma_\rho$ such that
$\gamma_{\mathfrak O}^{-1}p_{\mathfrak O}\in\{p_0,\ldots,p_\ell\}$.

\begin{Def}
\label{hd}
The \emph{height function}
$h=h_{\Gamma_\rho,p_0,\Delta_0}:
\Lambda_{\mathrm p}(\Gamma_\rho)\smallsetminus\{p_0\}\to\mathbb R_{>0}$
is the $\mathsf P_0$-invariant function defined by
$$
h(p):=
e^{-\chi(\beta_{p_{\mathfrak O}}(e,\gamma_{\mathfrak O}))} \quad\text{where $\mathfrak O=\mathsf P_0 p$}. $$
In particular, if $p\in\operatorname{int}\Delta_0$ and
$\gamma^{-1}p\in\{p_0,\ldots,p_\ell\}$, then 
$p_{\mathfrak O}=p$, so that
$$
h(p)=e^{-\chi(\beta_p(e,\gamma))}.
$$
\end{Def}

\begin{lemma}
The definition of $h$ is independent of the choice of
$\gamma_{\mathfrak O}$. Different choices of the representatives
$p_{\mathfrak O}\in\mathfrak O\cap\Delta_0$ change $h$ by a uniformly
bounded multiplicative factor.
\end{lemma}
\begin{proof}
Two choices of $\gamma_{\mathfrak O}$ differ by right
multiplication by an element fixing
$\gamma_{\mathfrak O}^{-1}p_{\mathfrak O}$.
Independence follows from the cocycle identity and
\Cref{lem.zerobuse}.
By proper discontinuity, only finitely many
$u\in\mathsf P_0$ satisfy
$u\Delta_0\cap\Delta_0\neq\emptyset$.
The cocycle identity and compactness of $\Delta_0$
therefore give the uniform multiplicative bound under
changes of $p_{\mathfrak O}$.
\end{proof}

If $\gamma^{-1}p\in\{p_0,\ldots,p_\ell\}$ and
$\widetilde\gamma p=p_{\mathfrak O}$, then
\begin{equation}
\label{hpgg}
h(p)=
e^{-\chi(\beta_{\widetilde\gamma p}
                 (e,\widetilde\gamma\gamma))}.
\end{equation}

\begin{lemma}[Height comparison on compact sets]
\label{lem.heightondelta0}
For every compact set $Q\subset\Lambda_\rho\smallsetminus\{p_0\}$,
$$
h(p)\asymp_Q e^{-\chi(\beta_p(e,\gamma))}
$$
whenever $p\in Q\cap\Lambda_{\mathrm p}(\Gamma_\rho)$ and
$\gamma\in\Gamma_\rho$ satisfy
$\gamma^{-1}p\in\{p_0,\ldots,p_\ell\}$.
In particular, the comparison is uniform for $p\in\Delta$.
\end{lemma}

\begin{proof}
Choose $\widetilde\gamma\in\mathsf P_0$ with
$\widetilde\gamma p=p_{\mathfrak O}\in\Delta_0$.
Proper discontinuity implies that the set
$\{u\in\mathsf P_0:uQ\cap\Delta_0\neq\emptyset\}$ is finite.
By \eqref{hpgg} and the cocycle identity,
$$
-\log h(p)
=\chi(\beta_p(e,\gamma))
+\chi(\beta_p(\widetilde\gamma^{-1},e))
=\chi(\beta_p(e,\gamma))+O_Q(1).
$$
This proves the assertion.
\end{proof}

\subsection{Top representatives}

For $p\in\Gamma_\rho p_i\cap\Delta$, choose an element
$\gamma_p\in\Gamma_\rho$ such that
\begin{equation}
\label{eqn.definitiontoprepresentative}
p=\gamma_p p_i
\qquad\text{and}\qquad \gamma_p^{-1}p_i\in\Delta_i.
\end{equation}
We call $\gamma_p$ a top representative of $p$ and set $x_p:=\gamma_p^{-1}p_i$.
There are at most two choices, differing, when distinct,
by right multiplication by a generator of $\mathsf P_i$.
We write
$\widehat\gamma:=\rho^{-1}(\gamma)\in\Gamma$
for $\gamma\in\Gamma_\rho$.
\begin{lemma}
\label{lem.representshadow}
There exists $r>0$ such that, for every
$p\in\Lambda_{\mathrm p}(\Gamma_\rho)\cap\Delta$,
$$
p\in \mathcal O_r(o,\gamma_p o),
\qquad
x_p\in \mathcal O_r(o,\gamma_p^{-1}o).
$$
\end{lemma}

\begin{proof} Suppose that the first claim does not hold. By \Cref{lem.shadowscomparison},
there exist $i$ and a sequence $ f^{-1}(p_n) = \widehat\gamma_{p_n} \xi_i$ such that $f^{-1}(p_n) \notin O_n (o, \widehat\gamma_{p_n} o)$ for all $n \ge 1$. This implies that $\widehat\gamma_{p_n}^{-1} o \to \xi_i$. Since $f^{-1}(p_n)$ ranges over a compact
set disjoint from $\xi_i$, the point $o$ lies within a uniform distance
of the geodesic $[f^{-1}(p_n),\xi_i]$. Applying
$\widehat\gamma_{p_n}^{-1}$, we find that
$\widehat\gamma_{p_n}^{-1}o$ lies within a uniform distance of
$[\xi_i,y_n]$, where $y_n = \widehat\gamma_{p_n}^{-1} \xi_i$ which also ranges over a fixed compact subset of $\La(\Ga) \smallsetminus \{ \xi_i\}$. Thus $\xi_i$ is a conical limit point, contradicting
the fact that $\xi_i$ is parabolic. The second claim can be shown similarly.
\end{proof}

\begin{lemma}
\label{lem.heightondelta}
For every $p\in\Lambda_{\mathrm p}(\Gamma_\rho)\cap\Delta$,
$$
h(p)\asymp e^{-\chi(\mu(\gamma_p))}.
$$
\end{lemma}

\begin{proof} 
By \Cref{lem.heightondelta0,lem.representshadow,lem.buseandcartan},
$h(p)\asymp e^{-\chi(\beta_p(e,\gamma_p))}
\asymp e^{-\chi(\mu(\gamma_p))}$.
\end{proof}

\subsection{Height functions based at the other cusps}
Set $g_0=e$ and
for $1\le i\le\ell$, choose
$$
g_i\in K\quad\text{such that}\quad g_ip_i=p_0.
$$
 Define the height function $h_i$ on
$g_i(\Lambda_{\mathrm p}(\Gamma_\rho)\smallsetminus\{p_i\})$ by
$$
h_i(g_ip):=
h_{g_i\Gamma_\rho g_i^{-1},p_0,g_i\Delta_i}(g_ip).
$$

Equivalently, for each $\mathsf P_i$-orbit $\mathfrak O$ in
$\Lambda_{\mathrm p}(\Gamma_\rho)\smallsetminus\{p_i\}$, fix
$p_{\mathfrak O}^{(i)}\in\mathfrak O\cap\Delta_i$. If
$\gamma_{\mathfrak O}^{-1}p_{\mathfrak O}^{(i)}\in
\{p_0,\ldots,p_\ell\}$, then
$$
\log h_i(g_ip)
={-\chi (\beta_{g_ip_{\mathfrak O}^{(i)}}
   (e,g_i\gamma_{\mathfrak O}g_i^{-1}) )},
\qquad p\in\mathfrak O.
$$

\begin{lemma}
\label{lem.LP5.11}
There exists $C>1$ such that, for every $0\le i\le\ell$ and 
$p\in\Lambda_{\mathrm p}(\Gamma_\rho)\cap\Delta$,
$$
C^{-1}\le\frac{h_i(g_ip)}{h(p)}\le C.
$$
\end{lemma}

\begin{proof}
Choose $\gamma\in\Gamma_\rho$ with
$\gamma^{-1}p\in\{p_0,\ldots,p_\ell\}$, and choose
$\widetilde\gamma_i\in\mathsf P_i$ with
$\widetilde\gamma_i p=p_{\mathfrak O_i}^{(i)}$, where
$\mathfrak O_i=\mathsf P_i p$. The cocycle identity gives
$$
\beta_{g_ip_{\mathfrak O_i}^{(i)}}
 (e,g_i\widetilde\gamma_i\gamma g_i^{-1})
=
\beta_p(e,\gamma)
+\beta_{\gamma^{-1}p}(e,g_i^{-1})
+\beta_p(\widetilde\gamma_i^{-1}g_i^{-1},e).
$$
For each $i$, proper discontinuity implies that
$\{\delta\in\mathsf P_i:\delta\Delta\cap\Delta_i\ne\emptyset\}$
is finite. Hence the elements $\widetilde\gamma_i$ range over a fixed
finite set. 
Since both $g_i$ and $\widetilde{\ga}_i$ range over finite sets, the last two terms in the displayed identity are bounded. Applying the same argument for $i=0$, with
$\widetilde\gamma_0p=p_{\mathsf P_0p}^{(0)}$, and comparing with
\eqref{hpgg}, gives
$|\log h_i(g_ip)-\log h(p)|=O(1)$
uniformly in $p$ and $i$.
\end{proof}

\subsection{Separation}

The height function gives a uniform lower bound separating
distinct parabolic limit points.

\begin{proposition}
\label{lem.LP6.16}
There exists $0<C_s\le1$ such that for all distinct
$p,p'\in\Lambda_{\mathrm p}(\Gamma_\rho)\smallsetminus\{p_0\}$,
$$
D_\chi(p,p')\ge C_s\sqrt{h(p)h(p')}.
$$
\end{proposition}

\begin{proof}
We first prove that there exists $C>0$ such that
\begin{equation}
\label{lem.separationprelim}
d_\chi(p,p')
\geq C \sqrt{e^{-\chi(
\beta_p(e,\gamma)+\beta_{p'}(e,\gamma'))}}
\end{equation}
whenever $p\neq p'\in \Lambda_{\mathrm p}(\Ga_\rho)$  and
$\gamma^{-1}p=p_k$, $(\gamma')^{-1}p'=p_j$ for some
$0\leq j,k\leq\ell$.

Choose an oriented geodesic $\sigma$ in $\mathbb H^2$
from $f^{-1}(p')$ to $f^{-1}(p)$, with $\sigma(0)$
at its exit from the horoball based at $f^{-1}(p')$
in the family $\mathscr H$.
Since $\Ga\ba Y_{\rm {th}}$ is compact, $\sigma(0)$ lies
within a uniform distance of $\Gamma o$.
Replacing $\gamma$ and $\gamma'$ within
$\gamma\mathsf P_k$ and $\gamma'\mathsf P_j$, we may
arrange that $\widehat\gamma o$ and
$\widehat{\gamma'}o$ lie within a uniform distance
of the entrance into the horoball based at $f^{-1}(p)$
and the exit from the horoball based at $f^{-1}(p')$,
respectively. 
These replacements do not change the Busemann terms 
by \Cref{lem.zerobuse}.

Choose $g\in G$ with $g^+=p$ and $g^-=p'$, and set
$t:=\chi(\beta_p(g,\gamma))$ and
$t':=\chi(\beta_{p'}(g,\gamma'))$.
Since $\chi\circ\i=\chi$, the cocycle identity gives
\begin{equation}
\label{eqn.pythagorian}
2\chi(\mathcal G(p,p'))+t+t'
=\chi\bigl(\beta_p(e,\gamma)+\beta_{p'}(e,\gamma')\bigr).
\end{equation}
Thus \eqref{lem.separationprelim} is equivalent to a
uniform lower bound for $t+t'$.

Choose $h\in\Gamma$ such that $h\sigma(0)$ lies in
a fixed bounded neighborhood of $o$.
Replacing $\sigma$ by $h\sigma$ and simultaneously
left translating $p,p',g,\gamma,\gamma'$ by $\rho(h)$
preserves $t$ and $t'$ by equivariance.
We may therefore assume that $\sigma(0)$ lies in
that neighborhood of $o$.
By \Cref{lem.gromovproductnearest},
$\mathcal G(p,p')$ is then uniformly bounded.
The choices of $\gamma,\gamma'$ and
\Cref{lem.shadowscomparison} give a uniform $r>0$ such that
$$
p\in \mathcal O_r(o,\gamma o),
\qquad
p'\in \mathcal O_r(o,\gamma'o).
$$
By \Cref{lem.buseandcartan} and
\eqref{eqn.pythagorian},
\begin{equation}
\label{eqn.translateseparateequation}
t+t'=\chi(\mu(\gamma))+\chi(\mu(\gamma'))+O(1).
\end{equation}
Since
$\chi(\mu(\delta))$ is uniformly bounded below for
$\delta\in\Gamma_\rho$,
this proves \eqref{lem.separationprelim}.

Now let $p,p'\neq p_0$ and choose
$\widetilde\gamma,\widetilde\gamma'\in\mathsf P_0$
with $\widetilde\gamma p=p_{\mathfrak O}$ and
$\widetilde\gamma'p'=p_{\mathfrak O'}$.
\Cref{lem.zerobuse} and \eqref{hpgg} give
\begin{equation}
\label{eqn.separationprelim2}
e^{2\chi(\mathcal G(p,p_0))-\chi(\beta_p(e,\gamma))}
=e^{2\chi(\mathcal G(p_{\mathfrak O},p_0))}h(p),
\end{equation}
and similarly for $p'$.
Multiplying \eqref{lem.separationprelim} by
$e^{\chi(\mathcal G(p,p_0))+\chi(\mathcal G(p',p_0))}$
therefore yields
$$
D_\chi(p,p')
\geq C e^{\chi(\mathcal G(p_{\mathfrak O},p_0))
             +\chi(\mathcal G(p_{\mathfrak O'},p_0))}
\sqrt{h(p)h(p')}.
$$
Since $p_{\mathfrak O},p_{\mathfrak O'}\in\Delta_0$
and $\Delta_0$ is compact and disjoint from $p_0$,
the exponential factor has a uniform positive lower
bound. This proves the proposition.
\end{proof}

\subsection{Ball flipping}
The first estimate below shows that
$\gamma^{-1}$ maps a punctured $D_\chi$-ball of radius $r$ around
$p=\gamma p_0$ into the complement of a ball around
$\gamma^{-1}p_0$ of radius comparable to $h(p)/r$. The second gives the corresponding two-point estimate.
\begin{lemma}
\label{lem.ballflipping}
\label{lem.LP5.7_1}
\label{lem.maybeuseful}
\label{lem.LP5.7_2}
Let $\gamma\in\Gamma_\rho\smallsetminus\mathsf P_0$
and set $p:=\gamma p_0$. For all
$x,y\in\Lambda_\rho\smallsetminus\{p_0,p\}$,
$$
D_\chi(\gamma^{-1}x,\gamma^{-1}p_0)
\asymp \frac{h(p)}{D_\chi(x,p)},\qquad D_\chi(\gamma^{-1}x,\gamma^{-1}y)
\asymp
\frac{h(p)D_\chi(x,y)}
     {D_\chi(x,p)D_\chi(y,p)},
$$
with constants independent of $\gamma,x,y$.
\end{lemma}

\begin{proof}
Choose $\widetilde\gamma\in\mathsf P_0$ such that
$\widetilde\gamma p=p_{\mathfrak O}\in\Delta_0$.
\Cref{lem.zerobuse}, \eqref{GPT}, and \eqref{hpgg} give
$e^{2\chi(\mathcal G(p,p_0))-\chi(\beta_p(e,\gamma))}
=e^{2\chi(\mathcal G(p_{\mathfrak O},p_0))}h(p)
\asymp h(p),
$
uniformly, since $\Delta_0$ is compact and $p_0\notin \Delta_0$.
Both estimates now follow from
\Cref{lem.metrics} and \eqref{dedef}.
\end{proof}

\section{Local behavior of Patterson--Sullivan measures} \label{sec.PS}
For the rest of the paper, fix $\psi\in\mathfrak a^*$
tangent to $\psi_\rho$, and set
$$
\psi^-:=\psi\circ\i,
\qquad
\bar\psi:=\frac{\psi+\psi^-}{2}.
$$
Since $\psi_\rho$ is $\i$-invariant, the form $\psi^-$
is also tangent and
\begin{equation}
\label{lb}
\varphi>0\quad\text{on }\mathcal L_\rho\smallsetminus\{0\}
\qquad(\varphi\in\{\psi,\psi^-,\bar\psi\}).
\end{equation}

Since $\bar\psi\circ\i=\bar\psi$, we use the geometric
constructions and notation of \Cref{sec.visual,sec.height}
with $\chi=\bar\psi$.
Let $\nu:=\nu_\psi$ be the unique
$(\Gamma_\rho,\psi)$-Patterson--Sullivan probability measure
on $\Lambda_\rho$; thus $\delta_\psi(\Gamma_\rho)=1$.
For $r>0$ and $\xi\in\Lambda_\rho\smallsetminus\{p_0\}$, write
$$
{\mathcal B_{\bar\psi}}(\xi,r):=
\{\eta\in\Lambda_\rho\smallsetminus\{p_0\}:
D_{\bar\psi}(\xi,\eta)<r\}.
$$
We use the same notation for balls relative to conjugate
limit sets, retaining the ambient premetric $D_{\bar\psi}$.
For $H<\Gamma_\rho$, let $\delta_\psi(H)$ be the abscissa of convergence
of $\sum_{\gamma\in H}e^{-s\psi(\mu(\gamma))}$.

\begin{theorem}[{\cite[Theorem~1.1 and Corollary~7.2]{CZZ_relative}}]
\label{thm.divergencetype}
\label{thm.entropydrop}
We have $\sum_{\gamma\in\Gamma_\rho}e^{-\psi(\mu(\gamma))}=\infty$ and $\delta_\psi(\mathsf P_i)<1$ for every $0\le i\le\ell$.
\end{theorem}

\begin{lemma}
\label{lem:antisymmetric-gromov}
There exists $C>0$ such that, for all distinct $\xi,\eta\in\Lambda_\rho$,
$$
\left|(\psi-\bar\psi)(\mathcal G(\xi,\eta))\right|\leq C.
$$
\end{lemma}
\begin{proof}
Let $z$ be the nearest-point projection of $o$ to the geodesic joining
$f^{-1}(\xi)$ and $f^{-1}(\eta)$.
By \Cref{lem.gromovproductnearest},
$\mathcal G(\xi,\eta)=\frac12(\mu(f(z))+\i\mu(f(z)))+O(1)$.
The form $\psi-\bar\psi$ vanishes on the first term, since
$(\psi-\bar\psi)\circ\i=-(\psi-\bar\psi)$.
\end{proof}

\subsection{Change of measure}
\begin{lemma}
\label{lem.LP5.18}
Fix $\xi_0\in\Lambda_\rho\smallsetminus\{p_0\}$ and $D>0$.
Uniformly for $\gamma\in\mathsf P_0$, Borel sets
$F\subset\Lambda_\rho\smallsetminus\{p_0\}$ with
$\nu(F)>0$ and $D_{\bar\psi}$-diameter at most $D$,
and $\xi\in F$, we have
\begin{equation}
\label{nuF}
\nu(F)\asymp
\frac{1+D_{\bar\psi}(\gamma^{-1}\xi,\xi_0)^2}
     {1+D_{\bar\psi}(\xi,\xi_0)^2}
\,\nu(\gamma^{-1}F).
\end{equation}
\end{lemma}

\begin{proof}
For $\gamma\in\mathsf P_0$, \Cref{lem.zerobuse} gives
$\beta_{p_0}(e,\gamma)=0$, so \eqref{GPT} yields
$$
\beta_\eta(e,\gamma)
=2\mathcal G(\eta,p_0)
 -2\mathcal G(\gamma^{-1}\eta,p_0).
$$
By \Cref{lem:antisymmetric-gromov}, applying $\psi$ or
$\bar\psi$ to this difference changes its value by a
uniformly bounded amount.
Conformality and \Cref{lem.sphereeuccompare} therefore give
$$
\begin{aligned}
\nu(\gamma^{-1}F)
&=\int_F e^{\psi(\beta_\eta(e,\gamma))}\,d\nu(\eta)\asymp\int_F
\frac{d_{\bar\psi}(\gamma^{-1}\eta,p_0)^2}
     {d_{\bar\psi}(\eta,p_0)^2}\,d\nu(\eta)\asymp\int_F
\frac{1+D_{\bar\psi}(\eta,\xi_0)^2}
     {1+D_{\bar\psi}(\gamma^{-1}\eta,\xi_0)^2}\,d\nu(\eta).
\end{aligned}
$$
By $\mathsf P_0$-invariance of $D_{\bar\psi}$, both
$F$ and $\gamma^{-1}F$ have diameter at most $D$.
Thus \Cref{lem.eucdistancetriangle} shows that the last
integrand is uniformly comparable to its value at $\xi$.
Integrating and rearranging proves \eqref{nuF}.
\end{proof}

\begin{lemma}
\label{lem.LP5.17}
Fix $c>1$.
Uniformly for all $p\in\Delta\cap\Lambda_{\mathrm p}(\Gamma_\rho)$, where $\gamma_p$
is the top representative from
\eqref{eqn.definitiontoprepresentative}, and Borel sets
$F\subset\mathcal B_{\bar\psi}(p,ch(p))$, we have
$$
\nu(F)\asymp_c
e^{-\psi(\mu(\gamma_p))}\nu(\gamma_p^{-1}F).
$$ 
\end{lemma}

\begin{proof}
We claim that there exists $R=R(c)>0$ such that
$\mathcal B_{\bar\psi}(p,ch(p))
\subset\mathcal O_R(o,\gamma_p o)
$
for every $p=\gamma_pp_0\in\Delta$.
Otherwise, choose
$p_n=\gamma_{p_n}p_0\in\Delta$ and
$x_n\in\mathcal B_{\bar\psi}(p_n,ch(p_n))
\smallsetminus\mathcal O_n(o,\gamma_{p_n}o).
$
By \Cref{lem.representshadow}, $x_n\neq p_n$ for all
sufficiently large $n$.
Set
$y_n:=\gamma_{p_n}^{-1}x_n$ and
$z_n:=\gamma_{p_n}^{-1}p_0=x_{p_n}\in\Delta_0.
$
Then \Cref{lem.ballflipping} gives
$D_{\bar\psi}(y_n,z_n)\gg_c1$.
Since $x_n\notin\mathcal O_n(o,\gamma_{p_n}o)$, we have
$d(o,\gamma_{p_n}o)\geq n$, and hence
$\gamma_{p_n}^{-1}\to\infty$.
By \Cref{lem.ambient-shadow-convergence}(3), after
passing to a subsequence, $y_n\to\zeta\in\Lambda_\rho$.
By \Cref{lem.representshadow} and
\Cref{lem.ambient-shadow-convergence}(2), $z_n\to\zeta$
as well.
Since $\Delta_0$ is compact and disjoint from $p_0$,
we have $\zeta\neq p_0$.
Continuity of $D_{\bar\psi}$ therefore gives
$D_{\bar\psi}(y_n,z_n)\to0$, a contradiction.

Now let $p=\gamma_pp_i\in\Delta$ for an arbitrary $i$.
For fixed $c$, \eqref{dedef} and
\Cref{lem.dpsidiameterfinite} imply that, when $h(p)$ is
sufficiently small, the ball
$\mathcal B_{\bar\psi}(p,ch(p))$ lies in a fixed compact
neighborhood of $\Delta$ disjoint from all the points
$p_0,\ldots,p_\ell$.
Hence \Cref{lem.metrics,lem.LP5.11} give a uniform $C>1$
such that
$$
g_i\mathcal B_{\bar\psi}(p,ch(p))
\subset
\mathcal B_{\bar\psi}(g_ip,Cc\,h_i(g_ip)).
$$
Apply the preceding argument in
$\Gamma_i=g_i\Gamma_\rho g_i^{-1}$.
The element $g_i\gamma_pg_i^{-1}$ sends $p_0$ to $g_ip$,
and its inverse sends $p_0$ into $g_i\Delta_i$.
Thus the ball on the right is contained in
$\mathcal O_R(o,(g_i\gamma_pg_i^{-1})o)$ for some
$R=R(c)$, uniform over the finitely many cusp types.
Since $g_i\in K$, pulling back by $g_i$ gives
$\mathcal B_{\bar\psi}(p,ch(p))
\subset\mathcal O_R(o,\gamma_po)$.

By \Cref{lem.heightondelta} and \eqref{lb}, only finitely
many top representatives $\gamma_p$ occur when $h(p)$
is bounded below by a fixed positive constant.
Increasing $R$ therefore handles the remaining points
$p$ and proves the claim for every cusp type.

The claim and \Cref{lem.buseandcartan} give
$\beta_\xi(e,\gamma_p)=\mu(\gamma_p)+O_c(1)$ uniformly
for $\xi\in\mathcal B_{\bar\psi}(p,ch(p))$. The result follows from conformality:
$
\nu(\gamma_p^{-1}F)
=\int_F e^{\psi(\beta_\xi(e,\gamma_p))}\,d\nu(\xi).
$
\end{proof}
\subsection{Cusp depth and global shadows}

A horoball in $\mathscr H$ has type $i$ if it is a
$\Gamma$-translate of $\mathsf H_i$; a parabolic point
$p\in\Lambda_{\mathrm p}(\Gamma_\rho)$ has type $i$
if $p\in\Gamma_\rho p_i$.

For a primitive generator $u_i$ of $\mathsf P_i$, put
$v_i:=v_{u_i}$. By \Cref{lem.regularunipotent_cartan},
$\i(v_i)=v_i$ and $v_i\in\operatorname{int}\mathfrak a^+$.
Set
$$
\kappa_i:=\frac12\psi(v_i)=\frac12\bar\psi(v_i)>0.
$$
For each $\varphi\in\{\psi,\bar\psi\}$, the same lemma gives
$\varphi(\mu(u_i^n))=2\kappa_i\log|n|+O(1)$
as $|n|\to\infty$.
Since $\mathsf P_i=\langle u_i\rangle$, the Poincar\'e series
$\sum_{\gamma\in\mathsf P_i}e^{-s\varphi(\mu(\gamma))}$
has the same abscissa of convergence as
$\sum_{n\ge1}n^{-2s\kappa_i}$.
Consequently,
$$
\delta_i:=\delta_\psi(\mathsf P_i)
=\delta_{\bar\psi}(\mathsf P_i)
=\frac{1}{2\kappa_i}<1,
$$
where the last inequality follows from \Cref{thm.entropydrop}.

The cusp calculation \eqref{cuspes} gives
\begin{equation}
\label{eqn:scalar-cusp-depth}
\varphi\bigl(\mu(f(x)^{-1}f(y))\bigr)
=\kappa_i d(x,y)+O(1)
\end{equation}
for $\varphi\in\{\psi,\bar\psi\}$ whenever $x,y\in Y$
lie in the same horoball of type $i$, with uniform error.

\begin{lemma}[Comparison with cusp depth]
\label{lem:weighted-cusp-depth}
For $\varphi\in\{\psi,\bar\psi\}$ and every horoball
$\mathsf H\in\mathscr H$ of type $i$,
$$
\inf_{g\in\Gamma_\rho}
\varphi\bigl(\mu(g^{-1}f(x))\bigr)
=\kappa_i d(x,\partial\mathsf H)+O(1)
\qquad(x\in Y\cap\mathsf H).
$$
For $x\in Y_{\rm th}$, the same infimum is $O(1)$.
All errors are uniform in $x$ and $\mathsf H$.
\end{lemma}

\begin{proof}
On $Y_{\rm th}$, the assertion follows from
cocompactness and \Cref{lem.scalar-Cartan-comparison}.
Suppose that $x\in Y\cap\mathsf H$, where $\mathsf H$
has type $i$. Choose $z\in Y\cap\partial\mathsf H$
such that
$d(x,z)=d(x,\partial\mathsf H)+O(1)$ and
$d(z,\Gamma o)=O(1)$.
The upper bound follows from
\eqref{eqn:scalar-cusp-depth} and \Cref{lem.cptcartan}.

For the lower bound, fix $g\in\Gamma_\rho$ and let
$z\in[\widehat g o,x]\cap\partial\mathsf H$.
By \Cref{lem.DKO_additive,lem.scalar-Cartan-comparison}
and \eqref{eqn:scalar-cusp-depth},
$$
\varphi\bigl(\mu(g^{-1}f(x))\bigr)
\geq\kappa_i d(z,x)-C
\geq\kappa_i d(x,\partial\mathsf H)-C.
$$
Taking the infimum over $g$ proves the claim.
\end{proof}

\begin{theorem}[Global shadow lemma,
{\cite[Theorems~1.4 and~5.1]{KO_globalshadow}}]
For every sufficiently large fixed $R>1$, the following
estimates hold uniformly for $\zeta\in\Lambda(\Gamma)$
and $x\in[o,\zeta]$.
If $x$ lies in a horoball $\mathsf H\in\mathscr H$
of type $i$, then
\begin{equation}
\label{eqn:general-shadow-depth}
\nu\bigl(\mathcal O_R(o,f(x))\bigr)
\asymp
e^{-\psi(\mu(f(x)))
   +(2\delta_i-1)\kappa_i d(x,\partial\mathsf H)}.
\end{equation}
If $x\in Y_{\rm th}$, then
$\nu(\mathcal O_R(o,f(x)))
\asymp e^{-\psi(\mu(f(x)))}$.
\end{theorem}

\begin{proof}
Apply the cited global shadow lemma using
\Cref{lem.shadowscomparison,lem:weighted-cusp-depth}.
The polynomial term is absent because the peripheral
groups are cyclic.
\end{proof}

For $p\in\Lambda_{\mathrm p}(\Gamma_\rho)$, let
$\mathsf H_p\in\mathscr H$ be the horoball based at
$f^{-1}(p)$, and let $t_p\geq0$ be the entrance time
of $\widehat z_p$ into $\mathsf H_p$.
If $p$ has type $i$, choose $g_p\in\Gamma_\rho$ such that
$g_pp_i=p$ and $\widehat g_po$ lies within a uniform
distance of $\widehat z_p(t_p)$.
Writing $\mathsf P_p:=\operatorname{Stab}_{\Gamma_\rho}(p)$,
we have
$\delta_\psi(\mathsf P_p)
=\delta_{\bar\psi}(\mathsf P_p)=\delta_i$.

For $t\geq t_p$,
$d(\widehat z_p(t),\partial\mathsf H_p)=t-t_p$.
Coarse additivity, \Cref{lem.cptcartan}, and
\eqref{eqn:scalar-cusp-depth} therefore give, for
$\varphi\in\{\psi,\bar\psi\}$,
\begin{equation}
\label{eqn:parabolic-ray-depth}
\begin{aligned}
&\varphi(\mu(z_p(t)))
=\varphi(\mu(g_p))+\kappa_i(t-t_p)+O(1),\\
&\varphi(\mu(g_p^{-1}z_p(t)))
=\kappa_i(t-t_p)+O(1).
\end{aligned}
\end{equation}

\begin{corollary}
\label{lem.globalshadowatparabolic}
For every sufficiently large fixed $R>1$, uniformly
in $p\in \La_{\rm p}(\Ga_\rho)$ of type $i$ and $t\geq t_p$,
\begin{equation}
\label{nufo}
\begin{aligned}
\nu\bigl(\mathcal O_R(o,z_p(t))\bigr)
&\asymp
e^{-\psi(\mu(g_p))
   +2(\delta_i-1)\psi(\mu(g_p^{-1}z_p(t)))}\asymp
e^{-\psi(\mu(g_p))
   -2(1-\delta_i)\kappa_i(t-t_p)}.
\end{aligned}
\end{equation}
\end{corollary}

\begin{proof}
Substitute \eqref{eqn:parabolic-ray-depth}, with
$\varphi=\psi$, into \eqref{eqn:general-shadow-depth}.
\end{proof}

\subsection{Doubling and reverse doubling}

\begin{proposition}[Doubling and reverse doubling]
\label{prop.LP5.14}
\begin{enumerate}
\item
For every $C>1$,
$$
\nu(\mathcal B_{\bar\psi}(\xi,Cr))
\asymp_C
\nu(\mathcal B_{\bar\psi}(\xi,r))
$$
uniformly for $\xi\in\Delta$ and $0<r\leq1/C$.

\item
There exists $C_0>1$ such that, for every
$\xi\in\Delta$ and $0<r<1/C_0$,
$$
2\nu(\mathcal B_{\bar\psi}(\xi,r))
\leq\nu(\mathcal B_{\bar\psi}(\xi,C_0r)).
$$
\end{enumerate}
\end{proposition}

\begin{proof}
Fix a sufficiently large $R$ and set
$\sigma:=\max_i|2\delta_i-1|<1$.
Let $\zeta\in\Lambda(\Gamma)$ and $x,y\in[o,\zeta]$
with $d(o,x)\leq d(o,y)$.

Suppose first that $x,y$ lie in the same horoball
$\mathsf H$ of type $i$.
By \Cref{lem.DKO_additive},
\eqref{eqn:scalar-cusp-depth}, and
\eqref{eqn:general-shadow-depth},
$$
\begin{aligned}
\log
\frac{\nu(\mathcal O_R(o,f(x)))}
     {\nu(\mathcal O_R(o,f(y)))}
={}&\kappa_i d(x,y) -(2\delta_i-1)\kappa_i
\bigl(d(y,\partial\mathsf H)-d(x,\partial\mathsf H)\bigr)
+O(1).
\end{aligned}
$$
Since distance to $\partial\mathsf H$ is $1$-Lipschitz,
this lies between
$(1-\sigma)\kappa_i d(x,y)-K$ and
$(1+\sigma)\kappa_i d(x,y)+K$
for a uniform $K$.

Otherwise, let $z$ be the exit point of $[x,y]$ from
the horoball $\mathsf H$ of type $i$ containing $x$,
and let $z'$ be its entrance point into the horoball
$\mathsf H'$ of type $j$ containing $y$.
If $x\in Y_{\rm th}$, set $z=x$; if $y\in Y_{\rm th}$,
set $z'=y$. Omit the corresponding cusp terms in these cases.
Coarse additivity and \eqref{eqn:scalar-cusp-depth} give
$$
\begin{aligned}
\psi\bigl(\mu(f(x)^{-1}f(y))\bigr)
={}&\kappa_i d(x,z)
+\psi\bigl(\mu(f(z)^{-1}f(z'))\bigr) +\kappa_j d(z',y)+O(1).
\end{aligned}
$$
Apply \Cref{lem.scalar-Cartan-comparison} to the middle
term and use
$$
d(x,\partial\mathsf H)\leq d(x,z),
\qquad
d(y,\partial\mathsf H')\leq d(z',y).
$$
Since $\sigma<1$, $\min_i\kappa_i>0$, and
$d(x,z)+d(z,z')+d(z',y)=d(x,y)$, the global shadow
estimates give, after enlarging $K$ to cover both cases,
\begin{equation}
\label{eqn:shadow-ratios-general}
K^{-1}d(x,y)-K
\leq
\log
\frac{\nu(\mathcal O_R(o,f(x)))}
     {\nu(\mathcal O_R(o,f(y)))}
\leq Kd(x,y)+K.
\end{equation}

Coarse additivity and \Cref{lem.scalar-Cartan-comparison}
also bound
$\bar\psi(\mu(f(y)))-\bar\psi(\mu(f(x)))$
above and below by positive multiples of $d(x,y)$,
up to uniform additive errors.
Together with \Cref{prop.shadowcompare} and
\eqref{eqn:shadow-ratios-general}, this gives power bounds
for ratios of $d_{\bar\psi}$-ball masses.
Here continuity and divergence of
$t\mapsto\bar\psi(\mu(z_\xi(t)))$ allow prescribed
sufficiently small scales to be realized in order along
each ray.
The uniform comparison of $d_{\bar\psi}$ and
$D_{\bar\psi}$ near $\Delta$ transfers these bounds to
$\mathcal B_{\bar\psi}$-balls.
Thus there exist $L_1\geq1$, $0<\beta_-\leq\beta_+$,
and $0<r_0\leq1$ such that, for every $\xi\in\Delta$
and $0<r\leq s\leq r_0$,
\begin{equation}
\label{eqn:ball-ratios-general}
L_1^{-1}(r/s)^{\beta_+}
\leq
\frac{\nu(\mathcal B_{\bar\psi}(\xi,r))}
     {\nu(\mathcal B_{\bar\psi}(\xi,s))}
\leq L_1(r/s)^{\beta_-}.
\end{equation}

Taking $s=Cr$ proves (1) when $Cr\leq r_0$;
compactness covers the remaining radii.
For (2), choose $C_0>1/r_0$ such that
$L_1(C_0r_0)^{-\beta_-}\leq1/2$.
Apply the upper bound in \eqref{eqn:ball-ratios-general}
with $s=\min\{C_0r,r_0\}$, noting that
$r/s\leq(C_0r_0)^{-1}$ whenever $0<r<1/C_0$.
\end{proof}

We fix a choice of $C_0$ satisfying (2) for use in the
geometric partition construction below.

\begin{corollary}
\label{cor.PS_polyuniform}
There exist $c,\beta>0$ such that, for every
$\xi\in\Delta$ and $0<r<1$,
$$
\nu(\mathcal B_{\bar\psi}(\xi,r))\geq cr^\beta.
$$
\end{corollary}

\begin{proof} Let $r_0>0$ be as in the proof of \Cref{prop.LP5.14}.
Since $\nu$ has full support and $\Delta$ is compact,
$\inf_{\xi\in\Delta}
\nu(\mathcal B_{\bar\psi}(\xi,r_0))>0$.
Apply \eqref{eqn:ball-ratios-general} with $s=r_0$
to obtain the estimate for $r\leq r_0$, with
$\beta=\beta_+$.
The remaining radii follow by monotonicity.
\end{proof}

\subsection{Parabolic ball estimates}

\begin{corollary}
\label{eqn.LP5.16}
Uniformly for $p\in\Delta\cap\Lambda_{\mathrm p}(\Gamma_\rho)$
and $0<r\leq h(p)$, we have
\begin{equation}
\label{eqn.LP5.16statement}
\nu(\mathcal B_{\bar\psi}(p,r))
\asymp
e^{-\psi(\mu(g_p))}
\left(\frac{r}{h(p)}\right)^{2(1-\delta_\psi(\mathsf P_p))}.
\end{equation}
\end{corollary}

\begin{proof}
Fix $p\in\Delta\cap\Lambda_{\mathrm p}(\Gamma_\rho)$
of type $i$.
By the choice of $g_p$ and
\Cref{lem.heightondelta0,lem.buseandcartan},
$h(p)\asymp e^{-\bar\psi(\mu(g_p))}$.
Thus \eqref{eqn:parabolic-ray-depth} gives
$$
r_p(t):=e^{-\bar\psi(\mu(z_p(t)))}
\asymp h(p)e^{-\kappa_i(t-t_p)}
\qquad(t\geq t_p).
$$

Since $d_{\bar\psi}$ and $D_{\bar\psi}$ are uniformly
comparable near $\Delta$, \Cref{prop.shadowcompare}
gives, for sufficiently large fixed $R$, constants
$T,t_0>0$ such that, for every $t\geq t_p+T$,
\begin{equation}
\label{eqn.shadowandEucballs}
\mathcal O_R(o,z_p(t+t_0))\cap\Lambda_\rho
\subset
\mathcal B_{\bar\psi}(p,r_p(t))
\subset
\mathcal O_R(o,z_p(t))\cap\Lambda_\rho.
\end{equation}
Applying \eqref{nufo} to the two shadows gives
$$
\nu(\mathcal B_{\bar\psi}(p,r_p(t)))
\asymp
e^{-\psi(\mu(g_p))}
\left(\frac{r_p(t)}{h(p)}\right)^{2(1-\delta_i)}.
$$

The function $r_p$ is continuous, tends to zero, and
satisfies $r_p(t_p+T)\asymp h(p)$.
Hence these radii cover $(0,ch(p)]$ for some uniform
$c>0$.
\Cref{prop.LP5.14}(1) and compactness extend the estimate
to the remaining radii $ch(p)<r\leq h(p)$.
\end{proof}

\section{Recurrence of the translation flow}
\label{sec.recur}

We describe the stable and unstable leaves of the translation
flow on $\Omega_{\bar\psi}$ and prove the uniform recurrence
property used in the coding construction. We then show that a
return to a fixed compact set forces the forward endpoint to
lie near a parabolic limit point whose height is comparable to
the scale determined by the return time.
\subsection{Stable and unstable leaves}
\label{subsec:stableunstable}

Recall that
$\widetilde\Omega_{\bar\psi}=\Lambda_\rho^{(2)}\times\mathbb R$ and
$\Omega_{\bar\psi}
=\Gamma_\rho\backslash\widetilde\Omega_{\bar\psi}$.

For $g\in G$ with $(g^+,g^-)\in\Lambda_\rho^{(2)}$, write
$$
[g]:=\bigl(g^+,g^-,\bar\psi(\beta_{g^+}(e,g))\bigr) \in\widetilde\Omega_{\bar\psi}.$$
For $x=[g]=(x^+,x^-,s_x)\in\widetilde\Omega_{\bar\psi}$, set
$s_x^-:=\bar\psi(\beta_{g^-}(e,g)).
$
Since $\bar\psi\circ\i=\bar\psi$, we have
$s_x^-=2\bar\psi(\mathcal G(x^+,x^-))-s_x$,
so $s_x^-$ depends only on $x$.

For $x,y\in\widetilde\Omega_{\bar\psi}$, define
$$
d_+(x,y):=e^{(s_x+s_y)/2}d_{\bar\psi}(x^+,y^+),
\qquad
d_-(x,y):=e^{(s_x^-+s_y^-)/2}d_{\bar\psi}(x^-,y^-).
$$
They are nonnegative, symmetric, 
$\Gamma_\rho$-invariant and satisfy that for all $t\in \br$,
$$
d_+(xa_t,ya_t)=e^t d_+(x,y),
\qquad
d_-(xa_t,ya_t)=e^{-t}d_-(x,y).
$$
For  $x=[g]$, set
$$
\begin{aligned}
W^-(x)
&:=\{(x^+,\eta,s_x):
       \eta\in\Lambda_\rho\smallsetminus\{x^+\}\} =[gN]\cap\widetilde\Omega_{\bar\psi},\\
W^+(x)
&:=\{(\xi,x^-,2\bar\psi(\mathcal G(\xi,x^-))-s_x^-):
       \xi\in\Lambda_\rho\smallsetminus\{x^-\}\}=[g\check N]\cap\widetilde\Omega_{\bar\psi}.
\end{aligned}
$$
These leaves are stable and unstable for the translation
flow in the following sense: by
\cite[Proposition~8.6]{KOW_ergodic}, there exists a metric
$d$ inducing the topology of $\Omega_{\bar\psi}$ such that,
as $t\to+\infty$,
$$
\begin{aligned}
d(\Gamma_\rho xa_t,\Gamma_\rho ya_t)\to 0\quad 
\text{for }y\in W^-(x)\quad\text{and}\quad  d(\Gamma_\rho xa_{-t},\Gamma_\rho ya_{-t})\to 0 \quad \text{for }y\in W^+(x).
\end{aligned}
$$
For $r>0$, set
$$
W^\pm(x,r):=
\{y\in W^\pm(x):d_\pm(x,y)<r\}.
$$
These sets form a neighborhood basis on the corresponding leaves.
\subsection{Local product structure}

For $x,y\in\widetilde\Omega_{\bar\psi}$ with
$(y^+,x^-)\in\mathcal F^{(2)}$, the intersection
$W^+(x)\cap W^-(ya_{\mathbb R})$ consists of a unique
point, denoted by $\langle x,y\rangle$.
Explicitly,
$$
\begin{aligned}
\langle x,y\rangle
=(y^+,x^-,s_y+\vartheta(x,y)),\quad\text{ where }\vartheta(x,y)
:=2\bar\psi(\mathcal G(y^+,x^-))-s_x^--s_y.
\end{aligned}
$$
That is, $\vartheta(x,y)$ is the unique time $t$ such that
$\langle x,y\rangle\in W^-(ya_t)$. The maps $\vartheta$ and $\langle\cdot,\cdot\rangle$ are
continuous wherever $(y^+,x^-)$ is transverse and, by
\eqref{GPT}, are respectively $\Gamma_\rho$-invariant
and $\Gamma_\rho$-equivariant.
Since $\vartheta(x,x)=0$ and $\langle x,x\rangle=x$,
continuity and compactness imply that, for every compact
$Q\subset\widetilde\Omega_{\bar\psi}$,
$$
|\vartheta(x,y)|\to0,
\qquad
d_+(x,\langle x,y\rangle)\to0,
\qquad
d_-(ya_{\vartheta(x,y)},\langle x,y\rangle)\to0
$$
uniformly as $(x,y)$ approaches the diagonal in $Q\times Q$.

We choose the compact section
$$
\widetilde\Delta_0
:=\{(\xi,p_0,0):\xi\in\Delta_0\}
\subset\widetilde\Omega_{\bar\psi}
$$
of the flow lines with negative endpoint $p_0$
and positive endpoint in $\Delta_0$.

\begin{lemma}
\label{lem.unstableplaquescompact}
For every $r>0$, the set
$\bigcup_{x\in\widetilde\Delta_0}W^+(x,r)$
is relatively compact in $\widetilde\Omega_{\bar\psi}$.
\end{lemma}
\begin{proof}
By compactness of $\Delta_0$, the formula for $d_+$
shows that $d_+(x,y)\to\infty$ as $y^+\to p_0$,
uniformly for $x\in\widetilde\Delta_0$ and $y\in W^+(x)$.
Thus $d_+(x,y)<r$ keeps $y^+$ in a fixed compact
subset of $\Lambda_\rho\smallsetminus\{p_0\}$.  Finally,
$s_y=2\bar\psi(\mathcal G(y^+,p_0)-\mathcal G(\xi,p_0))$
is uniformly bounded by continuity; since $y^-=p_0$,
the claim follows.
\end{proof}

\subsection{Uniform recurrence}
For $\delta\in\mathbb R$, let
$$
W^-_{p_0,\delta}
:=\{(p_0,\eta,\delta):
       \eta\in\Lambda_\rho\smallsetminus\{p_0\}\}
\subset\widetilde\Omega_{\bar\psi}
$$
denote the stable leaf with positive endpoint $p_0$
and $\mathbb R$-coordinate $\delta$.

\begin{lemma}
\label{lem.LP6.54}
Let $Q\subset\Omega_{\bar\psi}$ be compact and let
$\delta\in\mathbb R$.
There exists $T_*=T_*(Q,\delta)>0$ such that, for every
$x\in\widetilde\Omega_{\bar\psi}$ with $\Gamma_\rho x\in Q$,
there is $0<t<T_*$ for which
$$
W^+(x,1)a_t\cap\Gamma_\rho W^-_{p_0,\delta}
\neq\emptyset.
$$
\end{lemma}
\begin{proof}
Applying \Cref{koe} to a tangent multiple of $\bar\psi$
and rescaling time, we obtain a mixing invariant
probability measure of full support on $\Omega_{\bar\psi}$.

Choose a compact set $\widetilde Q\subset
\widetilde\Omega_{\bar\psi}$ whose image contains $Q$,
and choose $0<\varepsilon<1/4$ so that
$\bigcup_{x'\in W^+(x,\varepsilon)}W^+(x',\varepsilon)
\subset W^+(x,1)$ for all
$x\in\widetilde Q).$
After decreasing $\varepsilon$ if necessary, choose a
relatively compact product neighborhood $\mathcal V$
of a point in $W^-_{p_0,\delta}$ such that every
$y_0\in\mathcal V$ can be written as $y_0=y_1a_u$,
where $z\in W^-_{p_0,\delta}$,
$y_1\in W^+(z,\varepsilon)$, and $|u|<\varepsilon$,
with $z,y_1$ ranging over fixed compact sets.
We may also arrange
that, whenever $w\in W^-(y_1,\varepsilon)$, the bracket
$\langle w,z\rangle$ is defined and satisfies
$|\vartheta(w,z)|<1$ and
$d_+(w,\langle w,z\rangle)<1$.

Cover $\widetilde Q$ by finitely many sufficiently small
product neighborhoods $\mathcal U_1,\ldots,\mathcal U_k$.
Their images and that of $\mathcal V$ are nonempty open
sets in $\Omega_{\bar\psi}$.
Mixing therefore gives a common $L>2$, with
$e^{-(L-2\varepsilon)}<\varepsilon$, such that every
$\mathcal U_j$ meets $\Gamma_\rho\mathcal V a_{-L}$.

For $x\in\widetilde Q$, choose $j$ with $x\in\mathcal U_j$,
and choose $y_0=y_1a_u\in\mathcal V$ and
$\gamma\in\Gamma_\rho$ such that
$\gamma y_0a_{-L}\in\mathcal U_j$.
Then there exist
$|\sigma|<\varepsilon$ and
$$
x'\in W^+(x,\varepsilon)
\cap W^-(\gamma y_0a_{-L+\sigma},\varepsilon).
$$
Put $T:=L-u-\sigma$ and $w:=\gamma^{-1}(x'a_T)$.
Equivariance and the scaling relation for $d_-$ give
$w\in W^-(y_1,\varepsilon e^{-T})
\subset W^-(y_1,\varepsilon)$.
Set $\theta:=\vartheta(w,z)$ and
$q_0:=\gamma\langle w,z\rangle$.
By the choice of $\mathcal V$ and the definition of the
bracket, $|\theta|<1$ and
$q_0\in W^+(x'a_T,1)\cap W^-(\gamma z\,a_\theta)$.

Set $q:=q_0a_{-\theta}$ and $t:=T-\theta$.
Then $q\in\gamma W^-_{p_0,\delta}$, while
$$
qa_{-t}=q_0a_{-T}
\in W^+(x',e^{-T})
\subset W^+(x',\varepsilon)
\subset W^+(x,1).
$$
Thus $q$ belongs to the required intersection.
Since $0<t<L+2$, we may take $T_*:=L+2$.
Equivariance extends the conclusion to every lift
of a point of $Q$.
\end{proof}

The next lemma converts uniform recurrence into a
parabolic approximation with controlled height.

\begin{lemma}
\label{lem.LP6.55_var}
For every $\tau>1$, there exists $C>1$ with the following
property. For every compact $Q\subset\Omega_{\bar\psi}$,
there exists $U_0=U_0(\tau,Q)>0$ such that, if
$x=(x^+,p_0,0)\in\widetilde\Delta_0$ and
$\Gamma_\rho xa_{n\log\tau}\in Q$ for some $n\in\mathbb N$,
then there exists
$p\in\Lambda_{\mathrm p}(\Gamma_\rho)\smallsetminus\{p_0\}$ satisfying
\begin{equation}
\label{depx}
D_{\bar\psi}(p,x^+)\leq C\tau^{-n},
\qquad
e^{-U_0}\tau^{-n}\leq h(p)\leq\tau^{-n}.
\end{equation}
\end{lemma}

\begin{proof}
Fix $\tau>1$ and a parameter $\delta\in\mathbb R$ to be
chosen below. Let $T_*=T_*(Q,\delta)$ be given by
\Cref{lem.LP6.54}. Applying that lemma to
$xa_{n\log\tau}$, choose $\gamma\in\Gamma_\rho$,
$z\in W^-_{p_0,\delta}$, and $0<t<T_*$ such that
$$
\gamma z\in W^+(xa_{n\log\tau},1)a_t.
$$
Set $y:=\gamma z a_{-n\log\tau-t}$ and $p:=y^+=\gamma p_0$.
Then $y\in W^+(x,\tau^{-n})$.
By \Cref{lem.unstableplaquescompact}, all such $y$ lie
in a fixed relatively compact set, independent of
$Q,n,t,\delta$. Consequently, $p$ lies in a fixed compact
subset of $\Lambda_\rho\smallsetminus\{p_0\}$ and
$s_y$ is uniformly bounded.

Since $s_x=0$, the definition of $d_+$ gives
$
e^{s_y/2}d_{\bar\psi}(x^+,p)
=d_+(x,y)<\tau^{-n}.
$
The uniform comparison of $d_{\bar\psi}$ and
$D_{\bar\psi}$ on this compact set therefore gives
$D_{\bar\psi}(x^+,p)\leq C\tau^{-n}$.
It follows that
$s_y=\bar\psi(\beta_p(e,\gamma))
     +\delta-n\log\tau-t.
$
By \Cref{lem.heightondelta0} and the bound on $s_y$,
there is a constant $c>1$, independent of $Q$ and
$\delta$, such that
$$
c^{-1}e^\delta\tau^{-n}e^{-T_*}
\leq h(p)\leq ce^\delta\tau^{-n}.
$$
Choose $\delta:=-\log(2c)$ and then set
$U_0:=T_*(Q,\delta)+\log c-\delta$.
This gives \eqref{depx}.
\end{proof}

\section{Construction of the  geometric partition}
\label{sec.codingDelta0}
For a small parameter $\eta>0$ to be chosen below,
we construct neighborhoods $J_{p,\eta}$ of parabolic
points $p\in\Delta$, comparable to $D_{\bar\psi}$-balls
of radius $\eta h(p)$.
We remove selected neighborhoods from each $\Delta_i$
in successive height ranges, starting with the largest
heights. The residual set $\mathcal Z_{i,n}(\eta)$
consists of the points of $\Delta_i$ remaining after
$n$ steps. We prove that its measure decays
exponentially in $n$ (\Cref{prop.LP6.15ooo}).

When discussing intervals and their order, we use the
Euclidean coordinate transported by $f$ on
$\Lambda_\rho\smallsetminus\{p_0\}$.
For nonempty subsets $A,B\subset\Lambda_\rho\smallsetminus\{p_0\}$, set
$$
D_{\bar\psi}(A,B):=\inf_{\substack{a\in A,\,  b\in B}}D_{\bar\psi}(a,b),\quad 
{\mathcal B_{\bar\psi}}(A,r):=\{x\in\Lambda_\rho\smallsetminus\{p_0\}:D_{\bar\psi}(x,A)<r\}.
$$

Let $C\geq1$ be the constant in \Cref{lem.quasisymmetry}.
\begin{corollary}[Quasi-monotonicity of $D_{\bar\psi}$]
\label{lem.olive}
For any
$\xi_1,\xi_2,\xi_3\in\Lambda_\rho\smallsetminus\{p_0\}$ satisfying
$
d_{\operatorname{Euc}}(\xi_1,\xi_2)
\le d_{\operatorname{Euc}}(\xi_1,\xi_3)$ and $
d_{\operatorname{Euc}}(\xi_2,\xi_3)
\le d_{\operatorname{Euc}}(\xi_1,\xi_3),
$
one has
$$
D_{\bar\psi}(\xi_1,\xi_2)
\le C D_{\bar\psi}(\xi_1,\xi_3)
\quad\text{and}\quad
D_{\bar\psi}(\xi_2,\xi_3)
\le C D_{\bar\psi}(\xi_1,\xi_3).
$$
\end{corollary}

\begin{proof}
Each of the two relevant Euclidean distance ratios is at most one.
The upper bound in \Cref{lem.quasisymmetry}, applied at the appropriate
endpoint, gives both assertions.
\end{proof}

Recall from \eqref{eqn.choiceoffunddomain} that
$\Delta=\bigcup_{i=0}^{\ell}\Delta_i$.
\begin{lemma}
\label{lem.choiceoffunddom}
There exists $C_{\rm fd}>1$ such that, for every
$0<\eta<1$, there are a finite-index subgroup
$\mathsf P_0'=\mathsf P_0'(\eta)<\mathsf P_0$ and a compact
fundamental domain
$R_{0,\eta}\subset\Lambda_\rho\smallsetminus\{p_0\}$
for its action satisfying
\begin{equation}
\label{eqn.R0balls}
\mathcal B_{\bar\psi}(\Delta,\eta^{-1})
\subset R_{0,\eta}
\subset\mathcal B_{\bar\psi}(x,C_{\rm fd}\eta^{-1})
\qquad(x\in\Delta).
\end{equation}
\end{lemma}
\begin{proof}
Choose a generator $h_0$ of $\mathsf P_0$ acting by
$x\mapsto x+1$, and normalize
$\widehat\Delta_\infty=[0,1]$.
By \eqref{inv}, $D_{\bar\psi}$ is invariant under translation by $h_0$.
Let $a,b$ be the left and right endpoints of $\overline{\mathcal B_{\bar\psi}(\Delta,\eta^{-1})}$ and put
$n:=\lfloor a\rfloor$ and $m:=\lceil b\rceil$.
Since this set contains $\Delta_0$, we have $m-n\geq1$.
Define
$$
\mathsf P_0':=\langle h_0^{m-n}\rangle,
\qquad
R_{0,\eta}:=[n,m]\cap\Lambda(\Gamma).
$$
Then $R_{0,\eta}$ is a compact fundamental domain
for $\mathsf P_0'$ and contains
$\mathcal B_{\bar\psi}(\Delta,\eta^{-1})$.

For the upper bound, note
$D_{\bar\psi}(a,\Delta),D_{\bar\psi}(b,\Delta)\leq\eta^{-1}$.
Hence \Cref{lem.eucdistancetriangle}
 gives
$D_{\bar\psi}(x,a)+D_{\bar\psi}(x,b)\ll\eta^{-1}$
uniformly for $x\in\Delta$ and $0<\eta<1$.
As $x\in[a,b]$, \Cref{lem.olive} therefore yields
$$
D_{\bar\psi}(x,\xi)\ll\eta^{-1}
\qquad
\bigl(x\in\Delta,\ 
\xi\in[a,b]\cap\Lambda(\Gamma)\bigr).
$$
Since $a-n,m-b<1$,  $D_{\bar\psi}$-diameters for
$[n,a]\cap\Lambda(\Gamma)$ and $[b,m]\cap\Lambda(\Gamma)$ are uniformly bounded.
Applying \Cref{lem.eucdistancetriangle}
extends the preceding bound to every $\xi\in R_{0,\eta}$.
Choosing $C_{\rm fd}>1$ larger than the implied constant
proves \eqref{eqn.R0balls}.
\end{proof}

The proof applies without change to each of the finitely many conjugated
systems
$$
\Gamma_i:=g_i\Gamma_\rho g_i^{-1},
\qquad
g_i\mathsf P_i g_i^{-1}<\Gamma_i,
$$
using $g_i\Delta_i$ as the compact set of centers.  Hence, after taking
the maximum of finitely many constants, the same $C_{\rm fd}$ works for all
$0\le i\le\ell$.  For each $i$ and $\eta$, fix once and for all a
compact fundamental domain $R_{i,\eta}$ constructed as in \Cref{lem.choiceoffunddom} and a finite-index
subgroup $\mathsf P'_{i,\eta}<\mathsf P_i$ such that
$R_{i,\eta}$ is a fundamental domain for
$g_i\mathsf P'_{i,\eta}g_i^{-1}$ and
\begin{equation}
\label{eqn.Randballs}
{\mathcal B_{\bar\psi}}(g_ix,\eta^{-1})
\subset R_{i,\eta}
\subset {\mathcal B_{\bar\psi}}(g_ix,C_{\rm fd}\eta^{-1})
\qquad(x\in\Delta_i).
\end{equation}

\subsection{Building blocks of the coding}
\label{subsection.buildingblocks}
Fix $0<\eta<\eta_*$, where $\eta_*<1/4$ will be decreased a finite
number of times below.  Until $\eta$ is fixed in
\Cref{prop.LP6.15ooo}, we retain it in the notation. For each parabolic point $p\in\Delta$, we construct a
compact neighborhood $J_{p,\eta}$ of $p$ that will serve as
a building block of the coding.
Let $p\in\Delta\cap\Lambda_{\mathrm p}(\Gamma_\rho)$
have type $i$, with top representative $\gamma_p$ as in
\eqref{eqn.definitiontoprepresentative};
so $p=\gamma_pp_i$ and
$x_p=\gamma_p^{-1}p_i\in\Delta_i$.

Here and below, complements are taken in
$\Lambda_\rho\smallsetminus\{p_0\}$ and closures in
$\Lambda_\rho$.

\begin{Def}[$J_p$ and $N_p$]
\label{jpnp}
For $p\in\Delta\cap\Gamma_\rho p_i$, define
\begin{equation}
\label{np0}
\begin{aligned}
&J_p=J_{p,\eta}
:=\overline{\Lambda_\rho\smallsetminus
\gamma_pg_i^{-1}R_{i,\eta}}
=\gamma_p\overline{g_i^{-1}R_{i,\eta}^{\,c}},\\
&N_p
:=\{\gamma_1\in\mathsf P_i:
\gamma_p\gamma_1\Delta_i\subset J_p\}= \{\gamma_1\in\mathsf P_i:
\gamma_1\Delta_i\subset\overline{ g_i^{-1} R_{i,\eta}^{\,c}}\}.
\end{aligned}
\end{equation}
\end{Def}
The set $J_{p,\eta}$ is the image of the closed neighborhood
$\overline{R_{i,\eta}^{\,c}}$ of $p_0$ under
$\gamma_pg_i^{-1}$, which sends $p_0$ to $p$.

For $\gamma_1\in N_p$, the invariance \eqref{inv}
applied in the conjugate system gives
\begin{equation}
\label{eqn.LP6.6}
D_{\bar\psi}\bigl(g_i(\gamma_p\gamma_1)^{-1}p_i,
g_i\Delta_i\bigr)
=
D_{\bar\psi}(g_ix_p,g_i\gamma_1\Delta_i)
\ge\eta^{-1}.
\end{equation}
The last inequality follows from \eqref{eqn.Randballs},
since $\mathcal B_{\bar\psi}(g_ix_p,\eta^{-1})
\subset R_{i,\eta}$.

\begin{lemma}[Peripheral tiling of $J_p$]
\label{lem.tilingJp}
For fixed $\eta$, $0\le i\le\ell$, and
$p\in\Delta\cap\Gamma_\rho p_i$, we have:
\begin{enumerate}
\item The complement $\mathsf P_i\smallsetminus N_p$ is finite.
\item There is a countable set
$
\mathcal E_p
\subset
\{p\}\cup
\bigcup_{\gamma_1\in\mathsf P_i}
\gamma_p\gamma_1\partial\Delta_i
$
such that
$$
J_p\smallsetminus\mathcal E_p
=
\bigsqcup_{\gamma_1\in N_p}
\gamma_p\gamma_1\operatorname{int}\Delta_i.
$$
In particular,
$
\nu\left(
J_p\mathbin{\triangle}
\cup_{\gamma_1\in N_p}
\gamma_p\gamma_1\Delta_i
\right)=0
$
where $\triangle$ denotes the symmetric difference.
\end{enumerate}
\end{lemma}
\begin{proof} All but finitely many tiles lie in
$\overline{R_{i,\eta}^{\,c}}$, so the definition of
$N_p$ implies that $\mathsf P_i\smallsetminus N_p$
is finite. The map $\gamma_pg_i^{-1}$ sends each tile
$g_i\gamma_1\Delta_i$ to $\gamma_p\gamma_1\Delta_i$
and sends $p_0$ to $p$. This gives the asserted
decomposition, with $\mathcal E_p$ consisting of
$p$ and the endpoints of these images.
\end{proof}

We compare $J_{p,\eta}$ with $D_{\bar\psi}$-balls
centered at $p$ whose radii are proportional to $\eta h(p)$.

\begin{lemma}
\label{lem.LP6.13}
\label{lem.LP6.7}
There exist $\eta_*>0$, $0<C_1<1$, and $C_2>1$ such that,
for every $0<\eta<\eta_*$, $0\leq i\leq\ell$, and
$p\in\Delta\cap\Gamma_\rho p_i$, we have
$$
\mathcal B_{\bar\psi}\left(
g_ip,\frac{C_1\eta h_i(g_ip)}{C_{\rm fd}}
\right)
\subset g_iJ_{p,\eta}
\subset
\mathcal B_{\bar\psi}\left(
g_ip,\frac{\eta h_i(g_ip)}{C_1}
\right),
$$
$$
\mathcal B_{\bar\psi}\left(p,\frac{\eta h(p)}{C_2}\right)
\subset J_{p,\eta}
\subset\mathcal B_{\bar\psi}\left(p,C_2\eta h(p)\right).
$$
\end{lemma}
\begin{proof}
Fix $i$ and $p\in\Delta\cap\Gamma_\rho p_i$. Recall that
$\alpha_p:=g_i\gamma_pg_i^{-1}$ satisfies
$\alpha_pp_0=g_ip$ and
$\alpha_p^{-1}p_0=g_ix_p\in g_i\Delta_i$.
By \Cref{lem.ballflipping}, applied to $\Gamma_i$,
\begin{equation}
\label{eqn.ballflipconjugate}
D_{\bar\psi}(\alpha_p^{-1}y,g_ix_p)
\asymp
\frac{h_i(g_ip)}{D_{\bar\psi}(y,g_ip)}
\qquad
\bigl(y\in g_i\Lambda_\rho\smallsetminus\{p_0,g_ip\}\bigr).
\end{equation}

If $y\in g_iJ_{p,\eta}\smallsetminus\{g_ip\}$, then
$\alpha_p^{-1}y\in\overline{R_{i,\eta}^{\,c}}$.
Since $x_p\in\Delta_i$, the first inclusion in
\eqref{eqn.Randballs} gives
$
D_{\bar\psi}(\alpha_p^{-1}y,g_ix_p)\ge\eta^{-1}.
$
Together with \eqref{eqn.ballflipconjugate}, this implies
$D_{\bar\psi}(y,g_ip)\ll\eta h_i(g_ip).
$
Conversely, if
$y\in g_i\Lambda_\rho\smallsetminus
\bigl(g_iJ_{p,\eta}\cup\{p_0\}\bigr)$, then
$\alpha_p^{-1}y\in R_{i,\eta}$, so the second inclusion in
\eqref{eqn.Randballs} gives
$
D_{\bar\psi}(\alpha_p^{-1}y,g_ix_p)
<C_{\rm fd}\eta^{-1}.
$
Hence
$
D_{\bar\psi}(y,g_ip)
\gg \frac{\eta h_i(g_ip)}{C_{\rm fd}}.
$
Choosing $C_1>0$ sufficiently small proves the first pair of
inclusions.
By \Cref{lem.heightondelta,lem.LP5.11},
$h_i(g_ip)\asymp h(p)\ll1$ uniformly.
For sufficiently small $\eta_*$, \Cref{lem.LP5.12} therefore
gives uniform bi-Lipschitz bounds on the relevant neighborhoods.
Pulling back the first pair of inclusions by $g_i^{-1}$ yields
the second pair, with a uniform constant $C_2>1$.
\end{proof}
\subsection{The inductive partition}
Let $\nu_i:=(g_i)_*\nu$; this is the unique $(\Gamma_i, \psi)$-Patterson-Sullivan measure on $g_i\La_\rho$.
Apply \Cref{lem.olive} to each conjugated representation
$\gamma\mapsto g_i\rho(\gamma)g_i^{-1}$, with the same
form $\bar\psi$ and distinguished point $g_ip_i=p_0$.
Choose $\mathsf c\ge1$ larger than all the resulting constants.

\begin{lemma}
\label{lem.LP5.21}
There exists $\mathtt H>\mathsf c^2$ such that, for every
$0\le i\le\ell$,
$p\in\Lambda_{\mathrm p}(\Gamma_i)\smallsetminus\{p_0\}$,
and $0<r\le h_i(p)$,
\begin{equation}
\label{eqn.annularH}
\nu_i\bigl(\mathcal B_{\bar\psi}(p,r/\sqrt{\mathtt H})\bigr)
\le \frac12\,
\nu_i\bigl(\mathcal B_{\bar\psi}(p,r/\mathsf c)\bigr).
\end{equation}
\end{lemma}
\begin{proof} To apply \Cref{eqn.LP5.16}, first move the center $p$
into the compact fundamental domain: choose
$\gamma\in g_i\mathsf P_i g_i^{-1}$ such that
$q:=\gamma^{-1}p\in g_i\Delta_i$.
By invariance, $h_i(q)=h_i(p)$ and
$\gamma^{-1}\mathcal B_{\bar\psi}(p,s)
=\mathcal B_{\bar\psi}(q,s)$.
Since $r\le h_i(q)\ll1$, the balls under consideration
have uniformly bounded diameter.
Applying \Cref{lem.LP5.18} in $\Gamma_i$ to both balls,
with the same center $p$, gives
$$
\frac{\nu_i(\mathcal B_{\bar\psi}(p,r/\sqrt{\mathtt H}))}
     {\nu_i(\mathcal B_{\bar\psi}(p,r/\mathsf c))}
\asymp
\frac{\nu_i(\mathcal B_{\bar\psi}(q,r/\sqrt{\mathtt H}))}
     {\nu_i(\mathcal B_{\bar\psi}(q,r/\mathsf c))}.
$$
By \Cref{eqn.LP5.16} applied in $\Gamma_i$, the
right-hand side is
$\ll(\mathsf c/\sqrt{\mathtt H})^a$ for some uniform
$a>0$. Hence
taking $\mathtt H$ sufficiently large
proves the claim.
\end{proof}

Fix $\mathtt H$ as in \Cref{lem.LP5.21} and set
$$
h_n:=\mathtt H^{-n}\qquad(n\ge0).
$$

For $q=g_ip$ with $p\in\Delta$, write
$J_{\Gamma_i,q,\eta}:=g_iJ_{p,\eta}$.
Take the constants $C_2$ and $C_s$ in
\Cref{lem.LP6.13,lem.LP6.16} uniformly over the
finitely many groups $\Gamma_i$, and decrease
$\eta_*$ if necessary so that $4C_2\eta_*<1$.

\begin{Def}
\label{def.ZnPnDn}
For $0\le i\le\ell$ and $0<\eta<\eta_*$, set
$\mathcal Z_{i,0}(\eta):=\Delta_i$.
For $n\ge1$, define inductively
$$
P_{i,n}(\eta)
:=
\left\{
p\in\Delta_i\cap\Lambda_{\mathrm p}(\Gamma_\rho):
\begin{array}{l}
h_n<\eta h_i(g_ip)\le h_{n-1},\, 
\mathcal B_{\bar\psi}\left(g_ip,\dfrac{h_{n-1}}{4\eta}\right)
\subset g_i\mathcal Z_{i,n-1}(\eta)
\end{array}
\right\},
$$
where 
$\mathcal B_{\bar\psi}(g_ip,r)
:=\{q\in g_i\Lambda_\rho\smallsetminus\{p_0\}:
D_{\bar\psi}(g_ip,q)<r\}$. Thus $P_{i,n}(\eta)$ consists of the centers of the
sets $J_{p,\eta}$ removed from $\Delta_i$ at stage $n$. Set
$$
D_{i,n}(\eta)
:=
\bigcup_{p\in P_{i,n}(\eta)}J_{p,\eta},
\qquad
\mathcal Z_{i,n}(\eta)
:=
\mathcal Z_{i,n-1}(\eta)\smallsetminus D_{i,n}(\eta).
$$
For $i=0$, abbreviate
$P_n(\eta):=P_{0,n}(\eta)$,
$D_n(\eta):=D_{0,n}(\eta)$, and
$\mathcal Z_n(\eta):=\mathcal Z_{0,n}(\eta)$.
\end{Def}

For $p\in P_{i,n}(\eta)$, the ball condition implies
$p\in\mathcal Z_{i,n-1}(\eta)$.
By \Cref{lem.LP6.13}, applied to $\Gamma_i$,
$$
g_iJ_{p,\eta}
\subset
\mathcal B_{\bar\psi}(g_ip,C_2\eta h_i(g_ip))
\subset
\mathcal B_{\bar\psi}\left(g_ip,\frac{h_{n-1}}{4\eta}\right)
\subset g_i\mathcal Z_{i,n-1}(\eta).
$$
Thus $D_{i,n}(\eta)\subset\mathcal Z_{i,n-1}(\eta)$,
and all three sets lie in $\Delta_i$.

Moreover, \Cref{lem.LP6.16}, applied to $\Gamma_i$,
gives $D_{\bar\psi}(g_ip,g_ip')>C_sh_n/\eta$
for distinct $p,p'\in P_{i,n}(\eta)$.
Since $g_i\Delta_i$ is compact, $P_{i,n}(\eta)$ is finite.
Hence $D_{i,n}(\eta)$ is a finite union of the sets $J_{p,\eta}$.
The following proposition will be proved in \Cref{sec.expdecayomega}.
\begin{proposition}[Exponential decay of
residual mass]
\label{prop.LP6.15ooo}
For all sufficiently small $\eta>0$, there exists
$\varepsilon_0>0$ such that, for every $0\le i\le\ell$ and sufficiently large $n$,
$$
\nu\bigl(\mathcal Z_{i,n}(\eta)\bigr)
\le(1-\varepsilon_0)^n.
$$
\end{proposition}

We take $\eta>0$ sufficiently small for the preceding proposition
and all requirements above, allowing further reductions until
its final choice in the proof of \Cref{lem.LP6.62}.
We suppress $\eta$ from the notation:
$$
P_{i,n}:=P_{i,n}(\eta),
\quad
J_p:=J_{p,\eta},
\quad
D_{i,n}:=D_{i,n}(\eta),
\quad
\mathcal Z_{i,n}:=\mathcal Z_{i,n}(\eta),
$$
and write
$$
P_n:=P_{0,n},
\qquad
D_n:=D_{0,n},
\qquad
\mathcal Z_n:=\mathcal Z_{0,n}.
$$
We also abbreviate $R_i:=R_{i,\eta}$ and
$\mathsf P_i':=\mathsf P'_{i,\eta}$.

\subsection{Boundary conventions}
By construction, each $J_p$ has the form
$J_p=f(\widehat J_p\cap\Lambda(\Gamma))$
for a closed interval
$\widehat J_p\subset\partial\mathbb H^2$.
Define
$\partial J_p:=f(\partial\widehat J_p\cap\Lambda(\Gamma))$.
This set contains at most two points and is $\nu$-null.
For $p$ of type $i$, we have
$g_i\gamma_p^{-1}\partial J_p=\partial R_i$.

For $n\ge0$, define
\begin{equation}
\label{eqn.bdrdef}
\partial\mathcal Z_n
:=
\partial\Delta_0
\cup
\bigcup_{k=1}^{n}\ \bigcup_{p\in P_k}\partial J_p,
\end{equation}
where the union is empty when $n=0$.

We call the sets $\partial\Delta_0$ and $\partial J_p$
the \emph{boundary pieces}; a piece may contain two endpoints.
Thus $\partial\mathcal Z_n$ denotes their union, rather than a
family of connected components. The separation statement in
\Cref{lem.LP6.17} identifies the pieces relevant to the residual
set $\mathcal Z_n$.

\section{Exponential decay of the residual mass}
\label{sec.expdecayomega}

\Cref{prop.LP6.15ooo} follows from the case $i=0$
stated below. Indeed, applying this case to each
$\Gamma_i$ with measure $\nu_i=(g_i)_*\nu$ gives
the desired estimate, since
$\nu_i(g_i\mathcal Z_{i,n}(\eta))
=\nu(\mathcal Z_{i,n}(\eta))$.

\begin{proposition}
\label{prop.LP6.15}
For every sufficiently small $\eta>0$, there exists $\varepsilon_0>0$
such that, for all sufficiently large $n$,
$\nu(\mathcal Z_n)\leq(1-\varepsilon_0)^n.
$
\end{proposition}

We classify parabolic limit points according to their
proximity to the boundary pieces of the residual sets
$\mathcal Z_{i,n}(\eta)$.  The resulting geometric estimates,
combined with annular mass loss and recurrence of the reparametrized
flow, show that $D_{n+1}\cup\cdots\cup D_{n+N}$ removes a fixed
proportion of the residual mass, yielding uniform contraction
over blocks of $N$ generations.

For $n\geq1$, define
$$
Q_n:=
\left\{
p\in\Lambda_{\mathrm p}(\Gamma_\rho)\smallsetminus\{p_0\}:
\begin{array}{l}
h_n<\eta h(p)\leq h_{n-1}, p\notin P_n \\
{\mathcal B_{\bar\psi}}(p,\eta h(p))\cap\mathcal Z_{n-1}\neq\emptyset
\end{array}
\right\}.
$$
Set
$$
r(p,n):=\sqrt{\eta h(p)h_n},\qquad 
Q_n':=
\{p\in Q_n:{\mathcal B_{\bar\psi}}(p,r(p,n))\cap\partial\mathcal Z_{n-1}
\neq\emptyset\}.
$$

By \Cref{hd} and \Cref{lem.heightondelta}, the height $h$ is
uniformly bounded. Since
$\mathcal B_{\bar\psi}(p,\eta h(p))$ meets
$\mathcal Z_{n-1}\subset\Delta_0$ for every $p\in Q_n$,
all the sets $Q_n$ lie, for sufficiently small $\eta$,
in a fixed compact neighborhood of $\Delta_0$ avoiding
$p_0,\ldots,p_\ell$. On this neighborhood, \Cref{lem.heightondelta0}, applied
both to $\Gamma_\rho$ and to the conjugated groups
$\Gamma_i=g_i\Gamma_\rho g_i^{-1}$, together with
$g_i\in K$, gives
$
h_i(g_ip)\asymp h(p).
$
The maps $g_i$ also satisfy uniform bi-Lipschitz bounds there
by \Cref{lem.metrics}. We use both facts below without
further mention.

\subsection{A previous-generation boundary piece}
We first consider $p\in Q_n'$ for which
$\mathcal B_{\bar\psi}(p,r(p,n))$ meets $\partial J_q$
for some $q\in\bigcup_{k<n}P_k$.
Write $q=\gamma_qp_i$, where $\gamma_q$ is as in
\eqref{eqn.definitiontoprepresentative}, and 
$x_q=\gamma_q^{-1}p_i\in\Delta_i$.

\begin{lemma}
\label{lem.LP6.18}
There exists $C>1$ such that
$$
\eta h(p)\leq h_{n-1}\leq C\eta^3h(q)\;\;\text{and}\;\; 
C^{-1}\eta h_i(g_ip)\leq h_{n-1}
\leq C\eta^3h_i(g_iq).
$$
\end{lemma}

\begin{proof}
 It suffices to prove the first statement; the height comparison
$h_i(g_ip)\asymp h(p)$ then gives the second.
    Since $p \in Q_n$, we have $h_n < \eta h(p) \le h_{n-1}$, showing the first inequality.
    By  \Cref{lem.LP6.16}, we have
    \be \label{eqn.Qn}
    D_{\bar\psi}(p, q) > C_s \sqrt{h(p)h(q)} >  C_s \sqrt{ {h_{n} h(q)}{\eta}^{-1}}.
    \ee
    On the other hand, since $D_{\bar\psi}(p, \partial J_q) < r(p, n)$, we have for some $x \in \partial J_q$ that
    $$D_{\bar\psi}(p, q) \le N_0(D_{\bar\psi}(p, x) + D_{\bar\psi}(x, q)) \le N_0(r( p, n) + {C}_2 \eta h(q))$$
    by  \Cref{lem.eucdistancetriangle}, where the second inequality follows by  \Cref{lem.LP6.13}. Since $p \in Q_n$ and  $q \in P_1 \cup \cdots \cup P_{n-1}$,
    we have $r( p, n) \le \sqrt{h_{n-1}h_n} \le h_{n-1} < \eta h(q).$
    Therefore, we have
    $$D_{\bar\psi}(p, q) \le N_0(1 + {C}_2) \eta h(q).$$
    Together with \eqref{eqn.Qn}, we have
    $
 C_s \sqrt{ {h_{n-1} h(q)}{({\mathtt{H}} \eta})^{-1}} = C_s \sqrt{ {h_{n} h(q)}{\eta}^{-1}} \le N_0(1 + {C}_2) \eta h(q),$
    and hence
    $ h_{n-1} \le C \eta^3 h(q)$ with $C:=\frac{N_0^2 (1+{C}_2)^2 {\mathtt{H}}}{C_s^2}$.
\end{proof}

\begin{lemma}
\label{lem.LP6.19}
There exists $C>1$ such that
$$
{\mathcal B_{\bar\psi}}(g_ip,h_i(g_ip))
\subset
{\mathcal B_{\bar\psi}}(g_iq,C\eta h_i(g_iq))
\smallsetminus
{\mathcal B_{\bar\psi}}\left(g_iq,{C^{-1}\eta h_i(g_iq)}\right).
$$
\end{lemma}

\begin{proof}
    Let $x \in {\mathcal B_{\bar\psi}}(g_i p, h_i(g_i p))$. Then by  \Cref{lem.eucdistancetriangle},  \Cref{lem.LP5.12}, and  \Cref{lem.LP6.18}, there exist $c', c > 0$, depending only on $\bar\psi$, such that
    \be  \label{eqn.LP6.19pf}
    \begin{aligned}
    D_{\bar\psi}(g_i q, x)  & \le N_0(D_{\bar\psi}(g_i q, g_i p) + D_{\bar\psi}(g_i p, x))
    \le N_0(D_{\bar\psi}(g_i q, g_i p) + h_i(g_i p))
   \\ & \le   N_0(c' D_{\bar\psi}(q, p) + h_i(g_i p)) \le
      N_0(c' D_{\bar\psi}(q, p) + c \eta^2 h_i(g_i q))
     .\end{aligned}
     \ee

Since $p \in Q_n$,
we have $$D_{\bar\psi}(p, \partial J_q) < r(p, n) = \sqrt{ \eta h(p) h_n}\le
\sqrt{ h_{n-1} h_{n}} \le h_{n-1}.$$
By \Cref{lem.LP6.13}, \Cref{lem.eucdistancetriangle},
 and \Cref{lem.LP6.18},
there exists $c''>0$ such that
    $$D_{\bar\psi}(q, p) \le N_0( {C}_2 \eta h(q) + h_{n-1})\le  N_0 c'' \eta h_i(g_i q).$$
   Combining this with \eqref{eqn.LP6.19pf} yields the inclusion
   $${\mathcal B_{\bar\psi}}(g_i p, h_i(g_i p)) \subset {\mathcal B_{\bar\psi}}(g_i q, C \eta h_i(g_i q))\quad\text{ for some $C > 1$.}$$

 Now  by  \Cref{lem.eucdistancetriangle}, we have
    $$\begin{aligned}
    D_{\bar\psi}(g_i q, g_i(\partial J_q)) & \le N_0^2 (D_{\bar\psi}(g_i q, x) + D_{\bar\psi}(x, g_i p) + D_{\bar\psi}(g_i p, g_i(\partial J_q) ))\\
    & \le N_0^2 (D_{\bar\psi}(g_i q, x) + h_i(g_i p) + D_{\bar\psi}(g_i p, g_i(\partial J_q) )).
    \end{aligned}$$
 Again, by the bi-Lipschitz property of $g_i$ (\Cref{lem.LP5.12}),
    $$D_{\bar\psi}(g_ip, g_i(\partial J_q) ) \le c' D_{\bar\psi}(p, \partial J_q) < c' r(p, n) \le c' h_{n-1}.$$
    By  \Cref{lem.LP6.18}, this implies that  for some uniform constant $\widetilde c >  0$,
    $$h_i(g_i p) + D_{\bar\psi}(g_i p, g_i(\partial J_q) ) \le \widetilde c \eta^2 h_i(g_i q).$$
    Since $D_{\bar\psi}(g_i q, g_i(\partial J_q) ) \ge {C}_1 \eta h_i(g_i q)/C_{\rm fd}$ by  \Cref{lem.LP6.13}, we finally have that
    $$\left(\frac{{C}_1}{C_{\rm fd} N_0^2} - \widetilde c \eta \right) \eta h_i(g_i q) \le D_{\bar\psi}(g_i q, x).$$
     Since $x \in {\mathcal B_{\bar\psi}}(g_i p, h_i(g_i p))$ is arbitrary,
     it follows that $${\mathcal B_{\bar\psi}}(g_i p, h_i(g_i p))\cap  {\mathcal B_{\bar\psi}}\left(g_i q, C^{-1}{\eta h_i(g_i q)}\right)=\emptyset$$ for some $C > 1$.
    This completes the proof. \end{proof}
\begin{lemma}
\label{lem.LP6.20}
We have
$$
h_i(g_i\gamma_q^{-1}p)
\asymp
\frac{h_i(g_ip)}{\eta^2h_i(g_iq)}.
$$
\end{lemma}

\begin{proof}
By the definition of $h_i$, the cocycle identity, and
\Cref{lem.buseandcartan,lem.gromovproductalongparabolic},
$$
h_i(g_i\gamma_q^{-1}p)
\asymp
\frac{h_i(g_iq)h_i(g_ip)}
{d_{\bar\psi}(g_ip,g_iq)^2}.
$$
The points $p$ and $q$ lie in a fixed compact neighborhood
of $\Delta_0$ avoiding $p_0,\ldots,p_\ell$.
Hence $g_ip$ and $g_iq$ lie in a fixed compact subset
of $g_i\Lambda_\rho\smallsetminus\{p_0\}$, on which
$d_{\bar\psi}\asymp D_{\bar\psi}$.
Taking $x=g_ip$ in \Cref{lem.LP6.19} therefore gives
$d_{\bar\psi}(g_ip,g_iq)\asymp\eta h_i(g_iq)$.
Substitution proves the lemma.
\end{proof}

Define
\begin{equation}
\label{eqn.LP6.21}
\widetilde r_{p,n}
:=
\sqrt{\frac{h_nh_i(g_i\gamma_q^{-1}p)}
{\eta h_i(g_iq)}}.
\end{equation}
By \Cref{lem.LP6.20} and the assumption that $p\in Q_n$,
we have \be \label{eqn.LP6.22}
\widetilde r_{p, n} \ll \frac{ \sqrt{ \eta h_n h_i(g_i p)}}{\eta^2 h_i(g_i q)}
\ll \frac{h_i(g_i p)}{\eta h_i(g_i q)}\ll \eta h_i(g_i \ga_q^{-1} p).
\ee
In particular, $\widetilde r_{p,n}\ll\eta$, since
$h_i(g_i\gamma_q^{-1}p)$ is invariant under
$g_i\mathsf P_i g_i^{-1}$ and is bounded above on
a fixed fundamental domain.
\begin{lemma}
\label{lem.LP6.23}
There exists ${C}_3>1$ such that every $p'\in\Lambda_\rho$ with
$D_{\bar\psi}(g_i\gamma_q^{-1}p',
g_i\gamma_q^{-1}\partial J_q)
\leq\frac{\widetilde r_{p,n}}{{C}_3}$
satisfies $ D_{\bar\psi}(p',\partial J_q)\leq h_n. $
\end{lemma}
\begin{proof}
Choose $x\in\partial J_q$ realizing the distance in the hypothesis.
By \Cref{lem.metrics},
\begin{equation}
\label{eqn.6.23proof eqn0}
D_{\bar\psi}(g_ip',g_ix)
\asymp
\frac{
D_{\bar\psi}(g_i\gamma_q^{-1}p',g_i\gamma_q^{-1}x)
}{
D_{\bar\psi}(g_i\gamma_q^{-1}p',g_ix_q)
D_{\bar\psi}(g_ix_q,g_i\gamma_q^{-1}x)
}
h_i(g_iq).
\end{equation}
Here we use \Cref{lem.representshadow,lem.heightondelta,lem.LP5.11}
to obtain
$e^{\bar\psi(\beta_{p_0}(e,g_i\gamma_qg_i^{-1}))}
\asymp e^{-\bar\psi(\mu(\gamma_q))}
\asymp h_i(g_iq)$.

Since $g_i\gamma_q^{-1}\partial J_q=\partial R_{i}$,
\eqref{eqn.Randballs} gives
$D_{\bar\psi}(g_ix_q,g_i\gamma_q^{-1}x)\asymp\eta^{-1}$.
The hypothesis and \eqref{eqn.LP6.22} give
$D_{\bar\psi}(g_i\gamma_q^{-1}p',g_i\gamma_q^{-1}x)
\leq\widetilde r_{p,n}/C_3\ll\eta/C_3$.
Thus \Cref{lem.eucdistancetriangle} implies
$D_{\bar\psi}(g_i\gamma_q^{-1}p',g_ix_q)\asymp\eta^{-1}$
for sufficiently large $C_3$.
Substituting these bounds into \eqref{eqn.6.23proof eqn0}
and using \eqref{eqn.LP6.22}, we obtain
$$
D_{\bar\psi}(g_ip',g_ix)
\ll
\frac{\eta^2h_i(g_iq)}{C_3}\,\widetilde r_{p,n}
\ll
\frac{\eta h_i(g_ip)}{C_3}.
$$
Since $x\in\partial J_q\subset\Delta_0$ and the preceding
bound is $O(\eta)$, $p'$ lies in a fixed compact neighborhood
of $\Delta_0$ where the inverse coordinate change is
uniformly bi-Lipschitz.
Using \Cref{lem.metrics} and
$\eta h(p)\leq h_{n-1}=\mathtt Hh_n$, we obtain
$D_{\bar\psi}(p',x)\ll\mathtt Hh_n/C_3$.
Increasing $C_3$ proves the lemma.
\end{proof}

\subsection{Separation of boundary pieces}
We next estimate the separation of the sets $J_p$, $p\in P_n$,
from one another and from $\mathcal Z_{n-1}^{\,c}$.

\begin{lemma}
\label{lem.LP6.17}
After decreasing $\eta_*$, for every $n\geq1$ and distinct
$p,p'\in P_n$,
$$
D_{\bar\psi}(J_p,J_{p'})
>
\frac{C_s}{2\eta N_0^2}h_n\quad\text{and}\quad
D_{\bar\psi}(J_p,\mathcal Z_{n-1}^{\,c})
>
\frac{C_s}{2\eta N_0^2}h_n.
$$
Consequently, the sets $J_p$, $p\in\bigcup_{n\geq1}P_n$, are pairwise
disjoint.
\end{lemma}

\begin{proof}
For $x\in J_p$ and $y\in J_{p'}$, \Cref{lem.LP6.16} and
\Cref{lem.LP6.13} give
$$
C_s\sqrt{h(p)h(p')}
<
N_0^2\bigl(D_{\bar\psi}(x,y)+{C}_2\eta h(p)
+{C}_2\eta h(p')\bigr).
$$
Since $\eta h(p),\eta h(p')\in(h_n,h_{n-1}]$, it follows that
$
D_{\bar\psi}(x,y)
>
h_n\left(\frac{C_s}{\eta N_0^2}
-2\mathtt H{C}_2\right),
$
which is at least the claimed lower bound for small $\eta$. 

If $x\in\mathcal Z_{n-1}^{\,c}$ and $y\in J_p$, then the definition of
$P_n$ and \Cref{lem.LP6.13} give
$
\frac{h_{n-1}}{4\eta}
\leq
N_0\bigl(D_{\bar\psi}(x,y)+{C}_2\eta h(p)\bigr).
$
The same choice of $\eta_*$ gives the second estimate.  If $p\in P_j$,
$p'\in P_k$, and $j<k$, then $J_p\subset\mathcal Z_{k-1}^{\,c}$, so the
second estimate applied to $J_{p'}$ proves cross-generation disjointness.
\end{proof}

\begin{lemma}
\label{lem.uniqueassociatedboundary}
For every $p\in Q_n'$, the ball
$\mathcal B_{\bar\psi}(p,r(p,n))$ meets exactly one of the sets
$\partial\Delta_0$ and $\partial J_q$,
$q\in\bigcup_{k=1}^{n-1}P_k$.
\end{lemma}
\begin{proof}
The ball meets at least one of these sets by definition.
If it met two distinct sets $E$ and $F$, then
\Cref{lem.eucdistancetriangle} would give
$D_{\bar\psi}(E,F)\leq2N_0r(p,n)\leq2N_0h_{n-1}$.
On the other hand, \Cref{lem.LP6.17} gives
$D_{\bar\psi}(E,F)\geq C_sh_{n-1}/(2\eta N_0^2)$,
since $h_j\geq h_{n-1}$ for $j\leq n-1$.
For sufficiently small $\eta$, these bounds contradict each other.
\end{proof}

\subsection{Equivalence classes}

We construct the sets $C(p)$, the radii $r_{p,m}$, and the
neighborhoods $B_{p,m}$ inductively.
At stage $n\geq1$, consider
$p\in Q_n\smallsetminus\bigcup_{k<n}Q_k''$,
where the union is empty when $n=1$.

\noindent\textbf{Case 1: $p\in Q_n'$.}
By \Cref{lem.uniqueassociatedboundary},
$\mathcal B_{\bar\psi}(p,r(p,n))$ meets exactly one of
$\partial\Delta_0$ and $\partial J_q$,
$q\in\bigcup_{k<n}P_k$.

\noindent\textbf{Case 1-1:
$D_{\bar\psi}(p,\partial\Delta_0)<r(p,n)$.}
Set
$$
C(p):=
\left\{
\gamma_1p:
\gamma_1\in\mathsf P_0,\ 
\gamma_1\mathcal B_{\bar\psi}(p,r(p,n))
\cap\partial\Delta_0\neq\emptyset
\right\}.
$$
For $p'\in C(p)$ and $m\geq n$, define
$$
r_{p',m}:=\sqrt{\eta h_mh(p')},
\qquad
B_{p',m}:=\mathcal B_{\bar\psi}(p',r_{p',m}).
$$

\noindent\textbf{Case 1-2:
$D_{\bar\psi}(p,\partial J_q)<r(p,n)$.}
Write $q=\gamma_qp_i$ and use $\widetilde r_{p,n}$ and $C_3$
from \eqref{eqn.LP6.21} and \Cref{lem.LP6.23}.
If
$\mathcal B_{\bar\psi}
(g_i\gamma_q^{-1}p,\widetilde r_{p,n}/C_3)$
meets $g_i\gamma_q^{-1}\partial J_q$, set
\begin{equation}
\label{eqn.defequivclass}
C(p):=
\left\{
(g_i\gamma_q^{-1})^{-1}\gamma_1g_i\gamma_q^{-1}p:
\begin{array}{l}
\gamma_1\in g_i\mathsf P_i'g_i^{-1},
\gamma_1\mathcal B_{\bar\psi}
(g_i\gamma_q^{-1}p,\widetilde r_{p,n}/C_3)
\cap g_i\gamma_q^{-1}\partial J_q\neq\emptyset
\end{array}
\right\}.
\end{equation}
Otherwise, set $C(p):=\{p\}$.
For $p'\in C(p)$ and $m\geq n$, define
$$
r_{p',m}:=\frac1{C_3}
\sqrt{\frac{h_mh_i(g_i\gamma_q^{-1}p')}
{\eta h_i(g_iq)}},
\qquad
B_{p',m}:=(g_i\gamma_q^{-1})^{-1}
\mathcal B_{\bar\psi}(g_i\gamma_q^{-1}p',r_{p',m}).
$$

\noindent\textbf{Case 2.}
For every remaining point
$p\in Q_n\smallsetminus\bigcup_{k<n}Q_k''$
not covered by the classes above, set $C(p):=\{p\}$
and define $r_{p,m}$ and $B_{p,m}$ as in Case 1-1.

In each case, $r_{p',m}=r_{p,m}$ for $p'\in C(p)$,
by invariance of the corresponding height under the peripheral
group.
For each member $p'$ of a constructed class $C(p)$, we write
$C(p'):=C(p)$; the consistency of these assignments is proved
in \Cref{lem.LP6.35}.

Finally, set
$$
Q_n'':=
\bigcup_{p\in Q_n\smallsetminus\bigcup_{k<n}Q_k''}C(p).
$$
Then $\bigcup_{j\leq n}Q_j\subset\bigcup_{j\leq n}Q_j''$.
A class constructed at the $n$-th step is said to have
\emph{generation} $n$.
We call $\partial\Delta_0$ in Case 1-1 and $\partial J_q$
in Case 1-2 the \emph{associated boundary} of the class.
Classes constructed in Case 2 have no associated boundary.

\subsection{Uniformity among equivalence classes}

\begin{lemma}   \label{lem.LP6.27}
    There exists ${C}_4 > 1$
    such that for any $p \in Q_n \smallsetminus \bigcup_{k \le n-1} Q_{k}''$ and  $p' \in C(p)$, we have
    $${{C}_4}^{-1} \le \frac{h(p)}{h(p')} \le {C}_4.$$
\end{lemma}

The assertion is immediate for singleton classes. In Case 1-1, $p'\in\mathsf P_0p$ and $h$ is
$\mathsf P_0$-invariant by \Cref{hd}, so $h(p')=h(p)$.
It therefore remains to consider Case 1-2, with associated
boundary $\partial J_q$, where $q=\gamma_qp_i$.
We first establish the following estimate.

\begin{lemma} \label{lem.LP6.28}
    There exists $C>1$ such that for any $p\in Q_n' \smallsetminus \bigcup_{k \le n-1} Q_{k}''$ satisfying the hypothesis of Case 1-2 (i.e., ${\mathcal B_{\bar\psi}}(p, r(p, n)) \cap \partial J_q \neq \emptyset$
    for  some $q \in P_1 \cup \cdots \cup P_{n-1}$),
     any $p' \in C(p)$ and any $m \ge n$, we have
    $$
    \begin{aligned}
       & {\mathcal B_{\bar\psi}}(g_i \ga_q^{-1} p', r_{p', m})  \subset {\mathcal B_{\bar\psi}}(g_i x_q, C\eta^{-1}) \smallsetminus {\mathcal B_{\bar\psi}}(g_i x_q, (C \eta)^{-1}); \\
       & g_i B_{p', m}  \subset {\mathcal B_{\bar\psi}}(g_i q, C \eta h_i(g_i q)) \smallsetminus {\mathcal B_{\bar\psi}}(g_i q, C^{-1}\eta h_i(g_i q)).
    \end{aligned}
    $$
\end{lemma}
\begin{proof}
Since $g_i\gamma_q^{-1}\partial J_q=\partial R_{i}$,
\eqref{eqn.Randballs} gives
$D_{\bar\psi}(g_ix_q,z)\asymp\eta^{-1}$
for every $z\in g_i\gamma_q^{-1}\partial J_q$.

In the first alternative of Case 1-2,
\eqref{eqn.defequivclass} and \Cref{lem.metrics} imply that
$\mathcal B_{\bar\psi}(g_i\gamma_q^{-1}p',r_{p',n})$
meets $g_i\gamma_q^{-1}\partial J_q$.
In the singleton alternative, $p'=p$, and
\Cref{lem.LP6.19,lem.ballflipping} give
$D_{\bar\psi}(g_i\gamma_q^{-1}p,g_ix_q)\asymp\eta^{-1}$.
By \Cref{hd}, $h_i$ is $g_i\mathsf P_i g_i^{-1}$-invariant.
Together with \eqref{eqn.defequivclass}, this gives
$h_i(g_i\gamma_q^{-1}p')=h_i(g_i\gamma_q^{-1}p)$.
Hence, by the definitions of the radii and \eqref{eqn.LP6.22},
$r_{p',m}\leq r_{p',n}=r_{p,n}
=\widetilde r_{p,n}/C_3\ll\eta$.

Thus \Cref{lem.eucdistancetriangle} implies
$D_{\bar\psi}(z,g_ix_q)\asymp\eta^{-1}$
throughout
$\mathcal B_{\bar\psi}(g_i\gamma_q^{-1}p',r_{p',m})$.
This proves the first inclusion.
For every
$z\in\mathcal B_{\bar\psi}
(g_i\gamma_q^{-1}p',r_{p',m})$,
\Cref{lem.ballflipping} gives
$$
D_{\bar\psi}((g_i\gamma_qg_i^{-1})z,g_iq)
\asymp
\frac{h_i(g_iq)}{D_{\bar\psi}(z,g_ix_q)}
\asymp\eta h_i(g_iq).
$$
Since
$(g_i\gamma_qg_i^{-1})
\mathcal B_{\bar\psi}(g_i\gamma_q^{-1}p',r_{p',m})
=g_iB_{p',m}$
by definition, the second inclusion follows.
\end{proof}

\begin{proof}[Proof of \Cref{lem.LP6.27}]
It remains to consider Case 1-2.
Since $h_i(g_iq)\ll1$, \Cref{lem.LP6.28} places $p$ and $p'$
in a fixed compact neighborhood of $\Delta_0$ avoiding the
cusp representatives, after decreasing $\eta$.
In particular, \eqref{dedef} gives
$d_{\bar\psi}(g_ip',g_iq)\asymp D_{\bar\psi}(g_ip',g_iq)$.

The computation in the proof of \Cref{lem.LP6.20},
together with \Cref{lem.LP6.28}, therefore gives
\begin{equation}
\label{eqn.proofLP6.27}
h_i(g_i\gamma_q^{-1}p')
\asymp
\frac{h_i(g_iq)h_i(g_ip')}
     {d_{\bar\psi}(g_ip',g_iq)^2}
\asymp
\frac{h_i(g_ip')}{\eta^2h_i(g_iq)}.
\end{equation}
By the definition of $C(p)$,
$h_i(g_i\gamma_q^{-1}p')
=h_i(g_i\gamma_q^{-1}p)$.
Combining \eqref{eqn.proofLP6.27} with \Cref{lem.LP6.20}
yields $h_i(g_ip')\asymp h_i(g_ip)$.

Finally, \Cref{lem.heightondelta0}, applied to the original
and conjugated systems, and $g_i\in K$ give
$h(p')\asymp h_i(g_ip')$ and $h(p)\asymp h_i(g_ip)$
on these compact neighborhoods.
Hence $h(p')\asymp h(p)$, as required.
\end{proof}

\subsection{Size of the neighborhoods $B_{p,m}$}

\begin{lemma} \label{lem.LP6.32}
    There exists ${C}_5 > 1$
    such that for any $p \in Q_n''$ and any $m \ge n$,
    $${\mathcal B_{\bar\psi}}(p, \sqrt{ \eta h(p) h_m} /{C}_5) \subset B_{p, m} \subset {\mathcal B_{\bar\psi}}(p, {C}_5 \sqrt{ \eta h(p) h_m}).$$
\end{lemma}

\begin{proof}
The assertion follows directly from the definitions in Cases 1-1
and 2. Consider Case 1-2, with associated boundary $\partial J_q$,
where $q=\gamma_qp_i$ and $x_q=\gamma_q^{-1}p_i$.

By \Cref{lem.LP6.28}, $p$ lies in a fixed compact neighborhood
of $\Delta_0$ avoiding the cusp representatives.
On these neighborhoods, \Cref{lem.heightondelta0,lem.metrics}
give $h_i(g_ip)\asymp h(p)$ and uniform bi-Lipschitz bounds
for $g_i$. For $x\in\mathcal B_{\bar\psi}(g_i\gamma_q^{-1}p,r_{p,m})$,
\Cref{lem.ballflipping,lem.LP6.28} give
\begin{equation}
\label{eqn.pairingball}
D_{\bar\psi}(g_ip,(g_i\gamma_qg_i^{-1})x)
\asymp
\eta^2h_i(g_iq)\,
D_{\bar\psi}(g_i\gamma_q^{-1}p,x).
\end{equation}
Moreover, the definition of $r_{p,m}$ and
\eqref{eqn.proofLP6.27} imply
\begin{equation}
\label{eqn.prooflp6.32rpm}
r_{p,m}
\asymp
\frac{\sqrt{\eta h_i(g_ip)h_m}}{\eta^2h_i(g_iq)}.
\end{equation}
Since
$(g_i\gamma_qg_i^{-1})
\mathcal B_{\bar\psi}(g_i\gamma_q^{-1}p,r_{p,m})
=g_iB_{p,m}$,
these estimates give a uniform $C>1$ such that
$g_iB_{p,m}\subset
\mathcal B_{\bar\psi}(g_ip,C\sqrt{\eta h_i(g_ip)h_m})$. Pulling back by $g_i$ and using $h_i(g_ip)\asymp h(p)$
proves the second inclusion.

To prove the first inclusion, take
$x\in\mathcal B_{\bar\psi}
(g_ip,\sqrt{\eta h_i(g_ip)h_m}/C)$.
By \Cref{lem.LP6.27} and $m\geq n$,
$\sqrt{\eta h_i(g_ip)h_m}\ll h_n$.
Since $q\in\bigcup_{j<n}P_j$, we also have
$h_n\ll\eta h_i(g_iq)$.
Thus \Cref{lem.LP6.28,lem.eucdistancetriangle} imply
$D_{\bar\psi}(x,g_iq)\asymp\eta h_i(g_iq)$
for sufficiently large $C$.
Applying \Cref{lem.ballflipping} again, we obtain
\begin{equation}
\label{eqn.6.32 proof eqn 2}
\begin{aligned}
D_{\bar\psi}(g_i\gamma_q^{-1}p,
g_i\gamma_q^{-1}g_i^{-1}x)
&\asymp
\frac{h_i(g_iq)D_{\bar\psi}(g_ip,x)}
{D_{\bar\psi}(g_ip,g_iq)D_{\bar\psi}(x,g_iq)}
\ll
\frac{\sqrt{\eta h_i(g_ip)h_m}}
{C\eta^2h_i(g_iq)}
\asymp\frac{r_{p,m}}{C}.
\end{aligned}
\end{equation}
Increasing $C$ makes the last quantity smaller than $r_{p,m}$,
so $x\in g_iB_{p,m}$.
Pulling back the two inclusions by $g_i$ and using
$h_i(g_ip)\asymp h(p)$ proves the lemma.
\end{proof}

\subsection{Well-definedness of equivalence classes}
\begin{lemma}
\label{lem.LP6.35}
Any two equivalence classes are equal or disjoint.
\end{lemma}

\begin{proof}
Suppose that $C(p')\cap C(p'')\ne\emptyset$.
First consider the case when $C(p')=\{p'\}$.
If $C(p'')$ is also a singleton, the assertion is immediate.
Otherwise, let $k$ and $n$ be the generations at which
$C(p')$ and $C(p'')$ are defined.
Since $p'\in C(p'')\subset Q_n''$, the construction gives $k\le n$.

If the associated boundary of $C(p'')$ is $\partial\Delta_0$,
then $p'\in\mathsf P_0p''$.
The $\mathsf P_0$-invariance of $h$ (\Cref{hd}) gives
$h(p')=h(p'')$, hence $k=n$.
Moreover, $\mathcal B_{\bar\psi}(p',r(p',n))$ meets
$\partial\Delta_0$.
Thus $C(p')$ is also defined in Case~1-1, and
\Cref{lem.metrics} gives $C(p')=C(p'')$.

Suppose instead that the associated boundary is $\partial J_q$,
where $q=\gamma_qp_i$.
By \Cref{lem.LP6.18} and \Cref{lem.LP6.27},
$h(q)\gg h_k/\eta^3$.
For sufficiently small $\eta$, this implies
$q\in\bigcup_{j<k}P_j$.
By \eqref{eqn.defequivclass} and \Cref{lem.metrics},
the ball centered at $g_i\gamma_q^{-1}p'$ with radius
$\widetilde r_{p'',n}/C_3$ meets
$g_i\gamma_q^{-1}\partial J_q$.
Hence \Cref{lem.LP6.23} gives
$$
D_{\bar\psi}(p',\partial J_q)
\le h_n\le h_k<r(p',k).
$$
By \Cref{lem.uniqueassociatedboundary}, $C(p')$ is therefore
constructed in Case~1-2 with associated boundary $\partial J_q$.
The $g_i\mathsf P_i g_i^{-1}$-invariance of $h_i$
(\Cref{hd}) gives
$\widetilde r_{p',k}\ge
\widetilde r_{p',n}=\widetilde r_{p'',n}$.
Thus \eqref{eqn.defequivclass} implies
$C(p'')\subset C(p')$, contradicting the singleton assumption.

It remains to consider two classes containing at least two points.
Let their generations be $m\le n$, their associated boundaries
be $\partial_1$ and $\partial_2$, and choose
$p\in C(p')\cap C(p'')$.
The definitions and \Cref{lem.LP6.23} give
$D_{\bar\psi}(p,\partial_1)\le h_{m-1}$ and
$D_{\bar\psi}(p,\partial_2)\le h_{n-1}$.
By \Cref{lem.LP6.27},
$h(p')\asymp h(p'')$, so $h_{m-1}\asymp h_{n-1}$.
If $\partial_1\ne\partial_2$, then
\Cref{lem.LP6.17} and \Cref{lem.eucdistancetriangle} yield
$$
\frac{C_sh_{n-1}}{2\eta N_0^2}
\le D_{\bar\psi}(\partial_1,\partial_2)
\le N_0(h_{m-1}+h_{n-1})
\ll h_{n-1},
$$
a contradiction for small $\eta$.
Thus the associated boundaries agree.

If the common boundary is $\partial\Delta_0$, then
$h(p')=h(p)=h(p'')$, so $m=n$.
Both classes are consequently defined by the same
$\mathsf P_0$-orbit and the same radius, and hence are equal.

If the common boundary is $\partial J_q$, then
$g_i\gamma_q^{-1}p'$ and $g_i\gamma_q^{-1}p''$
lie in the same $g_i\mathsf P_i'g_i^{-1}$-orbit.
Their $h_i$-heights are therefore equal, and $m\le n$ gives
$\widetilde r_{p'',n}\le\widetilde r_{p',m}$.
Thus \eqref{eqn.defequivclass} gives
$C(p'')\subset C(p')$.
In particular, $p''\in Q_m''$, so the construction forces $m=n$.
The defining radii are then equal, and therefore
$C(p')=C(p'')$.
\end{proof}

\subsection{Auxiliary sets in $\mathcal Z_n$}

By construction and \Cref{lem.LP6.35}, the sets $Q_n''$
are pairwise disjoint. Define
$$
B_n:=\mathcal Z_n\cap
\bigcup_{j\le n}\ \bigcup_{p\in Q_j''}B_{p,n},
\qquad
A_n:=\mathcal Z_n\smallsetminus B_n.
$$

For $p=\gamma p_i\in\Lambda_{\mathrm p}(\Gamma_\rho)$,
the horoball based at $f^{-1}(p)$ whose boundary contains
$\rho^{-1}(\gamma\mathsf P_i)o$ is called the
\emph{horoball for $p$}.
This does not depend on the choice of $\gamma$.

Let $\Phi:\widetilde\Omega_\Gamma\to
\widetilde\Omega_{\bar\psi}$ be the reparametrization from
\Cref{thm.reparametrization}.
For $0<\varepsilon<1$, let $\mathscr C_\varepsilon$ consist
of those $x\in\widetilde\Omega_{\bar\psi}$ for which there
exists $p\in\Lambda_{\mathrm p}(\Gamma_\rho)$ such that
$(\Phi^{-1}(xa_t))(0)$ lies in the interior of the horoball
for $p$ whenever $|t|\le-\log\varepsilon$.

For each fixed $\varepsilon$, the set $\mathscr C_\varepsilon$
is open and $\Gamma_\rho$-invariant.
Let $Q\subset\Omega_{\bar\psi}$ be the image of the compact set
$\Gamma\backslash(\op{T}^1Y_{\rm th}\cap\widetilde\Omega_\Gamma)$
under the homeomorphism induced by $\Phi$.
Every point outside the image of $\mathscr C_\varepsilon$
meets $Q$ within flow time $|t|\le-\log\varepsilon$.
Hence
$\Gamma_\rho\backslash
(\widetilde\Omega_{\bar\psi}\smallsetminus\mathscr C_\varepsilon)$
is compact.

\begin{lemma}
\label{lem.LP6.37}
There exists $C>0$ such that, for every sufficiently small
$\varepsilon>0$, if $x=(x^+,p_0,0)$ satisfies
$x^+\in\mathcal Z_0$ and
$xa_{n\log\mathtt H}\in\mathscr C_{C\varepsilon}$
for some $n\ge1$, then there exists
$p\in\Lambda_{\mathrm p}(\Gamma_\rho)\smallsetminus\{p_0\}$
such that
$$
h_n<\varepsilon h(p),
\qquad
D_{\bar\psi}(x^+,p)<\sqrt{\varepsilon h(p)h_n}.
$$
\end{lemma}

\begin{proof}
Let $0<C<1$ be chosen below, set $T:=n\log\mathtt H$,
and choose $p$ as in the definition of
$\mathscr C_{C\varepsilon}$.
The set
$\Phi^{-1}(\{(\xi,p_0,0):\xi\in\mathcal Z_0\})$
is compact. Hence \Cref{thm.reparametrization} gives a
uniform upper bound for the forward exit time from the
horoball for $p_0$.
Since $T\ge0$, we have $p\ne p_0$ for sufficiently small
$\varepsilon$.

We also record that $p$ lies in a fixed compact subset of
$\Lambda_\rho\smallsetminus\{p_0\}$.
Indeed, place $f^{-1}(p_0)$ at $\infty$ in the upper
half-plane model.
The horoballs for $p\ne p_0$ have uniformly bounded
Euclidean diameters, since they are uniformly bounded
enlargements of the horoballs in $\mathscr H$.
As $f^{-1}(x^+)$ ranges over a bounded subset of $\mathbb R$,
the basepoints of horoballs meeting the vertical geodesic
from $\infty$ to $f^{-1}(x^+)$ also remain bounded.

Write $p=\gamma p_i$.
Since $x^-=p_0\ne p$, there is an entrance time
$T-t_1$, with $t_1>-\log(C\varepsilon)$.
Set $\sigma_1:=\Phi^{-1}(xa_{T-t_1})$ and choose
$\gamma_1\in\gamma\mathsf P_i$ so that
$\widehat\gamma_1o$ is uniformly close to $\sigma_1(0)$.
The entrance point is uniformly close to both rays
$[o,f^{-1}(p)]$ and $[o,f^{-1}(x^+)]$.
Thus \Cref{lem.shadowscomparison} and
\Cref{lem.buseandcartan} give
$\bar\psi(\beta_{x^+}(e,\gamma_1))
=\bar\psi(\beta_p(e,\gamma_1))+O(1)$.
Moreover, equivariance and continuity of $\Phi$ imply
that the $\mathbb R$-coordinate of
$\gamma_1^{-1}\Phi(\sigma_1)$ is uniformly bounded.
Consequently,
$$
T-t_1
=\bar\psi(\beta_{x^+}(e,\gamma_1))+O(1)
=\bar\psi(\beta_p(e,\gamma_1))+O(1).
$$
By \Cref{lem.heightondelta0},
$h_n\asymp e^{-t_1}h(p)\ll C\varepsilon h(p)$,
which proves the first inequality when $C$ is sufficiently
small.

If $x^+=p$, the second inequality is immediate.
Otherwise, there is an exit time $T+t_2$, with
$t_2>-\log(C\varepsilon)$.
Set $\sigma_2:=\Phi^{-1}(xa_{T+t_2})$ and choose
$\gamma_2\in\gamma\mathsf P_i$ so that
$\widehat\gamma_2o$ is uniformly close to $\sigma_2(0)$.
As above,
$T+t_2=\bar\psi(\beta_{x^+}(e,\gamma_2))+O(1)$.
The exit point is uniformly close to
$[f^{-1}(x^+),f^{-1}(p)]$.
Hence \Cref{lem.gromovproductnearest} gives
$\mathcal G(\gamma_2^{-1}x^+,p_i)=O(1)$.
The cocycle identity, $\bar\psi\circ\i=\bar\psi$, and
\Cref{lem.heightondelta0} therefore yield
$$
\begin{aligned}
2\bar\psi(\mathcal G(x^+,p))
&=\bar\psi(\beta_{x^+}(e,\gamma_2))
 +\bar\psi(\beta_p(e,\gamma_2))+O(1)=T+t_2-\log h(p)+O(1).
\end{aligned}
$$
Since $x^+$ and $p$ lie in a fixed compact set avoiding
$p_0$, \eqref{dedef} now gives
$$
D_{\bar\psi}(x^+,p)
\asymp d_{\bar\psi}(x^+,p)
\asymp \sqrt{e^{-t_2}h(p)h_n}
\ll \sqrt{C\varepsilon h(p)h_n}.
$$
Choosing $C$ sufficiently small proves the second inequality.
\end{proof}

\begin{lemma} \label{lem.LP6.38}
There exists ${C}_6 > 0$
such that if  $x = (x^+, p_0, 0) \in \widetilde \Omega_{\bar\psi}$ satisfies $x^+ \in \mathcal{Z}_n$ and $x a_{n \log {\mathtt{H}}} \in \mathscr C_{{C}_6 \eta}$ for some $n \in \N$, then $x^+ \in B_n$.

\end{lemma}
\begin{proof}
Let $C$ be the constant in \Cref{lem.LP6.37}, and choose
$0<C_6<C\min\{1,C_2^{-1},C_5^{-2}\}$.
Applying that lemma with $\varepsilon=C_6\eta/C$, we obtain
$p\in\Lambda_{\mathrm p}(\Gamma_\rho)\smallsetminus\{p_0\}$
such that
$$
h_n<\eta h(p),
\qquad
D_{\bar\psi}(x^+,p)<
\min\left\{
\frac{\eta h(p)}{C_2},
\frac{\sqrt{\eta h(p)h_n}}{C_5}
\right\}.
$$
Since $h(p)\ll1$, for sufficiently small $\eta$ there is
$1\le k\le n$ with $h_k<\eta h(p)\le h_{k-1}$.
The ball $\mathcal B_{\bar\psi}(p,\eta h(p))$ contains
$x^+\in\mathcal Z_n\subset\mathcal Z_{k-1}$,
so $p\in P_k\cup Q_k$.

If $p\in P_k$, then \Cref{lem.LP6.13} gives
$x^+\in J_p$, contradicting $x^+\in\mathcal Z_n$.
Thus $p\in Q_k\subset\bigcup_{j\le k}Q_j''$. Since $D_{\bar\psi}(x^+,p)<\sqrt{\eta h(p)h_n}/C_5$,
\Cref{lem.LP6.32} gives $x^+\in B_{p,n}$.
Together with $x^+\in\mathcal Z_n$, this yields $x^+\in B_n$.
\end{proof}
\subsection{Separation across generations}

\begin{lemma}
\label{lem.LP6.40}
For every $n\ge1$,
$$
P_{n+1}\cap\bigcup_{k\le n}Q_k''=\emptyset.
$$
\end{lemma}

\begin{proof}
Suppose that $p\in P_{n+1}\cap C(p')$, where $C(p')$
is defined at generation $m\le n$.
If $C(p')$ is a singleton or is associated with
$\partial\Delta_0$, then $h(p)=h(p')$, using the
$\mathsf P_0$-invariance of $h$ in the latter case
(\Cref{hd}).
This contradicts the disjoint height ranges of
$P_{n+1}$ and $Q_m$.

Thus $C(p')$ is a non-singleton class in Case~1-2,
with associated boundary $\partial J_q$ for some
$q\in\bigcup_{j<m}P_j$.
By \Cref{lem.LP6.23},
$D_{\bar\psi}(p,\partial J_q)\le h_m$.
On the other hand,
$\mathcal B_{\bar\psi}(p,h_n/(4\eta))\subset\mathcal Z_n$
and $\partial J_q\subset J_q\subset\mathcal Z_n^c$.
Together with \Cref{lem.LP6.27}, this gives
$$
\frac{h_n}{4\eta}
\le D_{\bar\psi}(p,\partial J_q)
\le h_m
<\eta h(p')
\le C_4\eta h(p)
\le C_4h_n,
$$
a contradiction for sufficiently small $\eta$.
\end{proof}

\begin{lemma}
\label{lem.LP6.41}
There exists $0<C_7<1$ such that, for every $n\ge1$,
$p\in P_{n+1}\cup Q_{n+1}''$, and
$q\in\bigcup_{j\le n}Q_j''$,
$$
D_{\bar\psi}(p,B_{q,n})\ge C_7\eta^{-1} {h_n}.
$$
In particular, $D_{\bar\psi}(p,B_n)\ge C_7\eta^{-1} h_n$.
\end{lemma}

\begin{proof}
By \Cref{lem.LP6.35} and \Cref{lem.LP6.40}, $p\ne q$.
The height ranges and \Cref{lem.LP6.27} give
$$
h(p)\ge\frac{h_n}{\mathtt H C_4\eta},
\qquad
h(q)\ge\frac{h_n}{C_4\eta}.
$$
By \Cref{lem.LP6.16}, \Cref{lem.LP6.32}, and
\Cref{lem.eucdistancetriangle},
$$
\begin{aligned}
D_{\bar\psi}(p,B_{q,n})
&\ge \frac{C_s}{N_0}\sqrt{h(p)h(q)}
     -C_5\sqrt{\eta h(q)h_n} \ge \frac{C_s}{2N_0}\sqrt{h(p)h(q)}
 \ge \frac{C_s}{2N_0C_4\sqrt{\mathtt H}}
       \frac{h_n}{\eta},
\end{aligned}
$$
where the second inequality holds for sufficiently small
$\eta$, since
$\sqrt{\eta h_n/h(p)}\le\eta\sqrt{\mathtt H C_4}$.
Thus we may take $C_7:=C_s/(2N_0C_4\sqrt{\mathtt H})$.
The assertion for $B_n$ follows from its definition.
\end{proof}

Recall that $D_{n+1}=\bigcup_{p\in P_{n+1}}J_p$ and
$\mathcal Z_{n+1}=\mathcal Z_n\smallsetminus D_{n+1}$.

\begin{lemma}
\label{lem.LP6.43}
For sufficiently small $\eta$ and every $n\ge1$, the following hold:
\begin{enumerate}
\item
$D_{n+1}\cap B_{q,n}=\emptyset$ for every
$q\in\bigcup_{j\le n}Q_j''$.
In particular, $D_{n+1}\cap B_n=\emptyset$.

\item
$B_{p,n+1}\cap B_{q,n}=\emptyset$ whenever
$p\in Q_{n+1}''$ and $q\in\bigcup_{j\le n}Q_j''$.
In particular, $B_{p,n+1}\cap B_n=\emptyset$
for every $p\in Q_{n+1}''$.

\item
$A_{n+1}
=\bigl(A_n\smallsetminus(D_{n+1}\cup B_{n+1})\bigr)
\sqcup(A_{n+1}\cap B_n).
$

\item
$
A_n\cap B_{n+1}
=\mathcal Z_{n+1}\cap\bigcup_{p\in Q_{n+1}''}B_{p,n+1}.
$

\item
$A_{n+1}\cap B_n=B_n\smallsetminus B_{n+1}$.
\end{enumerate}
Moreover, $B_{p,m}\cap\mathcal Z_m=B_{p,m}\cap\mathcal Z_n$
whenever $p\in Q_n''$ and $m\ge n$.
\end{lemma}

\begin{proof}
For $p\in P_{n+1}$, \Cref{lem.LP6.13} gives
$J_p\subset\mathcal B_{\bar\psi}(p,C_2h_n)$.
For $p\in Q_{n+1}''$, \Cref{lem.LP6.27} and
\Cref{lem.LP6.32} give
$B_{p,n+1}\subset
\mathcal B_{\bar\psi}(p,C_5\sqrt{C_4}\,h_n)$.
For sufficiently small $\eta$, both radii are smaller than
$C_7h_n/\eta$, so (1) and (2) follow from
\Cref{lem.LP6.41}.

By the definitions,
$A_{n+1}\smallsetminus B_n
=A_n\smallsetminus(D_{n+1}\cup B_{n+1})$.
Partitioning $A_{n+1}$ according to membership in $B_n$
proves (3).

The sets $B_{p,m}$ decrease with $m$.
Thus, for $q\in\bigcup_{j\le n}Q_j''$,
$B_{q,n+1}\cap\mathcal Z_{n+1}\subset B_n$.
The neighborhoods $B_{p,n+1}$, $p\in Q_{n+1}''$, are disjoint from $B_n$
by (2).
Splitting the defining union for $B_{n+1}$ into these two
groups proves (4).

By (1), $B_n\subset\mathcal Z_{n+1}$.
Hence
$A_{n+1}\cap B_n
=(\mathcal Z_{n+1}\smallsetminus B_{n+1})\cap B_n
=B_n\smallsetminus B_{n+1}$,
proving (5).

Finally, let $p\in Q_n''$ and $m\ge n$.
For every $n\le k<m$, nesting and (1) give
$B_{p,m}\subset B_{p,k}$ and
$B_{p,k}\cap D_{k+1}=\emptyset$.
Thus $B_{p,m}$ is disjoint from all sets removed between
$\mathcal Z_n$ and $\mathcal Z_m$, proving the last assertion.
\end{proof}

\subsection{Transfer of mass between $B_n$ and $A_{n+1}$}

The sets $A_n$ and $B_n$ partition $\mathcal Z_n$.
The following lemma shows that a uniform fraction of
the mass of $B_n$ passes to $A_{n+1}$.

\begin{lemma}
\label{lem.LP6.46}
There exists $C_8>0$, independent of $\eta$, such that,
for every sufficiently small $\eta$ and every $n\ge1$,
$$
\nu(B_n\cap A_{n+1})\ge C_8\nu(B_n).
$$
\end{lemma}

\begin{proof}
We first show that for all
$p,p'\in\bigcup_{j\le n}Q_j'',\ p\ne p'$,
\begin{equation}
\label{eqn.q''disjoint}
B_{p,n}\cap B_{p',n}=\emptyset.
\end{equation}
For points from different generations, this follows from
\Cref{lem.LP6.43}(2), since $B_{p,s+1}\subset B_{p,s}$.
If $p,p'\in Q_j''$, then \Cref{lem.LP6.27} gives
$h(p),h(p')\asymp h_j/\eta$, so \Cref{lem.LP6.16} yields
$D_{\bar\psi}(p,p')\gg h_j/\eta$.
An intersection of $B_{p,j}$ and $B_{p',j}$ would instead
give $D_{\bar\psi}(p,p')\ll h_j$ by
\Cref{lem.LP6.32} and \Cref{lem.eucdistancetriangle}.
This is impossible for small $\eta$, proving
\eqref{eqn.q''disjoint}.

We also record a consequence of \eqref{eqn.annularH}.
The ball estimates in \Cref{eqn.LP5.16}, together with
\Cref{lem.LP5.18}, give
\begin{equation}
\label{eqn.annularcore}
\begin{aligned}
\nu_i\left(
\mathcal B_{\bar\psi}(z,r/\mathsf c)
\smallsetminus
\mathcal B_{\bar\psi}(z,r/\sqrt{\mathtt H})
\right)
&\ge \frac12\,
\nu_i\bigl(\mathcal B_{\bar\psi}(z,r/\mathsf c)\bigr)\gg \nu_i\bigl(\mathcal B_{\bar\psi}(z,r)\bigr)
\end{aligned}
\end{equation}
uniformly for
$z\in\Lambda_{\mathrm p}(\Gamma_i)\smallsetminus\{p_0\}$
and $0<r\le h_i(z)$.

Fix a class $C(p)$ of generation $m\le n$, and set
$M_s:=\sum_{p'\in C(p)}
\nu(B_{p',s}\cap\mathcal Z_m)$ for $s\ge m$.
By the last assertion of \Cref{lem.LP6.43},
$\mathcal Z_m$ may be replaced by $\mathcal Z_s$ in
this definition.
Together with \eqref{eqn.q''disjoint} and
\Cref{lem.LP6.43}(2), (5), this reduces the proof to
$M_n-M_{n+1}\gg M_n$, uniformly over the classes.

We first record the separation needed below.
Every center in $C(p)$ lies outside $P_m$.
For singleton classes this follows from $p\in Q_m$.
Otherwise, the class definition and \Cref{lem.LP6.23}
give $D_{\bar\psi}(p',\partial\mathcal Z_{m-1})\ll h_m$
for every $p'\in C(p)$.
If $p'\in P_m$, then, by definition,
$
\mathcal B_{\bar\psi}\left(p',\frac{h_{m-1}}{4\eta}\right)
\subset\mathcal Z_{m-1}.
$
Since
$\partial\mathcal Z_{m-1}\subset
\overline{\mathcal Z_{m-1}^{\,c}}$,
this implies
$
D_{\bar\psi}(p',\partial\mathcal Z_{m-1})
\ge \frac{h_{m-1}}{4\eta}
=\frac{\mathtt H}{4\eta}h_m.
$
For sufficiently small $\eta$, this contradicts the preceding
bound $D_{\bar\psi}(p',\partial\mathcal Z_{m-1})\ll h_m$.

The heights of centers in $C(p)$ and $P_m$ are comparable
to $h_m/\eta$, so \Cref{lem.LP6.16} separates distinct
centers by a distance $\gg h_m/\eta$.
Since the radii in \Cref{lem.LP6.13} and
\Cref{lem.LP6.32} are $O(h_m)$, every $B_{p',m}$
is disjoint from $D_m$.
In either boundary case, the class centers are within
$O(h_m)$ of the associated boundary.
Thus \Cref{lem.LP6.17} also keeps each $B_{p',m}$,
$p'\in C(p)$, disjoint from every $J_q$ removed before
generation $m$, except the associated $J_q$ in Case~1-2;
in Case~1-2, it also keeps each $B_{p',m}$ inside $\Delta_0$.

Suppose first that $C(p)$ is constructed in Case~1-1.
The preceding separation gives
$B_{p',s}\cap\mathcal Z_m=B_{p',s}\cap\Delta_0$
for $p'\in C(p)$ and $s\ge m$.
For each $p'\in C(p)$, let $u_{p'}\in\mathsf P_0$
satisfy $p'=u_{p'}p$.
For small $\eta$, $B_{p,m}$ lies in the two adjacent
fundamental intervals sharing the endpoint it meets.
The maps $u_{p'}$ carry these intervals onto $\Delta_0$.
Thus, for every Borel set $E\subset B_{p,m}$,
the sets $E\cap u_{p'}^{-1}\Delta_0$, $p'\in C(p)$,
partition $E$ up to endpoints.
Applying \Cref{lem.LP5.18} with $E=B_{p,n}$ and
$E=B_{p,n}\smallsetminus B_{p,n+1}$ gives
$$
\begin{aligned}
M_n \asymp\nu(B_{p,n}), \qquad M_n-M_{n+1}
\asymp\nu(B_{p,n}\smallsetminus B_{p,n+1}).
\end{aligned}
$$
Since $r_{p,n+1}=r_{p,n}/\sqrt{\mathtt H}$ and
$r_{p,n}\le r_{p,m}<\eta h(p)\le h(p)$,
\Cref{lem.LP5.21} gives $M_n-M_{n+1}\gg M_n$.

Now suppose that $C(p)$ is constructed in Case~2.
Write $r:=r_{p,n}$.
The neighborhood $B_{p,m}$ misses $\partial\mathcal Z_{m-1}$.
Since $r\le r_{p,m}$, \Cref{lem.olive} implies that
$\mathcal B_{\bar\psi}(p,r/\mathsf c)$ is contained in
$\mathcal Z_{m-1}$ if $p\in\mathcal Z_{m-1}$,
and is disjoint from $\mathcal Z_{m-1}$ otherwise.
In the latter case, $\sqrt{\mathtt H}>\mathsf c$
implies $M_{n+1}=0$.
In the former case, separation from $D_m$ gives
$\mathcal B_{\bar\psi}(p,r/\mathsf c)\subset\mathcal Z_m$.
Consequently, \eqref{eqn.annularcore} yields
$$
\begin{aligned}
M_n-M_{n+1}
\ge
\nu\left(
\mathcal B_{\bar\psi}(p,r/\mathsf c)
\smallsetminus
\mathcal B_{\bar\psi}(p,r/\sqrt{\mathtt H})
\right) \gg \nu\bigl(\mathcal B_{\bar\psi}(p,r)\bigr)
\ge M_n.
\end{aligned}
$$

Finally, consider Case~1-2, with associated boundary
$\partial J_q$, where $q=\gamma_qp_i$.
Set $T:=g_i\gamma_q^{-1}$ and
$\mathcal U_s:=\mathcal B_{\bar\psi}(Tp,r_{p,s})$.
Since $\nu_i=(g_i)_*\nu$ and $g_i\in K$,
\Cref{lem.LP6.28} and \Cref{lem.LP5.17},
applied in the conjugated coordinates, give
$$
\nu(E)\asymp
e^{-\psi(\mu(\gamma_q))}\nu_i(TE)
\qquad
\bigl(E\subset B_{p',m}\text{ Borel},\ p'\in C(p)\bigr).
$$
By separation, for every $p'\in C(p)$ and $s\ge m$,
$$
T(B_{p',s}\cap\mathcal Z_m)
=
\mathcal B_{\bar\psi}(Tp',r_{p',s})
\cap R_{i},
$$
up to $\nu_i$-null endpoints.
Moreover,
$r_{p,n+1}=r_{p,n}/\sqrt{\mathtt H}$ and
$r_{p,n}\le r_{p,m}\ll\eta h_i(Tp)$ by
\eqref{eqn.LP6.22}, so $r_{p,n}\le h_i(Tp)$
for small $\eta$.

In the first alternative of Case~1-2, for each $p'\in C(p)$,
choose $u_{p'}\in g_i\mathsf P_i'g_i^{-1}$ with
$Tp'=u_{p'}Tp$.
Then $TB_{p',s}=u_{p'}\mathcal U_s$.
For small $\eta$, $\mathcal U_m$ lies in two adjacent
translates of $R_{i}$ sharing the endpoint it meets.
The class definition selects precisely the peripheral
maps carrying these intervals onto $R_{i}$.
Consequently, the sets
$$
u_{p'}^{-1}T(B_{p',s}\cap\mathcal Z_m)
=\mathcal U_s\cap u_{p'}^{-1}R_{i},
\qquad p'\in C(p),
$$
partition $\mathcal U_s$ up to null endpoints.
Apply this with $s=n$ and $s=n+1$.
By \Cref{lem.LP6.28}, all these portions lie at
$D_{\bar\psi}$-distance comparable to $\eta^{-1}$
from $g_ix_q$, while their diameters are uniformly
bounded by \eqref{eqn.LP6.22}.
Thus \Cref{lem.LP5.18} gives uniform measure comparisons
under these peripheral maps.
Applying these comparisons separately to the ball
and the annulus gives
$$
\begin{aligned}
M_n
\asymp e^{-\psi(\mu(\gamma_q))}\nu_i(\mathcal U_n),\qquad 
M_n-M_{n+1}
\asymp e^{-\psi(\mu(\gamma_q))}
\nu_i(\mathcal U_n\smallsetminus\mathcal U_{n+1}).
\end{aligned}
$$
It follows from \Cref{lem.LP5.21} that
$M_n-M_{n+1}\gg M_n$.

In the singleton alternative, $\mathcal U_m$ misses
$\partial R_{i}$.
Write $r:=r_{p,n}$.
By \Cref{lem.olive},
$\mathcal B_{\bar\psi}(Tp,r/\mathsf c)$ is contained
in $R_{i}$ if $Tp\in R_{i}$,
and is disjoint from $R_{i}$ otherwise.
In the latter case, $\sqrt{\mathtt H}>\mathsf c$
gives $M_{n+1}=0$.
In the former case, separation gives
$$
T^{-1}\left(
\mathcal B_{\bar\psi}(Tp,r/\mathsf c)
\smallsetminus\mathcal U_{n+1}
\right)
\subset\mathcal Z_m.
$$
Hence \eqref{eqn.annularcore} and the change-of-measure
comparison give
$M_n-M_{n+1}\gg
e^{-\psi(\mu(\gamma_q))}\nu_i(\mathcal U_n)\gg M_n$.

The implied lower bound is uniform over all classes.
Summing over them proves the lemma.
\end{proof}

\subsection{Non-concentration of PS measures near partition boundaries}

For any $r > 0$, set
\be \label{eqn.LP5.28}
N_r(\Delta_0) := \{ x \in \Delta_0 : D_{\bar\psi}(x, \partial \Delta_0) \le r \}
\ee
where $\partial\Delta_0=f(\partial\widehat\Delta_\infty)$.

\begin{lemma} \label{lem.LP5.29}
There exists $0 < \varepsilon_0 < 1$ such that for any $0 < \varepsilon < \varepsilon_0$, there exists $\la = \la (\varepsilon) \in (0, 1)$ so that
$$
\nu(N_{\varepsilon r}(\Delta_0)) \le \la \nu(N_r (\Delta_0)) \quad \text{for all } 0 < r < 1.
$$
Moreover, $\la(\varepsilon) \to 0$ as $\varepsilon \to 0$.

\end{lemma}
\begin{proof}
Write $\partial\Delta_0=\{x_1,x_2\}$, where
$x_2=\gamma x_1$ for a generator $\gamma\in\mathsf P_0$.
Let $C_0>1$ be as in \Cref{prop.LP5.14}(2).
Choose $0<r_*<C_0^{-1}$ so that each
$\mathcal B_{\bar\psi}(x_i,r_*)$ lies in the two
fundamental intervals adjacent to $x_i$.
For $0<r<r_*$, the peripheral translations satisfy
$$
\gamma\bigl(\mathcal B_{\bar\psi}(x_1,r)
\smallsetminus\Delta_0\bigr)
\subset\mathcal B_{\bar\psi}(x_2,r)\cap\Delta_0,\qquad 
\gamma^{-1}\bigl(\mathcal B_{\bar\psi}(x_2,r)
\smallsetminus\Delta_0\bigr)
\subset\mathcal B_{\bar\psi}(x_1,r)\cap\Delta_0.
$$
Thus \Cref{lem.metrics} and \Cref{lem.LP5.18} give
$$
\nu({\mathcal B_{\bar\psi}}(x_1,r))+\nu({\mathcal B_{\bar\psi}}(x_2,r))\ll\nu(N_r(\Delta_0))
\qquad(0<r<r_*).
$$
Iterating \Cref{prop.LP5.14}(2), we obtain
$$
\nu(\mathcal B_{\bar\psi}(x_i,C_0^{-k}r))
\le 2^{-k}\nu(\mathcal B_{\bar\psi}(x_i,r))
\qquad(0<r<r_*,\ k\ge1).
$$
If $C_0^{-k-1}\le\varepsilon<C_0^{-k}$, then
$N_{\varepsilon r}(\Delta_0)$ is contained in
$\bigcup_{i=1}^2\mathcal B_{\bar\psi}(x_i,C_0^{-k}r)$.
Consequently,
$\nu(N_{\varepsilon r}(\Delta_0))
\ll 2^{-k}\nu(N_r(\Delta_0))$
uniformly for $0<r<r_*$.

For $r_*\le r<1$, monotonicity gives
$$
\frac{\nu(N_{\varepsilon r}(\Delta_0))}
     {\nu(N_r(\Delta_0))}
\le
\frac{\nu(N_\varepsilon(\Delta_0))}
     {\nu(N_{r_*}(\Delta_0))}
\longrightarrow0
\qquad(\varepsilon\to0),
$$
since $\nu(\partial\Delta_0)=0$ by
\Cref{thm.ps atomless 0}.
Combining the two ranges proves the lemma.
\end{proof}

We will use the following counting result from \cite[Proposition~5.2]{KO_globalshadow}, which is based on \cite[Proposition 7.2]{BO}.

\begin{proposition}
\label{prop.countingperipheral}
    For any $\xi \in \La_{\mathrm p}(\Ga_{\rho})$ and all sufficiently large $T$,
    $$\# \{ \ga \in \mathsf P_\xi : \bar\psi(\mu(\ga)) \le T \} \asymp e^{\delta_\psi(\mathsf P_\xi) T}.
    $$
\end{proposition}

For $p\in\Delta\cap\Lambda_{\mathrm p}(\Gamma_\rho)$ and $r>0$,
define
\begin{equation}
\label{eqn.LP6.8}
N_r(\partial J_p)
:=\{x\in J_p^c:D_{\bar\psi}(x,\partial J_p)\leq r\}.
\end{equation}
When $p=\gamma_pp_0$, we define
$N_r(\partial(\gamma_p^{-1}J_p))$ in the same way.

\begin{lemma}
\label{lem.LP6.9}
Let $C>1$. For every sufficiently small $\eta>0$, there exists
$c=c(\eta)\in(0,1)$ such that for all $p\in\Delta\cap\Lambda_{\mathrm p}(\Gamma_\rho)$
and $0<r<h(p)$,
$$
\nu(N_{C\eta r}(\partial J_p))
\leq c\,\nu(N_r(\partial J_p)).
$$
Moreover, $c(\eta)\to0$ as $\eta\to0$.
\end{lemma}

\begin{proof}
Suppose that $p\in\Gamma_\rho p_i$ and write
$N_s^{(i)}:=N_s(\partial J_{\Gamma_i,g_ip})$.
Choose fixed constants $0<a<1$ and $L\ge1$ so that
the neighborhoods of radius $ar$ and $ar/L$ used below
lie in compact sets where the coordinate changes are
bi-Lipschitz with constant $L$.
For sufficiently small $\eta$, \Cref{lem.metrics} and
$\nu_i=(g_i)_*\nu$ then give
$\nu(N_{C\eta r}(\partial J_p))
\le\nu_i(N_{LC\eta r}^{(i)})$ and
$\nu_i(N_{ar/L}^{(i)})\le\nu(N_r(\partial J_p))$.
By \Cref{lem.LP5.11}, we may also choose $a$ so that
$r_i:=ar/L<h_i(g_ip)$.
Since $LC\eta r=(L^2C/a)\eta r_i$, it suffices to prove
the assertion at the cusp $p_0$ in each conjugated system,
after adjusting the fixed constant $C$.
We give the proof in the original notation, with $p=\gamma_pp_0$.
Write $x_p:=\gamma_p^{-1}p_0$ and set
$\beta:=C\sqrt{\eta}<1$.
We first estimate
$\nu(N_{\beta r}(\partial J_p))/\nu(N_r(\partial J_p))$
in three ranges of $r$.

Suppose first that $r\leq\eta h(p)$.
By \Cref{lem.LP6.7,lem.eucdistancetriangle},
every $x\in N_r(\partial J_p)$ satisfies
$D_{\bar\psi}(x,p)\asymp\eta h(p)$.
The ball-flipping estimates therefore give
$D_{\bar\psi}(\gamma_p^{-1}x,x_p)\asymp\eta^{-1}$ and
$$
D_{\bar\psi}(\gamma_p^{-1}x,\gamma_p^{-1}\partial J_p)
\asymp
\frac{D_{\bar\psi}(x,\partial J_p)}{\eta^2h(p)}.
$$
Applying these estimates in both directions, we obtain constants
$c_1,c_2>0$ such that
\begin{equation}
\label{eqn.69 proof eqn}
N_{c_1r/(\eta^2h(p))}(\gamma_p^{-1}\partial J_p)
\subset
\gamma_p^{-1}N_r(\partial J_p)
\subset
N_{c_2r/(\eta^2h(p))}(\gamma_p^{-1}\partial J_p).
\end{equation}
For the first inclusion, choose $c_1$ sufficiently small that
points in the left-hand neighborhood remain at distance
comparable to $\eta^{-1}$ from $x_p$, and apply the inverse
ball-flipping estimate.
Moreover, \Cref{lem.LP5.17} gives
\begin{equation}
\label{eqn.jan10.12.11}
\nu(\gamma_p^{-1}N_r(\partial J_p))
\asymp
e^{\psi(\mu(\gamma_p))}\nu(N_r(\partial J_p)).
\end{equation}
These estimates also hold with $\beta r$ in place of $r$.

\noindent\textbf{Case A1: $r\leq\eta^2h(p)$.}
The endpoints of $\gamma_p^{-1}J_p$ are peripheral translates
of the endpoints of $\Delta_0$.
By \eqref{eqn.69 proof eqn}, the part near each endpoint
$z\in\gamma_p^{-1}\partial J_p$ is contained in
$$
\left\{
x\in(\gamma_p^{-1}J_p)^c:
D_{\bar\psi}(x,z)
\le \frac{c_2r}{\eta^2h(p)}
\right\}.
$$
Since $r\le\eta^2h(p)$, each of these endpoint
neighborhoods has uniformly bounded diameter and lies at
distance comparable to $\eta^{-1}$ from $x_p$.
Thus \Cref{lem.LP5.18}, applied to the two endpoint neighborhoods
as in the proof of \Cref{lem.LP5.29}, gives a fixed
$x_0\in\partial\Delta_0$ and constants $a,b>0$ such that,
for $t=r$ and $t=\beta r$,
$$
\eta^2\nu\left(
\mathcal B_{\bar\psi}\left(x_0,\frac{at}{\eta^2h(p)}\right)
\right)
\ll
\nu(\gamma_p^{-1}N_t(\partial J_p))
\ll
\eta^2\nu\left(
\mathcal B_{\bar\psi}\left(x_0,\frac{bt}{\eta^2h(p)}\right)
\right).
$$
Together with \eqref{eqn.jan10.12.11},
\Cref{prop.LP5.14}(2) yields
$\nu(N_{\beta r}(\partial J_p))
\leq\omega(\beta)\nu(N_r(\partial J_p))$,
where $\omega(t)\to0$ as $t\to0$.
The bounded range of radii is handled as in
\Cref{lem.LP5.29}.

\noindent\textbf{Case A2:
$\eta^{3/2}h(p)<r\leq\eta h(p)$.}
Write $\delta_0:=\delta_\psi(\mathsf P_0)$.
For a generator $u$ of $\mathsf P_0$,
\Cref{lem.metrics,lem.dpsidiameterfinite,lem.zerobuse} and
\Cref{lem.gromovproductalongparabolic} give
$D_{\bar\psi}(x_0,u^j x_0)
\asymp e^{\bar\psi(\mu(u^j))/2}$ for large $|j|$.
Hence \Cref{prop.countingperipheral} gives a tile count
comparable to $T^{2\delta_0}$ at $D_{\bar\psi}$-scale $T$.

For $\eta^2h(p)\leq t\leq\eta h(p)$, apply this count with
$T=t/(\eta^2h(p))$ to the endpoint neighborhoods in
\eqref{eqn.69 proof eqn}.
For large $T$, quasi-monotonicity gives both an upper covering
count and a comparable lower count of full tiles contained
in $N_{c_1t/(\eta^2h(p))}
(\gamma_p^{-1}\partial J_p)$.
The bounded range of $T$ follows from the endpoint estimates
used in Case A1.
For the upper bound, use the peripheral tiles that meet
$\gamma_p^{-1}N_t(\partial J_p)$.
The preceding distance estimate and the uniformly bounded
diameters of the tiles place every such tile at distance
comparable to $\eta^{-1}$ from $x_p$.
For the lower bound, use the full tiles contained in
$N_{c_1t/(\eta^2h(p))}
(\gamma_p^{-1}\partial J_p)$ in \eqref{eqn.69 proof eqn};
the choice of $c_1$ keeps every point of that neighborhood
at distance comparable to $\eta^{-1}$ from $x_p$.
Thus \Cref{lem.LP5.18} gives measure comparable to $\eta^2$
for every tile used in either estimate.
Consequently,
$$
\nu(\gamma_p^{-1}N_t(\partial J_p))
\asymp
\eta^2\left(\frac{t}{\eta^2h(p)}\right)^{2\delta_0}.
$$
Both $t=r$ and $t=\beta r$ belong to this range, since
$\beta r>C\eta^2h(p)$.
Using \eqref{eqn.jan10.12.11}, we obtain
$\nu(N_{\beta r}(\partial J_p))
\ll\beta^{2\delta_0}\nu(N_r(\partial J_p))$.

\medskip
\noindent\textbf{Case B: $\sqrt{\eta}\,h(p)\leq r<h(p)$.}
Here $\eta h(p)\leq\beta r$.
By \Cref{lem.LP6.7,lem.olive,lem.eucdistancetriangle},
there are fixed constants $0<a<1<b$ such that
$$
\mathcal B_{\bar\psi}(p,ar)
\subset J_p\cup N_r(\partial J_p),
\qquad
J_p\cup N_{\beta r}(\partial J_p)
\subset\mathcal B_{\bar\psi}(p,b\beta r).
$$
Take $\eta$ sufficiently small that $b\beta<1$.
Since $N_{\beta r}(\partial J_p)\subset N_r(\partial J_p)$
and both sets are disjoint from $J_p$,
$$
\frac{\nu(N_{\beta r}(\partial J_p))}
     {\nu(N_r(\partial J_p))}
\leq
\frac{\nu(J_p\cup N_{\beta r}(\partial J_p))}
     {\nu(J_p\cup N_r(\partial J_p))}
\leq
\frac{\nu(\mathcal B_{\bar\psi}(p,b\beta r))}
     {\nu(\mathcal B_{\bar\psi}(p,ar))}
\ll\beta^{2(1-\delta_0)}
$$
by \Cref{eqn.LP5.16}.
The exponent is positive by \Cref{thm.entropydrop}.

The three cases give a common $c(\eta)\to0$ such that
$\nu(N_{\beta r}(\partial J_p))
\leq c(\eta)\nu(N_r(\partial J_p))$
throughout the indicated ranges.
Since $C\eta r\leq\beta r$, this proves the desired estimate
in those ranges.

In either remaining range,
$\eta^2h(p)<r\leq\eta^{3/2}h(p)$ or
$\eta h(p)<r<\sqrt{\eta}\,h(p)$,
set $s:=\sqrt{\eta}\,r$.
Then $s$ belongs to Case A1 or Case A2, respectively, and
$C\eta r=\beta s$.
Therefore,
$$
\nu(N_{C\eta r}(\partial J_p))
=
\nu(N_{\beta s}(\partial J_p))
\leq c(\eta)\nu(N_s(\partial J_p))
\leq c(\eta)\nu(N_r(\partial J_p)).
$$
For sufficiently small $\eta$, we also have $c(\eta)<1$.
\end{proof}

\subsection{Boundary-neighborhood estimates}
Let
$$
{C}_B = 2\mathsf{c}N_0{C}_2 {C}_0 + 4{C}_5 + 3N_0^2{C}_5\sqrt{{C}_4}\,{\mathtt{H}}.
$$
Here $C_0$, $C_2$, $C_4$, and $C_5$ are the constants from
\Cref{prop.LP5.14}, \Cref{lem.LP6.13}, \Cref{lem.LP6.27},
and \Cref{lem.LP6.32}, respectively.
The constant $\mathsf c$ was fixed before \Cref{lem.LP5.21}.

Set
$$
\mathcal{Z}_n' = \{ x \in \mathcal{Z}_n :  D_{\bar\psi}(x, \partial \mathcal{Z}_n) \le {C}_Bh_n\}.
$$

\begin{lemma}  \label{lem.LP6.50}
    There exists ${C}_9 > 0$,
    depending only on $\eta$, such that for all sufficiently large $n$,
    $$\nu(\mathcal{Z}_n') \le {C}_9 \nu(\mathcal{Z}_n);$$
    in particular, $\nu(A_n \cap \mathcal{Z}_n') \le {C}_9 \nu(\mathcal{Z}_n)$.
   Moreover,  ${C}_9\to 0$ as $\eta\to 0$.
\end{lemma}

\begin{proof}
Choose $0<a<\min\{1,C_s/(4N_0^3)\}$ and set $r_n:=ah_n/\eta$.
For $p\in P_1\cup\cdots\cup P_n$, we have $h_n<\eta h(p)$,
so $r_n<h(p)$. For sufficiently large $n$, we also have $r_n<1$.

Since $C_Bh_n=(C_B/a)\eta r_n$,
\Cref{lem.LP5.29,lem.LP6.9} give a common $c=c(\eta)>0$,
with $c(\eta)\to0$ as $\eta\to0$, such that
$$
\begin{aligned}
\nu(N_{C_Bh_n}(\Delta_0))
&\leq c\,\nu(N_{r_n}(\Delta_0)),\\
\nu(N_{C_Bh_n}(\partial J_p))
&\leq c\,\nu(N_{r_n}(\partial J_p))
\qquad(p\in P_1\cup\cdots\cup P_n).
\end{aligned}
$$
By \Cref{lem.LP6.17} and the choice of $a$, the sets
$N_{r_n}(\Delta_0)$ and $N_{r_n}(\partial J_p)$,
$p\in P_1\cup\cdots\cup P_n$, are pairwise disjoint subsets
of $\mathcal Z_n$. Hence
$$
\begin{aligned}
\nu(\mathcal Z_n')
&\leq \nu(N_{C_Bh_n}(\Delta_0))
+\sum_{p\in P_1\cup\cdots\cup P_n}
\nu(N_{C_Bh_n}(\partial J_p))\leq c\,\nu(\mathcal Z_n).
\end{aligned}
$$
Taking $C_9:=c$ proves the lemma.
\end{proof}
\begin{lemma}   \label{lem.LP6.51}
There exists $0 < {C}_{10} < 1$
such that for all $n\ge 1$,
$$
\nu(A_n \cap (D_{n+1} \cup B_{n+1})) \le {C}_{10} \nu(A_n) + \nu(A_n \cap \mathcal{Z}_n').
$$
\end{lemma}
\begin{proof}
Set
$$
Q_{n+1}'''
:=
\left\{
p\in Q_{n+1}'':
B_{p,n+1}\cap
\bigl((A_n\cap B_{n+1})\smallsetminus\mathcal Z_n'\bigr)
\neq\emptyset
\right\}.
$$ For $p\in Q_{n+1}'''$, choose
$x\in B_{p,n+1}\cap A_n\cap B_{n+1}$
with $x\notin\mathcal Z_n'$. 
By \Cref{lem.LP6.27,lem.LP6.32},
$D_{\bar\psi}(x,p)<C_5\sqrt{C_4}\,h_n$.

We claim that the class containing $p$ is constructed in Case~2.
Otherwise, its construction and \Cref{lem.LP6.23} give
$D_{\bar\psi}(p,\partial\mathcal Z_n)\leq h_n$.
The quasi-triangle inequality then implies
$D_{\bar\psi}(x,\partial\mathcal Z_n)
\leq N_0(C_5\sqrt{C_4}+1)h_n<C_Bh_n$,
contrary to $x\notin\mathcal Z_n'$.
Thus $p\in Q_{n+1}\smallsetminus Q_{n+1}'$, $C(p)=\{p\}$,
and
$B_{p,n+1}=\mathcal B_{\bar\psi}(p,r(p,n+1))
\subset\mathcal B_{\bar\psi}(p,h_n)$.

For $p\in P_{n+1}$, \Cref{lem.LP6.13} gives
$J_p\subset\mathcal B_{\bar\psi}(p,C_2h_n)$.
Together with \Cref{lem.LP6.43}(4), these inclusions show that
$$
\bigl(A_n\cap(D_{n+1}\cup B_{n+1})\bigr)
\smallsetminus\mathcal Z_n'
\subset
\bigcup_{p\in P_{n+1}\cup Q_{n+1}'''}
\mathcal B_{\bar\psi}(p,C_2h_n).
$$

Set $R_n:=C_Bh_n/(2\mathsf cN_0)$.
We next show that
$\mathcal B_{\bar\psi}(p,R_n)\subset A_n$
for every $p\in P_{n+1}\cup Q_{n+1}'''$.
By \Cref{lem.LP6.41},
$D_{\bar\psi}(p,B_n)\geq C_7h_n/\eta>R_n$
for sufficiently small $\eta$, so these balls avoid $B_n$.
For $p\in P_{n+1}$, their inclusion in $\mathcal Z_n$ follows
from the defining condition
$\mathcal B_{\bar\psi}(p,h_n/(4\eta))\subset\mathcal Z_n$.

For $p\in Q_{n+1}'''$, use the point $x$ chosen above,
which now satisfies $D_{\bar\psi}(x,p)<h_n$.
If $z\in\mathcal B_{\bar\psi}(p,R_n)\smallsetminus\mathcal Z_n$,
then \Cref{lem.olive,lem.eucdistancetriangle} give
$$
C_Bh_n
<
D_{\bar\psi}(x,\partial\mathcal Z_n)
\leq \mathsf c D_{\bar\psi}(x,z)
\leq \mathsf cN_0(h_n+R_n)
<C_Bh_n,
$$
since $C_B>2\mathsf cN_0$.
This contradiction proves the required inclusion in $A_n$.

By \Cref{lem.LP6.16}, the balls
$\mathcal B_{\bar\psi}(p,R_n)$,
$p\in P_{n+1}\cup Q_{n+1}'''$, are pairwise disjoint
for sufficiently small $\eta$.
Their centers lie in $\Delta_0$, so \Cref{lem.heightondelta}
and $h_{n+1}<\eta h(p)$ give $h_n\ll\eta$ whenever the
index set is nonempty. Thus $C_2h_n<1/C_0$ after decreasing
$\eta$.
Since $R_n\geq C_0C_2h_n$ by the choice of $C_B$,
\Cref{prop.LP5.14}(2) yields
$$
\begin{aligned}
\nu\bigl(
(A_n\cap(D_{n+1}\cup B_{n+1}))
\smallsetminus\mathcal Z_n'
\bigr)
&\leq
\sum_{p\in P_{n+1}\cup Q_{n+1}'''}
\nu(\mathcal B_{\bar\psi}(p,C_2h_n))\\
&\leq \frac12
\sum_{p\in P_{n+1}\cup Q_{n+1}'''}
\nu(\mathcal B_{\bar\psi}(p,R_n))
\leq \frac12\nu(A_n).
\end{aligned}
$$
Adding $\nu(A_n\cap\mathcal Z_n')$ proves the lemma
with $C_{10}=1/2$.
\end{proof}

\begin{lemma}
\label{lem.LP6.53}
There exist $N\in\mathbb N$ and $C_{11}>0$, depending on $\eta$,
such that, for all sufficiently large $n$,
$$
\nu\left(\cup_{j=1}^N D_{n+j}\right)
\geq C_{11}\nu(A_n\smallsetminus\mathcal Z_n').
$$
\end{lemma}

\begin{proof} 
Set $A_n':=A_n\smallsetminus\mathcal Z_n'$.
Let $U_0>0$ be given by \Cref{lem.LP6.55_var} for
$\tau=\mathtt H$ and the compact set
$Q:=\Gamma_\rho
(\widetilde\Omega_{\bar\psi}\smallsetminus\mathscr C_{C_6\eta})
\subset\Omega_{\bar\psi}$.
Set $k:=\lfloor-\log\eta/\log\mathtt H\rfloor$ and
$L:=\lceil U_0/\log\mathtt H\rceil$, and choose
$N>2k+L+2$.
We will show that there exists $C'>1$ such that
\begin{equation}
\label{eqn.653 proof claim}
\text{for every }x\in A_n',
\text{ there is }p\in\cup_{j=1}^N P_{n+j}
\text{ with }D_{\bar\psi}(x,p)<C'h_n.
\end{equation}

We first consider a point $y\in A_m$ satisfying
$D_{\bar\psi}(y,\partial\mathcal Z_m)>2h_m/\eta$.
We claim that there exists
$q\in P_{m+1}\cup\cdots\cup P_{m+k+L+1}$
with $D_{\bar\psi}(y,q)\ll h_m$.

Indeed, \Cref{lem.LP6.38} implies that
$\Gamma_\rho(y,p_0,0)a_{m\log\mathtt H}\in Q$.
By \Cref{lem.LP6.55_var}, there is a parabolic point $p$ such that
$D_{\bar\psi}(y,p)\leq Ch_m$ and
$e^{-U_0}h_m\leq h(p)\leq h_m$, where $C$ is independent of $\eta$.
For sufficiently small $\eta$, \Cref{lem.olive} and the distance
of $y$ from $\partial\mathcal Z_m$ imply that $p\in\mathcal Z_m$.

Set $j:=\lfloor-\log(\eta h(p))/\log\mathtt H\rfloor$.
Then $m+k\leq j\leq m+k+L$,
$h_{j+1}<\eta h(p)\leq h_j$, and
$h_j/\eta<\mathtt Hh_m$.
If $p\in P_{j+1}$, take $q=p$.
If $p\notin\mathcal Z_j$, then
$p\in J_q$ for some $q\in P_{m+1}\cup\cdots\cup P_j$.
By \Cref{lem.LP6.13,lem.eucdistancetriangle},
$D_{\bar\psi}(y,q)\leq N_0(C+C_2)h_m$.

It remains to consider $p\in\mathcal Z_j\smallsetminus P_{j+1}$.
The definition of $P_{j+1}$ and \Cref{lem.olive} give
$D_{\bar\psi}(p,\partial\mathcal Z_j)
\leq\mathsf c h_j/(4\eta)$.
On the other hand,
$D_{\bar\psi}(p,\partial\mathcal Z_m)
\geq2h_m/(N_0\eta)-Ch_m$.
Choose $z\in\partial\mathcal Z_j$ such that
$$
D_{\bar\psi}(p,z)\le\mathsf c h_j/(4\eta).
$$
Since $h_j/\eta<\mathtt Hh_m$, the lower bound for
$D_{\bar\psi}(p,\partial\mathcal Z_m)$ implies
$z\notin\partial\mathcal Z_m$ for sufficiently small $\eta$.
Thus there is $q\in P_{m+1}\cup\cdots\cup P_j$ with
$D_{\bar\psi}(p,\partial J_q)\leq\mathsf c h_j/(4\eta)$.
Another application of \Cref{lem.LP6.13,lem.eucdistancetriangle}
gives
$$
D_{\bar\psi}(y,q)
\leq N_0Ch_m
+N_0^2\left(\frac{\mathsf c h_j}{4\eta}+C_2h_m\right)
\ll h_m.
$$
This proves the auxiliary claim.

Now fix $x\in A_n'$ and put $m:=n+k$.
Recall that $\eta\leq\mathtt H^{-k}<\mathtt H\eta$.

If $x\notin\mathcal Z_m$, then
$x\in J_p$ for some $p\in P_{n+1}\cup\cdots\cup P_m$.
By \Cref{lem.LP6.13},
$D_{\bar\psi}(x,p)<C_2\eta h(p)\leq C_2h_n$.

Suppose next that $x\in\mathcal Z_m$ and
$D_{\bar\psi}(x,\partial\mathcal Z_m)<3N_0h_m/\eta$.
Since
$3N_0h_m/\eta<3N_0\mathtt Hh_n<C_Bh_n$
and $x\notin\mathcal Z_n'$,
there is $p\in P_{n+1}\cup\cdots\cup P_m$ with
$D_{\bar\psi}(x,\partial J_p)<3N_0h_m/\eta$.
Hence
$D_{\bar\psi}(x,p)
\leq N_0(3N_0h_m/\eta+C_2h_n)\ll h_n$.

We may therefore assume that $x\in\mathcal Z_m$ and
$D_{\bar\psi}(x,\partial\mathcal Z_m)\geq3N_0h_m/\eta$.
If $x\in A_m$, the auxiliary claim gives the required point
in $P_{m+1}\cup\cdots\cup P_{m+k+L+1}$.

It remains to consider $x\in B_m$.
Since $x\notin B_n$ and the sets $B_{p,t}$ decrease with $t$,
there is $p\in Q_{n+1}''\cup\cdots\cup Q_m''$
such that $x\in B_{p,m}$.
By \Cref{lem.LP6.27,lem.LP6.32},
\begin{equation}
\label{eqn.653 proof eqn}
B_{p,m}
\subset
\mathcal B_{\bar\psi}
\left(p,C_5\sqrt{C_4h_nh_m}\right).
\end{equation}
In particular, any two points of $B_{p,m}$ are at
$D_{\bar\psi}$-distance $O(\sqrt{h_nh_m})$.
Since $\sqrt{h_nh_m}\leq\sqrt{\eta}\,h_m/\eta$,
\Cref{lem.olive} and the distance of $x$ from
$\partial\mathcal Z_m$ imply that
$B_{p,m}\subset\mathcal Z_m$ for sufficiently small $\eta$.

As in the proof of \Cref{lem.LP6.46}, the annular estimate
\Cref{lem.LP5.21}, applied in the coordinates defining $B_{p,m}$,
gives $\nu(B_{p,m}\smallsetminus B_{p,m+1})>0$.
Choose $y$ in this difference.
By \eqref{eqn.q''disjoint} and \Cref{lem.LP6.43}(2),(5),
$y\in B_m\smallsetminus B_{m+1}\subset A_{m+1}$.
Moreover, for sufficiently small $\eta$
$$
D_{\bar\psi}(y,\partial\mathcal Z_m)
\geq \frac1{N_0}D_{\bar\psi}(x,\partial\mathcal Z_m)
-D_{\bar\psi}(x,y)
> \frac{2h_m}{\eta}
$$

If $D_{\bar\psi}(y,\partial\mathcal Z_{m+1})
>2h_{m+1}/\eta$, apply the auxiliary claim at level $m+1$.
It gives
$q\in P_{m+2}\cup\cdots\cup P_{m+k+L+2}$
with $D_{\bar\psi}(y,q)\ll h_{m+1}$.
Otherwise, the preceding lower bound shows that
$D_{\bar\psi}(y,\partial J_q)\leq2h_{m+1}/\eta$
for some $q\in P_{m+1}$.
Then \Cref{lem.LP6.13,lem.eucdistancetriangle} give
$D_{\bar\psi}(y,q)\ll h_{m+1}/\eta+h_m\ll h_n$,
since $h_{m+1}/\eta\leq h_n$.

In either case, $q\in\bigcup_{j=1}^N P_{n+j}$ and
$D_{\bar\psi}(x,q)
\leq N_0(D_{\bar\psi}(x,y)+D_{\bar\psi}(y,q))
\ll h_n$.
This proves \eqref{eqn.653 proof claim} after choosing $C'$
sufficiently large.

Finally, for every $p\in\bigcup_{j=1}^N P_{n+j}$,
\Cref{lem.LP6.13} gives
$\mathcal B_{\bar\psi}(p,h_{n+N}/C_2)\subset J_p$,
and these sets $J_p$ are pairwise disjoint by
\Cref{lem.LP6.17}.
Applying \Cref{prop.LP5.14}(1) with the fixed dilation factor
$C'C_2\mathtt H^N$, we obtain, for all sufficiently large $n$,
$$
\begin{aligned}
\nu\left(\cup_{j=1}^N D_{n+j}\right)
&\geq
\sum_p\nu\left(
\mathcal B_{\bar\psi}(p,h_{n+N}/C_2)\right)\geq  C_{11}\sum_p
\nu\left(\mathcal B_{\bar\psi}(p,C'h_n)\right)
\geq C_{11}\nu(A_n'),
\end{aligned}
$$
where both sums range over
$p\in\bigcup_{j=1}^N P_{n+j}$.
The last inequality follows from \eqref{eqn.653 proof claim}.
\end{proof}

\subsection{Proof of  \Cref{prop.LP6.15}}
By \Cref{lem.LP6.43}(1), (3), (5), \Cref{lem.LP6.46},
\Cref{lem.LP6.50}, and \Cref{lem.LP6.51},
$$
\nu(A_n)\ge(1-{C}_9-{C}_{10})\nu(A_{n-1})
+({C}_8-{C}_9)\nu(B_{n-1}).
$$
Since ${C}_9\to0$ as $\eta\to0$, choose $\eta$ so small that
$2{C}_9<\min\{1-{C}_{10},{C}_8\}$, and put
$m:=\min\{1-{C}_9-{C}_{10},{C}_8-{C}_9\}
>{C}_9$.
As $\mathcal Z_{n-1}=A_{n-1}\sqcup B_{n-1}$ and
$\mathcal Z_n\subset\mathcal Z_{n-1}$, we have
$\nu(A_n)\ge m\nu(\mathcal Z_n)$. Thus \Cref{lem.LP6.50} gives
$$
\nu(A_n-\mathcal Z_n')\ge\nu(A_n)-\nu(\mathcal Z_n')
\ge(m-{C}_9)\nu(\mathcal Z_n).
$$
By \Cref{lem.LP6.53}, there exist $N\ge1$ and $c_0\in(0,1)$ such
that, for all sufficiently large $n$,
$\nu\left(\cup_{j=1}^ND_{n+j}\right)\ge c_0\nu(\mathcal Z_n)$.
Since $\mathcal Z_{n+N}=\mathcal Z_n\smallsetminus\cup_{j=1}^ND_{n+j}$,
$\nu(\mathcal Z_{n+N})\le(1-c_0)\nu(\mathcal Z_n)$.
Iteration gives $\nu(\mathcal Z_n)\ll(1-c_0)^{n/N}$. Choose $\varepsilon_0>0$ so that
$(1-c_0)^{1/N}<1-\varepsilon_0<1$.
Then $\nu(\mathcal Z_n)\le(1-\varepsilon_0)^n$
for all sufficiently large $n$, proving \Cref{prop.LP6.15}.\qed

\section{The induced expanding map}
\label{sec.expandingmap}

We now use the geometric partition to construct an expanding map with
countably many branches. We first define $T_0$ on the sets $\gamma_p\gamma_1\Delta_i$
appearing in the decompositions of the sets $J_p$.
We then successively take first-return maps until we obtain a return
map on $\Delta_0$. At each step, we prove that the return time is finite
almost everywhere and establish exponential summability of the branches.
Finally, we take a fixed power of the return map so that all return
vectors lie uniformly in the prescribed positive cone.

Let $\eta$ be sufficiently small as in \Cref{prop.LP6.15ooo}, allowong further reduction until its final choice after \Cref{lem.LP6.62}. We suppress $\eta$
from the notation. Set
$$
\widetilde\Lambda'
:=
\bigcup_{j=0}^{\ell}\ \bigcup_{n\geq1}\
\bigcup_{p\in P_{j,n}}J_p,
\qquad
\widetilde\Lambda
:=
\widetilde\Lambda'\smallsetminus
\Lambda_{\mathrm p}(\Gamma_\rho).
$$
By \Cref{lem.tilingJp},
$$
\widetilde\Lambda
=
\bigcup_{i,j=0}^{\ell}\ \bigcup_{n\geq1}\
\bigcup_{p\in\Gamma_\rho p_i\cap P_{j,n}}\
\bigcup_{\gamma_1\in N_p}
\gamma_p\gamma_1
\bigl(\Delta_i\smallsetminus
\Lambda_{\mathrm p}(\Gamma_\rho)\bigr).
$$
For each $j$, the construction gives
$$
\Delta_j\smallsetminus
\bigcup_{n\geq1}\ \bigcup_{p\in P_{j,n}}J_p
=
\bigcap_{n\geq0}\mathcal Z_{j,n}.
$$
By \Cref{prop.LP6.15ooo},
$\nu(\mathcal Z_{j,n})\to0$ as $n\to\infty$.
Hence $\widetilde\Lambda'$ has full $\nu$-measure in $\Delta$.
Since $\Lambda_{\mathrm p}(\Gamma_\rho)$ is countable and
$\nu$ is atomless by \Cref{thm.ps atomless 0},
removing the parabolic fixed points leaves
$\widetilde\Lambda$ with full $\nu$-measure in $\Delta$.

Define $T_0:\widetilde\Lambda\to\Delta$ by
$$
T_0(x)=(\gamma_p\gamma_1)^{-1}x
\qquad
\text{if }x\in
\widetilde\Lambda\cap\gamma_p\gamma_1\Delta_i.
$$

For $0\leq i,j\leq\ell$, let $\mathcal H_0(i,j)$ be the set of
elements $\gamma_p\gamma_1$ with
$p\in\Gamma_\rho p_i\cap P_{j,n}$ for some $n\geq1$ and
$\gamma_1\in N_p$.
Each $\gamma\in\mathcal H_0(i,j)$ maps $\Delta_i$ onto a cell
$\gamma\Delta_i\subset\Delta_j$, on which $T_0$ acts as
$\gamma^{-1}$ wherever it is defined.
We call $\gamma$ an inverse branch of $T_0$.
Thus $i$ records the domain of the inverse branch, and $j$
records the set containing its image.

Set
$$
\widetilde\Lambda_+
:=
\bigcap_{m\geq0}T_0^{-m}(\widetilde\Lambda).
$$
This is precisely the set of points whose iterates $T_0^m$
are defined for every $m\geq1$, and it has full $\nu$-measure
in $\Delta$ by quasi-invariance of $\nu$.
By construction, $T_0$ maps $\widetilde\Lambda_+$ into itself.

The inverse branches also preserve this set: if
$x\in\widetilde\Lambda_+\cap\Delta_i$ and
$\gamma\in\mathcal H_0(i,j)$, then
$\gamma x\in\widetilde\Lambda\cap\Delta_j$ and
$T_0(\gamma x)=x$. Hence all iterates of $\gamma x$ are
defined, so $\gamma x\in\widetilde\Lambda_+\cap\Delta_j$.

The transition graph of $T_0$ is the directed graph with vertex
set $\{0,\ldots,\ell\}$ and an edge $j\to i$ whenever
$\mathcal H_0(i,j)\ne\emptyset$.
Since this graph is finite, we can choose a nonempty set
$I_0$ of indices such that no directed edge leaves $I_0$
and the graph on $I_0$ is strongly connected: every vertex
can be reached from every other by a directed path within
$I_0$. Thus $T_0$ maps
$\widetilde\Lambda_+\cap\bigcup_{i\in I_0}\Delta_i$
into itself. We may assume the index $0$ belongs to $I_0$.

\begin{Def}[First returns to $\Delta_0$]
\label{def:firstreturnDelta0}
Let $\Lambda_+$ be the set of points
$x\in\widetilde\Lambda_+\cap\Delta_0$ such that
$T_0^m x\in\Delta_0$ for infinitely many $m\geq1$.
For $x\in\Lambda_+$, let
$n(x):=\min\{m\geq1:T_0^m x\in\Delta_0\}$.
\end{Def}

\begin{proposition}
\label{prop.T is well def a.e.}
The set $\Lambda_+$ has full $\nu$-measure in $\Delta_0$.
\end{proposition}

We prove this by removing one index other than $0$ at a time and
taking the first-return map to the remaining domains.
After $|I_0|-1$ steps, only $\Delta_0$ remains.

Set $X_0:=\widetilde\Lambda_+\cap\bigcup_{i\in I_0}\Delta_i$,
and retain the notation $T_0$ for its restriction to $X_0$.
Then $T_0$ maps $X_0$ into itself, and
$\gamma(X_0\cap\Delta_i)\subset X_0\cap\Delta_j$
for every $i,j\in I_0$ and $\gamma\in\mathcal H_0(i,j)$.

Suppose inductively that we have defined an index set $I_k$
containing $0$ and a map $T_k:X_k\to X_k$, where
$X_k\subset\bigcup_{i\in I_k}\Delta_i$ has full $\nu$-measure
in that union. For $i,j\in I_k$, let $\mathcal H_k(i,j)$
denote the group elements defining the inverse branches of $T_k$
from $\Delta_i$ to $\Delta_j$. Assume that, for every
$\gamma\in\mathcal H_k(i,j)$ and $y\in X_k\cap\Delta_i$,
we have $\gamma y\in X_k\cap\Delta_j$ and $T_k(\gamma y)=y$.
The transition graph of $T_k$ has vertex set $I_k$ and an edge
$j\to i$ whenever $\mathcal H_k(i,j)\ne\emptyset$.
Assume that this graph is strongly connected.
If $I_k=\{0\}$, the construction stops.

Otherwise, put $\ell_k:=\max I_k>0$ and
$J_k:=I_k\smallsetminus\{\ell_k\}$, and set
$Y_{k+1}:=X_k\cap\bigcup_{i\in J_k}\Delta_i$.
For $x\in Y_{k+1}$, define its first-return time by
$$
n_{k+1}(x)
:=
\inf\{m\geq1:T_k^m x\in Y_{k+1}\},
$$
with $\inf\emptyset=\infty$.
Set
$$
X_{k+1}
:=
\{x\in Y_{k+1}:T_k^m x\in Y_{k+1}
\text{ for infinitely many }m\geq1\},
$$
and define $T_{k+1}(x):=T_k^{n_{k+1}(x)}x$ for
$x\in X_{k+1}$. Put $I_{k+1}:=J_k$.
After its first return, a point of $X_{k+1}$ still returns
infinitely often to $Y_{k+1}$.
Hence $T_{k+1}$ maps $X_{k+1}$ into itself.

The inverse branches of $T_{k+1}$ are compositions of
inverse branches of $T_k$ along paths that start and end
in $J_k$ and whose intermediate vertices are all $\ell_k$.
If such a composition $\gamma$ has domain $\Delta_i$ and
$y\in X_{k+1}\cap\Delta_i$, then $\gamma y$ first returns
to $Y_{k+1}$ at $y$. Hence $\gamma y$ also returns
infinitely often, so $\gamma y\in X_{k+1}$ and
$T_{k+1}(\gamma y)=y$.

\begin{lemma}[Connectivity of the transition graphs]
\label{lem.returngraphirreducible}
For every $k$, the transition graph of $T_k$ is strongly connected.
If $I_k\neq\{0\}$, there is an edge from $\ell_k$ to a vertex
of $J_k$.
\end{lemma}

\begin{proof}
The assertion holds for $k=0$ by the choice of $I_0$.
Suppose that the transition graph of $T_k$ is strongly connected
and $I_k\neq\{0\}$.
Since $J_k$ is nonempty, a path from $\ell_k$ to $J_k$
contains an edge from $\ell_k$ to a vertex of $J_k$.

Given $i,j\in J_k$, choose a directed path from $i$ to $j$
in the graph of $T_k$ and divide it at successive visits to $J_k$.
Each resulting segment has all its intermediate vertices
equal to $\ell_k$, so it defines a branch of $T_{k+1}$.
These branches give a directed path from $i$ to $j$ in the
transition graph of $T_{k+1}$.
Thus this graph is strongly connected, completing the induction.
\end{proof}

The following proposition, proved below, shows that the restriction
to $X_{k+1}$ removes only a null set.

\begin{proposition}
\label{prop.T is well def a.e. induction}
For every $k$ with $I_k\neq\{0\}$, the set $X_{k+1}$
has full $\nu$-measure in $Y_{k+1}$.
\end{proposition}

\subsection{Estimates for the initial branch map}

Recall that
\begin{equation}
\label{eqn:H0ij}
\mathcal H_0(i,j)
=
\bigcup_{n\geq1}
\left\{
\gamma_p\gamma_1:
p\in\Gamma_\rho p_i\cap P_{j,n},\ \gamma_1\in N_p
\right\},
\qquad i,j\in I_0.
\end{equation}
Write $\mathcal H_0:=\bigcup_{i,j\in I_0}\mathcal H_0(i,j)$
for the collection of all inverse branches of $T_0$.

For $i\in I_0$, define the following neighborhood of $p_i$
in $\Lambda_\rho$:
\begin{equation}
\label{eqn:defUi}
U_i
:=
\Lambda_\rho\smallsetminus
g_i^{-1}\mathcal B_{\bar\psi}
\left(g_i\Delta_i,\frac1{2N_0\eta}\right).
\end{equation}

\begin{lemma}
\label{lem.LP6.62}
For every $\eta>0$ sufficiently small, we have
$\gamma^{-1}U_j\subset U_i$ for all $\gamma\in\mathcal H_0(i,j)$.
\end{lemma}

\begin{proof}
Write $\gamma=\gamma_p\gamma_1$, so that $\gamma p_i=p\in\Delta_j$.
Since $\Delta_j$ is compact and disjoint from $p_j$, we may choose a neighborhood $W_j$ of $p_j$ so that $\overline{W_j}\cap \Delta_j=\emptyset$.
As $\eta\to0$, the sets $U_j$ shrink to $p_j$. Decreasing $\eta$ if necessary, we may therefore assume $U_j\subset W_j$.
Compactness and \eqref{dedef}  now give $\varepsilon>0$ such that,
independent of $\eta$,
\begin{equation}
\label{eqn.multicuspu_i}
D_{\bar\psi}(g_ix,g_ip)\geq\varepsilon
\qquad(x\in U_j\smallsetminus\{p_i\}),
\end{equation}
uniformly over the finitely many pairs $i,j$ and all such $p$.

Fix $x\in U_j$.
By \eqref{eqn.LP6.6}, applied in the conjugated system,
$D_{\bar\psi}(g_i\gamma^{-1}p_i,g_i\Delta_i)\geq\eta^{-1}$.
If $x=p_i$, this already gives $\gamma^{-1}x\in U_i$.
Otherwise, \Cref{lem.ballflipping}, applied to
$g_i\gamma g_i^{-1}$, together with
\eqref{eqn.multicuspu_i} and the uniform bound
$h_i(g_ip)\ll1$, gives
$D_{\bar\psi}(g_i\gamma^{-1}x,g_i\gamma^{-1}p_i)\leq C$,
where $C$ is independent of $\gamma$, $x$, and $\eta$.
Therefore, decreasing $\eta>0$ sufficiently small that $\eta\leq(2N_0C)^{-1}$,
\Cref{lem.eucdistancetriangle} gives
$$
D_{\bar\psi}(g_i\gamma^{-1}x,g_i\Delta_i)
\geq \frac1{N_0\eta}-C
\geq \frac1{2N_0\eta}.
$$
This means that
$\gamma^{-1}x\in U_i$.
\end{proof}

For the rest of the paper, we  fix $\eta>0$ sufficiently small to satisfy all the preceding requirements, including those of \Cref{lem.LP6.62}.
\begin{lemma}
\label{lem.inversebranchshadow}
There exists $r>0$ such that, for every $n\geq1$,
$i_0,\ldots,i_n\in I_0$, and
$\gamma_m\in\mathcal H_0(i_{m-1},i_m)$ $(1\leq m\leq n)$,
the product $g:=\gamma_n\cdots\gamma_1$ satisfies
$g\Delta_{i_0}\subset\mathcal O_r(o,go)$.
\end{lemma}
\begin{proof}
Suppose otherwise. For each $k\geq1$, choose a product
$g_k:=\gamma_{k,n_k}\cdots\gamma_{k,1}$ as in the statement
and $x_k\in\Delta_{i_{k,0}}$ such that
$g_kx_k\notin\mathcal O_k(o,g_ko)$, or equivalently, $x_k\notin \mathcal O_k(g_k^{-1}o, o)$. It follows that $g_k\to \infty$ and we may therefore assume that
the $g_k$ are pairwise distinct, $i_{k,0}=i$,
$i_{k,n_k}=j$, and $x_k\to x\in\Delta_i$.

By \Cref{lem.ambient-shadow-convergence}(1) and (3),
after passing to a further subsequence, there exists
$y\in\Lambda_\rho$ such that $g_k^{-1}\zeta\to x$
for every $\zeta\in\Lambda_\rho\smallsetminus\{y\}$.
Repeated application of \Cref{lem.LP6.62} gives
$g_k^{-1}U_j\subset U_i$.
Choose
$\zeta\in U_j\smallsetminus\{y\}$;
then $g_k^{-1}\zeta\in U_i$ and $g_k^{-1}\zeta\to x$.
As $U_i$ is closed, $x\in U_i$, contradicting
$U_i\cap\Delta_i=\emptyset$.
\end{proof}

\begin{proposition}
\label{prop.expsumh0finite}
There exists $\varepsilon>0$ such that
$$\sum_{\gamma\in\mathcal H_0}
e^{-(1-\varepsilon)\psi(\mu(\gamma))}<\infty.$$
\end{proposition}

\begin{proof}  By \Cref{lem.LP6.13} and the proof of \Cref{lem.LP5.17},
applied to the finitely many conjugated systems, there exists
$r>0$ such that, for every $i\in I_0$,
$p\in\Delta\cap\Gamma_\rho p_i$, and $\gamma_1\in N_p$,
\begin{equation}
\label{eqn:shadow-gammap}
\gamma_p\gamma_1\Delta_i
\subset J_p
\subset\mathcal O_r(o,\gamma_po).
\end{equation}
After increasing $r$, we also have
\begin{equation}
\label{eqn:shadow-gamma1}
\gamma_1\Delta_i\subset\mathcal O_r(o,\gamma_1o).
\end{equation}
Indeed, otherwise equivariance and
\Cref{lem.ambient-shadow-convergence}(3), applied to
a sequence in $\mathsf P_i$, would force a sequence
in $\Delta_i$ to converge to $p_i$, contradicting
compactness of $\Delta_i$ and $p_i\notin\Delta_i$.

Now fix $j\in I_0$ and restrict to
$p\in P_{j,n}\cap\Gamma_\rho p_i$, with $n\geq1$ and
$i\in I_0$. For $\gamma_1\in N_p$ and $x\in\Delta_i$,
the cocycle identity gives
$$
\beta_{\gamma_p\gamma_1x}(e,\gamma_p\gamma_1)
=
\beta_{\gamma_p\gamma_1x}(e,\gamma_p)
+
\beta_{\gamma_1x}(e,\gamma_1).
$$
By \eqref{eqn:shadow-gammap}, \eqref{eqn:shadow-gamma1},
and \Cref{lem.buseandcartan}, the two terms on the right
are $\mu(\gamma_p)+O(1)$ and $\mu(\gamma_1)+O(1)$,
respectively, with uniform errors.
By \Cref{lem.inversebranchshadow},
$\gamma_p\gamma_1\Delta_i
\subset\mathcal O_r(o,\gamma_p\gamma_1o)$.
Hence \Cref{lem.buseandcartan} and conformality of $\nu$ give
$\nu(\gamma_p\gamma_1\Delta_i)
\asymp e^{-\psi(\mu(\gamma_p\gamma_1))}$.
Combining this with the preceding cocycle estimate, we obtain
\begin{equation}
\label{eqn.intermediateareaJp}
e^{-\psi(\mu(\gamma_p\gamma_1))}
\asymp \nu(\gamma_p\gamma_1\Delta_i)
\asymp
e^{-\psi(\mu(\gamma_p))}
e^{-\psi(\mu(\gamma_1))}.
\end{equation}

By \Cref{thm.entropydrop}, choose $\varepsilon>0$ small enough
that $1-\varepsilon>\delta_\psi(\mathsf P_i)$ for every $i$.
For fixed $\eta$, the set $N_p$ depends only on the type $i$.
Both sums below are finite and strictly positive, since
$N_p$ is cofinite in $\mathsf P_i$.
As there are only finitely many types, we obtain, uniformly in $p$,
$$
\sum_{\gamma_1\in N_p}
e^{-(1-\varepsilon)\psi(\mu(\gamma_1))}
\asymp
\sum_{\gamma_1\in N_p}e^{-\psi(\mu(\gamma_1))}.
$$
Together with \eqref{eqn.intermediateareaJp} and
\Cref{lem.tilingJp}, this gives
\begin{equation}
\label{eqn.expsummableepsilon}
\begin{aligned}
\sum_{\gamma_1\in N_p}
e^{-(1-\varepsilon)\psi(\mu(\gamma_p\gamma_1))}
 \asymp
e^{\varepsilon\psi(\mu(\gamma_p))}
\sum_{\gamma_1\in N_p}\nu(\gamma_p\gamma_1\Delta_i)=
e^{\varepsilon\psi(\mu(\gamma_p))}\nu(J_p).
\end{aligned}
\end{equation}

By \eqref{lb}, $\psi^-(\mu(\gamma))$ is bounded below
uniformly on $\Gamma_\rho$, so
$\psi(\mu(\gamma))\leq2\bar\psi(\mu(\gamma))+O(1)$.
For $p\in P_{j,n}$,
\Cref{lem.heightondelta,lem.LP5.11} and the definition
of $P_{j,n}$ give
$\bar\psi(\mu(\gamma_p))\leq n\log\mathtt H+O(1)$,
with the fixed parameter $\eta$ absorbed in the error.
Hence, by \eqref{eqn:H0ij} and \eqref{eqn.expsummableepsilon},
$$
\begin{aligned}
\sum_{i\in I_0}\sum_{\gamma\in\mathcal H_0(i,j)}
e^{-(1-\varepsilon)\psi(\mu(\gamma))}
&\ll
\sum_{n\geq1}\mathtt H^{2\varepsilon n}
\sum_{p\in P_{j,n}}\nu(J_p)\ll
\sum_{n\geq1}\mathtt H^{2\varepsilon n}
\nu(\mathcal Z_{j,n-1})\ll
\sum_{n\geq1}
\bigl(\mathtt H^{2\varepsilon}(1-\varepsilon_0)\bigr)^n.
\end{aligned}
$$
The last estimate follows from \Cref{prop.LP6.15ooo}.
Taking $\varepsilon$ smaller if necessary so that
$\mathtt H^{2\varepsilon}(1-\varepsilon_0)<1$ proves
the desired summability for every $j\in I_0$, and hence
for $\mathcal H_0$.
\end{proof}

\subsection{The inductive argument}

We now prove the estimates needed at each step of the induction.
Fix $k$ with $I_k\neq\{0\}$.
For $I,J\subset I_k$, write
$\mathcal H_k(I,J):=
\bigcup_{i\in I,\ j\in J}\mathcal H_k(i,j)$,
and set $\mathcal H_k:=\mathcal H_k(I_k,I_k)$.

\begin{lemma}
\label{prop.hqexpsumdecay}
There exists $\varepsilon_k>0$ such that, for every $n\geq1$,
$$
\sum_{\gamma_1,\ldots,\gamma_n\in
\mathcal H_k(\ell_k,\ell_k)}
e^{-\psi(\mu(\gamma_1\cdots\gamma_n))}
\ll(1-\varepsilon_k)^n.
$$
\end{lemma}

\begin{proof}
Set $\mathcal J:=\mathcal H_k(\ell_k,\ell_k)$ and
$E:=\Delta_{\ell_k}\smallsetminus
\bigcup_{\gamma\in\mathcal J}\gamma\Delta_{\ell_k}$.
By \Cref{lem.returngraphirreducible}, there exist $i\in J_k=I_k\smallsetminus\{\ell_k\}$
and $\gamma\in\mathcal H_k(i,\ell_k)$.
The set $\gamma\Delta_i$ is disjoint from
$\bigcup_{\delta\in\mathcal J}\delta\Delta_{\ell_k}$
up to a $\nu$-null set.
Since $\nu(\Delta_i)>0$ and $\nu$ is quasi-invariant,
$\nu(E)\geq\nu(\gamma\Delta_i)>0$.

Each branch of $T_k$ is a composition of branches of $T_0$.
By conformality and
\Cref{lem.inversebranchshadow,lem.buseandcartan},
there exists $C>1$ such that, for $g=e$ or any finite
product of elements of $\mathcal J$, and every Borel
$A\subset\Delta_{\ell_k}$,
\begin{equation}
\label{eqn.hqexpsum2}
C^{-1}e^{-\psi(\mu(g))}\nu(A)
\leq\nu(gA)\leq
Ce^{-\psi(\mu(g))}\nu(A).
\end{equation}
Set
$\varepsilon_k:=C^{-2}\nu(E)/\nu(\Delta_{\ell_k})\in(0,1)$.
Then $\nu(gE)\geq\varepsilon_k\nu(g\Delta_{\ell_k})$.
The sets $g\gamma\Delta_{\ell_k}$, $\gamma\in\mathcal J$,
are disjoint up to $\nu$-null sets, and their complement
in $g\Delta_{\ell_k}$ is $gE$.
Thus their total measure is at most
$(1-\varepsilon_k)\nu(g\Delta_{\ell_k})$.
Iterating this estimate and using the lower bound in
\eqref{eqn.hqexpsum2}, we obtain
$$
\begin{aligned}
\sum_{\gamma_1,\ldots,\gamma_n\in\mathcal J}
e^{-\psi(\mu(\gamma_1\cdots\gamma_n))}
&\leq
\frac{C}{\nu(\Delta_{\ell_k})}
\sum_{\gamma_1,\ldots,\gamma_n\in\mathcal J}
\nu(\gamma_1\cdots\gamma_n\Delta_{\ell_k})\leq C(1-\varepsilon_k)^n.
\end{aligned}
$$
\end{proof}

We next estimate the return time to $Y_{k+1}$.

\begin{proposition}
\label{prop.aerecurrentexapndingmap}
There exists $\varepsilon_k>0$ such that, for every $n\geq1$,
$$
\nu\bigl(\{x\in Y_{k+1}:n_{k+1}(x)>n+1\}\bigr)
\ll(1-\varepsilon_k)^n.
$$
In particular, $X_{k+1}$ has full $\nu$-measure in $Y_{k+1}$.
\end{proposition}

\begin{proof}
Set $\mathcal J:=\mathcal H_k(\ell_k,\ell_k)$.
Every $x\in Y_{k+1}$ with $n_{k+1}(x)>n+1$ belongs to
a set $\gamma_0g\Delta_{\ell_k}$, where
$\gamma_0\in\mathcal H_k(\ell_k,J_k)$ and
$g=\gamma_1\cdots\gamma_n$ with $\gamma_m\in\mathcal J$.

Applying conformality and
\Cref{lem.inversebranchshadow,lem.buseandcartan}
to $\gamma_0$, and \eqref{eqn.hqexpsum2} to $g$, gives
$$
\begin{aligned}
\nu(\gamma_0g\Delta_{\ell_k})
&\ll
\frac{\nu(\gamma_0\Delta_{\ell_k})}
     {\nu(\Delta_{\ell_k})}
\nu(g\Delta_{\ell_k})\ll
\nu(\gamma_0\Delta_{\ell_k})e^{-\psi(\mu(g))}.
\end{aligned}
$$

The sets $\gamma_0\Delta_{\ell_k}$, with
$\gamma_0\in\mathcal H_k(\ell_k,J_k)$, are disjoint
up to null sets, so their measures sum to at most $1$.
Every point with $n_{k+1}(x)>n+1$ lies in a set
$\gamma_0\gamma_1\cdots\gamma_n\Delta_{\ell_k}$,
where $\gamma_0\in\mathcal H_k(\ell_k,J_k)$ and
$\gamma_1,\ldots,\gamma_n\in\mathcal J$.
Summing the preceding estimate over these choices gives
$$
\begin{aligned}
\nu\bigl(\{x\in Y_{k+1}:n_{k+1}(x)>n+1\}\bigr)
&\ll
\left(
\sum_{\gamma_0\in\mathcal H_k(\ell_k,J_k)}
\nu(\gamma_0\Delta_{\ell_k})
\right)
\left(
\sum_{\gamma_1,\ldots,\gamma_n\in\mathcal J}
e^{-\psi(\mu(\gamma_1\cdots\gamma_n))}
\right)\\
&\ll(1-\varepsilon_k)^n
\end{aligned}
$$
by \Cref{prop.hqexpsumdecay}. Letting $n\to\infty$ shows that the set of points in
$Y_{k+1}$ that never return to $Y_{k+1}$ has measure zero.
Since $Y_{k+1}\smallsetminus X_{k+1}
\subset \bigcup_{m\geq0}T_k^{-m}
\bigl(\{y\in Y_{k+1}:n_{k+1}(y)=\infty\}\bigr)$, $X_{k+1}$ has full $\nu$-measure in $Y_{k+1}$.
\end{proof}

The inverse branches of $T_{k+1}$ are
$$
\mathcal H_{k+1}
=
\mathcal H_k(J_k,J_k)
\cup
\bigcup_{n\geq0}
\mathcal H_k(\ell_k,J_k)
\mathcal H_k(\ell_k,\ell_k)^n
\mathcal H_k(J_k,\ell_k),
$$
where the zeroth power is $\{e\}$.

\begin{proposition}
\label{prop.expsumh1finite}
Assume that, for some $\varepsilon>0$,
$\sum_{\gamma\in\mathcal H_k}
e^{-(1-\varepsilon)\psi(\mu(\gamma))}<\infty$.
Then there exists $\varepsilon_1>0$ such that
$\sum_{\gamma\in\mathcal H_{k+1}}
e^{-(1-\varepsilon_1)\psi(\mu(\gamma))}<\infty$.
\end{proposition}

\begin{proof}
Decreasing $\varepsilon$ if necessary, we may assume that
$0<\varepsilon<1$.
Set $\mathcal J:=\mathcal H_k(\ell_k,\ell_k)$.
The cocycle identity and
\Cref{lem.inversebranchshadow,lem.buseandcartan}
give a constant $C>0$ such that
$$\psi(\mu(\gamma_1\cdots\gamma_n))
\geq\sum_{m=1}^n\psi(\mu(\gamma_m))-Cn$$
for every $\gamma_1,\ldots,\gamma_n\in\mathcal J$.
The hypothesis therefore implies that, for some $B\geq1$,
$$
\sum_{\gamma_1,\ldots,\gamma_n\in\mathcal J}
e^{-(1-\varepsilon)\psi(\mu(\gamma_1\cdots\gamma_n))}
\leq B^n.
$$ Choose $0<\theta<1$ small enough that
$(1-\varepsilon_k)^{1-\theta}B^\theta<1$,
where $\varepsilon_k$ is given by
\Cref{prop.hqexpsumdecay}. Since $1-\theta\varepsilon=(1-\theta)+\theta(1-\varepsilon)$,
Hölder's inequality with exponents $1/(1-\theta)$ and
$1/\theta$ gives
\begin{equation}
\label{eqn.segmentexpdecayh1}
\begin{aligned}
&\sum_{\gamma_1,\ldots,\gamma_n\in\mathcal J}
e^{-(1-\theta\varepsilon)\psi(\mu(\gamma_1\cdots\gamma_n))}\\
&\quad\leq
\left(
\sum_{\gamma_1,\ldots,\gamma_n\in\mathcal J}
e^{-\psi(\mu(\gamma_1\cdots\gamma_n))}
\right)^{1-\theta}
\left(
\sum_{\gamma_1,\ldots,\gamma_n\in\mathcal J}
e^{-(1-\varepsilon)\psi(\mu(\gamma_1\cdots\gamma_n))}
\right)^\theta\\
&\quad\ll
\bigl((1-\varepsilon_k)^{1-\theta}B^\theta\bigr)^n,
\end{aligned}
\end{equation}
where the last inequality follows from
\Cref{prop.hqexpsumdecay} and the preceding bound.
Set $\e_1=\theta \e$.
Then
\begin{equation}
\label{eqn.expsumh1finite}
\sum_{n\geq0}
\sum_{\gamma_1,\ldots,\gamma_n\in\mathcal J}
e^{-(1-\varepsilon_1)\psi(\mu(\gamma_1\cdots\gamma_n))}
<\infty.
\end{equation}

Every branch in
$\mathcal H_{k+1}\smallsetminus\mathcal H_k(J_k,J_k)$
has the form $\gamma_0g\gamma_0'$, where
$\gamma_0\in\mathcal H_k(\ell_k,J_k)$,
$\gamma_0'\in\mathcal H_k(J_k,\ell_k)$,
and $g$ is a product of elements of $\mathcal J$,
possibly empty.
The same cocycle argument gives
$$
\mu(\gamma_0g\gamma_0')
=
\mu(\gamma_0)+\mu(g)+\mu(\gamma_0')+O(1),
$$
with a uniform error. Consequently, the sum of
$e^{-(1-\varepsilon_1)\psi(\mu(\gamma_0g\gamma_0'))}$
over these branches is bounded by a constant times the product
of the sums over $\gamma_0$, $g$, and $\gamma_0'$.
The sum over $g$ is finite by \eqref{eqn.expsumh1finite}.
Since $\varepsilon_1<\varepsilon$ and
$\psi(\mu(\gamma))$ is uniformly bounded below by \eqref{lb},
the hypothesis implies that
$\sum_{\gamma\in\mathcal H_k}
e^{-(1-\varepsilon_1)\psi(\mu(\gamma))}<\infty$.
This bounds both outer sums and the sum over
$\mathcal H_k(J_k,J_k)$, proving the proposition.
\end{proof}

\subsection{Proofs of \Cref{prop.T is well def a.e. induction}
and \Cref{prop.T is well def a.e.}}

\begin{proof}
Starting with $X_0$, apply
\Cref{prop.aerecurrentexapndingmap} at each stage.
It follows inductively that $X_{k+1}$ has full $\nu$-measure
in $Y_{k+1}$, and hence in
$\bigcup_{i\in I_{k+1}}\Delta_i$.
This proves \Cref{prop.T is well def a.e. induction}.
Moreover, \Cref{prop.expsumh0finite,prop.expsumh1finite}
show that, at every stage $k$, there exists $\varepsilon>0$
such that
$\sum_{\gamma\in\mathcal H_k}
e^{-(1-\varepsilon)\psi(\mu(\gamma))}<\infty$.

Let $r:=|I_0|-1$, so that $I_r=\{0\}$.
Every point of $\Lambda_+$ returns to $\Delta_0$
infinitely often and therefore belongs to every $X_k$.
Thus $\Lambda_+\subset X_r$.
Conversely, the iterates of $T_r$ give successive returns
under $T_0$ to $\Delta_0$ at strictly increasing times.
Hence $X_r\subset\Lambda_+$, and therefore $X_r=\Lambda_+$.
Since $X_r$ has full $\nu$-measure in $\Delta_0$,
this proves \Cref{prop.T is well def a.e.}.
The map $T_r$ is the first-return map of $T_0$ to $\Delta_0$.
\end{proof}

We henceforth write $T_{\mathrm{ret}}:=T_r$ and
$\mathcal H_{\mathrm{ret}}:=\mathcal H_r$.
Thus $T_{\mathrm{ret}}(x)=T_0^{n(x)}x$ for
$x\in\Lambda_+$, and $\mathcal H_{\mathrm{ret}}$
is its collection of inverse branches.

\subsection{A fixed iterate with positive roof}

We will choose $n_0\geq1$ so that the return vectors of
$T:=T_{\mathrm{ret}}^{\,n_0}$ satisfy the required cone
condition and scalar lower bound.
In terms of $T_0$,
$$
T(x)
=T_0^{N_{n_0}(x)}(x), \qquad
N_{n_0}(x):=\sum_{j=0}^{n_0-1}n(T_{\mathrm{ret}}^jx).
$$

Each first-return branch is represented by a product
$\gamma_1\cdots\gamma_n$, where
$\gamma_m\in\mathcal H_0(i_m,i_{m-1})$,
$i_0=i_n=0$, and $i_m\neq0$ for $0<m<n$.
For $x\in\Lambda_+$, the corresponding sequence of
branches is uniquely determined by its first $n(x)$
iterates under $T_0$.

The sets $\gamma\Lambda_+$,
$\gamma\in\mathcal H_{\mathrm{ret}}$, are the first-return
cylinders.
Until \Cref{lem.inversebranchfreesemigroup_before} is proved,
$\mathcal H_{\mathrm{ret}}^{\,m}$ denotes the set of ordered
$m$-tuples of branches; in expressions such as
$\gamma\Lambda_+$, $\gamma$ denotes their group product.
For every $m\geq1$,
$$
\Lambda_+
=
\bigsqcup_{\gamma\in\mathcal H_{\mathrm{ret}}^{\,m}}
\gamma\Lambda_+.
$$
Since $\Delta_0$ is the closure of its relative interior, $\Lambda_+$ is dense in
$\Delta_0$. Thus $\gamma\Delta_0$ is the closure of
$\gamma\Lambda_+$.
The following lemma places these closures in the relative
interior of $\Delta_0$.

\begin{lemma}
\label{lem.branchcompactinterior}
For every $b\in\mathcal H_0(i,j)$,
$b\Delta_i\subset\operatorname{int}_{\Lambda_\rho}(\Delta_j)$.
Consequently,
$\gamma\Delta_0\subset
\operatorname{int}_{\Lambda_\rho}(\Delta_0)$
for every nonempty product $\gamma$ of elements of
$\mathcal H_{\mathrm{ret}}$.
\end{lemma}

\begin{proof}
Let $b\in\mathcal H_0(i,j)$.
Then $b\Delta_i\subset J_p$ for some $p\in P_{j,n}$.
Set $q:=g_jp$ and work in the conjugated system $\Gamma_j$.
Since $\eta h_j(q)\leq h_{n-1}$ and
$\eta<1/(4C_2)$, we have
$C_2\eta h_j(q)<h_{n-1}/(4\eta)$.
Hence \Cref{lem.LP6.13,def.ZnPnDn} give
$$
g_j(b\Delta_i)
\subset g_jJ_p
\subset
\mathcal B_{\bar\psi}
\left(q,\frac{h_{n-1}}{4\eta}\right)
\subset g_j\Delta_j.
$$
The ball in this display is open in $g_j\Lambda_\rho$.
Applying $g_j^{-1}$ proves the first assertion.

In a product defining a first-return branch, the leftmost
factor belongs to $\mathcal H_0(i,0)$ for some $i$,
and the remaining factors map $\Delta_0$ into $\Delta_i$.
Its image is therefore contained in the relative interior
of $\Delta_0$. Any further composition of first-return
branches has image contained in that of its leftmost
factor, proving the final assertion.
\end{proof}

\begin{lemma}
\label{lem.inversebranchfreesemigroup_before}
The set $\mathcal H_{\mathrm{ret}}$ freely generates a semigroup.
\end{lemma}

\begin{proof}
Suppose that
$\gamma_1\cdots\gamma_n=\gamma_1'\cdots\gamma_m'.
$
The common image of $\Lambda_+$ is a nonempty subset of both
$\gamma_1\Lambda_+$ and $\gamma_1'\Lambda_+$.
Since distinct first-return branches have disjoint images,
$\gamma_1=\gamma_1'$.
Cancelling and repeating identifies the factors up to the
shorter length.

By \Cref{lem.branchcompactinterior}, every nonempty product
maps $\Delta_0$ into
$\operatorname{int}_{\Lambda_\rho}(\Delta_0)$.
Since $\Delta_0$ has nonempty relative boundary, this image
is a proper subset of $\Delta_0$, so the product cannot
be the identity.
The two lengths must therefore agree, proving the lemma.
\end{proof}

Recall that
$U_0
=
\Lambda_\rho\smallsetminus
\mathcal B_{\bar\psi}
\left(\Delta_0,\frac1{2N_0\eta}\right)$.
Repeated application of \Cref{lem.LP6.62} gives
$\gamma^{-1}U_0\subset U_0$ for every nonempty product
$\gamma$ of elements of $\mathcal H_{\mathrm{ret}}$.

\begin{lemma}
\label{lem.U0shadow}
There exists $r>0$ such that
$U_0\subset\mathcal O_r(\gamma o,o)$ for every
$\gamma\in\bigcup_{m\geq1}\mathcal H_{\mathrm{ret}}^{\,m}$.
\end{lemma}

\begin{proof}
By \Cref{lem.inversebranchshadow}, there exists $r>0$
such that
\begin{equation}
\label{shadowH}
\gamma\Delta_0\subset\mathcal O_r(o,\gamma o)
\qquad
\left(\gamma\in
\cup_{m\geq1}\mathcal H_{\mathrm{ret}}^{\,m}\right).
\end{equation}
Suppose the assertion fails.
For each $k\geq1$, choose $x_k\in U_0$ and
$\gamma_k\in\bigcup_{m\geq1}\mathcal H_{\mathrm{ret}}^{\,m}$
such that $x_k\notin\mathcal O_k(\gamma_ko,o)$.
Then $d(o,\gamma_ko)\geq k$.
Passing to a subsequence, assume that the $\gamma_k$
are pairwise distinct and $x_k\to x\in U_0$.

Fix $z\in\Delta_0$.
By \eqref{shadowH} and
\Cref{lem.ambient-shadow-convergence}(2) and (3),
after passing to a further subsequence, $\gamma_kz\to x$.
Since $\gamma_kz\in\Delta_0$ and $\Delta_0$ is compact,
$x\in\Delta_0$, contradicting $U_0\cap\Delta_0=\emptyset$.
\end{proof}

\begin{lemma}
\label{lem.choice of n_0}
Let $\mathcal D_0\subset\mathfrak a$ be compact, and let
$\mathcal C_0\subset\mathfrak a^+$ be a closed cone such that
$\mathcal L_\rho\smallsetminus\{0\}
\subset\operatorname{int}\mathcal C_0$.
For every $C_0>0$, there exists $n_0\geq1$ such that,
for every $n\geq n_0$ and every
$\gamma=\gamma_1\cdots\gamma_n$ with
$\gamma_j\in\mathcal H_{\mathrm{ret}}$:
\begin{enumerate}
\item
$\beta_x(e,\gamma)+\mathcal D_0\subset\mathcal C_0$
for every $x\in\Lambda_+\cap\gamma\Delta_0$;
\item
$\varphi(\beta_x(e,\gamma))\geq C_0$
for every $x\in\Lambda_+\cap\gamma\Delta_0$ and
$\varphi\in\{\psi,\bar\psi\}$;
\item
$\varphi(\beta_y(\gamma,e))\geq C_0$
for every $y\in U_0$ and $\varphi\in\{\psi,\bar\psi\}$.
\end{enumerate}
\end{lemma}

\begin{proof}
By \eqref{shadowH} and
\Cref{lem.U0shadow,lem.buseandcartan},
$$
\beta_x(e,\gamma)=\mu(\gamma)+O(1),
\qquad
\beta_y(\gamma,e)=\i\mu(\gamma)+O(1),
$$
uniformly for the products $\gamma$ and points $x,y$
in the statement.

By the definition of $\mathcal L_\rho$, every accumulation
point of $\mu(\gamma)/\|\mu(\gamma)\|$ as
$\|\mu(\gamma)\|\to\infty$ belongs to $\mathcal L_\rho$. Since $\mathcal D_0$ is compact, the hypothesis on
$\mathcal C_0$ and positivity of $\psi$, $\psi^-$, and
$\bar\psi$ on $\mathcal L_\rho\smallsetminus\{0\}$
give all three conclusions whenever $\|\mu(\gamma)\|$
is sufficiently large. 
By discreteness, only finitely many elements of
$\Gamma_\rho$ have Cartan projection below this norm bound.
By \Cref{lem.inversebranchfreesemigroup_before}, each such
element has at most one expression as a product of elements
of $\mathcal H_{\mathrm{ret}}$.
The lengths of these exceptional products are therefore
bounded. Choosing $n_0$ larger than these lengths proves
the lemma.
\end{proof}

\subsection{The final expanding map}
\label{subsec.takingpowerexpandingmap}

We now fix $n_0$ and the map $T$.
Since $\Delta_0$ is compact and disjoint from $p_0$,
there is a compact set $\mathcal D_0\subset\mathfrak a$ such that
\begin{equation}
\label{Gbound}
2\mathcal G(x,p_0)-2\mathcal G(y,p_0)\in\mathcal D_0
\qquad(x,y\in\Delta_0).
\end{equation}
Choose a closed cone
$\mathcal C_0\subset\operatorname{int}\mathfrak a^+\cup\{0\}$
such that
$\mathcal L_\rho\smallsetminus\{0\}
\subset\operatorname{int}\mathcal C_0$
and $\psi,\psi^-,\bar\psi>0$ on
$\mathcal C_0\smallsetminus\{0\}$.
Apply \Cref{lem.choice of n_0} with
$$
C_0=C_R
:=
1+4\max_{\varphi\in\{\psi,\bar\psi\}}
\sup_{x\in\Delta_0}|\varphi(\mathcal G(x,p_0))|.
$$
Fix the resulting $n_0$ and define
$$
T:=T_{\mathrm{ret}}^{\,n_0}:\Lambda_+\to\Lambda_+.
$$
This is our final expanding map, whose inverse branches form
$\mathcal H:=\mathcal H_{\mathrm{ret}}^{\,n_0}$.
By
\Cref{lem.inversebranchfreesemigroup_before}, 
$\mathcal H$ freely generates a semigroup.
Write $\sH:=\bigcup_{n\geq0}\mathcal H^n$,
with $\mathcal H^0:=\{e\}$.

\part{The return cocycle} \label{part.return}
We establish the symbolic coding and the exponential moments
and uniform distortion estimates of the return cocycle.
We then prove Zariski density of the return semigroup,
non-arithmeticity of the return periods, and quantitative
non-integrability.

\section{Symbolic coding and cocycle estimates}
\label{sec.symbolic-coding}
We identify the expanding map $T:\Lambda_+\to\Lambda_+$
with the one-sided full shift on its countable family
$\mathcal H$ of inverse branches.
We then define the return cocycle and prove the moment
and regularity estimates needed for the spectral analysis.

Equip the countable alphabet $\mathcal H$ with the discrete
topology, and let $\Sigma^+:=\mathcal H^{\mathbb N}$.
For distinct $x,y\in\Sigma^+$, set
$n(x,y):=\min\{j\geq1:x_j\neq y_j\}$ and define
\begin{equation}
\label{eqn.symbolicmetric}
d_{\Sigma^+}(x,y):=2^{-n(x,y)},
\qquad d_{\Sigma^+}(x,x):=0.
\end{equation}
The one-sided shift is
$\sigma(x_1,x_2,\ldots):=(x_2,x_3,\ldots)$.
Since all transitions are allowed, this full shift is
topologically mixing and has the big-images-and-preimages
(BIP) property: any fixed symbol can precede and follow
every symbol; see \cite{Sarig_Countable,Sarig_Gibbs}.
For $n\geq1$, write $\Sigma_n^+:=\mathcal H^n$ for the words
of length $n$. The cylinder of
$\omega=(\omega_1,\ldots,\omega_n)\in\Sigma_n^+$ is
$$
[\omega]:=\{x\in\Sigma^+:x_j=\omega_j\text{ for }1\leq j\leq n\}.
$$

\subsection{Conjugacy with full shift} \label{subsec:symbolicspaceforcoding}

We retain the notation $\Delta_0$, $\Lambda_+$, $T$, and $\mathcal H$ from
\Cref{sec.expandingmap}.  Thus
$$
\Lambda_+
=
\bigsqcup_{\gamma\in\mathcal H}\gamma\Lambda_+,
\qquad
T|_{\gamma\Lambda_+}=\gamma^{-1},
$$
and, more generally,
\begin{equation}
\label{eqn:powered-branch-partition}
\Lambda_+
=
\bigsqcup_{\gamma\in\mathcal H^n}\gamma\Lambda_+
\qquad(n\geq1).
\end{equation}
The compact set $\gamma\Delta_0$ is the closure of the corresponding
branch cylinder.

\begin{lemma}
\label{lem.welldefPhi}
For every $(\gamma_1,\gamma_2,\ldots)\in\Sigma^+$, the compact sets
$
\gamma_1\cdots\gamma_n\Delta_0, n\geq1,
$
form a decreasing sequence whose intersection consists of a single point.
Moreover, this point belongs to $\Lambda_+$.
\end{lemma}
\begin{proof}
The sets are nonempty, compact, and nested because every
inverse branch maps $\Delta_0$ into itself.
Put $g_n:=\gamma_1\cdots\gamma_n$.
By \Cref{lem.inversebranchfreesemigroup_before}, the elements
$g_n$ are distinct. By \Cref{lem.inversebranchshadow},
there exists $r>0$, independent of $n$, such that
$$
g_n\Delta_0\subset\mathcal O_r(o,g_no)\cap\Lambda_\rho.
$$
Hence \Cref{lem.ambient-shadow-convergence} gives
$\operatorname{diam}_{d_{\mathrm{Riem}}}(g_n\Delta_0)\to0$.
The intersection therefore consists of a single point $x$.

Since $x\in\mathcal O_r(o,g_no)$ for every $n$,
\Cref{lem.shadowscomparison} shows that $f^{-1}(x)$
is a conical limit point of $\Gamma$.
In particular, $x\notin\Lambda_{\mathrm p}(\Gamma_\rho)$.

Expand each $\gamma_n$ into its finite word of basic
$T_0$-branches. The inclusions $x\in g_n\Delta_0$
show that every intermediate image of $x$ belongs to
the corresponding basic cell.
Each such image is nonparabolic, since the set of
parabolic fixed points is $\Gamma_\rho$-invariant.
It therefore belongs to $\widetilde\Lambda$, and the
prescribed $T_0$-iterate is defined.
Each $\gamma_n$ records $n_0$ successive returns to
$\Delta_0$, so this orbit returns to $\Delta_0$
infinitely often. Thus $x\in\Lambda_+$, and
$T^nx=g_n^{-1}x$ for every $n\geq1$.
\end{proof}
Define
$$
\Psi:\Sigma^+\to\Lambda_+
$$
by letting $\Psi((\gamma_j)_{j\geq1})$ be the unique point in
$\bigcap_{n\geq1}\gamma_1\cdots\gamma_n\Delta_0$.

\begin{proposition}
\label{prop.shiftandexpanding}
The map $\Psi:\Sigma^+\to\Lambda_+$ is a homeomorphism and
$$
\Psi\circ\sigma=T\circ\Psi.
$$
\end{proposition}
\begin{proof}
For each $x\in\Lambda_+$, the branch partition determines
a unique sequence $(\gamma_n)$ such that
$T^{n-1}x\in\gamma_n\Lambda_+$ for every $n\geq1$.
Then $x\in\gamma_1\cdots\gamma_n\Delta_0$ for every $n$,
so $x=\Psi((\gamma_n))$ by \Cref{lem.welldefPhi}.
Conversely, the proof of that lemma shows that the image
of each sequence follows its prescribed branches.
Thus $\Psi$ is a bijection and
$\Psi\circ\sigma=T\circ\Psi$.

Fix $s=(\gamma_1,\gamma_2,\ldots)\in\Sigma^+$ and put
$g_n:=\gamma_1\cdots\gamma_n$.
Sequences agreeing with $s$ in their first $n$ entries
have images in $g_n\Delta_0$.
Since the diameters of these sets tend to zero,
$\Psi$ is continuous.

To prove continuity of $\Psi^{-1}$, put $x:=\Psi(s)$.
By \Cref{lem.branchcompactinterior},
$$
x\in g_n\gamma_{n+1}\Delta_0
\subset g_n\operatorname{int}_{\Lambda_\rho}(\Delta_0).
$$
The branch definitions give
$\Lambda_+\cap g_n\Delta_0=g_n\Lambda_+
=\Psi([\gamma_1,\ldots,\gamma_n])$.
Thus
$\Lambda_+\cap g_n\operatorname{int}_{\Lambda_\rho}(\Delta_0)$
is a neighborhood of $x$ contained in the image of this
symbolic cylinder. Since symbolic cylinders form a basis,
$\Psi^{-1}$ is continuous.
\end{proof}

Recall the constant $C_R$ used in the choice of $n_0$
in \Cref{subsec.takingpowerexpandingmap}.

\begin{lemma}[Uniform branch contraction]
\label{lem:uniform-branch-contraction}
For every $n\geq1$, $\gamma\in\mathcal H^n$, and
$x,y\in\Lambda_+$,
\begin{equation}
\label{eqn:contraction metric}
d_{\bar\psi}(\gamma x,\gamma y)
\leq e^{-C_Rn}d_{\bar\psi}(x,y).
\end{equation}
\end{lemma}

\begin{proof}
For $\gamma\in\mathcal H$, \eqref{eqn.metrics_sph}
and \Cref{lem.choice of n_0}(2) give
$$
\begin{aligned}
d_{\bar\psi}(\gamma x,\gamma y)
&=d_{\bar\psi}(x,y)
e^{-\frac12\bar\psi(
\beta_{\gamma x}(e,\gamma)+\beta_{\gamma y}(e,\gamma))}\leq e^{-C_R}d_{\bar\psi}(x,y).
\end{aligned}
$$
Since every inverse branch maps $\Lambda_+$ into itself,
iteration gives the estimate for $\gamma\in\mathcal H^n$.
\end{proof}

\begin{proposition}
\label{lem.phiholder}
There exists $C>0$ such that, for all $s,t\in\Sigma^+$,
$$
d_{\bar\psi}(\Psi(s),\Psi(t))
\leq C d_{\Sigma^+}(s,t)^{\alpha_0},
\qquad
\alpha_0:=\frac{C_R}{\log2}.
$$
In particular, $\Psi$ is H\"older continuous.
\end{proposition}

\begin{proof}
The assertion is immediate when $s=t$.
Otherwise, put $N:=n(s,t)$ so that $d_{\Sigma^+}(s,t)=2^{-N}$. So $s$ and $t$ share their first $N-1$ symbols, so
\Cref{lem:uniform-branch-contraction,lem.dpsidiameterfinite}
give
$$
d_{\bar\psi}(\Psi(s),\Psi(t))
\leq e^{-C_R(N-1)}
\operatorname{diam}_{d_{\bar\psi}}(\Lambda_\rho)
\ll d_{\Sigma^+}(s,t)^{C_R/\log2}.
$$
\end{proof}

\subsection{Return cocycle and positivity}
\label{sec.exptail}

We define the return cocycle for $T=T_{\mathrm{ret}}^{\,n_0}$,
which records $n_0$ successive first returns to $\Delta_0$,
and establish its positivity and comparison with the
Cartan projection.

\begin{definition}[Return cocycle and roof function]
\label{def.return vector}\label{def.return vector2}
For $x\in \Lambda_+$, let $\gamma\in\mathcal H$
be the unique element such that $x\in \gamma \La_+$.
Define
$$
K(x):=\beta_x(e,\gamma)
+2\mathcal G(Tx,p_0)-2\mathcal G(x,p_0),
\qquad
R(x):=\psi(K(x)).
$$
Put $K_0:=0$ and $R_0:=0$.
For $n\geq1$, set
$K_n:=\sum_{j=0}^{n-1}K\circ T^j$ and
$R_n:=\psi\circ K_n=\sum_{j=0}^{n-1}R\circ T^j$.
\end{definition}

For $n\geq1$ and $\gamma\in\mathcal H^n$, put
\begin{equation}
\label{eqn:R gamma}
R_\gamma:=\sup_{x\in\gamma\Lambda_+}R_n(x).
\end{equation}

\begin{lemma}
\label{lem.uniformly positive}
There exists $C>0$ such that, for every $n\geq1$,
$\gamma\in\mathcal H^n$, and $x\in\gamma\Lambda_+$,
\begin{align}
K_n(x)
&=\beta_x(e,\gamma)
+2\mathcal G(\gamma^{-1}x,p_0)-2\mathcal G(x,p_0),
\label{eqn:telescoping-return-cocycle}\\
\|K_n(x)-\mu(\gamma)\|
&\leq C,
\qquad
|R_n(x)-\psi(\mu(\gamma))|\leq C.
\label{eqn:return-Cartan-comparison}
\end{align}
Moreover, for every $n\geq1$ and $x\in\Lambda_+$, $K_n(x)\in\mathcal C_0$ and $R_n(x)\geq n$.
\end{lemma}

\begin{proof}
Write $\gamma=\gamma_1\cdots\gamma_n$ with
$\gamma_1,\ldots,\gamma_n\in\mathcal H$.
The cocycle identity gives
$\sum_{j=0}^{n-1}\beta_{T^jx}(e,\gamma_{j+1})
=\beta_x(e,\gamma)$.
Since the $\mathcal G$-terms telescope and
$T^nx=\gamma^{-1}x$, this proves
\eqref{eqn:telescoping-return-cocycle}.

By \eqref{Gbound}, the endpoint term belongs to the
fixed compact set $\mathcal D_0$.
Together with \eqref{shadowH} and \Cref{lem.buseandcartan},
this gives \eqref{eqn:return-Cartan-comparison}.
Since $\gamma$ is a word of length $nn_0\geq n_0$
in $\mathcal H_{\mathrm{ret}}$,
\Cref{lem.choice of n_0}(1) also gives
$K_n(x)\in\mathcal C_0$.

For $n=1$, \Cref{lem.choice of n_0}(2) and the choice
of $C_R$ give
$
R(x)\geq
C_R-4\sup_{z\in\Delta_0}
|\psi(\mathcal G(z,p_0))|
\geq1.
$
Summing along the orbit yields $R_n(x)\geq n$.
\end{proof}

\subsection{Exponential summability and moments}

Exponential summability for $\mathcal H_{\mathrm{ret}}$
also holds for
$\mathcal H=\mathcal H_{\mathrm{ret}}^{\,n_0}$.

\begin{proposition}[Exponential branch summability]
\label{prop.expsumfinalfinite}
\label{Rinfinity2}
There exists $\varepsilon_*\in(0,1)$ such that, for every
$0<\varepsilon\leq\varepsilon_*$,
$$
\sum_{\gamma\in\mathcal H}
\sup_{x\in\Lambda_+}e^{-(1-\varepsilon)R(\gamma x)}
\asymp
\sum_{\gamma\in\mathcal H}
e^{-(1-\varepsilon)\psi(\mu(\gamma))}
<\infty.
$$
\end{proposition}

\begin{proof}
By \Cref{prop.expsumh0finite,prop.expsumh1finite},
there exists $\varepsilon_*\in(0,1)$ such that
$\sum_{\gamma\in\mathcal H_{\mathrm{ret}}}
e^{-(1-\varepsilon_*)\psi(\mu(\gamma))}<\infty$.
Since $\psi(\mu(\gamma))$ is uniformly bounded below,
the same convergence holds for
$0<\varepsilon\leq\varepsilon_*$.

For $\gamma_1,\ldots,\gamma_{n_0}\in\mathcal H_{\mathrm{ret}}$,
the cocycle identity, \eqref{shadowH}, and
\Cref{lem.buseandcartan} give
$$
\psi(\mu(\gamma_1\cdots\gamma_{n_0}))
=
\sum_{j=1}^{n_0}\psi(\mu(\gamma_j))+O_{n_0}(1),
$$
uniformly in the letters. Hence
$$
\sum_{\gamma\in\mathcal H}
e^{-(1-\varepsilon)\psi(\mu(\gamma))}
\ll
\left(
\sum_{\gamma\in\mathcal H_{\mathrm{ret}}}
e^{-(1-\varepsilon)\psi(\mu(\gamma))}
\right)^{n_0}
<\infty.
$$
Finally, \eqref{eqn:return-Cartan-comparison} with $n=1$
gives $R(\gamma x)=\psi(\mu(\gamma))+O(1)$ uniformly
in $\gamma$ and $x$, proving the comparison of the two sums.
\end{proof}

\begin{proposition}[Exponential moments]
\label{lem:roof-moment-branch-series}
\label{thm.exptail}
\label{Kk2}
\begin{enumerate}
\item For every $0<\varepsilon<1$,
$$
\int_{\Lambda_+}e^{\varepsilon R(x)}\,d\nu(x)
\asymp_\varepsilon
\sum_{\gamma\in\mathcal H}
e^{-(1-\varepsilon)\psi(\mu(\gamma))}.
$$

\item There exist $\varepsilon_0,C>0$ such that
$$
\int_{\Lambda_+}e^{\varepsilon_0R}\,d\nu<\infty,
\qquad
\nu(\{x\in\Lambda_+:R(x)>t\})
\leq Ce^{-\varepsilon_0t}
\quad(t\geq0).
$$

\item There exists $\varepsilon_1>0$ such that
$$
\int_{\Lambda_+}e^{\varepsilon_1\|K\|}\,d\nu<\infty.
$$
\end{enumerate}
\end{proposition}

\begin{proof} By conformality, \eqref{shadowH}, and
\Cref{lem.buseandcartan},
$$
\nu(\gamma\Lambda_+)
=
\int_{\Lambda_+}
e^{-\psi(\beta_{\gamma x}(e,\gamma))}\,d\nu(x)
\asymp e^{-\psi(\mu(\gamma))},
$$
uniformly in $\gamma\in\mathcal H$, since
$\nu(\Lambda_+)=\nu(\Delta_0)>0$.
Moreover, \eqref{eqn:return-Cartan-comparison} gives
$$
R(y)=\psi(\mu(\gamma))+O(1)
\qquad(y\in\gamma\Lambda_+),
$$
with a uniform error.
Therefore, summing over the branch partition, we obtain
$$
\begin{aligned}
\int_{\Lambda_+}e^{\varepsilon R(y)}\,d\nu(y)
&=\sum_{\gamma\in\mathcal H}
\int_{\gamma\Lambda_+}e^{\varepsilon R(y)}\,d\nu(y) \asymp_\varepsilon
\sum_{\gamma\in\mathcal H}
e^{\varepsilon\psi(\mu(\gamma))}
\nu(\gamma\Lambda_+) \asymp_\varepsilon
\sum_{\gamma\in\mathcal H}
e^{-(1-\varepsilon)\psi(\mu(\gamma))}.
\end{aligned}
$$
This proves (1).

Taking $\varepsilon_0=\varepsilon_*$ and applying
\Cref{prop.expsumfinalfinite} gives the finite moment in (2).
The tail bound follows from
$\nu(\{R>t\})\leq
e^{-\varepsilon_0t}\int_{\Lambda_+}e^{\varepsilon_0R}\,d\nu$.

Finally, $K(x)\in\mathcal C_0$ by
\Cref{lem.uniformly positive}, and $\psi$ is positive on
$\mathcal C_0\smallsetminus\{0\}$.
Thus
$m_0:=\min\{\psi(v):v\in\mathcal C_0,\ \|v\|=1\}>0$,
so $R(x)=\psi(K(x))\geq m_0\|K(x)\|$.
Taking $\varepsilon_1=m_0\varepsilon_0$ and applying (2)
proves (3).
\end{proof}
\subsection{Uniform distortion and H\"older regularity}
\label{sec.holder}

We estimate the variation of $K_n$ at $\gamma x,\gamma y$
in terms of $d_{\bar\psi}(x,y)$, uniformly in $n$ and
$\gamma\in\mathcal H^n$.

\begin{definition}
\label{def: locally Holder functions 1}
Let $\alpha>0$. A map $\phi:\Sigma^+\to\mathfrak a$
or $\phi:\Sigma^+\to\mathbb R$ is
\emph{locally $\alpha$-H\"older} if there exists $C>0$ such that
$$
\|\phi(s)-\phi(t)\|
\leq C d_{\Sigma^+}(s,t)^\alpha
\qquad\text{whenever }s_1=t_1.
$$
We call $\phi$ \emph{locally H\"older} if this holds
for some $\alpha>0$. 
\end{definition}

\begin{lemma}
\label{lem:Riemannian-visual-Holder}
There exist $\theta\in(0,1]$ and $C>0$ such that,
for all $\xi,\eta\in\Lambda_\rho$,
$$
d_{\mathrm{Riem}}(\xi,\eta)
\leq C d_{\bar\psi}(\xi,\eta)^\theta.
$$
\end{lemma}

\begin{proof}
For distinct $\xi,\eta\in\Lambda_\rho$, let $z$ be the
nearest-point projection of $o$ to the geodesic joining
$f^{-1}(\xi)$ and $f^{-1}(\eta)$.
By \cite[Corollary~B.7(2)]{CZZ_cusped}, the boundary map
$f$ is H\"older with respect to a visual metric on
$\Lambda(\Gamma)$ and $d_{\mathrm{Riem}}$ on $\mathcal F$.
Since the visual distance is comparable to $e^{-d(o,z)}$,
there exists $a>0$ such that
$d_{\mathrm{Riem}}(\xi,\eta)\ll e^{-a d(o,z)}$.

By \Cref{lem.gromovproductnearest,lem.scalar-Cartan-comparison},
we can choose $b\geq a$ such that
$\bar\psi(\mathcal G(\xi,\eta))\leq b\,d(o,z)+O(1)$.
Thus $d_{\bar\psi}(\xi,\eta)\gg e^{-b d(o,z)}$, and hence
$d_{\mathrm{Riem}}(\xi,\eta)
\ll d_{\bar\psi}(\xi,\eta)^{a/b}$.
Taking $\theta=a/b$ proves the lemma.
\end{proof}
\begin{remark}
A reverse H\"older estimate also holds: there exist
$\Theta\geq1$ and $C>0$ such that
$$
d_{\mathrm{Riem}}(\xi,\eta)
\geq C^{-1}d_{\bar\psi}(\xi,\eta)^\Theta
\qquad(\xi,\eta\in\Lambda_\rho).
$$
This follows from \cite[Lemma~3.11]{LO_invariant}, together
with \Cref{lem.gromovproductnearest,lem.scalar-Cartan-comparison}.
We will not use this estimate.
\end{remark}

\begin{proposition}[Uniform distortion of the return cocycle]
\label{prop.holderlocal}
\label{lem:Holder property of Birkoff sum}
There exist $c_0\in(0,1]$ and $C>0$ such that,
for every $n\geq1$, $\gamma\in\mathcal H^n$, and
$x,y\in\Lambda_+$,
\begin{equation}
\label{eqn:branch-K-distortion}
\begin{aligned}
&\|\beta_{\gamma x}(e,\gamma)-\beta_{\gamma y}(e,\gamma)\|
\leq C d_{\bar\psi}(x,y)^{c_0},\\
&\|K_n(\gamma x)-K_n(\gamma y)\|
\leq C d_{\bar\psi}(x,y)^{c_0},\\
&|R_n(\gamma x)-R_n(\gamma y)|
\leq C d_{\bar\psi}(x,y)^{c_0}.
\end{aligned}
\end{equation}
\end{proposition}

\begin{proof}
Take $c_0=\theta$ from \Cref{lem:Riemannian-visual-Holder}.
By \eqref{shadowH} and equivariance,
$x,y\in\mathcal O_r(\gamma^{-1}o,o)\cap\Lambda_\rho$.
Applying \Cref{prop.holderlocal_Buse_prelim} to
$\gamma^{-1}$ and using
$\beta_{\gamma x}(e,\gamma)
=-\beta_x(e,\gamma^{-1})$, we obtain
$$
\|\beta_{\gamma x}(e,\gamma)-\beta_{\gamma y}(e,\gamma)\|
\leq C d_{\mathrm{Riem}}(x,y)
\leq C d_{\bar\psi}(x,y)^{c_0}.
$$

The map $2\mathcal G(\,\cdot\,,p_0)$ is smooth on a
neighborhood of the compact set $\Delta_0$, so its
restriction to $\Delta_0$ is Lipschitz.
By \Cref{lem:Riemannian-visual-Holder,lem:uniform-branch-contraction},
$$
\begin{aligned}
&2\|\mathcal G(x,p_0)-\mathcal G(y,p_0)\|
+2\|\mathcal G(\gamma x,p_0)-\mathcal G(\gamma y,p_0)\|\\
&\qquad\ll
d_{\bar\psi}(x,y)^{c_0}
+d_{\bar\psi}(\gamma x,\gamma y)^{c_0}
\ll d_{\bar\psi}(x,y)^{c_0}.
\end{aligned}
$$
Combining this with the first estimate and
\eqref{eqn:telescoping-return-cocycle} proves the
estimate for $K_n$. Applying $\psi$ gives the estimate
for $R_n$.
\end{proof}

\begin{corollary}[Symbolic H\"older regularity]
\label{R Lipschitz}
For every $c\in(0,c_0]$, the map $K\circ\Psi$ is locally
$\theta_c$-H\"older, where $\theta_c:=cC_R/\log2$.
The same holds for $R\circ\Psi$.
\end{corollary}

\begin{proof}
Since $d_{\bar\psi}$ is bounded on $\Lambda_\rho$,
the estimates in \Cref{prop.holderlocal} also hold with
$c_0$ replaced by any $c\in(0,c_0]$, with a constant
depending on $c$.

Let $s\neq t$ satisfy $s_1=t_1$.
Applying \Cref{prop.holderlocal} to their common first
branch and then \Cref{lem.phiholder}, we obtain
$$
\begin{aligned}
\|K(\Psi(s))-K(\Psi(t))\|
&\ll_c
d_{\bar\psi}(\Psi(\sigma s),\Psi(\sigma t))^c\ll_c d_{\Sigma^+}(\sigma s,\sigma t)^{\theta_c}
=2^{\theta_c}d_{\Sigma^+}(s,t)^{\theta_c}.
\end{aligned}
$$
Composition with $\psi$ preserves this estimate and yields the last claim.
\end{proof}

\section{Non-arithmeticity of periodic return vectors}
\label{sec.nonari}

We prove that the periodic return vectors generate a dense
additive subgroup of $\mathfrak a$. We first identify these
vectors with Jordan projections and prove that the return
semigroup is Zariski dense. The conclusion then follows
from Benoist's non-arithmeticity theorem \cite{Benoist_properties2}.

For $n\geq1$, let $\operatorname{Fix}(T^n)\subset\Lambda_+$
denote the fixed-point set of $T^n$.

\begin{theorem}[Non-arithmeticity of return vectors]
\label{thm.denselength}
The additive subgroup generated by
$$
\{K_n(x):x\in\operatorname{Fix}(T^n),\ n\geq1\}
$$
is dense in $\mathfrak a$.
\end{theorem}

\begin{lemma}
\label{Zd}
The set $\Lambda_+$ is Zariski dense in $\mathcal F$.
\end{lemma}

\begin{proof}
By \Cref{thm.divergencetype},
$\sum_{\gamma\in\Gamma_\rho}e^{-\psi(\mu(\gamma))}=\infty$.
Hence \cite[Theorem~1.4]{KO_NC} implies that every proper
subvariety of $\mathcal F$ is $\nu$-null.
Since $\Lambda_+$ has full $\nu$-measure in $\Delta_0$
by \Cref{prop.T is well def a.e.} and $\nu(\Delta_0)>0$,
it cannot be contained in a proper subvariety.
\end{proof}

\begin{lemma}[Periodic return vectors]
\label{lem.proximal}
\label{lem.periodisjordan}
Every nontrivial element of $\sH$ is loxodromic.
For every $n\geq1$, the fixed points of $T^n$ are precisely
the attracting flags $y_\gamma$, $\gamma\in\mathcal H^n$,
and $K_n(y_\gamma)=\lambda(\gamma)$.
\end{lemma}

\begin{proof}
Let $\gamma=\gamma_1\cdots\gamma_n\in\mathcal H^n$.
The periodic repetition of $(\gamma_1,\ldots,\gamma_n)$
defines a point $x\in\Lambda_+$ with $T^nx=x$ and
$\gamma x=x$.
Since $K_{kn}(x)=kK_n(x)$, applying
\eqref{eqn:return-Cartan-comparison} to $\gamma^k$ gives
$$
\|kK_n(x)-\mu(\gamma^k)\|\leq C
\qquad(k\geq1).
$$
Dividing by $k$ and letting $k\to\infty$, we obtain
$K_n(x)=\lambda(\gamma)$.
By \Cref{lem.uniformly positive}, this vector belongs to
$\mathcal C_0$ and satisfies $\psi(K_n(x))\geq n$.
Thus $\lambda(\gamma)$ belongs to
$\operatorname{int}\mathfrak a^+$, so $\gamma$ is loxodromic.
Moreover, \eqref{shadowH} gives
$x\in\mathcal O_r(o,\gamma^ko)$ for every $k\geq1$.
By \Cref{lem.ambient-shadow-convergence}(2), $x=y_\gamma$.

Conversely, every fixed point of $T^n$ has an itinerary
obtained by repeating a word of length $n$.
The preceding argument identifies it with the attracting
flag of the corresponding element of $\mathcal H^n$.
\end{proof}

\begin{lemma}
\label{lem.attractingdense}
The set
$\{y_\gamma:\gamma\in\mathcal H^n,\ n\geq1\}$
is dense in $\Lambda_+$.
\end{lemma}

\begin{proof}
Periodic sequences are dense in $\Sigma^+$.
Apply the homeomorphism $\Psi$ and
\Cref{lem.periodisjordan}.
\end{proof}

Let $\Gamma_{\mathcal H}<\Gamma_\rho$ be the subgroup
generated by $\sH$.

\begin{lemma}
\label{lem.noglobalfixedpt}
No finite-index subgroup of $\Gamma_{\mathcal H}$ is
contained in a proper parabolic subgroup of $G$.
\end{lemma}

\begin{proof}
Suppose that a finite-index subgroup
$H<\Gamma_{\mathcal H}$ is contained in $gQg^{-1}$,
where $Q$ is a maximal proper parabolic subgroup
containing $P$. Then $H$ fixes $gQ\in G/Q$.

The flags transverse to $gP$ form a nonempty Zariski-open
subset of $\mathcal F$.
By \Cref{Zd,lem.attractingdense}, choose
$\gamma_1,\gamma_2\in\sH\smallsetminus\{e\}$ with distinct
attracting flags, both transverse to $gP$.
Replacing them by a common positive power, we may assume
that $\gamma_1,\gamma_2\in H$.
For $i=1,2$, transversality gives
$\gamma_i^{-m}gP\to y_{\gamma_i^{-1}}$ as $m\to\infty$.
Since $\gamma_i^{-m}gQ=gQ$, both flags
$y_{\gamma_i^{-1}}$ project to $gQ$.
The preimages of $\gamma_1,\gamma_2$ are hyperbolic. Two hyperbolic elements of a Fuchsian group
that share one fixed point share both fixed points.
Hence the distinct attracting flags of
$\gamma_1,\gamma_2$ imply that their repelling flags
are distinct.
As these repelling flags belong to $\Lambda_\rho$,  they have the form $hP$ and $hw_0P$ for some $h\in G$.
Their images $hQ$ and $hw_0Q$ in $G/Q$ are distinct because
$Q$ is proper, contradicting the preceding paragraph.
\end{proof}

\begin{lemma}
\label{Zd2}
Let $G$ be a connected semisimple real algebraic group, and let $L<G$ be a reductive
real algebraic subgroup.
Suppose that the projection of $L^\circ$ to every simple
factor of $G$ is non-compact.
If $\{y_g\in \mathcal F: g\in L^\circ, \text{loxodromic}\}$ is Zariski dense in $\mathcal F$, then $L=G$.
\end{lemma}

\begin{proof}
By \cite{Mostow_reductive}, after conjugation we may assume
that $L^\circ$ is invariant under $\theta$.
Write
$\mathfrak l=\mathfrak k_L\oplus\mathfrak p_L$
and $K_L=L^\circ\cap K$.
After a further conjugation by $K$, choose a maximal
abelian subspace $\mathfrak a_L\subset\mathfrak p_L$
contained in $\mathfrak a$.
Then
$L^\circ=K_L\exp(\mathfrak a_L^+)K_L$
for a closed positive Weyl chamber $\mathfrak a_L^+$.
For a loxodromic $g\in L^\circ$, write
$g^n=k_n\exp(H_n)k_n'$ in this decomposition.
Let $\mathcal W\subset N_K(A)$ be a finite set of
representatives of the Weyl group.
Choose $w_n\in\mathcal W$ such that
$\operatorname{Ad}(w_n^{-1})H_n=\mu(g^n)$.
Along a subsequence, $w_n=w$ and $k_n\to k\in K_L$.
Since
$n^{-1}\mu(g^n)\to\lambda(g)\in
\operatorname{int}\mathfrak a^+$,
we obtain $y_g=kwP\in \bigcup_{w\in\mathcal W}K_L\cdot(wP)$. By the hypothesis, the union $\bigcup_{w\in\mathcal W}K_L\cdot(wP)$ is Zariski dense.

Each $K_L\cdot(wP)$ is a compact semialgebraic
submanifold, and its real Zariski closure has the same
dimension. Therefore for some $w\in \mathcal W$, $K_L. w P$ must have
has dimension $\dim\mathcal F$.
That orbit is open and closed in the connected space
$\mathcal F$, so $K_L.wP=\mathcal F$.
Thus $K_L$ acts transitively on $\mathcal F=K/M$, and
$K=K_LM$.

Write $\mathfrak g=\bigoplus_i\mathfrak g_i$, where
$\mathfrak g_i=\mathfrak k_i\oplus\mathfrak p_i$
are the simple ideals.
By the non-compactness assumption, $\mathfrak p_L$
projects nontrivially to every $\mathfrak p_i$.
We can therefore choose $X\in\mathfrak p_L$ whose
projections to all $\mathfrak p_i$ are nonzero.
Conjugating by $K_L$, we obtain $H\in\mathfrak a_L$
whose projections $H_i$ are still nonzero.
Since $M$ centralizes $H$ and $K=K_LM$,
$
\operatorname{Ad}(K)H
=\operatorname{Ad}(K_L)H
\subset\mathfrak p_L.
$

Let $V$ be the linear span of $\operatorname{Ad}(K)H$.
For each $i$, the action of $\mathfrak k_i$ on
$\mathfrak p_i$ is nontrivial and irreducible.
Choose $Z_i\in\mathfrak k_i$ with $[Z_i,H_i]\neq0$.
Since $V$ is $\mathfrak k$-invariant and distinct
simple ideals commute,
$0\neq[Z_i,H]=[Z_i,H_i]\in V\cap\mathfrak p_i$. Since $V\cap\mathfrak p_i$ is a nonzero
$\mathfrak k_i$-invariant subspace of $\mathfrak p_i$,
irreducibility gives $\mathfrak p_i\subset V$.
Thus $\mathfrak p\subset V\subset\mathfrak l$.
As $G$ has no compact factors,
$[\mathfrak p,\mathfrak p]=\mathfrak k$, so
$\mathfrak l=\mathfrak g$.
Connectedness of $G$ now gives $L=G$.
\end{proof}

\begin{corollary}
\label{cor.Zdensesemigroup}
The semigroup $\sH$ is Zariski dense in $G$.
\end{corollary}

\begin{proof}
Let $L$ be the Zariski closure of $\sH$.
By \cite[Lemma~6.15]{BQ_book},  $L$ is an algebraic subgroup containing
$\Gamma_{\mathcal H}$.  By \Cref{lem.proximal}, every nontrivial element of
$\sH$ is loxodromic and hence the projection of $L^\circ$ to every simple
factor of $G$ is non-compact.

Since $L/L^\circ$ is finite,
$\Gamma_{\mathcal H}\cap L^\circ$ has finite index in
$\Gamma_{\mathcal H}$.
If $L^\circ$ were nonreductive, then it should be contained in a proper parabolic subgroup of $G$,
contrary to \Cref{lem.noglobalfixedpt}.
Hence $L$ is reductive.
For any $g\in \sH$, there is $n$ such that $g^n\in L^\circ$ and $y_{g}=y_{g^n}$.
By \Cref{lem.attractingdense,Zd} and \Cref{Zd2}, $L=G$.
\end{proof}

Since inversion is a biregular map of $G$,
$\sH^{-1}$ is also Zariski dense.

\begin{proof}[Proof of \Cref{thm.denselength}]
By \Cref{lem.periodisjordan},
$$
\{K_n(x):x\in\operatorname{Fix}(T^n),\ n\geq1\}
=
\{\lambda(\gamma):\gamma\in\sH\smallsetminus\{e\}\}.
$$
By \Cref{cor.Zdensesemigroup}, $\sH$ is Zariski dense.
Therefore \cite[Proposition, p.~34]{Benoist_properties2}
implies the claim; see also
\cite{OhSarkar} for a dynamical proof.
\end{proof}

\section{Quantitative non-integrability of the return cocycle}
\label{sec.nonintegrability}

We prove the quantitative non-integrability estimate for $K$
used in \Cref{prop:norm of transfer operator}.
The proof uses Zariski density of $\sH$ and the
peripheral root-balance condition
\eqref{eqn:peripheral-root-balance}.

For $C>0$, set
$\mathcal H_C:=\{\gamma\in\mathcal H:R_\gamma<C\}$,
where $R_\gamma$ is defined in \eqref{eqn:R gamma}.
By \eqref{eqn:return-Cartan-comparison},
$R_\gamma=\psi(\mu(\gamma))+O(1)$ uniformly in $\gamma$.
Thus $\mathcal H_C$ is finite by
\Cref{prop.expsumfinalfinite}.

\begin{theorem}[Quantitative non-integrability]
\label{thm:non-integrability}
There exist constants $C>1$,
$0<s_1 \leq s_2$, $s_3>0$,
$N_{\mathrm{NI}}>1$, and
$\epsilon_1,\delta_1,\eta\in(0,1)$ with the following
property.
For every unit vector $v_0\in\mathfrak a$, $b>3$,
and integer $n>N_{\mathrm{NI}}$, there exist
$\gamma_1,\gamma_2\in\mathcal H_C^n$,
$\gamma\in\mathcal H_C$, and $x_1,x_2\in\Lambda_+$
such that
$$
\begin{gathered}
x_1\in\gamma\Lambda_+,\qquad
B_{\bar\psi}\left(x_1,\frac{\epsilon_1}{b^{s_1}}\right)
\subset\gamma\Delta_0,\qquad x_2\in
B_{\bar\psi}\left(x_1,\frac{\epsilon_1}{2b^{s_1}}\right)
\smallsetminus
B_{\bar\psi}\left(x_1,
\frac{\delta_1\epsilon_1}{2b^{s_2}}\right),
\end{gathered}
$$
and
\begin{equation}
\label{eqn:non-integrability-estimate}
\begin{aligned}
\Big|&
\langle v_0,K_n(\gamma_1x_1)-K_n(\gamma_2x_1)\rangle -\langle v_0,K_n(\gamma_1x_2)-K_n(\gamma_2x_2)\rangle
\Big|
\geq\eta\,d_{\bar\psi}(x_1,x_2)^{s_3}.
\end{aligned}
\end{equation}
Moreover, writing
$\gamma_k=\gamma_{k,1}\cdots\gamma_{k,n}$ with
$\gamma_{k,i}\in\mathcal H_C$, we have
\begin{equation}
\label{eqn:inclusion of cylinder}
B_{\bar\psi}\left(\gamma_kx_1,
\frac{\epsilon_1}{2b^{s_1}}\right)
\cup
B_{\bar\psi}\left(\gamma_kx_2,
\frac{\epsilon_1}{2b^{s_1}}\right)
\subset\gamma_{k,1}\Delta_0
\qquad(k=1,2).
\end{equation}
\end{theorem}

\subsection{Root balance in the return limit cone}

Let $\mathcal L_{\mathcal H}$ be the asymptotic cone
of $\mu(\sH)$.
Since $\sH$ is Zariski dense by
\Cref{cor.Zdensesemigroup},
\cite[Theorem~1.2(a)]{Benoist1997proprietes} gives
$$
\mathcal L_{\mathcal H}
=
\overline{\{t\lambda(\gamma):
t\geq0,\ \gamma\in\sH\}}.
$$

\begin{lemma}[Peripheral root balance]
\label{lem.principal}
There exists
$v\in\mathcal L_{\mathcal H}
\cap\operatorname{int}\mathfrak a^+$
such that
$$
\frac{\min_{\alpha\in\Pi}\alpha(v)}
     {\max_{\alpha\in\Pi}\alpha(v)}
>\frac12.
$$
\end{lemma}
\begin{proof}
Recall that $\mathcal H=\mathcal H_{\mathrm{ret}}^{\,n_0}$,
where $\mathcal H_{\mathrm{ret}}$ consists of products
$\eta_1\cdots\eta_r$ satisfying
$$
\eta_j\in\mathcal H_0(i_j,i_{j-1}),\qquad
i_0=i_r=0,\qquad i_j\neq0\quad(0<j<r).
$$
The last factor has the form $\eta_r=\gamma_p h$,
where $p\in\Gamma_\rho p_0$ and
$h\in N_p\subset\mathsf P_0$.
By \eqref{eqn:H0ij},
$\gamma_pN_p\subset\mathcal H_0(0,i_{r-1})$,
and $N_p$ is cofinite in $\mathsf P_0$ by \eqref{np0}.

For every nontrivial $u\in\mathsf P_0$, the elements
$hu^m$ are distinct and hence belong to $N_p$
for all sufficiently large $m$.
Replacing the last factor of the last return word
therefore shows that, for every $g\in\mathcal H$ and
$u\in\mathsf P_0\smallsetminus\{e\}$, there exists
$m_0=m_0(g,u)$ such that
\begin{equation}
\label{eqn:iterationisinthecoding}
gu^m\in\mathcal H
\qquad(m\geq m_0).
\end{equation}

Now fix $g\in\mathcal H$ and 
$u_0\in\mathsf P_0\smallsetminus\{e\}$.
Then $u_0$ is weakly unipotent and hence,
by \Cref{lem.regularunipotent_cartan,lem.cptcartan},
$$
\mu(gu_0^m)=\mu(u_0^m)+O_g(1)
=(\log m)v_{u_0}+O(1).
$$
By \eqref{eqn:iterationisinthecoding},
$gu_0^m\in\mathcal H$ for all sufficiently large $m$.
Dividing by $\log m$ therefore gives
$v_{u_0}\in\mathcal L_{\mathcal H}$.
Moreover, $v_{u_0}\neq0$, and the Cartan asymptotic
and \eqref{eqn:peripheral-root-balance} imply
$$
\frac{\min_{\alpha\in\Pi}\alpha(v_{u_0})}
     {\max_{\alpha\in\Pi}\alpha(v_{u_0})}
=
\lim_{m\to\infty}
\frac{\min_{\alpha\in\Pi}\alpha(\mu(u_0^m))}
     {\max_{\alpha\in\Pi}\alpha(\mu(u_0^m))}
>\frac12.
$$
Thus $v_{u_0}\in\operatorname{int}\mathfrak a^+$,
and we may take $v=v_{u_0}$.
\end{proof}

For the remainder of this section, fix a vector
$v_*\in\mathcal L_{\mathcal H}\cap
\operatorname{int}\mathfrak a^+$ supplied by
\Cref{lem.principal}, and set
\begin{equation}
\label{eqn:root-balance-constant}
\theta_*:=
\frac{\min_{\alpha\in\Pi}\alpha(v_*)}
     {\max_{\alpha\in\Pi}\alpha(v_*)}
\in\left(\frac12,1\right].
\end{equation}

\subsection{Quantitative non-concentration}
Equip $\mathfrak g$ with the  inner product
$\langle X,Y\rangle:=-B(X,\theta Y)$, where $B$ is the
Killing form, and use the induced norms. Its restriction to
$\mathfrak a$ agrees with the inner product fixed above.
Consider the root decomposition
\begin{equation}
\label{eqn:restricted root decomp}
 \mathfrak g=\fa\oplus\mathfrak m\oplus
 \bigoplus_{\alpha\in\Phi^+}
 \bigl(\mathfrak g_\alpha\oplus\mathfrak g_{-\alpha}\bigr),  \qquad(\mathfrak m=\operatorname{Lie} M ).
\end{equation}
 This is an orthogonal decomposition which is $\operatorname{Ad}(M)$-invariant. For
$\kappa\in(0,1]$ and $\alpha\in\Pi$, put
$$
 W_\kappa(\mathfrak g_{-\alpha})
 :=\left\{Z\in\check{\mathfrak n}:\|Z\|=1,
   \sup_{\substack{Z'\in\mathfrak g_{-\alpha}, \|Z'\|=1}}
 |\langle Z,Z'\rangle|\geq\kappa\right\}.
$$

The following proposition gives an attracting fixed point $y_{\gamma^*}$
such that, in local Bruhat coordinates centered there,
$\Lambda_+$ escapes every hyperplane through the origin
whose normal has a definite component in a negative
simple-root space, with quantitative bounds at all
sufficiently small scales.
The powers of $d_{\bar\psi}$ match the contraction rates
in the corresponding root spaces; see
\eqref{eqn:adapted-radius-contraction}.

\begin{proposition}[Adapted non-concentration]
\label{prop: non-concentration self-joining}
Fix $s\in(\frac12,\theta_*)$. There exists a loxodromic element
$\ga^*\in\sH$ with the following property.
For every fixed $g_*\in G$ satisfying
$g_*^{-1}\ga^*g_*\in A^+M$
and $\kappa\in(0,1]$, there exist 
$c_0,\delta_0,\eta_0\in(0,1)$ such that, for every
$\alpha\in\Pi$, $Z\in W_\kappa(\mathfrak g_{-\alpha})$, and
$\varepsilon\in(0,c_0)$, there exists
$Y\in\check{\mathfrak n}$ with $g_*\exp(Y)P\in\Lambda_+$ and
\begin{equation}
\label{eqn:quantitative-non-concentration}
\begin{aligned}
 &\delta_0\varepsilon
 <d_{\bar\psi}\bigl(y_{\ga^*},g_*\exp(Y)P\bigr)^{
   \frac{\alpha(\lambda(\ga^*))}{\bar\psi(\lambda(\ga^*))}}
 <\varepsilon,\;\;  |\langle Z,Y\rangle|>\eta_0\varepsilon,
 \;  \; \|Y\|\leq\varepsilon^s.
\end{aligned}
\end{equation}

\end{proposition}
\begin{proof}
Choose $s_0\in(s,\theta_*)$. By
\eqref{eqn:root-balance-constant}, for every
$\beta\in\Phi^+$ and $\alpha\in\Pi$,
$$
 \beta(v_*)
 \geq\min_{\zeta\in\Pi}\zeta(v_*)
 >s_0\max_{\zeta\in\Pi}\zeta(v_*)
 \geq s_0\alpha(v_*).
$$

 Since $\mathcal L_{\mathcal H}$ is the closure of the rays through
$\lambda(\sH)$,
we may choose $\ga^*\in\sH\smallsetminus\{e\}$ such that
\begin{equation}
\label{eqn:near-principal-root-ineq}
 \beta(\lambda(\ga^*))
 \geq s_0\alpha(\lambda(\ga^*))
 \qquad(\beta\in\Phi^+,\ \alpha\in\Pi).
\end{equation}
By \Cref{lem.proximal}, $\ga^*$ is loxodromic, and
$\bar\psi(\lambda(\ga^*))>0$ by \eqref{lb}.
Fix $g_*\in G$ with $g_*^{-1}\ga^*g_*\in A^+M$. Then
$ g_*P=y_{\ga^*}$,
$ g_*w_0P=y_{(\ga^*)^{-1}}$, and $ g_*^{-1}\ga^*g_*\in\exp(\lambda(\ga^*))M.
$
For every $k\geq0$, we have 
\begin{equation*}
\begin{aligned}
& (\ga^*)^kg_*\exp(V)P
 =g_*\exp\bigl(
   \operatorname{Ad}(g_*^{-1}(\ga^*)^kg_*)V\bigr)P
 &&(V\in\check{\mathfrak n}),\\
 &\|\operatorname{Ad}(g_*^{-1}(\ga^*)^kg_*)V\|
 =e^{-k\beta(\lambda(\ga^*))}\|V\|
 &&(V\in\mathfrak g_{-\beta},\ \beta\in\Phi^+).
\end{aligned}
\end{equation*}

Since the exponential and logarithm maps of $\check N$ are
polynomial, $V\mapsto g_*\exp(V)P$ is an algebraic isomorphism onto
the open Bruhat cell $g_*\check NP$. By \Cref{Zd}, the set
$ \{V\in\check{\mathfrak n}:g_*\exp(V)P\in\Lambda_+\}$
is Zariski dense and hence spans $\check{\mathfrak n}$. Choose a
basis $Y_1,\ldots,Y_d$ from this set and $C>1$  such that
\begin{equation}
\label{eqn:nonconcentration basis bound}
 \max_j\|Y_j\|\leq C,
 \qquad
 \max_j|\langle V,Y_j\rangle|
 \geq C^{-1}\|V\|
 \quad(V\in\check{\mathfrak n}).
\end{equation}

By \Cref{lem.periodisjordan},
$y_{\ga^*}\in\Lambda_+$, and
\eqref{eqn:powered-branch-partition} gives
$(\ga^*)^k\Lambda_+\subset\Lambda_+$ for every $k\geq0$.
By \eqref{shadowH}, \Cref{lem.buseandcartan}, and
\Cref{lem.cptcartan},
$$
 \beta_{(\ga^*)^kz}(e,(\ga^*)^k)
 =\mu((\ga^*)^k)+O(1)
 =k\lambda(\ga^*)+O(1)
$$
uniformly for $k\geq0$ and $z\in\Lambda_+$.
Applying \eqref{eqn.metrics_sph} to
$y_{\ga^*}$ and $g_*\exp(Y_j)P$, whose distance is positive,
and enlarging $C$, we obtain
\begin{equation}
\label{eqn:adapted-radius-contraction}
\begin{aligned}
 C^{-1}e^{-k\alpha(\lambda(\ga^*))}
 &\leq
 d_{\bar\psi}\bigl(y_{\ga^*},
   (\ga^*)^kg_*\exp(Y_j)P\bigr)^{
   \frac{\alpha(\lambda(\ga^*))}{\bar\psi(\lambda(\ga^*))}} \leq C e^{-k\alpha(\lambda(\ga^*))}
\end{aligned}
\end{equation}
for every $\alpha\in\Pi$, $j$, and $k\geq0$. Fix $\kappa\in(0,1]$. For
$Z\in W_\kappa(\mathfrak g_{-\alpha})$, we have
$$
 \bigl\|\operatorname{Ad}(g_*^{-1}(\ga^*)^kg_*)^*Z\bigr\|
 \geq\kappa e^{-k\alpha(\lambda(\ga^*))},
$$
where ${}^*$ denotes the adjoint in  $\check{\mathfrak n}$.
Combining this with
\eqref{eqn:nonconcentration basis bound} gives the first
estimate below, while
\eqref{eqn:near-principal-root-ineq} gives the second:
\begin{equation}
\label{eqn:nonconcentration rescaled basis}
\begin{aligned}
& \max_j\bigl|
 \langle Z,\operatorname{Ad}(g_*^{-1}(\ga^*)^kg_*)Y_j\rangle
 \bigr|\geq C^{-1}\kappa e^{-k\alpha(\lambda(\ga^*))},\\
& \max_j\|\operatorname{Ad}(g_*^{-1}(\ga^*)^kg_*)Y_j\|
 \leq C e^{-ks_0\alpha(\lambda(\ga^*))}.
\end{aligned}
\end{equation}

Since $s_0>s$, choose $c_0\in(0,1)$ so small that
$ C\left(\frac{\varepsilon}{2C}\right)^{s_0}
 \leq\varepsilon^s$ for all $0<\varepsilon<c_0$.
Given $\alpha\in\Pi$, $Z\in W_\kappa(\mathfrak g_{-\alpha})$, 
and $\varepsilon\in(0,c_0)$, choose $k\geq1$ so that
$$
 C e^{-k\alpha(\lambda(\ga^*))}
 \leq\frac{\varepsilon}{2}
 <C e^{-(k-1)\alpha(\lambda(\ga^*))}.
$$
Choose $j$ attaining the maximum in the first line of
\eqref{eqn:nonconcentration rescaled basis}, and put
$
 Y:=\operatorname{Ad}(g_*^{-1}(\ga^*)^kg_*)Y_j.
$
Then $g_*\exp(Y)P
 =(\ga^*)^kg_*\exp(Y_j)P\in\Lambda_+$.
By \eqref{eqn:adapted-radius-contraction} and the choice of $k$,
$$
 \delta_0\varepsilon
 <
 d_{\bar\psi}\bigl(y_{\ga^*},g_*\exp(Y)P\bigr)^{
   \frac{\alpha(\lambda(\ga^*))}{\bar\psi(\lambda(\ga^*))}}
 \leq\frac{\varepsilon}{2}<\varepsilon.
$$ where $\delta_0:=(2C^2)^{-1} {e^{-\max_{\beta\in\Pi}\beta(\lambda(\ga^*))}}$.
Likewise, \eqref{eqn:nonconcentration rescaled basis} gives
$$
 |\langle Z,Y\rangle|
 >
 \frac{\kappa\varepsilon}
 {2C^2e^{\alpha(\lambda(\ga^*))}}
 \geq\kappa \delta_0 \varepsilon,
 \qquad
 \|Y\|
 \leq C\left(\frac{\varepsilon}{2C}\right)^{s_0}
 \leq\varepsilon^s.
$$ Taking $\eta_0:=\kappa\delta_0\in(0,1)$ proves
\eqref{eqn:quantitative-non-concentration}.
\end{proof}

\subsection{A geometric formula for cocycle differences}
We express differences between return branches in the Bruhat coordinates associated with the attracting and repelling flags of $\ga^*$
 used in \Cref{prop: non-concentration self-joining}.

Recall that $p_0=P$. For $g\in G$, recall
$$
 [g]:=(g^+,g^-,\beta_{g^+}(e,g))\in\mathcal F^{(2)}\times\fa,
 \qquad g^+=gp_0,\quad g^-=gw_0p_0.
$$

\begin{lemma}[{\cite[Lemma~8.4]{KOW_ergodic}}]
\label{lem.zigzagmove}
For $g\in G$, $n\in N$, and $h\in\check N$,
$$
 \begin{aligned}
 [gn]&=(g^+,gn^-,\beta_{g^+}(e,g)),\\
 [gh]&=\bigl(gh^+,g^-,\beta_{g^+}(e,g)
       +2\mathcal G(gh^+,g^-)-2\mathcal G(g^+,g^-)\bigr).
 \end{aligned}
$$
\end{lemma}

For the rest of this section, fix $s\in(1/2,\theta_*)$ and
$\gamma^*$ as in \Cref{prop: non-concentration self-joining}.
Choose $g_*\in G$ such that
$g_*^{-1}\gamma^*g_*\in\exp(\lambda(\gamma^*))M$, and set
\begin{equation}
\label{eqn:normalized-fixed-pair}
x_1:=g_*P=y_{\gamma^*},\qquad
q_*:=g_*w_0P=y_{(\gamma^*)^{-1}}.
\end{equation}
For every $x\in\mathcal F$ transverse to $q_*$, let
$h_x\in\check N$ be the unique element such that
$x=g_*h_xP$. In particular, $h_{x_1}=e$, and
$h_{g_*\exp(Y)P}=\exp(Y)$ for $Y\in\check{\mathfrak n}$.

For $(x,y)\in\mathcal F^{(2)}$ with $x$ transverse to $q_*$
and $y$ transverse to $x_1$, set $g_1:=g_*$ and successively
choose $g_2\in g_1N$, $g_3\in g_2\check N$,
$g_4\in g_3N$, and $g_5\in g_4\check N$ such that
$$
g_2^-=y,\qquad g_3^+=x,\qquad
g_4^-=q_*,\qquad g_5^+=x_1.
$$
Each choice is unique.
For $(h,n)\in\check N\times N$ such that
$(g_*hP,g_*nw_0P)\in\mathcal F^{(2)}$, apply this
construction with $x=g_*hP$ and $y=g_*nw_0P$.
Since $g_5$ and $g_*$ have the same two flags,
$g_*^{-1}g_5\in AM$. Define $\Xi(h,n)$ to be the unique
element of $A$ satisfying
\begin{equation}
\label{eqn:def-Xi}
g_*M=g_5\Xi(h,n)M.
\end{equation}

For fixed $n\in N$, write $\Xi_n(h):=\Xi(h,n)$.
For $\gamma\in\sH\smallsetminus\{e\}$, let $n_\gamma\in N$
be determined by
$$
q_\gamma:=\gamma^{-1}p_0=g_*n_\gamma w_0P.
$$
By \Cref{lem.LP6.62} and inverse iteration in $U_0$, the flags
$p_0$, $q_*$, and $q_\gamma$ belong to $U_0$, which is disjoint
from $\overline{\Lambda_+}$.
They are therefore transverse to every point of
$\overline{\Lambda_+}$, so $h_x$ and $\Xi(h_x,n_\gamma)$
are defined for $x\in\Lambda_+$.

\begin{lemma}
\label{lem:nh map and return map}
For $n\geq1$, $\gamma,\eta\in\mathcal H^n$, and
$x\in\Lambda_+$,
$$
\begin{aligned}
\log\Xi(h_x,n_\gamma)-\log\Xi(h_x,n_\eta)
={}K_n(\gamma x)-K_n(\eta x) -K_n(\gamma x_1)+K_n(\eta x_1).
\end{aligned}
$$
\end{lemma}
\begin{proof}
Apply \Cref{lem.zigzagmove} to the associated sequence for
$(x,q_\gamma)$. Since
$\log\Xi(h_x,n_\gamma)=\beta_{x_1}(e,g_*)-\beta_{x_1}(e,g_5)$,
we obtain
$$
\begin{aligned}
\log\Xi(h_x,n_\gamma)
={}&2\mathcal G(x,q_*)-2\mathcal G(x_1,q_*)-2\mathcal G(x,q_\gamma)+2\mathcal G(x_1,q_\gamma).
\end{aligned}
$$
On the other hand, \eqref{eqn:telescoping-return-cocycle}
and \eqref{GPT} give
$$
\begin{aligned}
K_n(\gamma x)-K_n(\gamma x_1)
={}&2\mathcal G(x,p_0)-2\mathcal G(x_1,p_0)-2\mathcal G(x,q_\gamma)+2\mathcal G(x_1,q_\gamma).
\end{aligned}
$$
Subtracting the corresponding formulas for $\eta$ cancels
the terms involving $q_*$ and $p_0$, respectively, and
proves the identity.
\end{proof}

Define the backward limit set
$$
\Lambda_b
:=
\overline{\{y_{\gamma^{-1}}:
\gamma\in\sH\smallsetminus\{e\}\}}.
$$
By \Cref{lem.LP6.62},
$\gamma^{-m}p_0\in U_0$ for every
$\gamma\in\sH\smallsetminus\{e\}$ and $m\geq1$.
Since $p_0$ is transverse to $y_\gamma\in\Lambda_+$,
these points converge to $y_{\gamma^{-1}}$.
Closedness of $U_0$ therefore gives $\Lambda_b\subset U_0$.
In particular, $\Lambda_b$ is disjoint from
$\overline{\Lambda_+}$.

\begin{lemma}
\label{lem:backward-limit-Zariski-density}
The set $\Lambda_b$ is Zariski dense in $\mathcal F$,
contains $p_0$, and satisfies
$\gamma^{-1}\Lambda_b\subset\Lambda_b$
for every $\gamma\in\sH$.
\end{lemma}

\begin{proof}
Since $\sH^{-1}$ is Zariski dense by
\Cref{cor.Zdensesemigroup},
\cite[Section~3.6]{Benoist1997proprietes}
gives the invariance and Zariski density of $\Lambda_b$.
Fix $g\in\mathcal H$ and
$u\in\mathsf P_0\smallsetminus\{e\}$.
By \eqref{eqn:iterationisinthecoding},
$gu^m\in\mathcal H$ for all sufficiently large $m$.
Choose $\zeta\in\Lambda_b\smallsetminus\{gp_0\}$.
The invariance just proved gives
$u^{-m}g^{-1}\zeta=(gu^m)^{-1}\zeta\in\Lambda_b
$
for all sufficiently large $m$.
Since $g^{-1}\zeta\neq p_0$ and $u\in \mathsf P_0$, these points converge to $p_0$.
Thus $p_0\in\Lambda_b$ by closedness.
\end{proof}

Since $\overline{\Lambda_+}$ and $\Lambda_b$ are disjoint
subsets of $\Lambda_\rho$, every pair in
$\overline{\Lambda_+}\times\Lambda_b$ is transverse.
As $x_1\in\Lambda_+$ and $q_*\in\Lambda_b$, the inverse
image of this product under
$(h,n)\mapsto(g_*hP,g_*nw_0P)$ is a compact product
contained in the domain of $\Xi$.

Choose relatively compact open neighborhoods
$\check N_1\subset\check N$ and $N_1\subset N$
of its two factors such that
$\overline{\check N_1}\times\overline{N_1}$
is still contained in that domain.
Then $h_x\in\check N_1$ for
$x\in\overline{\Lambda_+}$, and $n_\gamma\in N_1$
for $\gamma\in\sH\smallsetminus\{e\}$.
Moreover, $e\in\check N_1\cap N_1$, since
$h_{x_1}=e$ and $q_*=g_*w_0P$.

Let $\pi_\fa:\mathfrak g\to\fa$ be the projection associated to
\eqref{eqn:restricted root decomp}.

\begin{lemma}
\label{lem:nh map}
The map $\Xi$ is real analytic on its domain, $\Xi_n(e)=e$, and
\begin{equation}
\label{eqn:differential-Xi}
 \bigl(d(\log\circ\Xi_n)\bigr)_e
 =\pi_\fa\circ\operatorname{Ad}_{n^{-1}}|_{\check{\mathfrak n}}.
\end{equation}
\end{lemma} \begin{proof} 
Let $(h, n) \in \check N \times N$ belong to the domain of $\Xi$. 
Then $(n^{-1}hP,w_0P)\in\mathcal F^{(2)}$, and hence
$n^{-1}h\in\check NP$
Write $n^{-1}h=h'amn'$, with $h'\in\check N$, $a\in A$, $m\in M$,
and $n'\in N$. The associated sequence satisfies
$$
 g_3=g_1nh',\qquad
 g_4=g_1ha^{-1}m^{-1},\qquad
 g_5=g_1a^{-1}m^{-1}.
$$
Thus \eqref{eqn:def-Xi} gives $\Xi(h,n)=a$. The components of
this decomposition are real analytic, and the $A$-component of
$n^{-1}$ is $e$, so $\Xi_n(e)=e$.

Since right multiplication by $n$ preserves the $A$-component, for $Y \in \check{\mathfrak n}$ and $t \in \R$ close to $0$,
$\Xi_n(\exp(tY))$ is the $A$-component of
$ n^{-1}\exp(tY)n=\exp(t\operatorname{Ad}_{n^{-1}}Y)$.
The differential at $e$ of the logarithm of the $A$-component is
$\pi_\fa$, by \eqref{eqn:restricted root decomp}. This proves
\eqref{eqn:differential-Xi}.\end{proof}

By compactness,
\begin{equation}
\label{eqn:constant C1}
 C_{1}:=\sup_{n\in\overline{N_1}}\|\operatorname{Ad}_{n^{-1}}\|<\infty.
\end{equation}

\subsection{A finite family of nondegenerate adjoint directions}

\begin{lemma}
\label{lem.LP6.78}
Let $\xi\in\Lambda_b$. For every $\varepsilon>0$, there
exists a loxodromic element $\gamma_0\in\sH$ such that
$$
d_{\bar\psi}((\gamma\gamma_0)^{-1}p_0,\xi)<\varepsilon
\qquad\text{for every }\gamma\in\sH.
$$
\end{lemma}

\begin{proof}
By the definition of $\Lambda_b$, choose
$\eta\in\sH\smallsetminus\{e\}$ such that
$d_{\bar\psi}(y_{\eta^{-1}},\xi)<\varepsilon$.
Every point of the compact set $\Lambda_b$ is transverse
to $y_\eta\in\Lambda_+$. Hence
$\eta^{-m}y\to y_{\eta^{-1}}$ uniformly for $y\in\Lambda_b$.
By continuity of $d_{\bar\psi}$, we may choose $m\geq1$
such that
$\sup_{y\in\Lambda_b}
d_{\bar\psi}(\eta^{-m}y,\xi)<\varepsilon.
$
Set $\gamma_0:=\eta^m$.
Since $\gamma^{-1}p_0\in\Lambda_b$ for every
$\gamma\in\sH$ by
\Cref{lem:backward-limit-Zariski-density},
the assertion follows from
$(\gamma\gamma_0)^{-1}p_0=\eta^{-m}\gamma^{-1}p_0$.
\end{proof}

For $\beta\in\Pi$, let $H_\beta\in\mathfrak a$ be defined by
$\langle H,H_\beta\rangle=\beta(H)$ for every
$H\in\mathfrak a$.

\begin{lemma}[{\cite[Chapter~VI, Section~3, Lemma~3.1]{Hel}}]
\label{lem:Hel}
For every $\beta\in\Pi$, there are unit vectors
$E_\beta\in\mathfrak g_\beta$ and
$E_{-\beta}\in\mathfrak g_{-\beta}$ such that
$$
[E_\beta,E_{-\beta}]\in\mathbb R_{>0}H_\beta.
$$
\end{lemma}

Fix the unit vectors $E_{\pm\beta}$ from \Cref{lem:Hel}.
Since the simple roots span $\mathfrak a^*$, compactness
of the unit sphere gives $\delta_2>0$ such that every
unit vector $v_0\in\mathfrak a$ satisfies
$|\beta(v_0)|\geq\delta_2$ for some $\beta\in\Pi$.

\begin{proposition}
\label{prop:adjoint map approximation}
There exist $C_2,C_3>0$, $j_0,k_0\in\mathbb N$, and
$\gamma_j'\in\mathcal H_{C_2}^{k_0}$
$(1\leq j\leq j_0)$ with the following property.
For every integer $k>k_0$ and 
$\gamma_j\in\mathcal H^{k-k_0}\gamma_j'$
$(1\leq j\leq j_0)$, the following hold:
\begin{enumerate}
\item\label{pro:1}
$n_{\gamma_j}\in N_1$ for every $1\leq j\leq j_0$;
\item\label{pro:2}
for every unit vector $v_0\in\mathfrak a$ and
$\beta\in\Pi$ satisfying $|\beta(v_0)|\geq\delta_2$,
there is $1\leq j\leq j_0$ such that
$
\bigl|\langle v_0,
\operatorname{Ad}_{n_{\gamma_j}^{-1}}E_{-\beta}\rangle\bigr|
\geq C_3.
$
\end{enumerate}
\end{proposition}

\begin{proof}
By \Cref{lem:backward-limit-Zariski-density},
$\gamma^{-1}p_0\in\Lambda_b$ for every $\gamma\in\sH$.
Thus (1) follows from the choice of $N_1$.
Fix a unit vector $v_0\in\mathfrak a$ and $\beta\in\Pi$
with $|\beta(v_0)|\geq\delta_2$.
The function
$n\mapsto\langle v_0,\operatorname{Ad}_{n^{-1}}E_{-\beta}\rangle$
is polynomial in exponential coordinates on $N$, since
$\operatorname{Ad}_{\exp(-X)}=e^{-\operatorname{ad}X}$
is a finite sum for $X\in\mathfrak n$.
Its derivative at $e$ in the direction $E_\beta$ is
$-\langle v_0,[E_\beta,E_{-\beta}]\rangle\neq0$
by \Cref{lem:Hel}.

The map $n\mapsto g_*nw_0P$ is an algebraic isomorphism
from $N$ onto the open Bruhat cell transverse to $x_1$,
which contains $\Lambda_b$.
By \Cref{lem:backward-limit-Zariski-density}, we may
therefore choose $n\in N_1$ such that
$g_*nw_0P\in\Lambda_b$ and
$\langle v_0,\operatorname{Ad}_{n^{-1}}E_{-\beta}\rangle
\neq0$.
Continuity and \Cref{lem.LP6.78} then give
$\gamma'\in\sH\smallsetminus\{e\}$ and $c>0$ such that,
for every unit vector $v$ sufficiently close to $v_0$,
$$
\bigl|\langle v,
\operatorname{Ad}_{n_{\omega\gamma'}^{-1}}E_{-\beta}\rangle\bigr|
>c
\qquad\text{for every }\omega\in\sH.
$$

By compactness of the sets
$\{v:\|v\|=1,\ |\beta(v)|\geq\delta_2\}$, finitely many such choices give elements
$\gamma_j'\in\mathcal H^{k_j}$ and a common lower bound
$C_3>0$, valid after any left multiplication by an
element of $\sH$. Fix $\vartheta\in\mathcal H$ and put $k_0:=\max_j k_j$.
Replacing $\gamma_j'$ by
$\vartheta^{k_0-k_j}\gamma_j'$ preserves this property
and makes every word have length $k_0$.
Choose $C_2>0$ so that all letters in these finitely
many words belong to $\mathcal H_{C_2}$.
This proves (2).
\end{proof}

\subsection{Proof of \Cref{thm:non-integrability}}

Recall $x_1=g_*P$ and $q_*=g_*w_0P$ from
\eqref{eqn:normalized-fixed-pair}.
We use the constants and branches supplied by
\Cref{prop:adjoint map approximation}.

\smallskip
\noindent\emph{Step 1: infinitesimal non-integrability.}
Since $q_*=g_*w_0P$, its $N$-coordinate is $e$.
Also, $\pi_{\mathfrak a}|_{\check{\mathfrak n}}=0$.
Applying \Cref{lem.LP6.78} to $q_*\in\Lambda_b$ and
using continuity gives
$\gamma_0'\in\mathcal H^{m_0}$, for some
$m_0\in\mathbb N$, such that
\begin{equation}
\label{eqn:control of ajoint map}
\bigl\|\pi_{\mathfrak a}\circ
\operatorname{Ad}_{n_{\omega\gamma_0'}^{-1}}
|_{\check{\mathfrak n}}\bigr\|<C_3/2
\qquad\text{for every }\omega\in\sH.
\end{equation}

Let $\gamma\in\mathcal H$ be the first letter of the word
representing $\gamma^*$.
The periodic itinerary of $x_1$ gives
$\gamma^{-1}x_1\in\Lambda_+$.
Since $\Lambda_+\subset
\operatorname{int}_{\Lambda_\rho}(\Delta_0)$ by
\Cref{lem.branchcompactinterior}, we have
$x_1\in\gamma\Lambda_+\cap
\operatorname{int}_{\Lambda_\rho}(\gamma\Delta_0)$.

Fix $\gamma_{\mathrm b}\in\mathcal H$.
By \Cref{lem.branchcompactinterior},
$\gamma_{\mathrm b}^2\Delta_0$ is a compact subset of
$\operatorname{int}_{\Lambda_\rho}
(\gamma_{\mathrm b}\Delta_0)$.
We may therefore choose $r_0>0$ such that
\begin{equation}
\label{hat}
B_{\bar\psi}(x_1,r_0)\subset\gamma\Delta_0,
\qquad
B_{\bar\psi}(\gamma_{\mathrm b}^2\Delta_0,r_0)
\subset\gamma_{\mathrm b}\Delta_0.
\end{equation}
Choose $C>1$ so that $\gamma,\gamma_{\mathrm b}$ and
all letters of $\gamma_0',\gamma_1',\ldots,\gamma_{j_0}'$
belong to $\mathcal H_C$, and put
$N_{\mathrm{NI}}:=\max\{m_0,k_0\}+2$.
Fix an integer $\ell>N_{\mathrm{NI}}$ and define
\begin{equation}
\label{eqn:finite-branch-family}
\gamma_0:=\gamma_{\mathrm b}^{\,\ell-m_0}\gamma_0',
\qquad
\gamma_j:=\gamma_{\mathrm b}^{\,\ell-k_0}\gamma_j'
\quad(1\leq j\leq j_0).
\end{equation}
Every $\gamma_j$, including $\gamma_0$, belongs to
$\mathcal H_C^\ell$ and begins with $\gamma_{\mathrm b}^2$.

For $1\leq j\leq j_0$, define
$$
\BP_j(h)
:=
\log\Xi_{n_{\gamma_j}}(h)
-\log\Xi_{n_{\gamma_0}}(h)
\qquad(h\in\check N_1).
$$
Then $\BP_j(e)=0$, and
\Cref{lem:nh map and return map} gives
\begin{equation}
\label{eqn:BP and Xi}
\BP_j(h_x)
=
K_\ell(\gamma_jx)-K_\ell(\gamma_0x)
-K_\ell(\gamma_jx_1)+K_\ell(\gamma_0x_1)
\qquad(x\in\Lambda_+).
\end{equation}
By \eqref{eqn:differential-Xi},
$$
(d\BP_j)_e
=
\pi_{\mathfrak a}\circ
\bigl(\operatorname{Ad}_{n_{\gamma_j}^{-1}}
-\operatorname{Ad}_{n_{\gamma_0}^{-1}}\bigr)
|_{\check{\mathfrak n}}.
$$
Since $\pi_{\mathfrak a}$ is an orthogonal projection,
\eqref{eqn:constant C1} gives
\begin{equation}
\label{eqn:uniform upper bp}
\|(d\BP_j)_e\|\leq2C_1.
\end{equation}
By \Cref{prop:adjoint map approximation} and
\eqref{eqn:control of ajoint map}, for every unit vector
$v_0\in\mathfrak a$ and $\alpha\in\Pi$ with
$|\alpha(v_0)|\geq\delta_2$, some $1\leq j'\leq j_0$
satisfies
\begin{equation}
\label{eqn:ineq at e}
|\langle v_0,(d\BP_{j'})_eE_{-\alpha}\rangle|
\geq C_3/2.
\end{equation}

\smallskip
\noindent\emph{Step 2: finite-scale non-integrability.}
Put
$$
\begin{aligned}
s_1&:=\min\left\{1,\min_{\alpha\in\Pi}
\frac{\bar\psi(\lambda(\gamma^*))}
     {\alpha(\lambda(\gamma^*))}\right\}, \qquad s_2:=\max\left\{1,\max_{\alpha\in\Pi}
\frac{\bar\psi(\lambda(\gamma^*))}
     {\alpha(\lambda(\gamma^*))}\right\},
\qquad s_3:=1/s_1.
\end{aligned}
$$
For the fixed $s\in(1/2,\theta_*)$, let
$c_0,\delta_0,\eta_0$ be supplied by
\Cref{prop: non-concentration self-joining} for $g_*$
and $\kappa:=\min\{1,C_3/(4C_1)\}$.

Since $n_{\gamma_j}\in N_1$ for $0\leq j\leq j_0$,
analyticity of $\Xi$ on a neighborhood of
$\overline{\check N_1}\times\overline{N_1}$ gives $M>0$,
independent of $\ell$ and $j$, such that for $Y$ in a fixed neighborhood of $0$ in
$\check{\mathfrak n}$,
\begin{equation}
\label{eqn:uniform-BP-Taylor}
\|\BP_j(\exp Y)-(d\BP_j)_eY\|
\leq M\|Y\|^2.
\end{equation}

Choose $\epsilon_3\in(0,c_0)$ so small that this estimate
holds whenever $\|Y\|\leq\epsilon_3^s$, that
$2\epsilon_3^{s_1}<\min\{1,r_0\}$, and that
$M\epsilon_3^{2s-1}\leq C_3\eta_0/4$.
Set
$$
\epsilon_1:=2\epsilon_3^{s_1},
\qquad
\delta_1:=\delta_0^{s_2}\epsilon_3^{s_2-s_1},
\qquad
\eta:=\min\{1/2,C_3\eta_0/4\}.
$$
These constants are independent of $\ell$.

Fix a unit vector $v_0\in\mathfrak a$ and $b>3$, and
choose $\alpha\in\Pi$ with
$|\alpha(v_0)|\geq\delta_2$.
Let $j'$ be given by \eqref{eqn:ineq at e}, and put
$Z:=((d\BP_{j'})_e)^*v_0\in\check{\mathfrak n}$.
By \eqref{eqn:ineq at e} and
\eqref{eqn:uniform upper bp},
$$
|\langle Z,E_{-\alpha}\rangle|\geq C_3/2,
\qquad
C_3/2\leq\|Z\|\leq2C_1.
$$
Thus $Z/\|Z\|\in W_\kappa(\mathfrak g_{-\alpha})$.
Apply \Cref{prop: non-concentration self-joining} to
this unit vector with $\varepsilon=\epsilon_3/b$,
and let $Y$ be the resulting vector.
Set $x_2:=g_*\exp(Y)P$.
Then $h_{x_2}=\exp(Y)$ and $x_2\in\Lambda_+$.
Since
$s_1\leq
\bar\psi(\lambda(\gamma^*))/\alpha(\lambda(\gamma^*))
\leq s_2$ and $\varepsilon<1$, we obtain
\begin{equation}
\label{eqn:existence of points x1 x2}
\begin{aligned}
\left(\frac{\delta_0\epsilon_3}{b}\right)^{s_2}
<d_{\bar\psi}(x_1,x_2)
<\left(\frac{\epsilon_3}{b}\right)^{s_1},\qquad 
|\langle Z,Y\rangle|
\geq\frac{C_3\eta_0}{2}\frac{\epsilon_3}{b},
\qquad
\|Y\|\leq\left(\frac{\epsilon_3}{b}\right)^s.
\end{aligned}
\end{equation}

By \eqref{eqn:uniform-BP-Taylor} and
\eqref{eqn:existence of points x1 x2},
$$
\begin{aligned}
|\langle v_0,\BP_{j'}(h_{x_2})\rangle|
&\geq
\frac{C_3\eta_0}{2}\frac{\epsilon_3}{b}
-M\left(\frac{\epsilon_3}{b}\right)^{2s}\geq
\frac{C_3\eta_0}{4}\frac{\epsilon_3}{b}
\geq
\eta\,d_{\bar\psi}(x_1,x_2)^{s_3}.
\end{aligned}
$$
Here the second inequality uses $2s>1$ and the choice
of $\epsilon_3$, and the last uses $s_1s_3=1$.
Taking $\gamma_{j'}$ and $\gamma_0$ as the two branches
in the theorem, \eqref{eqn:BP and Xi} gives
\eqref{eqn:non-integrability-estimate}.
The definitions of $\epsilon_1,\delta_1$ turn
\eqref{eqn:existence of points x1 x2} into the required
annulus. Also, $\epsilon_1/b^{s_1}<r_0$, so
\eqref{hat} gives the required ball about $x_1$.

Finally, every word in \eqref{eqn:finite-branch-family}
begins with $\gamma_{\mathrm b}^2$, so
$\gamma_jx_k\in\gamma_{\mathrm b}^2\Delta_0$.
By \eqref{hat},
$$
B_{\bar\psi}\left(
\gamma_jx_k,\frac{\epsilon_1}{2b^{s_1}}\right)
\subset\gamma_{\mathrm b}\Delta_0
\qquad(0\leq j\leq j_0,\ k=1,2).
$$
Since $\gamma_{\mathrm b}$ is the first letter of both
selected branches, this proves
\eqref{eqn:inclusion of cylinder}.
\qed

\part{Spectral theory and BMS local mixing}
\label{part.thermo}

\section{Spectral theory of transfer operators}
\label{sec.spectra of complex transfer operators}

We study the transfer operators for
$(\Lambda_+,T)$, with weights determined by the roof function
$R$ and the return cocycle $K$.
We apply the Ruelle--Perron--Frobenius theory for countable
shifts and establish regularity directly with respect to
$d_{\bar\psi}$.
We obtain a positive eigenfunction and a spectral gap for
the normalized real operator, and show that the gap persists
under small complex perturbations.
Finally, we compute the second-order expansion of the
centered leading eigenvalue and prove that its negative
Hessian at the origin is positive definite.

\begin{definition}
\label{def:spectral gap}
Let $\mathcal L$ be a bounded linear operator on a complex Banach space
$X$, with spectrum $\sigma(\mathcal L)$ and spectral radius
$\rho(\mathcal L)
:=
\sup\{ |z| : z\in\sigma(\mathcal L)\}$.
We say that $\mathcal L$ has a \emph{spectral gap} if there exist
$\kappa\in\mathbb C$, a bounded rank-one projection $\mathcal P$ on
$X$, and a bounded linear operator $\mathcal R$ on $X$ such that
$$
\mathcal L=\kappa\mathcal P+\mathcal R,
\qquad
\mathcal P\mathcal R=\mathcal R\mathcal P=0,
\qquad
\rho(\mathcal R)<|\kappa|.
$$
In this case, $\kappa$ is a simple eigenvalue of $\mathcal L$,
$\rho(\mathcal L)=|\kappa|$, and we call $\kappa$ the
\emph{leading eigenvalue} of $\mathcal L$.
\end{definition}

For a topological space $X$ and 
$\mathbb K\in\{\mathbb R,\mathbb C\}$, let
$\mathcal C( X,\mathbb K)$ and $\mathcal C_b(X,\mathbb K)$  denote the spaces of continuous and bounded continuous
$\mathbb K$-valued functions on $X$, respectively.
When $\mathbb K=\mathbb R$, we suppress $\mathbb R$ from the notation. 

Recall the constant $c_0>0$ from \Cref{prop.holderlocal} and let
\begin{equation}\label{fixc}
    c\in(0,c_0]
\end{equation}
be any number in this section and in \Cref{sec.localmixingsymbolic}.

\begin{definition}[Bounded $c$-locally H\"older functions]
\label{def:bounded c locally}
For $f\in\mathcal C(\Lambda_+,\mathbb C)$ and
$\gamma\in\mathcal H$, define
\begin{align*}
|f|_{c,\gamma}
&:=\sup_{{x,y\in\gamma\Lambda_+, x\ne y}}
\frac{|f(x)-f(y)|}{d_{\bar\psi}(x,y)^c}.
\end{align*}
Set
$$
|f|_c:=\sup_{\gamma\in\mathcal H}|f|_{c,\gamma},
\qquad
\|f\|_c:=\|f\|_\infty+|f|_c,
$$
where $\|f\|_\infty:=\sup_{x\in\Lambda_+}|f(x)|$.
We define
$$
\mathcal F_c(\Lambda_+,\mathbb C)
:=
\{f\in\mathcal C(\Lambda_+,\mathbb C):\|f\|_c<\infty\}.
$$
On this space, $|\cdot|_c$ is a seminorm and
$\|\cdot\|_c$ is a norm. Equipped with $\|\cdot\|_c$, this is a complex Banach space. We write
$\mathcal F_c(\Lambda_+)$ for the corresponding real Banach space.
\end{definition}

\subsection{Thermodynamic formalism for the $-R$-potential}
\label{subsec:formalism for minus R}

By \Cref{Rinfinity2}, there exists $\epsilon_0\in (0,1)$ such that
\begin{equation}
\label{eqn:real branch summability Section15}
\sum_{\gamma\in\mathcal H}
\sup_{x\in\Lambda_+}e^{-(1-\epsilon_0)R(\gamma x)}<\infty .
\end{equation}

\begin{definition}[Real transfer operator on $\Lambda_+$]
Define $\widetilde{\mathcal L}_{-R}\colon
\mathcal C_b(\Lambda_+)\to\mathcal C_b(\Lambda_+)$ by
$$
\widetilde{\mathcal L}_{-R}H(x)
:=\sum_{x'\in T^{-1}(x)}e^{-R(x')}H(x')
=\sum_{\gamma\in\mathcal H}e^{-R(\gamma x)}H(\gamma x);
$$
by \eqref{eqn:real branch summability Section15},  the series on the right converges absolutely and uniformly on $\Lambda_+$.
\end{definition}

Setting $\lambda:=e^{-C_R}\in(0,1)$, where $C_R$ is the constant in
\eqref{eqn:contraction metric}, we have
\begin{equation}
\label{eqn:recall contraction metric}
d_{\bar\psi}(\gamma x,\gamma y)
\le \lambda^n d_{\bar\psi}(x,y)
\end{equation}
for every $n\ge1$, $\gamma\in\mathcal H^n$, and
$x,y\in\Lambda_+$.

\subsubsection{Ruelle--Perron--Frobenius theorem on
$(\Sigma^+,d_{\Sigma^+})$}

Recall from \Cref{subsec:symbolicspaceforcoding} that
$\Sigma^+=\mathcal H^{\mathbb N}$ is equipped with the metric
$d_{\Sigma^+}$ and that
$\sigma\colon\Sigma^+\to\Sigma^+$ is the one-sided shift map. Recall
also from \Cref{prop.shiftandexpanding} the homeomorphism
$\Psi\colon\Sigma^+\to\Lambda_+$
conjugating $\sigma$ to $T$. Set
$$
R^\diamond:=R\circ\Psi
\qquad\text{and}\qquad
K^\diamond:=K\circ\Psi.
$$

Fix  the H\"older exponent
\begin{equation}
\label{eqn:theta is lambda c}
\theta:=-\frac{c\log\lambda}{\log2}>0 ,
\qquad\text{equivalently}\qquad
2^{-\theta}=\lambda^{c}.
\end{equation}
Thus, for $z\ne z'$, $d_{\Sigma^+}(z,z')^{\theta}=\lambda^{c\,n(z,z')}$ where $n(z, z')$ is as in \eqref{eqn.symbolicmetric}.
For
$\mathbb K\in\{\mathbb R,\mathbb C\}$, let
 $\F_\theta(\Sigma^+,\mathbb K)$ denote the space of bounded
$\mathbb K$-valued functions on $\Sigma^+$ that are locally
$\theta$-H\"older, equipped with the norm
$$
\norm{f}_{\theta} := \norm{f}_{\infty} + \sup_{z \neq z' \in \Sigma^+, n(z, z')\ge 2} \frac{\abs{f(z) - f(z')}}{d_{\Sigma^+}(z, z')^{\theta}}.
$$

The potential
$-R^\diamond$ has the following properties:
\begin{condenum}
\item\label{(R1)}
The function $R^\diamond$ is locally $\theta$-H\"older by
\Cref{R Lipschitz}, since $\lambda=e^{-C_R}$ and hence
$\theta=cC_R/\log2$.
\item\label{(R2)} The potential $-R^\diamond$ is summable by \Cref{Rinfinity2}:
$$
\sum_{\gamma\in\mathcal H}
\exp\Bigl(\sup_{z\in[\gamma]}\bigl(-R^\diamond(z)\bigr)\Bigr)<\infty .
$$
\end{condenum}
Therefore,  the transfer operator
$
\widetilde{\mathcal L}_{-R^\diamond}\colon
\mathcal C_b(\Sigma^+)\to \mathcal C_b(\Sigma^+)$, defined by
$$ \widetilde{\mathcal L}_{-R^\diamond}F(z)
:=\sum_{\sigma z'=z}e^{-R^\diamond(z')}F(z'),
$$ is bounded.
By \Cref{prop.shiftandexpanding}, this operator is conjugate to
$\widetilde{\mathcal L}_{-R}$:
\begin{equation}
\label{eqn:transfer operator conjugation}
\widetilde{\mathcal L}_{-R^\diamond}
=\Psi^* \circ \widetilde{\mathcal L}_{-R} \circ (\Psi^{-1})^* 
\end{equation}
 where  $\Psi^*$ denotes the pull-back by $\Psi$.

We recall several notions from the thermodynamic formalism for
countable-state shifts. Let $f\colon\Sigma^+\to\mathbb R$ be a locally
H\"older summable potential. Its \emph{Gurevich pressure} is
\begin{equation*}
\pr(f)
:=\lim_{n\to\infty}\frac1n
\log\Biggl(
\sum_{\substack{\sigma^nx=x\\ x_1=\gamma}}
\exp\bigl(S_nf(x)\bigr)
\Biggr),
\qquad
S_nf:=\sum_{j=0}^{n-1}f\circ\sigma^j ,
\end{equation*}
where $\gamma\in\mathcal H$ is any fixed element.
Since the full shift is topologically mixing, the limit exists,
possibly with value $+\infty$, and is independent of the choice of
$\gamma$ \cite[Theorem~1]{Sarig_Countable}. Since $f$ is summable, the
limit is finite. The full shift also has the BIP property. A Borel probability measure $\mu$ on $\Sigma^+$ is a \emph{Gibbs state} for $f$ if
there exists $C>1$ such that for every $n\ge1$ and $z\in\Sigma^+$,
$$
C^{-1}\le
\frac{\mu([z_1,\ldots,z_n])}{\exp(S_nf(z)-n\pr(f))}
\le C.
$$
 We define the action of
$\widetilde{\mathcal L}_{-R^\diamond}^*$
on finite Borel measures on $\Sigma^+$ by duality:
$$\int F\,d\bigl(\widetilde{\mathcal L}_{-R^\diamond}^*\mu\bigr)
=
\int \widetilde{\mathcal L}_{-R^\diamond}F\,d\mu
\qquad(F\in\mathcal C_b(\Sigma^+)).$$ By \ref{(R1)} and \ref{(R2)}, we have the following;
see \cite[Corollary~2.7.5]{MU03} and \cite[Theorem~2.4 and Corollary 2.5]{Kessebohmer}.

\begin{theorem}[Real Ruelle--Perron--Frobenius theorem]
\label{thm:real Perron-Frobenius restated}
\begin{enumerate}
\item There is a unique Borel probability eigenmeasure
$\widetilde m^\diamond$ of
$\widetilde{\mathcal L}_{-R^\diamond}^*$, and the corresponding
eigenvalue is $e^{\pr(-R^\diamond)}$. Moreover
$\widetilde m^\diamond$ is a Gibbs state for $-R^\diamond$.
\item  There is a unique eigenfunction
$h^\diamond\in\mathcal F_\theta(\Sigma^+)$ of $\widetilde{\mathcal L}_{-R^\diamond}$ with eigenvalue
$e^{\pr(-R^\diamond)}$ normalized by
$\int_{\Sigma^+}h^\diamond\,d\widetilde m^\diamond=1$. Moreover, $h^\diamond$ is bounded above and 
below by positive constants.
\item The measure
$\widetilde\nu^\diamond:=h^\diamond\widetilde m^\diamond$
is the unique $\sigma$-invariant Gibbs state for $-R^\diamond$, and it
is ergodic.

\item There exist $M>0$ and  $\varrho\in(0,1)$ such that for every
$g\in \mathcal{F}_{\theta}(\Sigma^+)$ and  $n\ge0$,
\begin{equation}
\label{eqn:MU spectral gap}
\Bigl\|
e^{-n\pr(-R^\diamond)}\widetilde{\mathcal L}_{-R^\diamond}^ng
-\Bigl(\int_{\Sigma^+}g\,d\widetilde m^\diamond\Bigr)
h^\diamond
\Bigr\|_\theta
\le M\varrho^n\|g\|_\theta .
\end{equation}
\end{enumerate}
\end{theorem}
Applying \eqref{eqn:MU spectral gap} to the real and imaginary parts
of $g$ shows that, after enlarging $M$, the same estimate holds for
every $g\in\mathcal F_\theta(\Sigma^+,\mathbb C)$.

\subsubsection{The conformal eigenmeasures}
We continue to write $\nu$ for the restriction of the
Patterson--Sullivan measure to $\Lambda_+$.
Since $\nu(\Lambda_+)>0$ and $\psi(\mathcal G(x,p_0))$ is uniformly
bounded on $\Lambda_+$, the measure
$$
d\widehat\nu(x)
:=\frac{e^{2\psi(\mathcal G(x,p_0))}}
{\int_{\Lambda_+}e^{2\psi(\mathcal G(y,p_0))}\,d\nu(y)}
\,d\nu(x)
$$
is a well-defined probability measure on $\Lambda_+$.

\begin{corollary}
\label{cor:conformal eigenmeasure Section15}
We have $\Psi_*\widetilde m^\diamond=\widehat\nu$ and
$\pr(-R^\diamond)=0$. \end{corollary}
\begin{proof}
For $\gamma\in\mathcal H$ and $x\in\gamma\Lambda_+$, conformality and
the definition of $R$ give
$
\frac{d\gamma_*\widehat\nu}{d\widehat\nu}(x)
=e^{R(x)} $ on $\gamma\La_+$.
Thus $\widetilde{\mathcal L}_{-R}^*\widehat\nu
=\widehat\nu$. By
\eqref{eqn:transfer operator conjugation},
$\Psi^*\widehat\nu$ is a  Borel probability eigenmeasure of
$\widetilde{\mathcal L}_{-R^\diamond}^*$ with eigenvalue $1$.
The uniqueness in
\Cref{thm:real Perron-Frobenius restated}(1) therefore gives
$\Psi^*\widehat\nu=\widetilde m^\diamond$ and
$e^{\pr(-R^\diamond)}=1$, proving both assertions.
\end{proof}

Since 
$e^{-n\pr(-R^\diamond)}=1$, this factor will henceforth be omitted from \eqref{eqn:MU spectral gap}.

\subsubsection{The eigenfunctions and spectral gap}
The following Lasota--Yorke inequality is the main input for establishing the $d_{\bar\psi}$-regularity of the transported eigenfunction in
\Cref{prop:eigenfunction on Lambda}.  It can be proved by a standard argument, based on \Cref{lem:Holder property of Birkoff sum}.

\begin{lemma}
\label{lem:ITM restated}
There is $C>0$ such that, for every
$f\in\mathcal F_c(\Lambda_+,\mathbb C)$ and $n\in\mathbb N$,
$$
\bigl\|\widetilde{\mathcal L}_{-R}^nf\bigr\|_c
\le C\|f\|_\infty+C\lambda^{cn}\|f\|_c .
$$
 In
particular, $\widetilde{\mathcal L}_{-R}$ preserves
$\mathcal F_c(\Lambda_+,\mathbb C)$ and
$\mathcal F_c(\Lambda_+)$.
\end{lemma}

The following follows from \eqref{eqn:recall contraction metric}
by the same argument as in \ref{(R1)}.

\begin{lemma}
\label{lem:Psi pullback bounded}
With $\theta$ as in \eqref{eqn:theta is lambda c}, $\Psi^*$ maps
$\mathcal F_c(\Lambda_+,\mathbb C)$ boundedly into
$\mathcal{F}_{\theta}(\Sigma^+,\mathbb C)$, i.e., 
$$
\|f\circ\Psi\|_\theta\ll\|f\|_c
\qquad\text{for } f\in\mathcal F_c(\Lambda_+,\mathbb C).
$$
\end{lemma}

With $h^{\diamond}$ given by \Cref{thm:real Perron-Frobenius restated},
 set
\begin{equation}
\label{eqn:hFLa}
h:=h^\diamond\circ\Psi^{-1}  : \La_+ \to \R.
\end{equation}

\begin{proposition}
\label{prop:eigenfunction on Lambda}
\begin{enumerate}
\item We have $h\in\mathcal F_c(\Lambda_+)$ and $h$ is bounded
above and below by positive constants. Moreover,
$\widetilde{\mathcal L}_{-R}h=h$, and $\widehat\nu(h)=~1$. It is
the unique eigenfunction of $\widetilde{\mathcal L}_{-R}$ in
$\mathcal F_c(\Lambda_+)$ with eigenvalue $1$ and this normalization.
\item There exist $M'>0$ and $\varrho'\in(0,1)$ such that, for every
$g\in\mathcal F_c(\Lambda_+,\mathbb C)$ and  $n\ge0$,
\begin{equation}
\label{eqn:gap on Lambda}
\Bigl\|
\widetilde{\mathcal L}_{-R}^ng
-\Bigl(\int_{\Lambda_+}g\,d\widehat\nu\Bigr)h
\Bigr\|_c
\le M'(\varrho')^n\|g\|_c .
\end{equation}
In particular, $\widetilde{\mathcal L}_{-R}$ has a spectral gap on
$\mathcal F_c(\Lambda_+,\mathbb C)$, with leading eigenvalue $1$ and
rank-one spectral projection
$g\mapsto\bigl(\int_{\Lambda_+}g\,d\widehat\nu\bigr)h$.
\end{enumerate}
\end{proposition}

\begin{proof}
To prove (1), it suffices to show that $\norm{h}_c < \infty$; the remaining claims follow from \eqref{eqn:transfer operator conjugation}, \Cref{cor:conformal eigenmeasure Section15}, \Cref{lem:Psi pullback bounded}, and \Cref{thm:real Perron-Frobenius restated}.
Let $\mathbf1$ denote the constant function $1$ on $\Lambda_+$, and
set $u_n:=\widetilde{\mathcal L}_{-R}^n\mathbf1$ for each $n \in \N$. Since $\|\mathbf 1\|_c=1$,
\Cref{lem:ITM restated} gives
$\|u_n\|_c\le 2C$ for all $n\ge 1$.

On the other hand, by \eqref{eqn:transfer operator conjugation},
\eqref{eqn:MU spectral gap}, and \Cref{cor:conformal eigenmeasure Section15},
$$
\|u_n-h\|_\infty
=
\bigl\|
\widetilde{\mathcal L}_{-R^\diamond}^n\mathbf1-h^\diamond
\bigr\|_\infty
\le M\varrho^n\to 0.
$$ Passing to the limit in the H\"older inequalities gives
$\|h\|_c\leq2C$.

Now to prove (2), fix $g \in \F_c(\La_+, \mathbb C)$ and set
$
D_n:=\widetilde{\mathcal L}_{-R}^ng
-\widehat\nu (g) h .
$
Since $\widetilde{\mathcal L}_{-R}h=h$ by part (1), 
\begin{equation}
\label{eqn:Dn recursion}
D_{n+m}=\widetilde{\mathcal L}_{-R}^mD_n
\qquad (n,m\ge0).
\end{equation}

By \Cref{lem:Psi pullback bounded},
$g\circ\Psi\in \mathcal F_\theta(\Sigma^+,\mathbb C)$ and
$\|g\circ\Psi\|_\theta\ll\|g\|_c$. Applying
\eqref{eqn:MU spectral gap} to $g\circ\Psi$, and using  \Cref{cor:conformal eigenmeasure Section15} and \eqref{eqn:transfer operator conjugation} gives \begin{equation}
\label{eqn:Dn weak}
\|D_n\|_\infty\ll\varrho^n\|g\|_c .
\end{equation}
Here we used $\int_{\Sigma^+} g\circ\Psi\,d\widetilde m^\diamond
=\int_{\Lambda_+} g\,d\widehat\nu$. By \Cref{lem:ITM restated} and part (1), we also have
\begin{equation}
\label{eqn:Dn strong}
\|D_n\|_c
\ll \|g\|_\infty+\lambda^{cn}\|g\|_c
+\Bigl|\int g\,d\widehat\nu\Bigr|\,\|h\|_c
\ll\|g\|_c ,
\end{equation} with the case $n=0$ following directly from the definition of $D_0$.
Combining \eqref{eqn:Dn recursion} with \Cref{lem:ITM restated}, \eqref{eqn:Dn weak}, and \eqref{eqn:Dn strong},
 we obtain
$$
\|D_{n+m}\|_c
\ll \|D_n\|_\infty+\lambda^{cm}\|D_n\|_c
\ll\bigl(\varrho^n+\lambda^{cm}\bigr)\|g\|_c .
$$
For $N\ge1$, take $n=\lfloor N/2\rfloor$ and $m=N-n$.
It follows that \eqref{eqn:gap on Lambda} holds with any
$\varrho'\in\bigl(\max\{\varrho,\lambda^c\}^{1/2},1\bigr)$ and a
suitable $M'$, enlarged if necessary to include $N=0$.
Finally, define $\cal P, \cal R : \F_c(\La_+, \mathbb C) \to \F_c(\La_+, \mathbb C)$ by 
$$
\cal P g:=\widehat\nu (g) h,
\qquad
\mathcal R:=\widetilde{\mathcal L}_{-R}- \cal P,
$$
which are bounded. By \Cref{cor:conformal eigenmeasure Section15} and part~(1),
$\cal P^2= \cal P$ and $\widetilde{\mathcal L}_{-R} \cal P
=\cal P \widetilde{\mathcal L}_{-R}= \cal P.
$
Hence $\cal P\mathcal R=\mathcal R \cal P=0$ and
$\mathcal R^n=\widetilde{\mathcal L}_{-R}^n- \cal P$ for every $n\ge1$.
By \eqref{eqn:gap on Lambda},
$\|\mathcal R^n\|\le M'(\varrho')^n$, so
$\rho(\mathcal R)\le\varrho'<1$.
This proves the spectral-gap assertion in the sense of
\Cref{def:spectral gap}.
\end{proof}

Let $h\in\mathcal F_c(\Lambda_+)$ be as in \eqref{eqn:hFLa}, and
define the probability measure
\begin{equation}
\label{nuplus}
d\nu_+:=h\,d\widehat\nu
\end{equation}
on $\Lambda_+$. By \Cref{prop:eigenfunction on Lambda}, both $h$ and $h^{-1}$
belong to $\mathcal F_c(\Lambda_+)$, and multiplication by either
is bounded on $\mathcal F_c(\Lambda_+,\mathbb C)$.
Hence, by \Cref{lem:ITM restated}, the normalized operator
\begin{equation}
\label{eqn:L0 def}
\mathcal L_0H:=h^{-1}\widetilde{\mathcal L}_{-R}(hH)
\qquad
\bigl(H\in\mathcal F_c(\Lambda_+,\mathbb C)\bigr)
\end{equation}
is bounded on
$\mathcal F_c(\Lambda_+,\mathbb C)$.
 It is the
case $a=w=0$ of the family introduced in
\Cref{def:normalized transfer operators} below. By \eqref{eqn:L0 def},
\eqref{nuplus}, \Cref{cor:conformal eigenmeasure Section15}, and \Cref{prop:eigenfunction on Lambda}(1),
\begin{equation}
\label{eqn:invariance of transfer operator and measure}
\mathcal L_0\mathbf 1=\mathbf 1,
\qquad
\mathcal L_0^*\nu_+=\nu_+ .
\end{equation}

Define the operators $\cal P_0, \cal R_0 : \F_c (\La_+, \mathbb C) \to \F_c(\La_+, \mathbb C)$ by 
\begin{equation}
\label{eqn:P0}
\mathcal P_0f:=\Bigl(\int_{\Lambda_+}f\,d\nu_+\Bigr)\mathbf 1,
\qquad
\mathcal R_0:=\mathcal L_0-\mathcal P_0 .
\end{equation}

\begin{corollary}
\label{prop:L0 decomposition}
The operators $\mathcal P_0$ and $\mathcal R_0$ are bounded on
$\mathcal F_c(\Lambda_+,\mathbb C)$ and satisfy
\begin{equation}
\label{eqn:spectral gap decomposition}
\mathcal P_0\mathcal R_0=\mathcal R_0\mathcal P_0=0,
\qquad
\rho(\mathcal R_0)<1.
\end{equation}

\end{corollary}
\begin{proof} 
By \eqref{eqn:L0 def}, $\mathcal L_0$ is conjugate to
$\widetilde{\mathcal L}_{-R}$ by multiplication by $h$.
Multiplication by $h$ and by $h^{-1}$ is bounded on
$\mathcal F_c(\Lambda_+,\mathbb C)$, so the spectral gap in
\Cref{prop:eigenfunction on Lambda}(2) transfers to $\mathcal L_0$.
By \eqref{nuplus}, the conjugated spectral projection is
$$
f\mapsto
\left(\int_{\Lambda_+}hf\,d\widehat\nu\right)\mathbf1
=
\left(\int_{\Lambda_+}f\,d\nu_+\right)\mathbf1
=\mathcal P_0f.
$$
This proves the corollary.
\end{proof}

\begin{corollary}
\label{cor: T invariant measure}
The probability measure $\nu_+$ is $T$-invariant and ergodic.
\end{corollary}

\begin{proof}
By \eqref{nuplus}, \eqref{eqn:hFLa},
\Cref{cor:conformal eigenmeasure Section15} and
\Cref{thm:real Perron-Frobenius restated},
$$
\Psi^*\nu_+
=(h\circ\Psi)\,\Psi^*\widehat\nu
=h^\diamond\,\widetilde m^\diamond
=\widetilde\nu^\diamond ,
$$
which is $\sigma$-invariant and ergodic by
\Cref{thm:real Perron-Frobenius restated}(3). Since $\Psi$ conjugates $\sigma$ and $T$ (\Cref{prop.shiftandexpanding}), this finishes the proof.
\end{proof}

Recall from \Cref{lem.uniformly positive} that $R \ge 1$. Since $d\nu_+/d\nu$ is bounded above and below by positive constants, \Cref{thm.exptail}  implies that
$R\in L^1(\nu_+)$ and $K\in L^1(\nu_+;\mathfrak a)$.
 Hence, the \emph{$\psi$-symbolic vector}
\begin{equation}
\label{eqn:mixing direction}
\mathtt u := \mathtt u_\psi
:=\frac{\int_{\Lambda_+}K\,d\nu_+}
{\int_{\Lambda_+}R\,d\nu_+}\in\mathfrak a
\end{equation}
is well-defined, and $\psi(\mathtt u)=1$ because $R=\psi\circ K$.

\begin{theorem}
\label{thm.mixingdirectiontangent}
The symbolic direction $\uu$ satisfies
$\psi_\rho(\uu)=\psi(\uu)=1$.
Consequently, if $\psi=\psi_u$ for some
$u\in\operatorname{int}\mathcal L_\rho$, then
$$
u=\psi_u(u)\,\ut_{\psi_u}.
$$
\end{theorem}

\begin{proof}
Setting
$\bar K:=\int_{\Lambda_+}K\,d\nu_+$ and $\bar R:=\int_{\Lambda_+}R\,d\nu_+$, 
it suffices to prove that
$\psi_\rho(\bar K)= \psi(\bar K)$.

The measure $\nu_+$ is the $T$-invariant Gibbs probability
measure for the potential $-R$, whose pressure is zero;
see \Cref{thm:real Perron-Frobenius restated},
\Cref{cor:conformal eigenmeasure Section15},
\Cref{cor: T invariant measure}.
Thus there exists $C\geq1$ such that
$$
\nu_+(\gamma\Lambda_+)
\leq C e^{-R_n(x)}
\qquad
(n\geq1,\ \gamma\in\mathcal H^n,\
x\in\gamma\Lambda_+).
$$

Fix $\varepsilon>0$ and set
$E_n:=\{x\in\Lambda_+:
\|K_n(x)/n-\bar K\|<\varepsilon\}$.
By ergodicity and the Birkhoff ergodic theorem,
$\nu_+(E_n)\to1$.
Let $\mathcal W_n\subset\mathcal H^n$ consist of the
words whose cylinders meet $E_n$, and choose
$x_\gamma\in\gamma\Lambda_+\cap E_n$ for each
$\gamma\in\mathcal W_n$.
Writing $\|\psi\|_*$ for the dual norm, we have
$R_n(x_\gamma)\geq n(\bar R-\|\psi\|_*\varepsilon)$ for all sufficiently large $n$.
Since these cylinders cover $E_n$, the Gibbs estimate
gives, for all sufficiently large $n$,
\be\label{wn}
\#\mathcal W_n
\geq \frac{1}{2C}
e^{n(\bar R-\|\psi\|_*\varepsilon)}.
\ee 
Distinct words represent distinct group elements.
Moreover, by \eqref{eqn:return-Cartan-comparison}, there is
 a uniform constant
$C_0>0$  such that
$$
\|\mu(\gamma)-n\bar K\|
\leq n\varepsilon+C_0
\qquad(\gamma\in\mathcal W_n).
$$

Let $\mathcal O$ be any open cone containing $\bar K$.
Choose $\varepsilon>0$ sufficiently small that
$B(\bar K,2\varepsilon)\subset\mathcal O$ and
$\bar R-\|\psi\|_*\varepsilon>0$.
For all sufficiently large $n$, every
$\gamma\in\mathcal W_n$ then satisfies
$\mu(\gamma)\in\mathcal O$ and
$\|\mu(\gamma)\|\leq n(\|\bar K\|+\varepsilon)+C_0$.
The bound \eqref{wn} implies that the
abscissa of convergence $\tau_{\mathcal O}$ of the
Poincar\'e series $\sum_{\gamma\in \Ga_\rho, \mu(\gamma)\in \mathcal O} e^{-s \|\mu(\ga)\|} $ satisfies
$$
\tau_{\mathcal O}
\geq
\frac{\bar R-\|\psi\|_*\varepsilon}
     {\|\bar K\|+\varepsilon}.
$$
Indeed, for any positive $s$ below this ratio, the sum
over $\mathcal W_n$ of $e^{-s\|\mu(\gamma)\|}$ tends
to infinity as $n\to\infty$, so the restricted series
diverges.
Letting $\varepsilon\to0$ and then taking the infimum
over $\mathcal O$, the definition of the growth
indicator gives
$$
\psi_\rho(\bar K)
=\|\bar K\|\inf_{\mathcal O\ni\bar K}\tau_{\mathcal O}
\geq\bar R=\psi(\bar K).
$$
The reverse inequality follows from
$\psi_\rho\leq\psi$.
Homogeneity now yields
$\psi_\rho(\uu)=\psi(\uu)=1$.

By the strict concavity of $\psi_\rho$
\cite{Reyes-Wang,KOZ_Manhattan}, there exists a unique tangent direction of $\psi$. Hence $\psi=\psi_u$ gives
$u\in\mathbb R_{>0}\ut_{\psi_u}$.
\end{proof}

\subsection{Complex transfer operators}
\label{subsec:larger Banach space}

Throughout this section, we fix an inner product
$\langle\cdot,\cdot\rangle_{\mathtt u}$ on $\mathfrak a$ for which the
decomposition
$$
\mathfrak a=\mathbb R\mathtt u\oplus\ker\psi
$$
is orthogonal and $\langle\mathtt u,\mathtt u\rangle_{\mathtt u}=1$. We
write $\|\cdot\|_{\mathtt u}$ for the associated norm, omitting the subscript when it is clear from the context.

Recall the constant $\varepsilon_0 > 0$ from \eqref{eqn:real branch summability Section15}. 

\begin{definition}[Normalized complex transfer operators]
\label{def:normalized transfer operators}
For $a\in\mathbb C$ with $\Re a\ge-\epsilon_0$ and $w\in\mathfrak a$,
define the normalized transfer operator
$\mathcal L_{a,w}:\mathcal F_c(\Lambda_+,\mathbb C)\to\mathcal F_c(\Lambda_+,\mathbb C)$
by
$$ \mathcal L_{a,w}H(x)
=\frac1{h(x)}\sum_{\gamma\in\mathcal H}
e^{-(1+a)R(\gamma x)+i\langle w,K(\gamma x)\rangle_{\mathtt u}}
h(\gamma x)H(\gamma x).
$$
The defining series converges absolutely and uniformly by
\eqref{eqn:real branch summability Section15} and the
boundedness of $h$ and $h^{-1}$.
Boundedness on $\mathcal F_c(\Lambda_+,\mathbb C)$ follows
from \Cref{prop.holderlocal},
\eqref{eqn:recall contraction metric}, and
$h,h^{-1}\in\mathcal F_c(\Lambda_+)$.
It preserves $\mathcal F_c(\Lambda_+)$ when $a\in\mathbb R$ and $w=0$.
We abbreviate
$$
\mathcal L_w:=\mathcal L_{0,w},
\qquad
\mathcal L_0:=\mathcal L_{0,0},
$$
the latter agreeing with the operator in \eqref{eqn:L0 def}.
\end{definition}
Recall from \eqref{eqn:R gamma} that
$
R_\gamma:=\sup_{x\in\gamma\Lambda_+}R(x)$ for $\gamma\in\mathcal H .$
 By \Cref{lem.uniformly positive}, $R_\gamma\geq1$, and, after enlarging
$C>0$ if necessary,
\begin{equation}
\label{eqn:one branch distortion Section15}
\left|R_\gamma-\psi(\mu(\gamma))\right|\leq C,
\qquad
R_\gamma-C\leq R(x)\leq R_\gamma
\end{equation}
for every $\gamma\in\mathcal H$ and $x\in\gamma\Lambda_+$.

The following  provides the estimates needed for the spectral analysis.
\begin{lemma}[Lasota--Yorke inequalities]
\label{lem:L0 action on B}
There exists $C>0$ such that for all $w\in\mathfrak a$,
$n\in\mathbb N$ and $f\in \mathcal F_c(\Lambda_+,\mathbb C)$,
\begin{align}
&\|\mathcal L_w^nf\|_{\infty}\le C\|f\|_{\infty},
\label{eqn:Lnsup}\\
& |\mathcal L_w^nf|_{c}
\le C\max\{\|w\|,1\}\,\|f\|_{\infty}+C\lambda^{cn}\|f\|_{c},
\label{eqn:Lnf} \\
& \|\mathcal L_w^nf\|_{c}
\le C\max\{\|w\|,1\}\,\|f\|_{\infty}+C\lambda^{cn}\|f\|_{c} .
\label{eqn:Lnc}
\end{align}
\end{lemma}

\begin{proof} It suffices to prove \eqref{eqn:Lnsup} and \eqref{eqn:Lnf}, since
\eqref{eqn:Lnc} then follows after enlarging $C$.
Fix $w\in\mathfrak a$,
$n\in\mathbb N$, and $f\in \mathcal F_c(\Lambda_+,\mathbb C)$.  We first prove \eqref{eqn:Lnsup}. Since
$\widetilde{\mathcal L}_{-R}^nh=h$ by
\Cref{prop:eigenfunction on Lambda}(1),
$$
|\mathcal L_w^nf(x)|
\leq\frac{\|f\|_\infty}{h(x)}
\sum_{\gamma\in\mathcal H^n}e^{-R_n(\gamma x)}h(\gamma x)
=\|f\|_\infty.
$$
Next we prove \eqref{eqn:Lnf}. Let $x, y \in \ga_0 \La_+$ be distinct where $\ga_0 \in \cal H$. We have
$$
\mathcal L_w^nf (x)- \mathcal L_w^nf (y)
=\sum_{\gamma\in\mathcal H^n} \left( I_1(\ga) + I_2(\ga) + I_3(\ga) + I_4(\ga)\right)
$$
where
\begin{align*}
I_1(\ga )
&:=\Bigl(\frac1{h(x)}-\frac1{h(y)}\Bigr)
e^{-R_n(\ga x)+i\langle w,K_n(\ga x)\rangle_{\mathtt u}}
h(\ga x)f(\ga x),\\
I_2(\ga )
&:=\frac1{h(y)}
\Bigl(
e^{-R_n(\ga x)+i\langle w,K_n(\ga x)\rangle_{\mathtt u}}
-e^{-R_n(\ga y)+i\langle w,K_n(\ga y)\rangle_{\mathtt u}}
\Bigr)h(\ga x)f(\ga x),\\
I_3(\ga )
&:=\frac1{h(y)}
e^{-R_n(\ga y)+i\langle w,K_n(\ga y)\rangle_{\mathtt u}}
\bigl(h(\ga x)-h(\ga y)\bigr)f(\ga x),\\
I_4(\ga )
&:=\frac1{h(y)}
e^{-R_n(\ga y)+i\langle w,K_n(\ga y)\rangle_{\mathtt u}}
h(\ga y)\bigl(f(\ga x)-f(\ga y)\bigr).
\end{align*}
Since $h,h^{-1}\in\mathcal F_c(\Lambda_+)$, the preceding summation
bound gives
$$
\sum_{\gamma\in\mathcal H^n}|I_1(\gamma)|
\ll\|f\|_{\infty}d_{\bar\psi}(x,y)^c.
$$
Similarly, using \eqref{eqn:recall contraction metric}, 
$$
 \sum_{\ga \in \cal H^{n}} \left| I_3(\ga) \right| \ll  \norm{f}_{\infty} d_{\bar\psi}(x, y)^c.
$$
By the uniform distortion estimates for $R_n$ and $K_n$ in
\Cref{prop.holderlocal} and  the mean value estimate for the
exponential function,
$$\left|
e^{-R_n(\ga x)+i\langle w,K_n(\ga x)\rangle_{\mathtt u}}
-e^{-R_n(\ga y)+i\langle w,K_n(\ga y)\rangle_{\mathtt u}}
\right|\ll 
\max\{\|w\|,1\}
e^{-R_n(\ga x)}d_{\bar\psi}(x,y)^c.
$$
Applying the same summation argument as in the proof of
\eqref{eqn:Lnsup}, we obtain
$$ \sum_{\ga \in \cal H^n} \left| I_2(\ga) \right| \ll \max\{\norm{w},1\} \norm{f}_{\infty} d_{\bar\psi}(x, y)^c.
$$
Finally, for each $\gamma \in\mathcal H^n$,
$$
|f(\gamma x)-f(\gamma y)|
\leq\|f\|_c d_{\bar\psi}(\gamma x,\gamma y)^c
\leq\lambda^{cn}\|f\|_c d_{\bar\psi}(x,y)^c.
$$
Hence
$$\sum_{\ga \in \cal H^n} \left| I_4(\ga) \right| \ll \la^{cn} \norm{f}_{c} d_{\bar\psi}(x, y)^c.
$$
Combining these estimates and taking the supremum over all distinct
$x,y$ in a common one-cylinder proves \eqref{eqn:Lnf}.
\end{proof}

\begin{lemma}
\label{lem:spectral radius at most one}
For every $w\in\mathfrak a$, we have $\rho(\mathcal L_w)\leq1$.
\end{lemma}

\begin{proof}
By \eqref{eqn:Lnc},
$\|\mathcal L_w^nf\|_c
\leq(C\max\{\|w\|,1\}+C)\|f\|_c$
for all $n\ge1$ and $f\in\mathcal F_c(\Lambda_+,\mathbb C)$.
The spectral-radius formula gives the assertion.
\end{proof}

\subsection{Spectral gap of complex transfer operators}
\label{subsec:spectral gap complex}

\begin{definition}
\label{def:a-function}
Let $w\in\mathfrak a$ and $0\le a<2\pi$. We call
$\langle w,K(\cdot) \rangle_{\mathtt u}$ an \emph{$a$-function} if there exists
$\zeta\in\mathcal F_c(\Lambda_+,\mathbb C)$ with $|\zeta|\equiv1$ such
that
\begin{equation}
\label{eqn:cohomologous equation}
e^{i(\langle w,K\rangle_{\mathtt u}-a)}\zeta=\zeta\circ T .
\end{equation}
\end{definition}

The argument of \cite[Proposition~2.6]{Kessebohmer} gives the following:
\begin{proposition}[{\cite[Proposition~2.6]{Kessebohmer}}]
\label{prop:lattice function}
Let $w\in\mathfrak a$ and $0\le a<2\pi$. Then $e^{ia}$ is an eigenvalue
of $\mathcal L_w$ on $\mathcal F_c(\Lambda_+,\mathbb C)$ if and only if
$\langle w,K (\cdot) \rangle_{\mathtt u}$ is an $a$-function.
\end{proposition}

It follows from \Cref{prop:lattice function} and
\Cref{lem:spectral radius at most one} that, if
$\langle w,K(\cdot)\rangle_{\mathtt u}$ is an $a$-function, then
\begin{equation}
\label{eqn:rho equals one}
\rho(\mathcal L_w)=1.
\end{equation}

\begin{lemma}
\label{lem:periodic obstruction}
Let $w\in\mathfrak a$ and $0\le a<2\pi$. If
$\langle w,K(\cdot) \rangle_{\mathtt u}$ is an $a$-function, then
$$
\langle w,K_n(x)\rangle_{\mathtt u}-an\in2\pi\mathbb Z
\qquad\text{for every }n\ge1\text{ and }x\in\operatorname{Fix}(T^n).
$$
\end{lemma}
\begin{proof}
Iterating \eqref{eqn:cohomologous equation} gives
$e^{i(\langle w,K_n(x)\rangle_{\mathtt u}-an)}\zeta(x)
=\zeta(T^nx)$.
If $T^nx=x$, then $|\zeta(x)|=1$ implies the assertion.
\end{proof}

\begin{proposition}
\label{prop:irrationality}
Let $w\in\mathfrak a\smallsetminus\{0\}$. If $\mathcal L_w$ has an
eigenvalue $e^{ia}$ with $0\le a<2\pi$, then $a/\pi\notin\mathbb Q$.
\end{proposition}

\begin{proof}
Suppose $a/\pi\in\mathbb Q$. By \Cref{prop:lattice function} and
\Cref{lem:periodic obstruction}, the character
$v \in \fa \mapsto e^{i\langle w,v\rangle_{\mathtt u}}$ takes only finitely many values on periods of $T^n$.
On the other hand, by the non-arithmeticity of periods (\Cref{thm.denselength}), this implies $w = 0$, a contradiction.
\end{proof}

Write $\mathfrak a_{\mathbb C}$ for the complexification of
$\mathfrak a$, and extend $\langle\cdot,\cdot\rangle_{\mathtt u}$
complex-bilinearly to $\frak a_{\mathbb C}$.
The bound $\|K(x)\|\ll R(x)$ from \Cref{lem.uniformly positive} 
and
\eqref{eqn:real branch summability Section15} give an open
neighborhood $\widetilde{\mathcal U}\subset\mathfrak a_{\mathbb C}$
of $\mathfrak a$ on which the series below converges absolutely
and uniformly for every $H\in\mathcal F_c(\Lambda_+,\mathbb C)$.

We extend the family $\mathcal L_w$, defined for
$w\in\mathfrak a$, to
$w\in\widetilde{\mathcal U}\subset\mathfrak a_{\mathbb C}$
by
\begin{equation}
\label{eqn:complexification}
\mathcal L_wH(x)
:=\frac1{h(x)}\sum_{\gamma\in\mathcal H}
e^{-R(\gamma x)+i\langle w,K(\gamma x)\rangle_{\mathtt u}}
h(\gamma x)H(\gamma x).
\end{equation}
The argument in the proof of \Cref{lem:L0 action on B} with $n=1$
shows that $\mathcal L_w$ is bounded, after shrinking
$\widetilde{\mathcal U}$ if necessary.

We now perturb the decomposition \eqref{eqn:spectral gap decomposition} in the parameter $w$.

\begin{theorem}[{Spectral decomposition near zero}]
\label{thm:decomposition of transfer operators}
The following hold: 
\begin{itemize}
\item There exist an open neighborhood
$\mathcal U\subset\mathfrak a_{\mathbb C}$ of $0$ 
such that
$w\mapsto \mathcal L_w$
is analytic on $\mathcal U$.

\item For each $w\in\mathcal U$, 
$\mathcal L_w $ admits a spectral decomposition
\begin{equation}
\label{eqn:PR decomposition of transfer operator}
\mathcal L_w=\kappa_w\mathcal P_w+\mathcal R_w,
\end{equation}
where $\kappa_w$ is a simple eigenvalue of $\mathcal L_w$ with $\kappa_0 = 1$,
$\mathcal P_w$ is the corresponding
rank-one spectral projection, and
$\mathcal R_w$ is a bounded operator on $\mathcal F_c(\La_+, \mathbb C)
$  satisfying:
\begin{enumerate}
\item The maps $w\mapsto\kappa_w$,
$w\mapsto\mathcal P_w$, and $w\mapsto\mathcal R_w$
are analytic on $\mathcal U$.

\item There exists $\eta>0$ such that for every $w\in\mathcal U$, 
\begin{equation}
\label{eqn:PR decomposition}
\mathcal P_w\mathcal R_w=\mathcal R_w\mathcal P_w=0,
\qquad
\rho(\mathcal R_w)<|\kappa_w|-\eta.
\end{equation}

\item Setting $\mathcal O:=\mathcal U\cap\mathfrak a$,
we have
$$
|\kappa_w|<1
\qquad\text{for every }w\in\mathcal O\smallsetminus\{0\}.
$$
\end{enumerate}

\item For every $w\in\mathcal U$, $
\int_{\Lambda_+}\mathcal P_w\mathbf 1\,d\nu_+\ne0$ and thus
\begin{equation}
\label{eqn:Hw}
H_w:= \nu_+(
\mathcal P_w\mathbf 1)^{-1}
\mathcal P_w\mathbf 1
\in\mathcal F_c(\Lambda_+,\mathbb C)
\end{equation}
is well defined. Moreover, 
 $H_w$ is the unique eigenfunction of 
$\mathcal L_w$ 
with eigenvalue $\kappa_w$ satisfying
$
\int_{\Lambda_+}H_w\,d\nu_+=1,
$, $H_0=\mathbf 1$;
and the map $w \mapsto H_w$
is analytic on $\mathcal U$.

\item Setting
$A_w:=\kappa_w^{-1}\mathcal R_w$,
there exists $\vartheta\in(0,1)$ such that, for every multi-index
$\alpha$, there is a constant $C_\alpha>0$ for which
\begin{equation}
\label{eqn:uniform remainder powers Section15}
\sup_{w\in\mathcal O}
\left\|
\partial_w^\alpha\bigl(A_w^n\bigr)
\right\|
\le C_\alpha\vartheta^n
\qquad\text{for all }n\ge0
\end{equation}
where $\norm{\cdot}$ is the operator norm.
In particular, 
there exists $C>0$ such that
$
\|\mathcal R_0^n\|
\le C\vartheta^n$ for all $n \ge 0$.
\end{itemize}
\end{theorem}

\begin{proof}
\emph{Analyticity of $\L_w$.}
For $\gamma\in\mathcal H$, let $\mathcal L_w^\gamma$ denote the
operator on $\mathcal F_c(\Lambda_+,\mathbb C)$ given by the
corresponding summand in \eqref{eqn:complexification}.
Since $K$ is bounded on $\gamma\Lambda_+$,
\Cref{prop.holderlocal} shows that $w\mapsto\mathcal L_w^\gamma$
is analytic in operator norm.

Using $\|K(x)\|\ll R(x)$ from \Cref{lem.uniformly positive}, choose
a bounded neighborhood $\mathcal U\subset\widetilde{\mathcal U}$
of $0$ so small that
$$
|\langle {\Im}(w),K(x)\rangle_{\mathtt u}|
\leq\frac{\epsilon_0}{2}R(x)
\qquad(w\in\mathcal U,\ x\in\Lambda_+).
$$
By \Cref{prop.holderlocal},
\eqref{eqn:one branch distortion Section15},
\eqref{eqn:recall contraction metric}, and
$h,h^{-1}\in\mathcal F_c(\Lambda_+)$,
$$
\sum_{\gamma\in\mathcal H}
\sup_{w\in\mathcal U}\|\mathcal L_w^\gamma\|
\ll\sum_{\gamma\in\mathcal H}
e^{-(1-\epsilon_0/2)R_\gamma}<\infty,
$$
where the last inequality follows from
\eqref{eqn:real branch summability Section15}.
Thus $\mathcal L_w=\sum_{\gamma\in\mathcal H}\mathcal L_w^\gamma$
converges uniformly in operator norm on $\mathcal U$, and
$w\mapsto\mathcal L_w$ is analytic there.

\smallskip

\emph{The decomposition and (1)--(3).} By \Cref{prop:L0 decomposition}, $1$ is a
simple isolated eigenvalue of $\mathcal L_0$, so we may fix a contour
enclosing $1$ and no other point of $\sigma(\mathcal L_0)$. Analytic 
perturbation theory \cite[Chapter~VII, Section~1]{Kat95} and \cite[Proposition 4.6]{PP_Zeta} apply:
because the unperturbed eigenvalue is simple, no branching occurs, and
one obtains, after shrinking $\cal U$ if necessary, analytic maps $w\mapsto\kappa_w \in \mathbb C$, $w\mapsto\mathcal P_w $, 
and $w\mapsto\mathcal R_w $ on $\cal U$ satisfying
\eqref{eqn:PR decomposition of transfer operator} and the first half of
\eqref{eqn:PR decomposition}, with $\mathcal P_w$ the rank-one
eigenprojection of $\kappa_w$. Hence (1) follows. By the upper semicontinuity of the spectral radius, shrinking $\cal U$ while
keeping the same contour yields the uniform spectral separation
$\rho(\mathcal R_w)<|\kappa_w|-\eta$ on $\cal U$ in (2) for some $\eta > 0$; in particular, $\kappa_w \neq 0$.

We now prove (3), by applying \cite[Lemma~6.12]{CS_local}. Let $w\in\mathcal O := \cal U \cap \fa$. By
\Cref{lem:spectral radius at most one} we have $\rho(\mathcal L_w)\le1$,
and $\kappa_w$ is an eigenvalue of $\mathcal L_w$, so
$|\kappa_w|\le1$. By \Cref{prop:irrationality}, if
 $|\kappa_w|=1$, then $\kappa_w \in \{e^{ia}:a/\pi\notin\mathbb Q\} \cup \{1 \} $, which is
totally disconnected.

To apply \cite[Lemma~6.12]{CS_local}, it remains to show that there is no 
nonconstant analytic curve $t\mapsto w(t)$ in $\mathcal O$ with
$w(0)=0$ and $\kappa_{w(t)} = 1$ for all $t$. If such a curve exists, then by \Cref{prop:lattice function} and
\Cref{lem:periodic obstruction} with $a=0$, we have for each $t$ that 
$\langle w(t),K_n(x)\rangle_{\mathtt u}\in2\pi\mathbb Z$ for every
$n\ge1$ and $x\in\operatorname{Fix}(T^n)$. By the non-arithmeticity in \Cref{thm.denselength}, this implies $w(t) = 0$ for all $t$, which is a contradiction. Hence, we can apply \cite[Lemma~6.12]{CS_local} and have that, after shrinking
$\mathcal U$ once more, $|\kappa_w|<1$ for every
$w\in\mathcal O\smallsetminus\{0\}$.

\smallskip

\emph{The eigenfunction.}
Since $\nu_+(\mathcal P_0\mathbf1)=1$, analyticity of
$\mathcal P_w$ allows us to shrink $\mathcal U$ so that
$\nu_+(\mathcal P_w\mathbf1)\neq0$ there.
Formula \eqref{eqn:Hw} therefore defines an analytic family
$H_w$ with $H_0=\mathbf1$ and $\nu_+(H_w)=1$.
Uniqueness follows because the range of $\mathcal P_w$ is the
one-dimensional eigenspace of $\kappa_w$.

\smallskip

\emph{The uniform bound
\eqref{eqn:uniform remainder powers Section15}.} 
Since $w \mapsto \kappa_w$ is analytic on $\cal U$ with $\kappa_0 = 1$, we can shrink $\cal U$ so that $0 < \abs{\kappa_w} < 2$ for $w \in \cal U$. As $\rho(\cal R_w) < \abs{\kappa_w} - \eta$ for $w \in \cal U$ and $A_w := \kappa_w^{-1} \cal R_w$, there exist constants $0<\vartheta_0<\vartheta<1$ such that for each $w \in \cal U$,
$$
\sigma(A_w)\subset\{\zeta\in\mathbb C:|\zeta|\le\vartheta_0\}.
$$

After shrinking $\mathcal U$, we may assume that
$\overline{\mathcal U}$ is compact and that $A_w$ is analytic
on a neighborhood of $\overline{\mathcal U}$, where the same
spectral bound holds. Set $\mathcal O:=\mathcal U\cap\mathfrak a$. 
The circle $|\zeta|=\vartheta$ is
therefore a resolvent contour common to all these operators, and by
Cauchy's estimates on a neighborhood of $\overline{\mathcal U}$, the quantity
$
M_\alpha:=\sup_{w\in\overline{\mathcal O},\ |\zeta|=\vartheta}
\bigl\|\partial_w^\alpha(\zeta I-A_w)^{-1}\bigr\|
$
is finite for each $\alpha$. Dunford's formula and differentiation
under the integral sign give
$$
\partial_w^\alpha A_w^n
=\frac1{2\pi i}\oint_{|\zeta|=\vartheta}
\zeta^n\,\partial_w^\alpha(\zeta I-A_w)^{-1}\,d\zeta .
$$
Thus $\|\partial_w^\alpha A_w^n\|\leq
\vartheta M_\alpha\vartheta^n$, proving
\eqref{eqn:uniform remainder powers Section15}.
\end{proof}

\subsection{Expansion of the leading eigenvalue}
\label{subsec:expansion} 
Let
$\mathcal O\subset\mathfrak a$ be the neighborhood of $0$ provided by
\Cref{thm:decomposition of transfer operators}, and let
$\mathtt u\in\mathfrak a$ be as in \eqref{eqn:mixing direction}. When
convenient, we write $w\in\mathfrak a$ in the coordinates
$(s,v)\in\mathbb R\times\ker\psi$ determined by $w=s\mathtt u+v$, and we
set
\begin{equation}
\label{eqn:tilde kappa}
\widetilde\kappa(w):=\kappa_w e^{-is\nu_+(R)} = \kappa_w e^{- i \psi(w) \nu_+ (R)},
\qquad w\in\mathcal O .
\end{equation}

\begin{proposition}
\label{cor:expansion of complex eigenvalue around 0}
There is a positive definite symmetric linear operator $S$ on
$\mathfrak a$ such that for all $w$ in a neighborhood of $0$ in $\mathfrak a$,
$$
\widetilde\kappa(w)
=1-\frac12\,\langle Sw,w\rangle_{\mathtt u}
+o\bigl(\|w\|_{\mathtt u}^{2}\bigr).
$$
\end{proposition}

The main point is the positive
definiteness of $-D^2\widetilde\kappa(0)$.
The analogous conclusion in
\cite[Lemma~8(3)]{Pollicott_Sharp_Inv} follows from symmetry
of a homological cocycle. That symmetry is unavailable for
our $\mathfrak a$-valued Busemann cocycle.

Before proving the proposition, we establish a regularity
lemma to handle the unbounded return cocycle.
For $w\in \frak a$, set
$$
\Phi_w(x):=\langle w,K(x)\rangle_{\mathtt u}-\psi(w) \nu_+(R)  \qquad (x\in \La_+).
$$
For every $G\in\mathcal F_c(\Lambda_+,\mathbb C)$ and $j\ge 1$, 
the series
$$
\mathcal L_0(\Phi_w^jG)(x)
:=\frac1{h(x)}\sum_{\gamma\in\mathcal H}
e^{-R(\gamma x)}h(\gamma x)\Phi_w(\gamma x)^jG(\gamma x)
$$
converges absolutely and uniformly on $\Lambda_+$.
Indeed, $|\Phi_w|\ll_w1+R$ by \Cref{lem.uniformly positive}.
Using $(1+u)^je^{-u}\ll_j e^{-(1-\epsilon_0)u}$ for $u\geq1$
and \eqref{eqn:real branch summability Section15}, we obtain
\begin{equation}
\label{eqn:centered moment summability}
\sum_{\gamma\in\mathcal H}\sup_{x\in\Lambda_+}
e^{-R(\gamma x)}|\Phi_w(\gamma x)|^j
\ll_{w,j}
\sum_{\gamma\in\mathcal H}\sup_{x\in\Lambda_+}
e^{-(1-\epsilon_0)R(\gamma x)}<\infty.
\end{equation}
The claimed convergence follows since $G$, $h$, and $h^{-1}$ are bounded.

\begin{lemma}
\label{lem:regularity of centered moments}
For every $w\in\mathfrak a$,
$G\in\mathcal F_c(\Lambda_+,\mathbb C)$, and integer $j\geq1$,
we have $\mathcal L_0(\Phi_w^jG)\in\mathcal F_c(\Lambda_+,\mathbb C)$.
In particular, $\mathcal L_0\Phi_w\in\mathcal F_c(\Lambda_+)$.
\end{lemma} 
\begin{proof}
By \Cref{thm:decomposition of transfer operators}, the map
$t\mapsto e^{-it\psi(w)\nu_+(R)}\mathcal L_{tw}G$ is analytic near
$0$ with values in $\mathcal F_c(\Lambda_+,\mathbb C)$.
Since $|e^{it\Phi_w}|=1$ for real $t$,
\eqref{eqn:centered moment summability} justifies termwise
differentiation of its branch formula, giving
\begin{equation}
\label{eqn:regularized derivatives Section15}
\left.\frac{d^j}{dt^j}\right|_{t=0}
\bigl(e^{-it\psi(w)\nu_+(R)}\mathcal L_{tw}G\bigr)
=i^j\mathcal L_0(\Phi_w^jG)\qquad(j\geq1).
\end{equation}
The left-hand side belongs to
$\mathcal F_c(\Lambda_+,\mathbb C)$ by analyticity.
 \end{proof}

\begin{proof}[Proof of \Cref{cor:expansion of complex eigenvalue around 0}]
For $w\in\mathcal O$,  recall $H_w \in \F_c(\La_+, \mathbb C)$ from 
\eqref{eqn:Hw}. By \Cref{thm:decomposition of transfer operators} the
maps $w\mapsto\kappa_w$ and $w\mapsto H_w$ are analytic on
$\mathcal O$, and $H_0=\mathbf 1$.

By positivity of the branch weights,
\eqref{eqn:invariance of transfer operator and measure}, and monotone
convergence, $\mathcal L_0$ extends to $L^1(\nu_+)$,
given almost everywhere by the same branch formula:
\begin{equation}
\label{contraction}
 \|\mathcal L_0F\|_{L^1(\nu_+)}
 \leq\int_{\Lambda_+}\mathcal L_0|F|\,d\nu_+
 =\|F\|_{L^1(\nu_+)}.
\end{equation}

\emph{The linear term vanishes.}
Differentiating the identity 
$\kappa_w = \int_{\Lambda_+}\mathcal L_w(H_w)\,d\nu_+$ at $w = 0$, it follows from  $\mathcal L_0^*\nu_+=\nu_+$ and $\int_{\La_+} H_w d \nu_+ = 1$ that 
\begin{align*}
D\kappa_0 (\cdot)
&=\int_{\Lambda_+}\bigl(D_w\mathcal L_w|_{w=0}\bigr)(H_0)\,d\nu_+
 +\int_{\Lambda_+}\mathcal L_0\bigl(D_wH_w|_{w=0}\bigr)\,d\nu_+\\
&=\int_{\Lambda_+}\mathcal L_0
 \bigl(i\langle\,\cdot\,,K\rangle_{\mathtt u}H_0\bigr)\,d\nu_+
 +D_w\Bigl|_{w=0}\Bigl(\int_{\Lambda_+}H_w\,d\nu_+\Bigr)\\
&=i\Bigl\langle\,\cdot\,,\int_{\Lambda_+}K\,d\nu_+\Bigr\rangle_{\mathtt u}
 =i\bigl\langle\,\cdot\,,\nu_+(R)\,\mathtt u\bigr\rangle_{\mathtt u},
\end{align*}
the last equality by \eqref{eqn:mixing direction}. Hence
$D\kappa_0(w)=i\nu_+(R) \psi(w)$, and therefore $(D\widetilde\kappa(0))(\cdot)=0$. By
Taylor's theorem,
\begin{equation}
\label{eq:taylor-tildekappa}
\widetilde\kappa(w)
=1+\frac12\,D^2\widetilde\kappa(0)(w,w)
+o\bigl(\|w\|_{\mathtt u}^2\bigr),
\end{equation}
so it suffices to prove that the quadratic form
$$
Q(w):=-D^2\widetilde\kappa(0)(w,w)
$$
is positive definite on $\mathfrak a$.
Writing
$Q(w)=\langle Sw,w\rangle_{\mathtt u}$ for the corresponding positive
definite symmetric operator $S$, \eqref{eq:taylor-tildekappa} gives the
asserted expansion.

\emph{Setup along a ray.} Fix
$w=s\mathtt u+v\in\mathfrak a\smallsetminus\{0\}$ in the coordinate $\R \times \ker \psi$, and write the function $\Phi := \Phi_w$ as in \Cref{lem:regularity of centered moments}.
By \eqref{eqn:mixing direction},
$\int_{\La_+}\Phi\,d\nu_+=0$.
By \Cref{Kk2} and the boundedness of $d\nu_+/d\nu$, there exists $\varepsilon_1 > 0$ such that 
\begin{equation}
\label{eqn:exponential moment}
\int_{\Lambda_+}e^{\epsilon_1|\Phi|}\,d\nu_+<\infty .
\end{equation}

For $t\in\mathbb R$ with $|t|$ small, write $\widetilde\kappa(t):=\widetilde\kappa(tw)$ and
$H_t:=H_{tw}$, and define
\begin{equation}
\label{eq:centered-operator}
\widetilde{\mathcal L}_t(F):=\mathcal L_0\bigl(e^{it\Phi}F\bigr),
\qquad F\in \mathcal F_c(\Lambda_+,\mathbb C).
\end{equation}
Then 
$\widetilde{\mathcal L}_t=e^{-it\nu_+(R)s}\mathcal L_{tw}$, and hence
$\widetilde{\mathcal L}_tH_t=\widetilde\kappa(t)H_t$. By \Cref{lem:regularity of centered moments},
$\mathcal L_0\Phi\in\mathcal F_c(\Lambda_+)$.

\emph{The covariance formula.}
The derivative identities
\eqref{eqn:regularized derivatives Section15} allow us to
differentiate the eigenrelation in $\mathcal F_c(\Lambda_+,\mathbb C)$.
Writing $\dot H_0=\left.\frac{d}{dt}\right|_{t=0}H_t$ and using
$\widetilde\kappa'(0)=0$, we obtain
\begin{equation}
\label{eq:poisson}
(I-\mathcal L_0)\dot H_0=i\mathcal L_0\Phi,
\qquad \nu_+(\dot H_0)=0.
\end{equation}
Both $\mathcal L_0\Phi$ and $\dot H_0$ have mean zero.
The spectral gap of $\mathcal L_0$ therefore gives
\begin{equation}
\label{eq:dotH}
\dot H_0
=i(1-\mathcal R_0)^{-1}\mathcal L_0\Phi
=i\sum_{n\geq1}\mathcal L_0^n\Phi,
\end{equation}
with convergence in $\mathcal F_c(\Lambda_+,\mathbb C)$.
Differentiating the eigenrelation twice and integrating against
$\nu_+$ gives
$$
\widetilde\kappa''(0)
=-\nu_+(\Phi^2)+2i\nu_+(\Phi\dot H_0).
$$
Consequently,
\begin{equation}
\label{eq:D2-formula}
Q(w)=\nu_+(\Phi^2)
+2\sum_{n\geq1}\nu_+(\Phi\,\mathcal L_0^n\Phi).
\end{equation}
The series is absolutely convergent because
$\|\mathcal L_0^n\Phi\|_\infty$ decays exponentially for $n\geq1$
and $\Phi\in L^1(\nu_+)$.

Set
$$
f:=\Phi+(1-\mathcal R_0)^{-1}\mathcal L_0\Phi
=\sum_{k\ge0}\mathcal L_0^k(\Phi).
$$
Then $f$ is real-valued and
$f-\Phi\in\mathcal F_c(\Lambda_+)$. 
Since $\Phi\in L^2(\nu_+)$ by \eqref{eqn:exponential moment} and
$\mathcal L_0f=f-\Phi$ is bounded, we have
$f,\mathcal L_0f\in L^2(\nu_+)$.
Substituting $\Phi=f-\mathcal L_0f$ and
$\sum_{k\geq1}\mathcal L_0^k\Phi=\mathcal L_0f$ into
\eqref{eq:D2-formula}, we obtain
\be 
\label{eq:jensen-form}
\begin{aligned}
Q(w)
&=\int_{\Lambda_+}(f-\mathcal L_0f)^2\,d\nu_+
 +2\int_{\Lambda_+}(f-\mathcal L_0f)(\mathcal L_0f)\,d\nu_+\\
&=\int_{\Lambda_+} f^2-(\mathcal L_0f)^2 \,d\nu_+
 =\int_{\Lambda_+} \mathcal L_0(f^2)-(\mathcal L_0f)^2 \,d\nu_+ .
\end{aligned}
\ee
In the last equality, we used $\L_0^* \nu_+ = \nu_+$.

Moreover, $\mathcal L_0f=f-\Phi\in\mathcal F_c(\Lambda_+)$.
Applying \Cref{lem:regularity of centered moments} to the first
two terms in
$$
f^2=\Phi^2+2\Phi(f-\Phi)+(f-\Phi)^2,
$$
and using the boundedness of $\mathcal L_0$ on
$\mathcal F_c(\Lambda_+)$ for the last term, gives
$\mathcal L_0(f^2)\in\mathcal F_c(\Lambda_+)$.
In particular, the integrand in \eqref{eq:jensen-form} is continuous.

\emph{Positivity.}
 Since $\L_0$ is a positive operator on $\C_b(\La_+)$ and 
$\mathcal L_0\mathbf 1=\mathbf 1$, Jensen's inequality gives
$\mathcal L_0(f^2)\ge(\mathcal L_0f)^2$ pointwise, so $Q(w)\ge0$.
Suppose $Q(w)=0$ for some $w\ne0$. Since $Q$ is a quadratic form and
$\Phi_{tw}=t\Phi_w$, we may replace $w$ by $tw$ for small $t>0$ and so
assume $w\in\mathcal O$. By \eqref{eq:jensen-form}, the full support of $\nu_+$, and the continuity of $\L_0(f^2)- (\L_0 f)^2$ on $\La_+$, we have $\L_0(f^2) = (\L_0 f)^2$ on $\La_+$. Together with the definition of $\L_0$ whose branch weights are all strictly positive, this implies that $f$ is constant on $T^{-1}(x)$ for each $x \in \La_+$.
Fix $\ga \in \cal H$ and define $\varphi : \La_+ \to \R$ by $\varphi(x) := f(\ga x)$. Then $f = \varphi \circ T$ and
$\varphi=\mathcal L_0f\in\mathcal F_c(\Lambda_+)$, and hence $\Phi = \varphi \circ T - \varphi$, i.e.,
\begin{equation}
\label{eq:coboundary}
\langle w,K(\cdot)\rangle_{\mathtt u}-\nu_+(R)s=(\varphi\circ T-\varphi)(\cdot)
\qquad\text{on }\Lambda_+ .
\end{equation}

Set $\zeta:=e^{i\varphi}$, so that $|\zeta|\equiv1$. 
Since $\varphi\in\mathcal F_c(\Lambda_+)$,
we have $\zeta\in\mathcal F_c(\Lambda_+,\mathbb C)$.
Exponentiating \eqref{eq:coboundary} now gives
$$
e^{\,i(\langle w,K\rangle_{\mathtt u}-\nu_+(R)s)}\,\zeta
=\zeta\circ T
\qquad\text{on }\Lambda_+ ,
$$
so $\langle w,K(\cdot)\rangle_{\mathtt u}$ is an $a$-function in the sense of
\Cref{def:a-function}, for the $a\in[0,2\pi)$ with
$\nu_+(R)s\in a+2\pi\mathbb Z$. Hence
$\rho(\mathcal L_w)=1$ by \eqref{eqn:rho equals one}. On the other hand, by \Cref{thm:decomposition of transfer operators}, this is impossible unless $w = 0$.
Therefore $Q(w)>0$
for every $w\ne0$, so $Q$ is positive definite, completing the proof.
\end{proof}

\section{Suspension models and factor maps}
\label{sec.localmixingsymbolic}
We construct suspension models for the translation flow on
$\Omega_\psi$ and the diagonal flow on $\Omega:=\Omega_\rho$.
Write $\widetilde\Omega:=\Lambda_\rho^{(2)}\times\mathfrak a$
and retain the projection $\pi_\psi$ and BMS measures of \Cref{ss:rep}.
The measures have the normalization \eqref{eqn:normalizedBMS},
so $m_\psi(\Omega_\psi)=1$.
Retain the splitting $\mathfrak a=\mathbb R\ut\oplus\ker\psi$
from \eqref{eqn:mixing direction} and define the following functions on $\La_+$:
$$
K^\circ(\cdot):=K(\cdot)-R(\cdot)\ut,\quad
K_n^\circ(\cdot):=\sum_{j=0}^{n-1}(K^\circ\circ T^j)(\cdot)=K_n(\cdot)-R_n(\cdot)\ut,
\quad K_0^\circ :=0
$$
where $n \in \N$.
The translations
$$
\mathcal W_t(\xi,\eta,z):=(\xi,\eta,z+t\ut),\qquad
\mathcal W_t(\xi,\eta,s):=(\xi,\eta,s+t)
$$
on $\widetilde\Omega$ and $\widetilde\Omega_\psi$ descend to the
quotients $\Omega$ and $\Omega_{\psi}$, respectively. 
For fixed $\ww\in\ker\psi$, we study correlations under the maps
$$
q\mapsto qa_{t\ut+\ell(t)\ww},
\qquad \ell(t)=0\ \text{or}\ \sqrt t.
$$

Recall from \Cref{ss:rep} that $\pi_\psi:\Omega\to\Omega_\psi$
is a trivial principal $\ker\psi$-bundle. Fix a continuous section
$\sigma:\Omega_\psi\to\Omega$. It gives the identification
$$
\Omega_\psi\times\ker\psi\to\Omega,\qquad
(q,v)\mapsto\sigma(q)a_v.
$$
By \eqref{mpsi}, for every $F\in C_c(\Omega)$,
\begin{equation}
\label{eqn:BMS product in section coordinates}
\int_\Omega F\,d\mathsf m_\psi
=\int_{\Omega_\psi}\int_{\ker\psi}
F(\sigma(q)a_v)\,d\operatorname{Leb}_{\ker\psi}(v)\,dm_\psi(q).
\end{equation}

\subsection{The contracting skew product and the suspension spaces}
\label{subsection:correlation function on shift space}

Set
$$ \Lambda_-:=U_0=
\Lambda_\rho\smallsetminus
{\mathcal B_{\bar\psi}}\left(\Delta_0,\frac1{2N_0\eta}\right).
$$

\begin{lemma}
\label{lem:shrinking fibers}
The set $\Lambda_-$ is a compact neighborhood of $p_0$ in $\Lambda_\rho$
and satisfies
$\Lambda_-\cap\overline{\Delta_0}=\emptyset$ and
$ \Lambda_b\subset\Lambda_-$.
For every $\gamma\in\mathcal H$ and $y,y'\in\Lambda_-$,
$$
\gamma^{-1}\Lambda_-\subset\Lambda_-,
\qquad
d_{\bar\psi}(\gamma^{-1}y,\gamma^{-1}y')
\leq\lambda d_{\bar\psi}(y,y'),
$$
where $\lambda$ is as in \eqref{eqn:recall contraction metric}.
\end{lemma}

\begin{proof}
The compact neighborhood assertion and disjointness follow from
the definition. Iterating \Cref{lem.LP6.62} gives
$\gamma^{-1}\Lambda_-\subset\Lambda_-$. Since $\sH^{-1}$ is Zariski dense (\Cref{cor.Zdensesemigroup}), a nonempty $\sH^{-1}$-invariant compact set 
$\Lambda_-$ contains the unique $\sH^{-1}$-minimal set $\Lambda_b$
\cite{Benoist1997proprietes}.
The contraction estimate follows from \eqref{eqn.metrics_sph},
\Cref{lem.choice of n_0}(3), and the choice of $\lambda$ in
\eqref{eqn:recall contraction metric}.
\end{proof}

For $(x,x_-),(y,y_-)\in\Lambda_+\times\Lambda_-$ with $x$ and $y$
belonging to a common one-cylinder, define the premetric
$$
d_{\Lambda_+\times\Lambda_-}
\bigl((x,x_-),(y,y_-)\bigr)
:=d_{\bar\psi}(x,y)+d_{\bar\psi}(x_-,y_-).
$$

Let $c \in (0, c_0]$ be any number as in \eqref{fixc}.
Define $\mathcal F_c(\Lambda_-,\mathbb C)$ as in 
\Cref{def:bounded c locally}, imposing the $c$-H\"older condition on
all of $\Lambda_-$.
Similarly, define $\mathcal F_c(\Lambda_+\times\Lambda_-,\mathbb C)$
using $d_{\Lambda_+\times\Lambda_-}$ and imposing the H\"older condition
on each $(\gamma\Lambda_+)\times\Lambda_-$, $\gamma\in\mathcal H$.
In both spaces, use the norm $\|\cdot\|_c=\|\cdot\|_\infty+|\cdot|_c$.

Recall that $Tx=\gamma^{-1}x$ for $x\in\gamma\Lambda_+$.
Define
$$
\widehat T(x,x_-):=(Tx,\gamma^{-1}x_-)=(\gamma^{-1}x, \gamma^{-1} x_-)
\qquad(x\in\gamma\Lambda_+,\ x_-\in\Lambda_-).
$$
\begin{proposition}
\label{lem:measure on product space}
There is a unique $\widehat T$-invariant probability measure $\nu^\diamond$ on
$\Lambda_+\times\Lambda_-$ whose projection to $\Lambda_+$ is $\nu_+$, and it is ergodic.
Its disintegration $\{\nu_{+,x}:x\in\Lambda_+\}$ over $\nu_+$ can be chosen to be defined
at every $x\in\Lambda_+$, and there is $C_c>0$ with the following two properties. 
\begin{enumerate}
\item For every
$f\in\mathcal F_c(\Lambda_+\times\Lambda_-,\mathbb{C})$,
\begin{equation}
\label{eqn:stable disintegration Holder}
\|\bar f\|_c\leq C_c\|f\|_c,
\qquad
\bar f(x):=
\int_{\Lambda_-}f(x,x_-)\,d\nu_{+,x}(x_-).
\end{equation}
\item For $\varphi\in\mathcal F_c(\Lambda_-,\mathbb{C})$ and $x, y \in \La_+$ in a common
one-cylinder,
\begin{equation}
\label{eqn:conditional-kernel-Holder}
\left|
\int_{\Lambda_-}\varphi\,d\nu_{+,x}
-\int_{\Lambda_-}\varphi\,d\nu_{+,y}
\right|
\leq C_c\|\varphi\|_{\infty}d_{\bar\psi}(x,y)^c.
\end{equation}
\end{enumerate}
\end{proposition}
\begin{proof}
 We construct the conditional measures at every
$x\in\Lambda_+$.
For $\gamma\in\mathcal H^n$, put
$$
p_\gamma^{(n)}(x)
:=
e^{-R_n(\gamma x)}{h(\gamma x)}{h(x)}^{-1}
$$
where $h$ is given by \eqref{eqn:hFLa}. 
By \eqref{eqn:invariance of transfer operator and measure}, 
$\sum_{\gamma\in\mathcal H^n}p_\gamma^{(n)}(x)=1$, and for every $\gamma_1\in\mathcal H^\ell$ and
$\gamma_2\in\mathcal H^n$, 
\begin{equation}
\label{eqn:Markov weight factorization}
p_{\gamma_1\gamma_2}^{(n + \ell)}(x)
=p_{\gamma_2}^{(n)}(x)p_{\gamma_1}^{(\ell)}(\gamma_2 x).
\end{equation}

Fix $q_-\in\Lambda_-$ and define the probability measure
$$
\nu_{+,x}^{(n)}
:=
\sum_{\gamma\in\mathcal H^n}
p_\gamma^{(n)}(x)\,\delta_{\gamma^{-1}q_-}
$$
where $\delta_{\ga^{-1} q_-}$ is the Dirac mass at $\ga^{-1} q_-$.
Since $(\gamma_1\gamma_2)^{-1}q_-, \gamma_2^{-1}q_-\in \gamma_2^{-1}\Lambda_-$, by
\Cref{lem:shrinking fibers}, their distance is at most
$\lambda^n\operatorname{diam}_{d_{\bar\psi}}(\Lambda_-)$.
For every $\varphi\in C(\Lambda_-,\mathbb C)$,
\eqref{eqn:Markov weight factorization},
\Cref{lem:Riemannian-visual-Holder}, and uniform continuity show
that $\nu_{+,x}^{(n)}(\varphi)$ is Cauchy uniformly in $x$.
Thus $\nu_{+,x}^{(n)}$ converges weakly to a Borel probability measure
$\nu_{+,x}$.
 By \eqref{eqn:Markov weight factorization}, we have
$\nu_{+,x}^{(n+k)}
=\sum_{\gamma\in\mathcal H^k}p_\gamma^{(k)}(x)
(\gamma^{-1})_*\nu_{+,\gamma x}^{(n)}$.
Letting $n\to\infty$ gives that for every $k\geq1$,
\begin{equation}
\label{eqn:stable kernel recursion}
\nu_{+,x}
=
\sum_{\gamma\in\mathcal H^k}
p_\gamma^{(k)}(x)(\gamma^{-1})_*\nu_{+,\gamma x}.
\end{equation}
Define a Borel probability measure on $\La_+\times \La_-$:
$$
 d\nu^\diamond(x,x_-):=d\nu_+(x)\,d\nu_{+,x}(x_-).
$$ We first show the $\widehat T$-invariance of $\nu^{\diamond}$.
For a bounded Borel function $F : \La_+ \times \La_- \to \mathbb C$, we have from $\L_0^* \nu_+ = \nu_+$, as extended in \eqref{contraction}, and \eqref{eqn:stable kernel recursion} that
\begin{align*}
 \int F\circ\widehat T\,d\nu^\diamond
 &=\int_{\Lambda_+}\sum_{\gamma\in\mathcal H}p_\gamma^{(1)}(x)
   \int_{\Lambda_-}F(x,\gamma^{-1}x_-)\,d\nu_{+,\gamma x}(x_-)
   \,d\nu_+(x)=\int F\,d\nu^\diamond.
\end{align*} 
It is clear from the definition that $\nu^{\diamond}$ projects to $\nu_{+}$ on $\La_+$.

To prove uniqueness, let $\eta$ be any $\widehat T$-invariant
probability measure projecting to $\nu_+$. For a bounded Borel
function $a:\Lambda_+\to\mathbb C$ and $\varphi\in \C(\Lambda_-)$,
put $F(x,x_-):=a(x)\varphi(x_-)$. Since $\eta$ is $\widehat T$-invariant and $\varphi$ is uniformly continuous,  \Cref{lem:shrinking fibers} gives
$$
 \int F\,d\eta
 =\lim_{n\to\infty}
   \int_{\Lambda_+}F\bigl(\widehat T^n(x,q_-)\bigr)\,d\nu_+(x).
$$
The right-hand side is independent of $\eta$. Since these product
functions determine a probability measure, the invariant lift is
unique.
Finally, let $E$ be an invariant set with $\nu^\diamond(E)>0$.
The normalized restriction
$\nu^\diamond|_E/\nu^\diamond(E)$ projects to a $T$-invariant
probability measure absolutely continuous with respect to $\nu_+$.
By \Cref{cor: T invariant measure}, this projection equals $\nu_+$.
Uniqueness of the invariant lift gives
$\nu^\diamond|_E/\nu^\diamond(E)=\nu^\diamond$, hence
$\nu^\diamond(E)=1$. Thus $\nu^\diamond$ is ergodic.

To prove (2), let $x,y\in\Lambda_+$ belong to a common one-cylinder.
Applying the estimates for $I_1,I_2,I_3$ in the proof of
\Cref{lem:L0 action on B} with $w=0$ and $f=\mathbf1$, we obtain
$$
\sum_{\gamma\in\mathcal H^n}
|p_\gamma^{(n)}(x)-p_\gamma^{(n)}(y)|
\leq C_c d_{\bar\psi}(x,y)^c
\qquad(n\geq1);
$$
hence $\left|\nu_{+,x}^{(n)}(\varphi)-\nu_{+,y}^{(n)}(\varphi)\right|
\leq C_c \|\varphi\|_{\infty} d_{\bar\psi}(x,y)^c$.
Letting $n\to\infty$ proves (2).

We now prove (1). 
For $f\in\mathcal F_c(\Lambda_+\times\Lambda_-,\mathbb C)$, define
$\bar f$ by \eqref{eqn:stable disintegration Holder}. Then  $\|\bar f\|_\infty\leq\|f\|_\infty$ is clear. For
$x$ and $y$ in a common one-cylinder, write
\begin{align*}
\bar f(x)-\bar f(y)
&=\int_{\Lambda_-}\bigl(f(x,x_-)-f(y,x_-)\bigr)\,d\nu_{+,x}(x_-)\\
&\quad+\left(\int_{\Lambda_-}f(y,x_-)\,d\nu_{+,x}(x_-)
-\int_{\Lambda_-}f(y,x_-)\,d\nu_{+,y}(x_-)\right).\nonumber
\end{align*}
The first integral is bounded in absolute value by
$|f|_cd_{\bar\psi}(x,y)^c$. The remaining difference is bounded by $C_c\|f\|_\infty d_{\bar\psi}(x,y)^c$, by applying (2) to  $f(y,\cdot)$. Together with $\|\bar f\|_\infty\leq\|f\|_\infty$, this proves (1)
after enlarging $C_c$.

\end{proof}

Let
$$
\widetilde\Lambda_{\mathfrak a}
:=\Lambda_+\times\Lambda_-\times\mathbb R\times\ker\psi,
\qquad
\widetilde\Lambda_\psi
:=\Lambda_+\times\Lambda_-\times\mathbb R.
$$
Define the skew-product maps
$\zeta\colon\widetilde\Lambda_{\mathfrak a}\to\widetilde\Lambda_{\mathfrak a}$ and
$\zeta_\psi\colon\widetilde\Lambda_\psi\to\widetilde\Lambda_\psi$ by
\begin{align}
\label{eqn:zeta action}
&\zeta(x,x_-,s,v)
:=
(\widehat T(x,x_-),s-R(x),v-K^\circ(x)),\\
&\zeta_\psi(x,x_-,s)
:=
(\widehat T(x,x_-),s-R(x)).
\nonumber
\end{align}

Let 
$$
\begin{aligned}
&\Lambda_{\mathfrak a}
:=
\{(x,x_-,s,v)\in\widetilde\Lambda_{\mathfrak a}: 0\leq s<R(x)\},\\
& \Lambda_\psi
:=
\{(x,x_-,s)\in\widetilde\Lambda_\psi: 0\leq s<R(x)\}.
\end{aligned}
$$
For $t\geq0$, define the semiflow $T_t$ on $\Lambda_{\mathfrak a}$ by
\begin{equation}
\label{eqn:defn T flow}
T_t(x,x_-,s,v)
:=\zeta^{n}(x,x_-,s+t,v)=
\bigl(
\widehat T^n(x,x_-),
s+t-R_n(x),
v-K_n^\circ(x)
\bigr),
\end{equation}
where $n=n(t,x,s)$ is the unique nonnegative integer such that
$R_n(x)\leq s+t<R_{n+1}(x)$.
Such an integer exists because $R\geq1$ by
\Cref{lem.uniformly positive}.
This descends to the semiflow on $\La_\psi$ via the projection $
\pi\colon\Lambda_{\mathfrak a}\to\Lambda_\psi$, which we also denote by $T_t$.
For $v\in\ker\psi$, define
$$
(x,x_-,s,v_0)a_v:=(x,x_-,s,v_0+v).
$$
These translations commute with $T_t$. On the product  spaces, define
\be
\label{eqn:def M+}
d\widetilde{\mathtt M}_\psi
 := 
\frac{1}{\nu_+(R)} d\nu^\diamond\,ds\,d\operatorname{Leb}_{\ker\psi}(v), \qquad
d\widetilde M_\psi
:= \frac{1}{\nu_+(R)}  d\nu^\diamond\,ds.
\ee
 Since $\nu^{\diamond}$ is $\widehat T$-invariant (\Cref{lem:measure on product space}), $\widetilde{\mathtt M}_{\psi}$ and $\widetilde M_{\psi}$ are invariant under $\zeta$ and $\zeta_\psi$,
respectively.  Consider the following $T_t$-invariant measures for $t\ge 0$:
$$
\mathtt M_{\psi} := \widetilde{\mathtt M}_{\psi} |_{\La_{\fa}}, \qquad M_{\psi} = \widetilde{M}_{\psi}|_{\La_{\psi}}.
$$

\begin{lemma}
\label{Mpsi_finite}
The measure $M_\psi$ is an ergodic probability measure for $\{T_t:t\geq0\}$.
\end{lemma}
\begin{proof}
Since $\nu^\diamond$ projects to $\nu_+$,
$M_\psi(\Lambda_\psi)
=\frac{1}{\nu_+(R)}
\int_{\Lambda_+\times\Lambda_-}R(x)\,d\nu^\diamond(x,x_-)=1$. Every measurable set invariant under all $T_t$, $t\geq0$, is the suspension of a $\widehat T$-invariant set.
Thus ergodicity follows from \Cref{lem:measure on product space}.
\end{proof}

\subsection{The factor maps}
To relate the symbolic suspensions to the geometric flows on $\Omega$ and $\Omega_\psi$,
define 
\begin{align}
\label{eqn:LaUpsi}
&\widetilde\Phi:
\widetilde \Lambda_{\mathfrak a}\to\widetilde\Omega, & \widetilde\Phi(x,x_-,s,v):=
(x,x_-,s\ut+v+2\mathcal G(x,p_0)),\\
\label{eqn:Phi psi lift}
&\widetilde\Phi_\psi:
\widetilde \Lambda_\psi \to\widetilde\Omega_\psi,
&\widetilde\Phi_\psi(x,x_-,s)
:=
(x,x_-,s+2\psi(\mathcal G(x,p_0))).
\end{align}
Restricting to the suspension domains and composing
 with the projections
$\widetilde\Omega\to\Omega$ and $\widetilde\Omega_\psi\to\Omega_\psi$
yields  continuous maps $$
\Phi:\Lambda_{\mathfrak a}\to\Omega,
\qquad
\Phi_\psi:\Lambda_\psi\to\Omega_\psi.
$$
By \eqref{eqn:telescoping-return-cocycle} and
\eqref{eqn:zeta action}, these maps satisfy 
\begin{equation}
\label{eqn:factor intertwining identities}
\Phi\circ T_t=\mathcal W_t\circ\Phi,
\qquad
\Phi_\psi\circ T_t=\mathcal W_t\circ\Phi_\psi,
\qquad
\Phi(za_v)=\Phi(z)a_v.
\end{equation}

In particular,
$$
\Phi\bigl((T_tz)a_{\ell(t)\ww}\bigr)
=\Phi(z)a_{t\ut+\ell(t)\ww}.
$$

\begin{proposition}
\label{prop:factor map R}
We have $(\Phi_\psi)_*M_\psi=m_\psi$. 
\end{proposition}

We first record two elementary measure-theoretic facts needed in the proof.

\begin{lemma}
\label{lem:bijection}
There is a Borel set $U\subset\Lambda_\psi$ such that
$m_\psi(\Phi_\psi(U))>0$  and
$
\Phi_\psi
$ is injective on $U$.
\end{lemma}

\begin{proof} 
The map $\widetilde\Phi_\psi$ 
 is an embedding by \eqref{eqn:Phi psi lift}. Since $\inf_{x \in \La_+} R(x) > 0$ by 
\Cref{lem.uniformly positive}, we have $\widetilde m_{\psi}( \widetilde\Phi_\psi(\La_{\psi})) > 0$.
Since $\Gamma_\rho$ is torsion-free, there exists an open $O \subset \widetilde \Omega_{\psi}$ on which $\widetilde \Omega_{\psi} \to \Omega_{\psi}$ is injective and $\widetilde m_\psi(O \cap \widetilde \Phi_{\psi}(\La_\psi))>0 $. Setting $
 U:= \widetilde\Phi_\psi^{-1}(O) \cap \La_{\psi}$  finishes the proof. 
\end{proof}

Here stable leaves are taken in $\Omega_\psi$, rather than the
auxiliary flow $\Omega_{\bar\psi}$ of \Cref{sec.recur}: two points lie on
the same stable leaf if they admit lifts $(x,x_-,s)$ and $(x,y_-,s)$
with the same positive endpoint and $\mathbb R$-coordinate.
By \cite[Proposition~8.6]{KOW_ergodic}, there exists a metric $d$ on $\Omega_{\psi}$ compatible with the topology of $\Omega_\psi$ such that 
\begin{equation}
\label{eqn:stable convergence base flow}
d(\mathcal W_tq,\mathcal W_tq')\to0
\qquad(t\to+\infty).
\end{equation}
whenever $q, q' \in \Omega_{\psi}$ lie on the same stable leaf.

\begin{lemma}
\label{lem:same stable leaf}
Let $Q\subset\Lambda_\psi$ be $M_\psi$-conull and
$Q'\subset\Omega_\psi$ be $m_\psi$-conull.
Then $\Phi_\psi(Q)$ and $Q'$ meet a common stable leaf of $\Omega_\psi$.
\end{lemma}

\begin{proof} 
Let $U$ be as in \Cref{lem:bijection} and set $V:=\Phi_\psi(U)$. 
Also set  $m:=((\Phi_\psi|_U)^{-1})_*(m_\psi|_V)$.
For $p(x,x_-,s):=(x,s)$, the local product formula for BMS measure gives
 $p_* m\ll d\nu ds$. Since $\nu_+$ and $\nu|_{\La_+}$ are equivalent, it follows from \Cref{lem:measure on product space} that
$p_*m\ll p_*M_\psi$.
Since $Q$ is $M_\psi$-conull, by \Cref{lem:measure on product space},
$\nu_{+, x} \left( \{x_- : (x, x_-, s ) \in Q \} \right) = 1$ for $p_* M_{\psi}$-a.e. $(x, s)$  and hence for $p_*m$-a.e. $(x,s)$.
Since $\Phi_{\psi}^{-1}(Q' \cap V)$ is $m$-conull, there exists $(x, x_-', s) \in \Phi_{\psi}^{-1}(Q') \cap U$ such that  
$\nu_{+, x} \left( \{x_- : (x, x_-, s ) \in Q \} \right) = 1.$
Choose $x_-$ with $(x,x_-,s)\in Q$.
By \eqref{eqn:Phi psi lift},
$\Phi_\psi(x,x_-,s)$ and
$\Phi_\psi(x,x_-',s)\in Q'$ lie on the same stable leaf in $\Omega_{\psi}$.
\end{proof}

\begin{proof}[Proof of \Cref{prop:factor map R}]
Let $f\in\mathcal C_c(\Omega_\psi)$.  By \Cref{Mpsi_finite}, \Cref{koe}, and the Birkhoff ergodic theorem, there exist  conull
sets $Q'\subset\Omega_\psi$ and $Q\subset\Lambda_\psi$ on which,
respectively,
$$\lim_{T\to\infty}\frac1T\int_0^T
f(\mathcal W_tq)\,dt
=m_\psi(f), \quad \lim_{T\to\infty}\frac1T\int_0^T
f(\Phi_\psi(T_tz))\,dt
=M_\psi(f\circ\Phi_\psi).
$$
By \Cref{lem:same stable leaf}, choose $z\in Q$ and $q\in Q'$
such that $\Phi_\psi(z)$ and $q$ lie on the same stable leaf.
Then by \eqref{eqn:factor intertwining identities}, the uniform continuity of $f$, and \eqref{eqn:stable convergence base flow}, we have 
$$
\frac1T\int_0^T
f(\mathcal W_tq)-f(\mathcal W_t\Phi_\psi(z)) \,dt
\to0.
$$
It follows that
$m_\psi(f)=M_\psi(f\circ\Phi_\psi).$
Since this holds for every $f\in\mathcal C_c(\Omega_\psi)$,  $(\Phi_\psi)_*M_\psi = m_\psi$. 
\end{proof}

\begin{theorem}
\label{cor:factor map} We have $\Phi_*\mathtt M_\psi=\mathsf m_\psi.$
Moreover, for every compact set $Q\subset\Omega$ and $a\geq0$, there is $C_{Q,a}>0$ such that
$$
\|s\ut+v\|\leq C_{Q,a}R(x)+C_{Q,a}
$$
whenever $(x,x_-,s,v)\in\widetilde\Lambda_{\mathfrak a}$,
$-a\leq s<R(x)$, and $[\widetilde\Phi(x,x_-,s,v)]\in Q$.
Here $[\cdot]$ denotes projection to $\Omega$.
\end{theorem}

\begin{proof} In the bundle trivialization determined by the section $\sigma : \Omega_{\psi} \to \Omega$, there is
a unique continuous map $\tau_\sigma:\Lambda_\psi\to\ker\psi$
such that $\Phi(z,0)=\sigma(\Phi_\psi(z))a_{\tau_\sigma(z)}$.  
\Cref{prop:factor map R} and
\eqref{eqn:BMS product in section coordinates} give,
for $F\in\mathcal C_c(\Omega)$,
$$\begin{aligned}
\int_{\Lambda_{\mathfrak a}}F(\Phi(z,v))
\,d\mathtt M_\psi(z,v)
=
\int_{\Lambda_\psi}\int_{\ker\psi}
F\left(\sigma(\Phi_\psi(z))a_v\right)
\,dv\,dM_\psi(z)\\ =
\int_{\Omega_\psi}\int_{\ker\psi}
F\left(\sigma(q)a_v\right)
\,dv \,dm_\psi(q)
= \mathsf m_\psi(F),
\end{aligned}$$
where $dv= d\Leb_{\ker \psi}(v)$. This proves
$\Phi_*\mathtt M_\psi=\mathsf m_\psi$.
We now prove the remaining claim. Fix a compact subset $Q \subset \Omega$ and $a \ge 0$. We fix a compact subset $\widetilde Q \subset \widetilde \Omega$ whose image under the projection to $\Omega$ contains $Q$. Let $(x, x_-, s, v) \in \widetilde \La_{\fa}$ be such that $-a \le s < R(x)$ and $\left[ \widetilde \Phi (x, x_-, s, v) \right] \in Q$. Then there exists $\ga \in \Ga_{\rho}$ so that
$$
\gamma\widetilde\Phi(x,x_-,s,v) = (\ga x, \ga x_-, s \mathtt u + v + 2 \cal G(x, p_0) + \beta_{x}(\ga^{-1}, e)) \in \widetilde Q.
$$
Since $\cal G(x, p_0)$ is uniformly bounded by \Cref{lem.gromovproductnearest} and $\widetilde Q$ is compact, this implies that $s \mathtt u + v + \beta_{x}(\ga^{-1}, e)$ is uniformly bounded, depending only on $Q$.

By the proof of
\cite[Proposition~9.7]{KOW_indicators} and compactness of $Q$, either
$$
\beta_x(\gamma^{-1},e)=\mu(\gamma)+O_Q(1)
\quad\text{or}\quad
\beta_x(\gamma^{-1},e)=-\mu(\gamma^{-1})+O_Q(1).
$$

By \eqref{lb}, 
$\psi>0$ on $\mathcal L_\rho\smallsetminus\{0\}$, so
$$
\norm{s \mathtt u + v} \approx_{Q} \norm{\beta_{x}(\ga^{-1}, e)}  \ll_Q 1+ |\psi(\beta_{x}(\ga^{-1}, e))| \approx_Q  |s|  \le a + R(x).
$$
This completes the proof.
\end{proof}

\section{BMS local mixing and the correlation expansion}
\label{sec.local mixing}
In this section, we prove local mixing for
$(\Omega,\mathcal W_t,\mathsf m_\psi)$.
All spaces of smooth functions in this section are understood
to consist of complex-valued functions.
The factor map
$\Phi:(\Lambda_{\mathfrak a},T_t,\mathtt M_\psi)
\to(\Omega,\mathcal W_t,\mathsf m_\psi)$
from \Cref{cor:factor map}
allows us to express correlations on $\Omega$ in terms of the suspension
flow $T_t$ on $\Lambda_{\mathfrak a}$.
The principal technical difficulty is that $\Phi$ involves
an unbounded $\fa$-coordinate: the pullback of a compactly supported
function on $\Omega$ need not be compactly supported on
$\La_\fa$. 
Set 
$$
\begin{gathered}
r:=\operatorname{rank}_{\mathbb R}G=\dim\mathfrak a, \qquad
c_R:=\left(\int_{\Lambda_+}R\,d\nu_+\right)^{1/2}.\end{gathered} 
$$
Throughout this section, we omit the subscripts in $\langle\cdot,\cdot\rangle_\ut$ and $\|\cdot\|_\ut$. We use the orthogonal splitting $\mathfrak a=\mathbb R \mathtt u \oplus\ker\psi$,
and write elements of $\mathfrak a$ as pairs $(s,v)$ accordingly.
Fix an orthonormal basis for $\langle\cdot,\cdot\rangle$ containing $\mathtt u$.
Let $S$ be the matrix in this basis of the positive definite symmetric
operator in \Cref{cor:expansion of complex eigenvalue around 0}.

Write $dz:=d\operatorname{Leb}_{\mathfrak a}(z)=ds\,d\Leb_{\ker \psi}(v)$, as in \Cref{subsection:PSBMS}. 
For $F\in L^1( dz;\mathbb C)$, our Fourier convention is 
$$
\widehat F(\xi):=\int_{\mathfrak a}
e^{i\langle\xi,z\rangle} F(z)\,dz.
$$
Frequency integrals use the dual Lebesgue measure $d\xi$, characterized
by the inversion formula
$$
F(z)=\frac1{(2\pi)^r}\int_{\mathfrak a}
e^{-i\langle\xi,z\rangle}\widehat F(\xi)\,d\xi
\qquad(F\in C_c^\infty(\mathfrak a)).
$$

Let $\mathcal G$ be the centered Gaussian probability density
on $\mathfrak a$ with covariance $S$, defined by
$$
\mathcal G(z):=
\frac
1{\int_{\mathfrak a} e^{-\frac12 y^tS^{-1}y}\,dy} {e^{-\frac12 z^tS^{-1}z}},
\qquad z\in\mathfrak a.
$$
Then
$\int_{\mathfrak a}\mathcal G(z)\,dz=1$,
$\int_{\mathfrak a}z\,\mathcal G(z)\,dz=0$, $\int_{\mathfrak a}zz^t\,\mathcal G(z)\,dz=S$
and its Fourier transform is
$$
\widehat{\mathcal G}(\xi)=e^{-\frac12\xi^tS\xi},
\qquad \xi\in\mathfrak a.
$$

We sometimes write $\mathcal G(s,v):=\mathcal G(s\mathtt u+v)$
for $s\in\mathbb R$ and $v\in\ker\psi$.
Fix a basis $X_1,\ldots,X_{\dim G}$ of left-invariant vector fields
on $G$, regarded also as vector fields on $\Gamma_\rho\backslash G$.
For  $f\in\mathcal C_c^\infty(\Gamma_\rho\backslash G)$, $k\in\mathbb Z_{\geq0}$ and $\beta\in(0,1]$, set
$$
\|f\|_{k,\beta}
:=
\max_{|\alpha|\leq k}
\left(\|X^\alpha f\|_\infty+|X^\alpha f|_\beta\right),
$$
where
$X^\alpha:=X_1^{\alpha_1}\cdots X_{\dim G}^{\alpha_{\dim G}}$ and $|g|_\beta
:=
\sup_{x\neq y}
\frac{|g(x)-g(y)|}{d_{\mathrm{Riem}}(x,y)^\beta}$. Here $d_{\mathrm{Riem}}$ is the quotient metric induced by a fixed
left $G$-invariant and right $K$-invariant Riemannian metric on $G$. We also write $d_{\mathrm{Riem}}$ for the induced quotient metric on $\Gamma_\rho\backslash G/M$.

For a compact subset $\mathcal K \subset \Ga_{\rho} \ba G$, 
denote by $\mathcal C_{\mathcal K}^{\infty}(\Gamma_\rho\backslash G)^M$ the space of smooth $M$-invariant functions whose supports are contained in $\mathcal K$.
For simplicity, we write $\mathcal C_{\mathcal K}^{\infty}:=\mathcal C_{\mathcal K}^{\infty}(\Gamma_\rho\backslash G)^M$.

\begin{theorem}[BMS local mixing]
\label{bmslocal}
The following assertions hold.
\begin{enumerate}
\item There exist bilinear forms  
$\mathfrak c_j:
\mathcal C_c^\infty(\Gamma_\rho\backslash G)^M
\times
\mathcal C_c^\infty(\Gamma_\rho\backslash G)^M
\to\mathbb C$, $ j\geq0$,
such that, for every compact set
$\mathcal K\subset\Gamma_\rho\backslash G$ and
$L\in\mathbb Z_{\geq0}$, there exist $k\in\mathbb N$ and
$\beta\in(0,1)$ for which
$$
|\mathfrak c_j(f_1,f_2)|
\ll_{\mathcal K,L}
\|f_1\|_{k,\beta}\|f_2\|_{k,\beta}
\qquad(0\leq j\leq L)
$$
and the following expansion holds for all
$f_1,f_2\in\mathcal C_{\mathcal K}^\infty$ as $t\to\infty$:
$$
\begin{aligned}
&t^{\frac{r-1}{2}}
\int_{\Gamma_\rho\backslash G}
f_1(xa_{t \mathtt u})f_2(x)\,d\mathsf m_\psi(x)=
\sum_{j=0}^L\mathfrak c_j(f_1,f_2)t^{-j}
+
O_{\mathcal K,L}
\left(
\|f_1\|_{k,\beta}\|f_2\|_{k,\beta}t^{-(L+1)}
\right).
\end{aligned}
$$
Moreover,
$$\mathfrak c_0(f_1,f_2)
=
c_R^r\left(
\int_{\mathbb R}\mathcal G(c_Rs\mathtt u)\,ds
\right)
\mathsf m_\psi(f_1)\mathsf m_\psi(f_2).
$$

\item
There exists $a>0$ such that,
for every  compact set
$\mathcal K\subset\Gamma_\rho\backslash G$,
there exist $k \in \N$,  $\beta\in(0,1)$, and 
$C,T>0$ with the following property.
For all $f_1, f_2\in\mathcal C_{\mathcal K}^\infty$, $t\geq T$, and 
$\ww\in\ker\psi$,
$$
t^{\frac{r-1}{2}}
\int_{\Gamma_\rho\backslash G}
f_1(xa_{t \mathtt u +\sqrt t\,\ww})f_2(x)\,d\mathsf m_\psi(x)
=\mathfrak c_{\ww}\mathsf m_\psi(f_1)\mathsf m_\psi(f_2)
+\mathcal E_{f_1,f_2}(t,\ww),
$$
where 
$\mathfrak c_{\ww}
=
c_R^r\int_{\mathbb R}
\mathcal G\bigl(c_R(s\mathtt u+\ww)\bigr)\,ds$
and
\begin{equation}
\label{eqn:global weighted error}
|\mathcal E_{f_1,f_2}(t,\ww)|
\leq
C\|f_1\|_{k,\beta}\|f_2\|_{k,\beta}\,
t^{-\frac12}
\exp\bigl(-a \cdot \min\{\|\ww\|^2,\|\ww\|\}\bigr).
\end{equation}
\end{enumerate}
\end{theorem}
\begin{remark}[More general transverse displacements]
Let $\ell:[1,\infty)\to(0,\infty)$ be any function.
Since the constants in \Cref{bmslocal}(2) are independent
of $\ww$, we may apply that theorem with $\ww$ replaced by
$\ell(t)\ww/\sqrt t$. This gives, for all sufficiently large $t$,
$$
\begin{aligned}
&t^{\frac{r-1}{2}}
\int_{\Gamma_\rho\backslash G}
f_1(xa_{t\mathtt u+\ell(t)\ww})f_2(x)\,
d\mathsf m_\psi(x)=
\mathfrak c_{\ell(t)\ww/\sqrt t}\,
\mathsf m_\psi(f_1)\mathsf m_\psi(f_2)
+\mathcal E_{f_1,f_2}
\left(t,\frac{\ell(t)}{\sqrt t}\ww\right),
\end{aligned}
$$
with the remainder bounded by
\eqref{eqn:global weighted error} evaluated at
$\ell(t)\ww/\sqrt t$.
In particular, if $\ell(t)^2/t\to q\in[0,\infty)$, then
$$
\lim_{t\to\infty}
t^{\frac{r-1}{2}}
\int_{\Gamma_\rho\backslash G}
f_1(xa_{t\mathtt u+\ell(t)\ww})f_2(x)\,
d\mathsf m_\psi(x)
=
\mathfrak c_{\sqrt q\,\ww}\,
\mathsf m_\psi(f_1)\mathsf m_\psi(f_2),
$$
by continuity of $v\mapsto \mathfrak c_v$.
If $\ell(t)^2/t\to\infty$, the limit is zero for
every $\ww\neq0$.
This is analogous to the more general transverse local
mixing statement in \cite[Theorem~1.3]{CS_local}.
\end{remark}

We first prepare the estimates needed for the proof of \Cref{bmslocal}.
After introducing the pullback functions and fixing the parameters,
we establish return-time tail bounds and a rough transverse correlation
bound in \Cref{subsec:estimatesonthereturnmap}.
The transfer-operator decay in \Cref{subsec:transfer-decay}, together
with the Fourier and stable-projection estimates in
\Cref{subsubsection:auxiliary estimates}, controls frequencies away from
zero and reduces the remaining contribution to the expanding map.
We then obtain the limiting Taylor coefficients in
\Cref{subsec:normalized-expectations} and the Gaussian sum estimates in
\Cref{subsec:uniformrenewalsums}.
The proof of \Cref{bmslocal} is given in \Cref{subsec:proof-bmslocal}.

\subsection{Pullbacks of compactly supported functions}
\label{subsec:pullbacks}

For the remainder of this section we fix 
$$\text{ a compact subset }\mathcal K\subset\Gamma_\rho\backslash G,\; L\in\mathbb Z_{\geq0}, \;\text{and}\; f_1,f_2\in C_{\mathcal K}^\infty$$
as in \Cref{bmslocal}.
Unless otherwise indicated, all implicit constants depend only on
$\mathcal K$, $L$. 

In the following, we abuse the notation and write $\Phi : \widetilde \La_{\fa} \to \Omega$ for the composition of $\widetilde \Phi : \widetilde \La_{\fa} \to \widetilde \Omega$ and the projection $\widetilde \Omega \to \Omega$. We pull $f_1$ and $f_2$ back to the suspension $\widetilde \La_{\fa}$ and replace the second
pullback by a function smooth in the $\mathfrak a$-variable.
More precisely, define $F_i : \widetilde \La_{\fa} \to \mathbb C$, $i = 1, 2$, by 
\begin{equation}
\label{eqn:Fi}
F_i(x,x_-,s,v)
:=f_i\bigl(\Phi(x,x_-,s, v)
\bigr)\mathbf 1_{[0,R(x))}(s).
\end{equation}
Fix a function  $\zeta_0\in\mathcal C^\infty(\mathbb R,[0,1])$ such that
$\zeta_0(s)=0$ for $s\leq- \inf_{x \in \La_+} \frac{R(x)}{10}$ and
$\zeta_0(s)=1$ for $s\geq0$. Define $F_2^{\sharp} : \widetilde \La_{\fa} \to \mathbb C$ by 
\begin{equation}
\label{eqn:F2 sharp}
F_2^\sharp(x,x_-,s,v)
:=\zeta_0(s)f_2\bigl( \Phi(x,x_-,s,  v)
\bigr)\bigl(1-\zeta_0(s-R(x))\bigr).
\end{equation}
Then $F_2^\sharp$ is smooth in the $\mathfrak a$-variable,
is nonzero only if 
$-\inf_{x \in \La_+} \frac{R(x)}{10}\leq s\leq R(x)$, and satisfies
\begin{equation}
\label{eqn:F2 decomposition}
F_2=F_2^\sharp+F_2^\sharp\circ\zeta
\quad\text{on }\Lambda_{\mathfrak a},
\end{equation}
where $\zeta$ is defined in \eqref{eqn:zeta action}, since $\Phi\circ\zeta=\Phi$ by \eqref{eqn:LaUpsi} and
\eqref{eqn:telescoping-return-cocycle}, and the cutoff factors
on $\Lambda_{\mathfrak a}$ are
$1-\zeta_0(s-R(x))$ and $\zeta_0(s-R(x))$, respectively.

By \Cref{cor:factor map},
there is $C_1=C_1(\mathcal K)>0$ such that
\begin{equation}
\label{eqn:C1 new}
\supp F_1\cup\supp F_2\cup\supp F_2^\sharp
\subset
\{(x,x_-,z):\|z\|\leq C_1R(x)+C_1\}.
\end{equation}

Note that
by \eqref{eqn:LaUpsi}, the map
$(x,x_-)\mapsto\Phi(x,x_-,0)$ extends smoothly to a neighborhood
of the compact set $\overline{\Lambda_+}\times\Lambda_-$
in $\mathcal F^{(2)}$.
Using \Cref{lem:Riemannian-visual-Holder} and
$\|\operatorname{Ad}(a_{-z})\|\ll e^{C_2\|z\|}$, we obtain
 $C_2>0$ and $\eta\in(0,1]$ such that
\begin{equation}
\label{eqn:C3 Riem}
d_{\rm Riem}(\Phi(x,x_-,z),\Phi(y,y_-,z))
\ll e^{C_2\|z\|}
\bigl(d_{\bar\psi}(x,y)^\eta+d_{\bar\psi}(x_-,y_-)^\eta\bigr)
\end{equation}
for all $(x,x_-),(y,y_-)\in\Lambda_+\times\Lambda_-$
and $z\in\mathfrak a$.

We also introduce the stable projections used below.
Recall the conditional measures $\nu_{+,x}$ from
\Cref{lem:measure on product space}.
For a Borel function
$H:\Lambda_+\times\Lambda_-\to\mathbb C$ and $k\geq0$,
define its \emph{$k$-step stable projection} by
\begin{equation}
\label{eqn:def stable projection}
(\boldsymbol\pi_kH)(x)
:=
\int_{\Lambda_-}
H\bigl(\widehat T^k(x,x_-)\bigr)\,d\nu_{+,x}(x_-),
\qquad x\in\Lambda_+,
\end{equation}
whenever the integral converges absolutely.

\subsection{Parameters and the local spectral expansion}
\label{subsec:constants}
Recall $\lambda\in(0,1)$, $c_0>0$, and $\epsilon_0>0$ from
\Cref{lem:shrinking fibers}, \Cref{prop.holderlocal}, and
\eqref{eqn:real branch summability Section15}, respectively;
in particular, $\int_{\Lambda_+}e^{\epsilon_0R}\,d\nu_+<\infty$.
With $C_1,\eta,C_2>0$ as in
\eqref{eqn:C1 new} and \eqref{eqn:C3 Riem},
choose $\beta,\epsilon_1\in(0,1)$ such that
\begin{equation}\label{eqn:beta-choiceofe1}
8\beta C_1C_2<\epsilon_0,
\qquad
\epsilon_1<\frac14
\left(\frac{r-1}{2}+L+\frac32\right)^{-1},
\end{equation}
and set
\begin{equation}\label{eqn:kappa}
c_1:=\frac12\min\left\{\beta\eta,c_0,{\epsilon_0}/{8}\right\},
\qquad
\tau:=\lambda^{-\beta\eta}>1,
\qquad
J:=2L+5.
\end{equation}
We apply the results of \Cref{sec.spectra of complex transfer operators,sec.localmixingsymbolic} with $c=c_1$ in the estimates below.

For $t>1$, set
\begin{equation}
\label{eqn:kt}
k_t:=
\left\lceil
\left(\frac{r-1}{2}+L+2J\right)
\frac{\log t}{\beta\eta|\log\lambda|}
\right\rceil.
\end{equation}
Then $\tau^{-k_t}:=\lambda^{\beta\eta k_t}
\leq t^{-((r-1)/2+L+2J)}$.

Finally, choose $b_0>0$ such that
\Cref{thm:decomposition of transfer operators} and
\Cref{cor:expansion of complex eigenvalue around 0} hold on
a neighborhood of $[-b_0,b_0]^r$, and 
for every $z\in(-b_0,b_0)^r$,
\be
\label{eqn:kappa25-kappafw}
\left(\frac{(r-1)/2+L+2J}{\beta\eta|\log\lambda|}\right)
\bigl|\log|\kappa_z|\bigr|
<\frac{1-\epsilon_1}{10}, \qquad
|\kappa_z|\leq\widehat{\mathcal G}(z/\sqrt2).
\ee

These choices are possible because $\kappa_0=1$ and
$$
\log|\kappa_z|=-\frac12 z^tSz+o(\|z\|^2)
$$
by \eqref{eqn:tilde kappa} and
\Cref{cor:expansion of complex eigenvalue around 0}.
Let $\mathcal O:=(-b_0,b_0)^r$.
By analyticity and
\Cref{cor:expansion of complex eigenvalue around 0},
$\frac{d}{ds}\log|\kappa_{s\xi}|
=-s\xi^tS\xi+O(s^2\|\xi\|^3)$.
Since $S$ is positive definite, decreasing $b_0$ if necessary
makes this derivative nonpositive for $\xi\in\mathcal O$
and $0\leq s\leq1$.
Thus $s\mapsto\log|\kappa_{s\xi}|$ is nonincreasing, and hence
\begin{equation}
\label{eqn:radial eigenvalue modulus Section18}
|\kappa_{s\xi}|\geq|\kappa_\xi|
\qquad(\xi\in\mathcal O,\ 0\leq s\leq1).
\end{equation}

We use the symmetrized tensor product $\otimes$ and write
$*$ for full tensor contraction in the fixed orthonormal coordinates.
Apply \cite[Lemma~A.1]{DNP} to $\widetilde\kappa$ with its parameters
$K=2L+3$ and $J=3$; the latter is distinct from our $J=2L+5$.
We now record the expansion of $\kappa_\xi^n$ that
produces the Gaussian terms after Fourier inversion.
After decreasing $b_0$ further if necessary, the following
expansion and remainder estimate hold uniformly for
$\xi\in[-b_0,b_0]^r$ and all sufficiently large $n$:
\begin{equation}
\label{eqn:kappawn}
\begin{aligned}
\kappa_\xi^n
&=e^{in c_R^{2}\xi_1}\widehat{\mathcal G}(\sqrt n\,\xi)
\sum_{k=0}^{2L+4}\sum_{j=3k}^{2L+4+2k}
n^k A_{j,k}*\xi^{\otimes j}+O\left(
\widehat{\mathcal G}(\sqrt{n/2}\,\xi)n^{-(L+2)}
(1+\|\sqrt n\,\xi\|^{L_0})\varphi(\xi)\right),
\end{aligned}
\end{equation}
where $A_{j,k}$ are fixed symmetric tensors, $A_{0,0}=1$,
$A_{j,0}=0$ for $j>0$, $L_0\in\mathbb N$, and $\varphi$ is bounded
on $\mathcal O$ with $\varphi(\xi)\to0$ as $\xi\to0$.

All these choices are independent of $t$.
For each $\xi\in\mathfrak a$, the operator $\mathcal L_\xi$
extends to a contraction on $L^1(\nu_+)$, given almost everywhere
by the same branch formula. Indeed,
$|\mathcal L_\xi F|\leq\mathcal L_0|F|$ and \eqref{contraction}
give $\|\mathcal L_\xi F\|_{L^1(\nu_+)}
\leq\|F\|_{L^1(\nu_+)}$.

\subsection{Return decomposition and tail estimates}
\label{subsec:estimatesonthereturnmap}
We express the correlations as finite sums indexed by $n$,
the number of iterations of the return map $T$. 
The return-time tail bound in \Cref{lem:N return time}
will allow us to discard terms with $n\leq t^{1-\epsilon_1}$
in the return decomposition.
The rough correlation bound in
\Cref{lem:rough transverse correlation bound}
provides the transverse decay needed for the integrable
remainder in \Cref{bmslocal}(2).

Let $\ell$ be either the zero function or $t\mapsto\sqrt t$
on $[0,\infty)$. 
 For $t \ge 0$ and $\ww \in \ker \psi$, set 
$$
C_t^\ell(f_1,f_2;\ww)
:=\int_{\Gamma_\rho\backslash G}
f_1(xa_{t\ut+\ell(t)\ww})f_2(x)\,d\mathsf m_\psi(x).
$$
Put $t_+:=2t/\inf R+2$. If $x\in\Lambda_+$, $0\leq s<R(x)$, and
$R_n(x)\leq s+t$ for some $n\geq1$, then
$(n-1)\inf R\leq R_n(x)-R(x)<t$.
Hence $n<t/\inf R+1=t_+/2$.

By \Cref{cor:factor map}, \eqref{eqn:factor intertwining identities},
and \eqref{eqn:defn T flow},
\begin{equation}
\label{eqn:return decomposition Section18}
\begin{aligned}
C_t^\ell(f_1,f_2;\ww)
&=\sum_{0\leq n<t_+}\int_{\Lambda_{\mathfrak a}}
(F_1\circ\zeta^n)(x,x_-,s+t,v+\ell(t)\ww) F_2(x,x_-,s,v)\,d\mathtt M_\psi.
\end{aligned}
\end{equation}

We first estimate the total measure of $N$-step branches
$\gamma\Lambda_+$ for which $R_\gamma>t$.
By \Cref{lem.uniformly positive}, there exists $C_3>0$
such that, for every $N\geq1$, $\gamma\in\mathcal H^N$, and
$x\in\gamma\Lambda_+$,
\begin{equation}
\label{eqn:C3}
R_\gamma-C_3\leq R_N(x)\leq R_\gamma.
\end{equation}

\begin{lemma}
\label{lem:N return time}
There exists $C_4>0$ such that for every $N\geq1$ and $t>0$,
$$
\sum_{{\gamma\in\mathcal H^N,\, R_\gamma>t}}
\nu_+(\gamma\Lambda_+)\leq C_4Ne^{-\epsilon_0t/N}.
$$
\end{lemma}
\begin{proof}
By \eqref{eqn:C3}, the definition of $R_N$, and $T$-invariance
of $\nu_+$,
\begin{align*}
&\sum_{{\gamma\in\mathcal H^N,\, R_\gamma>t}} \nu_+(\gamma\Lambda_+)
\leq \nu_+(\{R_N>t-C_3\})\\
&\leq \sum_{j=0}^{N-1} \nu_+\bigl(T^{-j}\{R>(t-C_3)/N\}\bigr) =N\nu_+(\{R>(t-C_3)/N\})\\
&\leq Ne^{-\epsilon_0(t-C_3)/N}
\int_{\Lambda_+}e^{\epsilon_0R}\,d\nu_+.
\end{align*}
The integral is finite by the choice of $\epsilon_0$ and hence the claim follows.
\end{proof}

The next bound gives transverse decay but does not yet give decay in
$t$. In the proof of \Cref{bmslocal}(2), we combine it with a sufficiently
high-order approximation to obtain the required integrable error.

\begin{lemma}
\label{lem:rough transverse correlation bound} There exists $a_0>0$  such that, for $\ell(t)=\sqrt t$, for all $\ww\in\ker\psi$ and sufficiently large $t$,
$$
t^{(r - 1) /2}|C_t^{\ell}(f_1,f_2;\ww)|
\ll \|f_1\|_\infty\|f_2\|_\infty\,
t^{(r-1)/2+1}e^{-a_0\min\{\|\ww\|^2,\|\ww\|\}}.
$$
\end{lemma}
\begin{proof} 
Apply \Cref{thm:decomposition of transfer operators} on the
fixed space $\mathcal F_{c_0}(\Lambda_+,\mathbb C)$. By analyticity and $\nu_+(K^\circ)=0$, we have
$\kappa_{-i\theta}=1+O(\|\theta\|^2)$ for
$\theta\in\ker\psi$ near $0$. The  spectral decomposition, with uniformly bounded remainder
powers on a smaller complex neighborhood of $0$, therefore gives 
$\delta,C>0$ such that for every $n \ge 0$, $\theta\in\ker\psi,$ and $\|\theta\|\leq\delta$,
\begin{equation}
\label{eqn:transverse exponential moment Section18}
\nu_+\bigl(e^{\langle\theta,K_n\rangle}\bigr)
=\nu_+(\mathcal L_{-i\theta}^{\,n}\mathbf1)
\leq Ce^{Cn\|\theta\|^2}.
\end{equation}

Suppose $\ww\ne0$ and set $w := \ww / \norm{\ww}$. For $\ell(t)=\sqrt t$, the $F_1$-factor in the $n$th integrand
of \eqref{eqn:return decomposition Section18} is evaluated at
$\bigl(\widehat T^n(x,x_-),\,s+t-R_n(x),\,v'\bigr)$ for $v':=v+\sqrt t\,\ww-K_n^\circ(x)$. On the support of this integrand, \eqref{eqn:C1 new} gives
$\|v\|+\|v'\|
\leq C_1\bigl(2+R(x)+R(T^nx)\bigr)$.
Taking the inner product with $w$ in
$\sqrt t\,\ww=K_n^\circ(x)+v'-v$, we obtain
$$
\sqrt t\,\|\ww\|
\leq \langle w,K_n(x)\rangle
+D\bigl(1+R(x)+R(T^nx)\bigr),
$$
where $D>0$ depends only on $\mathcal K$.
The volume of the initial $(s,v)$-support is $O((1+R(x))^r)$.

Choose $q_0\in (0,1)$ with $2q_0\leq\delta$ and
$4Dq_0\leq\epsilon_0/2$. For $0\leq q\leq q_0$, two
applications of Cauchy--Schwarz, $T$-invariance of $\nu_+$, and
\eqref{eqn:transverse exponential moment Section18} give
\begin{align*}
&\nu_+\bigl((1+R)^r
 e^{q\langle w,K_n\rangle +qD(R+R\circ T^n)}\bigr)\\
&\quad\leq
\nu_+\bigl(e^{2q\langle w,K_n\rangle }\bigr)^{1/2}
\nu_+\bigl((1+R)^{4r}e^{4qDR}\bigr)^{1/4}
\nu_+(e^{4qDR})^{1/4}
\ll e^{Cnq^2}.
\end{align*}
Here, the last two factors are uniformly bounded by the exponential
moment $\nu_+(e^{\epsilon_0R})<\infty$.
Summing over $0\leq n<t_+$ and absorbing the factor $e^{qD}$
into the implied constant yields
$$
|C_t^{\ell}(f_1,f_2;\ww)|
\ll\|f_1\|_\infty\|f_2\|_\infty\,
t\exp(-q\sqrt t\,\|\ww\|+Ctq^2).
$$
Taking $q=\min\left\{\frac{\|\ww\|}{2C\sqrt t},q_0\right\}$
bounds the exponent above by $-c\min\{\|\ww\|^2,q_0\sqrt t\,\|\ww\|\}$,
 with $c>0$ uniform. For $t\geq q_0^{-2}$,
this proves the assertion. The case $\ww=0$ follows from the same
support and moment bounds with $q=0$.
\end{proof}

\subsection{Decay of transfer operators}
\label{subsec:transfer-decay}
To control the Fourier integral outside $\mathcal O=(-b_0,b_0)^r$ fixed in \Cref{subsec:constants}, we need decay
estimates with constants independent of the frequency parameter $w\in \frak a$.  We derive these estimates from the quantitative non-integrability
established in \Cref{thm:non-integrability}. In the next subsection (\Cref{lem:Lkttv}), we establish the regularity and norm bounds
needed to apply this decay estimate to the functions arising
from the stable-projection reduction.

Throughout this subsection, we abbreviate
$\F(\La_+,\mathbb{C}):=\F_{c_1}(\La_+,\mathbb{C})$. For each $w\in \fa\smallsetminus\{0\}$ we equip this
space with the norm
\begin{equation}
\label{eqn:w norm}
\|H\|_{w}:=\max\left\{\|H\|_{\infty},\ \frac{C_0\,|H|_{c_1}}{\|w\|+1}\right\},
\qquad H\in \F(\La_+,\mathbb{C}),
\end{equation}
where $C_0>0$ will be chosen below so that
\eqref{eqn:estimate of n0 holder norm} holds.
The main result of this subsection is the following exponential decay estimate.

\begin{proposition}[Transfer-operator decay away from zero]
\label{prop:norm of transfer operator}
There exist $C,\beta_1,\beta_2>0$ such that for every $w\in \Wd$ with
$\|w\|\geq b_0$, $n\in \mathbb{N}$ and  $H\in \F(\La_+,\mathbb{C})$,
\begin{equation*}
\|\mathcal{L}^n_wH\|_{\infty}
\leq C e^{-\beta_1 n\|w\|^{-\beta_2}} \|H\|_{w}.
\end{equation*}
\end{proposition}
Absorbing the term $\lambda^{c_1n}\|H\|_\infty$ in
\eqref{eqn:Lnf}, we obtain $C_{0,1}>0$ such that
\begin{equation}
\label{eqn:Lasota-Yorke}
|\L_w^nH|_{c_1}
\leq C_{0,1}(\|w\|+1)\|H\|_\infty
+C_{0,1}\lambda^{c_1n}|H|_{c_1}
\end{equation}
for every $w\in\Wd$, $n\in\mathbb N$, and
$H\in\F(\Lambda_+,\mathbb C)$. We  use the following two conventions throughout this subsection. First, for $t>0$, we write $\L^{t}_{w}:=\L^{\lfloor t\rfloor}_{w}$.
Second, we denote again by $\|\cdot\|_{w}$ the operator norm on
$\F(\La_+,\mathbb{C})$ induced by \eqref{eqn:w norm}. 

For $z\in\mathbb C\smallsetminus\{0\}$, let
$\arg z:=\{\vartheta\in\mathbb R:z=|z|e^{i\vartheta}\}$
and let $\Arg z$ be its unique element in $(-\pi,\pi]$.
Write $\Log$ for the principal logarithm on
$\mathbb C\smallsetminus(-\infty,0]$.

Fix $H\in\mathcal F(\Lambda_+,\mathbb C)$.
On a cylinder $\gamma\Lambda_+$, $\ga \in \sH$,  where $H$ does not vanish,
a branch of $\arg H$ is a function
$\varphi_\gamma:\gamma\Lambda_+\to\mathbb R$ satisfying
$\varphi_\gamma(y)\in\arg H(y)$ for every $y\in\gamma\Lambda_+$.
Set \begin{equation}
\label{den:N0H}
N_0(H):=\min\left\{n\in\N :|H|_{c_1}\lambda^{nc_1}
\diam_{d_{\bar\psi}}(\Lambda_+)^{c_1}\leq 1/8\right\}.
\end{equation}
\begin{lemma}
\label{lem:holder-branch-arg}
Let $n\geq N_0(H)$ and $\gamma\in\mathcal H^n$ satisfy
$\inf_{\gamma\Lambda_+}|H|\geq1/4$, and fix $x_\gamma\in\gamma\Lambda_+$.
There is a branch $\varphi_\gamma$ of $\arg H$ on $\gamma\Lambda_+$ such that
\begin{enumerate}
\item $\varphi_\gamma(x_\gamma)=\Arg (H(x_\gamma))$;
\item for every $y,y'\in\gamma\Lambda_+$,
$|\varphi_\gamma(y)-\varphi_\gamma(y')|
\leq8|H|_{c_1}d_{\bar\psi}(y,y')^{c_1}$.
\end{enumerate}
\end{lemma}
\begin{proof}
Put $a:=H(x_\gamma)$. By \eqref{eqn:recall contraction metric} and
\eqref{den:N0H},
$$
|H(y)-a|\leq|H|_{c_1}\lambda^{nc_1}\diam_{d_{\bar\psi}}(\Lambda_+)^{c_1}
\leq 1/8\leq {|a|}/{2}
\qquad(y\in\gamma\Lambda_+).
$$
Thus $\varphi_\gamma(y):=\Arg (a) +\Im (\Log\bigl(H(y)/a\bigr))$
is a branch with the required normalization.
The segment joining $H(y)/a$ and $H(y')/a$ lies in
$\{z:|z-1|\leq1/2\}$, where $|(\Log)'(z)|\leq2$.
Integrating along this segment gives
$$
|\varphi_\gamma(y)-\varphi_\gamma(y')|
\leq {2}{|a|^{-1}}|H(y)-H(y')|
\leq8|H|_{c_1}d_{\bar\psi}(y,y')^{c_1}.
$$
\end{proof}

The following elementary lemma will be used in the proof of
\Cref{lem:v norm of transfer operators} to control the angular
separation of summands whose weighted sum has modulus close to one.
\begin{lemma}
\label{lem:angular-clustering}
Let $(x_j)$ be a countable set of positive numbers $x_j$ and $\sum_jx_j\leq1$, let $y_j\in\mathbb R$, and let $\eta\in(0,1)$. If
$\abs{\sum_j x_j e^{i y_j}} \;>\; 1 - \eta $,
then for all $j,k$, 
$  d_{{\mathbb S}^1}(y_j, y_k) \;<\; \frac{\pi}{\sqrt{2}}
  \left( \sqrt{\frac{\eta}{x_j}} + \sqrt{\frac{\eta}{x_k}} \right)$
  where 
$d_{{\mathbb S}^1}(u,v):=\inf_{m\in\mathbb Z}|u-v+2\pi m|$ denotes the angular
distance in $\mathbb S^1=\mathbb R/(2\pi\mathbb Z)$.
\end{lemma}

\begin{proof}
If $\theta$ is an argument of $\sum_jx_je^{iy_j}$, then
$\sum_jx_j(1-\cos(y_j-\theta))=\sum_jx_j-\Bigl|\sum_jx_je^{iy_j}\Bigr|<\eta$.
Use $1-\cos u\geq2u^2/\pi^2$ for $|u|\leq\pi$ and the triangle
inequality on $\mathbb S^1$.
\end{proof}

The following is the main ingredient in the proof of
\Cref{prop:norm of transfer operator}.
\begin{lemma}
\label{lem:v norm of transfer operators}
There exist $a_1>0$, $D_1>1$, and $D_2>0$ such that,
for every $w\in\Wd$ with $\|w\|\geq b_0$ and every
$H\in\mathcal F(\Lambda_+,\mathbb C)$,
\begin{equation}
\label{eqn:estimate v norm}
\left\|
\mathcal L_w^{\,\lfloor D_1\log(3\|w\|/b_0)\rfloor}H
\right\|_w
\leq
\left(1-D_2\|w\|^{-a_1}\right)\|H\|_w.
\end{equation}
\end{lemma}

\begin{proof}
Choose $C_0>0$ so that $C_{0,1}C_0\leq1/4$, and fix
$D_{1,1}>0$ sufficiently large that
$C_{0,1}\lambda^{c_1\lfloor D_{1,1}\log3\rfloor}\leq1/4$.
Now fix $w\in\Wd$ with $\|w\|\geq b_0$ and
$H\in\mathcal F(\Lambda_+,\mathbb C)$ with $\|H\|_w\leq1$. By \eqref{eqn:w norm}, $\|H\|_\infty\leq1$ and
$|H|_{c_1}\leq(\|w\|+1)/C_0$.
Thus \eqref{eqn:Lasota-Yorke} gives, for every
$n\geq\lfloor D_{1,1}\log(3\|w\|/b_0)\rfloor$,  \begin{equation}
\label{eqn:estimate of n0 holder norm}
\frac{C_0|\L_w^nH|_{c_1}}{\|w\|+1}
\leq C_{0,1}C_0+C_{0,1}\lambda^{c_1n}
\leq\frac12.\end{equation}
Also, $|\L_w^nH|\leq\L_0^n|H|\leq1$ by
\eqref{eqn:invariance of transfer operator and measure}.
It remains to find $a_1,D_2>0$, independent of $w$
and $H$, such that
\begin{equation}\label{supbdd}
\|\L_w^nH\|_\infty\leq1-D_2\|w\|^{-a_1}
\end{equation} for $n=\lfloor D_1\log(3\|w\|/b_0)\rfloor$, with $D_1>0$
to be chosen below.

Let $a_{1,2}>0$ be a constant to be specified in Step~3.
We divide the proof into three steps. The constants $\beta'$
in Step~1 and $D_{1,2}$ in Step~2 are chosen after
$a_{1,2}$ is fixed.

\textbf{Step 1.} 
Denote the constants $C$, $\epsilon_1$, and $\eta$ in
\Cref{thm:non-integrability} by $C_{\mathrm{NI}}$,
$\epsilon_{\mathrm{NI}}$, and $\eta_{\mathrm{NI}}$, respectively,
so as to distinguish them from the constants fixed in
\Cref{subsec:constants}; retain $s_1,s_2,s_3,\delta_1$ from that theorem.
We use that theorem with $\langle\cdot,\cdot\rangle_\ut$, decreasing
$\eta_{\mathrm{NI}}$ if necessary. Indeed, let $\langle\cdot,\cdot\rangle_{\mathrm K}$ denote
the Killing-form inner product on $\mathfrak a$ used in
\Cref{thm:non-integrability}. For each $v\in\mathfrak a$ with $\|v\|_{\mathtt u}=1$,
choose $v^\flat$ so that
$\langle v^\flat,Y\rangle_{\mathrm K}
=\langle v,Y\rangle_{\mathtt u}$ for all $Y\in\mathfrak a$.
Apply the theorem to $v^\flat/\|v^\flat\|_{\mathrm K}$; the quantity
$\|v^\flat\|_{\mathrm K}$ has a positive lower bound independent of $v$ with
$\|v\|_{\mathtt u}=1$. 
Choose $s_4>1$ such that $s_1s_4c_1 > 1$ and for any $w'\in \frak{a}$ with $\|w'\|\geq b_0$, any $n\in \N$, any $\ga_1,\ga_2\in \mathcal{H}^n$ and $x_1,x_2\in \La_+$ with $d_{\bar\psi}(x_1,x_2)\leq \frac{\epsilon_{\mathrm{NI}}}{2\|3w'/b_0\|^{s_1s_4}}$, we have
\be
\label{eqn:choice of s4}
\begin{aligned}
&\left|\left(\langle w',K_n(\gamma_1 x_1)\rangle-\langle w',K_n(\gamma_2 x_1) \rangle\right)-\left(\langle w',K_n(\gamma_1 x_2)\rangle-\langle w',K_n(\gamma_2 x_2) \rangle\right)\right|\\
& \leq  2\|w'\|\cdot C_K\cdot d_{\bar\psi}(x_1,x_2)^{c_1}\leq \frac{b_0}{3}2^{1-c_1}C_K\epsilon_{\mathrm{NI}}^{c_1}\cdot \|3w'/b_0\|^{1-s_1s_4c_1}<\frac{\pi}{2},
\end{aligned}
\ee
where $C_K>0$ is the constant in \eqref{eqn:branch-K-distortion}, with $c_0$ replaced by $c_1$ using the finite $d_{\bar\psi}$-diameter of $\Lambda_+$.
We may increase $s_4$ further as required in Step~3.
Once $a_{1,2}$ is fixed in Step~3, choose $\beta'$ sufficiently large
that
\begin{equation}
\label{eqn:inequality beta}
\beta' > \max\left\{s_1s_4, \frac{a_{1,2}+1}{c_1}\right\}.
\end{equation}

Suppose that there exist $\gamma\in\mathcal H_{C_{\mathrm{NI}}}$
and $x\in\gamma\Lambda_+$ such that
\begin{equation}
\label{eqn:local decrease}
\begin{gathered}
B_{\bar\psi}\left(
x,\frac{\epsilon_{\mathrm{NI}}}
{2\|3w/b_0\|^{\beta'}}
\right)\subset\gamma\Delta_0,\quad 
|H(x)|<1-\frac{1}{2\|3w/b_0\|^{a_{1,2}}}.
\end{gathered}
\end{equation}
We first prove that \eqref{eqn:local decrease} implies
\begin{equation}\label{lone}
\|H\|_{L^1(\nu_+)}
\leq1-\frac{D_{2,1}}{\|3w/b_0\|^{a_{1,1}}}
\end{equation}
for some $a_{1,1},D_{2,1}>0$ independent of $w$ and $H$.
By \eqref{eqn:w norm}, $|H|_{c_1}\leq(\|w\|+1)/C_0$.
For every $x'\in\Lambda_+$ in the ball in \eqref{eqn:local decrease},
\begin{equation}
\label{eqn:difference}
\begin{aligned}
&|H(x')|
\leq |H(x)|+|H|_{c_1}d_{\bar\psi}(x,x')^{c_1}\\
&\leq1-\frac{1}{2\|3w/b_0\|^{a_{1,2}}}
+\frac{\|w\|+1}{C_0}
\left(\frac{\epsilon_{\mathrm{NI}}}
{2\|3w/b_0\|^{\beta'}}\right)^{c_1}\leq1-\frac{1}{4\|3w/b_0\|^{a_{1,2}}}.
\end{aligned}
\end{equation}

Here $x'$ lies in $\gamma\Lambda_+$ by
\eqref{eqn:powered-branch-partition}, and the last inequality holds
for all $\|w\|\geq b_0$ after increasing $\beta'$ in
\eqref{eqn:inequality beta}.
By \Cref{cor.PS_polyuniform},  \eqref{dedef},
\Cref{prop.T is well def a.e.}, and \eqref{nuplus}, 
$$
\nu_+\left(B_{\bar\psi}\left(x,\frac{\epsilon_{\mathrm{NI}}}{2\|3w/b_0\|^{\beta'}}\right)\right)
\geq\frac{D_{2,3}}{\|3w/b_0\|^{\beta'a_{1,3}}}
$$
for some $a_{1,3}>1$ and $D_{2,3}\in(0,1)$.
Combining this with \eqref{eqn:difference} and $\|H\|_\infty\leq1$
gives \eqref{lone}, with
$$
a_{1,1}:=a_{1,2}+\beta'a_{1,3},\qquad D_{2,1}:=D_{2,3}/4.
$$

\textbf{Step 2.}
Assuming \eqref{eqn:local decrease}, we use the spectral gap of $\L_0$
to turn \eqref{lone} into the uniform supremum-norm bound \eqref{supbdd}.
By \eqref{eqn:w norm},
$\||H|\|_{c_1}\leq\|H\|_{c_1}\ll\|3w/b_0\|$.
Thus \Cref{prop:L0 decomposition} and \eqref{lone} give 
$C>0$ and $\vartheta\in(0,1)$, independent of $w$, $H$, and $n$,
such that, for every $n\geq1$,
\begin{align*}
\|\L_w^nH\|_\infty
&\leq\|\L_0^n|H|\|_\infty \leq\nu_+(|H|)+\|\mathcal R_0^n|H|\|_\infty \leq1-\frac{D_{2,1}}{\|3w/b_0\|^{a_{1,1}}}
+C\vartheta^n\|3w/b_0\|.
\end{align*}
Since $\|3w/b_0\|\geq3$, we can choose $D_{1,2}>0$ sufficiently
large that, for
$n=\lfloor D_{1,2}\log\|3w/b_0\|\rfloor\geq1$, the last term
is strictly less than
$D_{2,1}/(2\|3w/b_0\|^{a_{1,1}})$ for every $\|w\|\geq b_0$.
Consequently,
$$
\|\L_w^{\lfloor D_{1,2}\log\|3w/b_0\|\rfloor}H\|_\infty
<1-\frac{D_{2,1}}{2\|3w/b_0\|^{a_{1,1}}}.
$$
The same bound holds for all further iterates, since $\L_w$
contracts the supremum norm.

\textbf{Step 3.} We now show that \eqref{eqn:local decrease} holds for either
$H$ or an iterate $\L_w^{n_0}H$. Step~2 then gives the required
supremum-norm bound \eqref{supbdd}.
Set $C_h:=\inf_{x,y\in\Lambda_+}h(x)/h(y)>0$, choose $a_{1,6}>0$ so that
$C_h3^{a_{1,6}}\geq2$, and put $a_{1,7}:=a_{1,2}/2$.

It suffices to find $\gamma\in\mathcal H_{C_{\mathrm{NI}}}$ and
$x\in\gamma\Lambda_+$ satisfying the ball inclusion in \eqref{eqn:local decrease}
such that either $H$ satisfies \eqref{eqn:local decrease} at $x$, or
\begin{equation}
\label{eqn:decay of norm}
|\L_w^{n_0}H(x)|<1-\frac{1}{2\|3w/b_0\|^{a_{1,2}}},
\qquad
n_0:=\left\lfloor\frac{a_{1,7}-a_{1,6}}{C_{\mathrm{NI}}}\log\|3w/b_0\|\right\rfloor.
\end{equation}

Choose $a_{1,2}$ sufficiently large that, for all 
$\|w\|\geq b_0$, we have 
$n_0>N_{\mathrm{NI}}$, $n_0\geq\lfloor D_{1,1}\log\|3w/b_0\|\rfloor$, and
\begin{equation}
\label{eqn:choice of n0}
\frac{\|w\|+1}{C_0}\lambda^{n_0c_1}
\diam_{d_{\bar\psi}}(\Lambda_+)^{c_1}\leq\frac18.
\end{equation}

In the former case, Step~2 applies to $H$. In the latter, it applies to
$\L_w^{n_0}H$, since \eqref{eqn:estimate of n0 holder norm} gives
$\|\L_w^{n_0}H\|_w\leq1$.

Applying \Cref{thm:non-integrability} to the unit vector $\frac{w}{\|w\|}\in \fa$, $b=\|3w/b_0\|^{s_4}$, and $n_0$, we obtain 
$\ga_1,\ga_2\in \mathcal{H}^{n_0}_{C_{\mathrm{NI}}}$ as the two inverse branches, and two points $x_1,x_2\in \La_+$ satisfying
$$
x_2\in \Lambda_+\cap \left(
B_{\bar\psi}\left(x_1, \frac{\epsilon_{\mathrm{NI}}}{2\|3w/b_0\|^{s_4s_1}} \right)
\smallsetminus
B_{\bar\psi}\left(x_1,\frac{\delta_1\epsilon_{\mathrm{NI}}}{2\|3w/b_0\|^{s_4s_2}}\right)
\right).
$$

We also need the ball inclusion in \eqref{eqn:local decrease} at these points.
In the proof of \Cref{thm:non-integrability}, $x_1$ is fixed in the
interior of a one-cylinder by \eqref{hat}, and $x_2\to x_1$ uniformly
as $b\to\infty$ by \eqref{eqn:existence of points x1 x2}.
For $s_4$ sufficiently large, both points lie in a fixed compact
subset of that interior. Increasing $\beta'$ then gives the ball
condition at both points. For $\gamma_jx_i$, $i,j\in\{1,2\}$, the same
condition follows directly from \eqref{eqn:inclusion of cylinder}
and $\beta'>s_1s_4$. These requirements on $s_4$ are independent of $a_{1,2}$;
we now fix $s_4$.

Suppose that no pair $(\gamma,x)$ as above satisfies
\eqref{eqn:local decrease} for $H$.
We show that \eqref{eqn:decay of norm} holds at $x_1$ or $x_2$.
Assume to the contrary that it holds at neither point.  
Put $F:=|H|$. Omitting terms for which $H(\gamma x_1)=0$, write
$$
(\L_w^{n_0}H)(x_1)
=
\sum_{\gamma\in\mathcal H^{n_0}}
\frac{e^{-R_{n_0}(\gamma x_1)}F(\gamma x_1)h(\gamma x_1)}
     {h(x_1)}
e^{i\langle w,K_{n_0}(\gamma x_1)\rangle+i\Arg H(\gamma x_1)}.
$$
The weights are positive and sum to
$\L_0^{n_0}F(x_1)\leq1$.

Set $\eta:=\|3w/b_0\|^{-a_{1,2}}$.
Since \eqref{eqn:decay of norm} fails at $x_1$,
$|\L_w^{n_0}H(x_1)|\geq1-\eta/2>1-\eta$.
Thus \Cref{lem:angular-clustering} applies, and any two terms
with weights at least $\|3w/b_0\|^{-a_{1,7}}$
have angular distance at most

$$
\sqrt2\pi\left(\|3w/b_0\|^{-a_{1,2}}\|3w/b_0\|^{a_{1,7}}\right)^{1/2}
=\frac{\sqrt2\pi}{\|3w/b_0\|^{a_{1,8}}},\qquad a_{1,8}:=a_{1,2}/4.
$$
By this assumption and the ball inclusions at $\gamma_jx_i$,
$F(\gamma_jx_i)
\geq1-\frac{1}{2\|3w/b_0\|^{a_{1,2}}}
>\frac12$
for all $i,j\in\{1,2\}$. We claim that
\begin{equation}
\label{eqn:weight lower bound}
\frac{e^{-R_{n_0}(\gamma_jx_i)}h(\gamma_jx_i)}{h(x_i)}
\geq\frac{2}{\|3w/b_0\|^{a_{1,7}}}
\qquad(i,j\in\{1,2\}).
\end{equation}
Indeed, since $\gamma_j\in\mathcal H_{C_{\mathrm{NI}}}^{n_0}$, we have
$$
\frac{e^{-R_{n_0}(\gamma_jx_i)}h(\gamma_jx_i)}{h(x_i)}
\geq C_he^{-n_0C_{\mathrm{NI}}}
\geq\frac{C_h}{\|3w/b_0\|^{a_{1,7}-a_{1,6}}}.
$$

The choice $C_h3^{a_{1,6}}\geq2$ proves
\eqref{eqn:weight lower bound}. Including the factor $F(\gamma_jx_i)$,
the weights of the selected terms are at least $\|3w/b_0\|^{-a_{1,7}}$.
Thus 
\begin{equation}
\label{eqn:difference x1}
\begin{aligned}
d_{{\mathbb S}^1} (
&\langle w,K_{n_0}(\gamma_1x_1)\rangle+\Arg H(\gamma_1x_1), \langle w,K_{n_0}(\gamma_2x_1)\rangle+\Arg H(\gamma_2x_1))
\leq\frac{\sqrt2\pi}{\|3w/b_0\|^{a_{1,8}}}.
\end{aligned}
\end{equation}

The same argument at $x_2$ gives 
\begin{equation}
\label{eqn:difference x2}
\begin{aligned}
d_{{\mathbb S}^1} (
&\langle w,K_{n_0}(\gamma_1x_2)\rangle+\Arg H(\gamma_1x_2), \langle w,K_{n_0}(\gamma_2x_2)\rangle+\Arg H(\gamma_2x_2) )
 \leq\frac{\sqrt2\pi}{\|3w/b_0\|^{a_{1,8}}}.
\end{aligned}
\end{equation}

Fix $j\in\{1,2\}$. By \eqref{eqn:choice of n0} and
$|H(\gamma_jx_1)|\geq1/2$, we have $n_0\geq N_0(H)$ as in
\eqref{den:N0H} and $\inf_{\gamma_j\Lambda_+}|H|\geq3/8$.
Thus \Cref{lem:holder-branch-arg} gives a branch
$\varphi_{\gamma_j}:\gamma_j\Lambda_+\to\mathbb R$ satisfying
\begin{equation}
\label{eqn:difference branch}
\left|\varphi_{\gamma_j}(\gamma_j x_1)-\varphi_{\gamma_j}(\gamma_j x_2)\right|
\leq 8|H|_{c_1}d_{\bar\psi}(\gamma_j x_1,\gamma_j x_2)^{c_1}
\leq\frac{D_{2,7}}{\|3w/b_0\|^{a_{1,9}}},
\end{equation}
where we use \eqref{eqn:recall contraction metric} and the estimate $|H|_{c_1}\leq(\|w\|+1)/C_0\ll\|3w/b_0\|$, absorbing the
factor $\lambda^{-c_1}$ from the integer part in $n_0$ into $D_{2,7}$, and where
$a_{1,9}:=-1-c_1(a_{1,7}-a_{1,6})(\log \lambda)/C_{\mathrm{NI}}>0$.

Replacing $\Arg H$ by $\varphi_{\gamma_j}$ does not change
the angular distances in \eqref{eqn:difference x1} and
\eqref{eqn:difference x2}. The triangle inequality for
$d_{{\mathbb S}^1}$ and \eqref{eqn:difference branch} therefore give
\begin{equation}
\label{eqn:difference upper bound}
\begin{aligned}
&d_{{\mathbb S}^1} (
\langle w,K_{n_0}(\gamma_1x_1)-K_{n_0}(\gamma_2x_1)\rangle,\langle w,K_{n_0}(\gamma_1x_2)-K_{n_0}(\gamma_2x_2)\rangle)\\
&\leq
\frac{4\pi}{\|3w/b_0\|^{a_{1,8}}}
+\frac{2D_{2,7}}{\|3w/b_0\|^{a_{1,9}}}
\leq\frac{D_{2,8}}{\|3w/b_0\|^{a_{1,10}}}.
\end{aligned}
\end{equation}
where $a_{1,10}=\min\{a_{1,8},a_{1,9}\}
=\min\{a_{1,2}/4,\;-1-c_1(a_{1,2}/2-a_{1,6})(\log \lambda)/C_{\mathrm{NI}}\}$.

By \eqref{eqn:choice of s4}, the two arguments of
$d_{\mathbb S^1}$ in \eqref{eqn:difference upper bound}
differ by less than $\pi/2$ in absolute value.
Hence their angular distance equals their absolute difference.
 Thus \eqref{eqn:difference upper bound} gives
\begin{align*}
&\left|\left(\langle w,K_{n_0}(\gamma_1 x_1)\rangle-\langle w,K_{n_0}(\gamma_2 x_1) \rangle\right)-\left(\langle w,K_{n_0}(\gamma_1 x_2)\rangle-\langle w,K_{n_0}(\gamma_2 x_2) \rangle\right)\right|
\leq \frac{D_{2,8}}{\|3w/b_0\|^{a_{1,10}}},\nonumber
\end{align*}

On the other hand, \eqref{eqn:non-integrability-estimate} gives 
\begin{align*}
&\left|\left(\langle w,K_{n_0}(\gamma_1 x_1)\rangle-\langle w,K_{n_0}(\gamma_2 x_1) \rangle\right)-\left(\langle w,K_{n_0}(\gamma_1 x_2)\rangle-\langle w,K_{n_0}(\gamma_2 x_2) \rangle\right)\right|\\
&\geq \eta_{\mathrm{NI}} \|w\|\left(\frac{\delta_1 \epsilon_{\mathrm{NI}}}{2\|3w/b_0\|^{s_4s_2}}\right)^{s_3}.\nonumber
\end{align*}

The lower bound is a fixed positive multiple of
$\|3w/b_0\|^{1-s_4s_2s_3}$.
The constant $D_{2,8}$ can be chosen independently of $a_{1,2}$,
whereas $a_{1,10}\to\infty$ as $a_{1,2}\to\infty$.
Since $\|3w/b_0\|\geq3$, we may therefore fix $a_{1,2}$
sufficiently large that the upper bound is strictly smaller
than the lower bound for every $\|w\|\geq b_0$.
Now choose $\beta'$ and $D_{1,2}$ as in Steps~1 and~2.
In either case, Step~2 gives
$$
\|\L_w^nH\|_\infty
<
1-\frac{D_{2,1}}{2\|3w/b_0\|^{a_{1,1}}}
\qquad
\left(
n\geq n_0+\lfloor D_{1,2}\log\|3w/b_0\|\rfloor
\right).
$$

Choose $D_{1,3}\geq D_{1,1} +1$ with
$D_{1,3}\geq D_{1,2}+(a_{1,7}-a_{1,6})/C_{\mathrm{NI}}$.
Since $D_{2,1}<1/4$, this bound and
\eqref{eqn:estimate of n0 holder norm} prove
\eqref{eqn:estimate v norm}, with
$a_1:=a_{1,1}$, $D_1:=D_{1,3}$, and
$D_2:=\frac{D_{2,1}}4\left(\frac{b_0}{3}\right)^{a_1}$.
\end{proof}

\begin{proof}[Proof of \Cref{prop:norm of transfer operator}]
Fix any $w\in \fa\smallsetminus\{0\}$ with $\|w\|\geq b_0$,  $n\in \mathbb{N}$, and any $H\in \mathcal{F}(\La_+,\mathbb{C})$ with $\|H\|_w=1$. 
Let $a_1>0,D_1>1,D_2>0$ be the constants from \Cref{lem:v norm of transfer operators}. Set $n_0=\lfloor D_1\log\|3w/b_0\|\rfloor\geq1$,
where the last inequality follows from
$D_1\log3\geq1$ and $\|3w/b_0\|\geq3$.

Write $n=kn_0+l$ with $0\leq l<n_0$. Using \Cref{lem:v norm of transfer operators}, $1-u\leq e^{-u}$ for $u\geq0$, and the fact that $\|\L_w\|_{\infty}\leq 1$, we obtain 
$$
\begin{aligned}
&\|\L_w^nH\|_\infty
\leq\|\L_w^{kn_0}H\|_\infty
\leq\|\L_w^{kn_0}H\|_w\\
&\leq\left(1-\frac{D_2}{\|w\|^{a_1}}\right)^k
\leq e^{-kD_2/\|w\|^{a_1}} \leq e^{D_2b_0^{-a_1}}
e^{-\frac{D_2n}{n_0\|w\|^{a_1}}}  \leq Ce^{-\beta_1n\|w\|^{-(a_1+1)}},
\end{aligned}
$$
where we used $n_0\ll\|w\|$ for $\|w\|\geq b_0$.
This proves the proposition with $\beta_2:=a_1+1$.
\end{proof}

\subsection{Fourier and stable-projection estimates}
\label{subsubsection:auxiliary estimates}
We now estimate the Fourier transforms of the pullback functions and
their stable projections. These estimates justify Fourier inversion,
allow us to apply \Cref{prop:norm of transfer operator}, and control the error introduced by averaging over the
contracting $\Lambda_-$ coordinate after applying $\widehat T^k$.
 To streamline the arguments, we introduce $$
\cal S:=\{F_1,\,\partial^{\alpha}_{\fa}\Fn:\alpha \,\,\text{a multi-index}\}, \;\; 
C_\infty(F_1):=\|f_1\|_\infty,\;\;
C_\infty(\partial_\fa^\alpha\Fn):=\|f_2\|_{C^{|\alpha|}}.
$$
By \eqref{eqn:Fi}, \eqref{eqn:F2 sharp}, and the product rule,
$\|H\|_\infty\ll C_\infty(H)$ for $H\in\mathcal S$.
Constants in estimates involving derivatives may depend on their orders.
Derivatives of $\widehat F$ are taken in $\xi$; derivatives of $F$
are taken in its last variable $z$.
For $F\in\{F_1,\Fn\}$, write $f=f_1$ or $f_2$, respectively.

\subsubsection{Bounds for Fourier transforms}
Since $\Fn(x,x_-,\cdot)$ is smooth and compactly supported,
integration by parts $|\alpha|$ times gives the following identity.

\begin{lemma}
\label{lem:Fourierder}
Let $\alpha$ be a multi-index and let $(x,x_-)\in \La_+\times \La_-$.  Then, for
every $\xi\in \fa$ with $\xi^{\alpha}\neq 0$,
\begin{equation*}
\widehat{\Fn}(x,x_-,\xi)=\frac{i^{|\alpha|}}{\xi^{\alpha}}\widehat{\partial^{\alpha}_{\fa}\Fn}(x,x_-,\xi).
\end{equation*}
\end{lemma}

\begin{lemma}[Supremum norm]
\label{lem:hatFsup}
\begin{enumerate}
\item Let $H\in \cal S$.  Then, for all $(x,x_-)\in \La_+\times \La_-$, all
$\xi\in \fa$ and every integer $k\ge 0$,
\begin{align}
\label{eqn:inffa}
&|\widehat{H}(x,x_-,\xi)|\ll C_\infty(H) (R(x)+1)^r,\\
&\bigl|\bigl(\boldsymbol\pi_{k}\widehat{H}(\cdot,\xi)\bigr)(x)\bigr|
\ll C_\infty(H) (R(T^{k}x)+1)^r.\nonumber
\end{align}

\item For all $(x,x_-)\in \La_+\times \La_-$, all $\xi\in \fa\smallsetminus\{0\}$,  $m\in \mathbb{N}$ and integers $k\ge 0$,
\begin{equation}
\label{eqn:HS2sup}
|\widehat{\Fn}(x,x_-,\xi)|\ll_m \|\xi\|^{- m}  \|f_2\|_{C^m} (R(x)+1)^{r},
\end{equation}
$$
\bigl|\bigl(\boldsymbol{\pi}_{k}\widehat{\Fn}(\cdot,\xi)\bigr)(x)\bigr|\ll_m \|\xi\|^{- m}  \|f_2\|_{C^m} (R(T^{k}x)+1)^{r}.
$$

\item For  $F\in\{F_1,\Fn\}$, every multi-index $\alpha$, all $(x,x_-)\in \La_+\times \La_-$,
$\xi\in \fa$ and $k\ge 0$,
$$
\begin{aligned}
& |\paral\widehat F(x,x_-,\xi)|
\ll \|f\|_{\infty} (R(x)+1)^{r+|\alpha|},\\
&\bigl|\bigl(\boldsymbol{\pi}_{k}\paral\widehat F(\cdot,\xi)\bigr)(x)\bigr|
\ll \|f\|_{\infty} (R(T^{k}x)+1)^{r+|\alpha|}.
\end{aligned}
$$
\end{enumerate}
\end{lemma}

\begin{proof}
By \eqref{eqn:C1 new}, the support of every $H(x,x_-,\cdot)$,
$H\in\mathcal S$, lies in the ball of radius $C_1(R(x)+1)$.
Hence
$$
|\widehat H(x,x_-,\xi)|\leq\int_\fa|H(x,x_-,z)|\,dz
\ll C_\infty(H)(R(x)+1)^r.
$$

Integrating this bound against the probability measure in
\eqref{eqn:def stable projection} proves (1).
For (2), choose $j$ with $|\xi_j|\geq\|\xi\|/\sqrt r$ and apply
\Cref{lem:Fourierder} with $\alpha=me_j$, followed by (1).

For (3), differentiation under the Fourier integral gives
\begin{equation}
\label{eqn:partialF1}
\partial_{\fa}^\alpha\widehat F(x,x_-,\xi)
=i^{|\alpha|}\int_\fa e^{i\langle\xi,z\rangle}
z^\alpha F(x,x_-,z)\,dz.
\end{equation}
On the support, $|z^\alpha|\ll_\alpha(R(x)+1)^{|\alpha|}$ by
\eqref{eqn:C1 new}. Apply the argument for (1).
\end{proof}

\begin{lemma}[$L^1$-norm]
\label{lem:L1norm}
For all $(x,x_-)\in \La_+\times \La_-$,
\begin{equation}
\label{eqn:L1fa}
\|\widehat{\Fn}(x,x_-,\cdot)\|_{L^1(\fa)}\ll  \|f_2\|_{C^{r+1}} (R(x)+1)^{r}.
\end{equation}
\end{lemma}
\begin{proof}
On $\|\xi\|\leq 1$ we bound the integrand by
\eqref{eqn:inffa}, and on $\|\xi\|\geq 1$ by \eqref{eqn:HS2sup} with $m=r+1$,
which gives
\begin{equation*}
\|\widehat{\Fn}(x,x_-,\cdot)\|_{L^1(\fa)}
\ll \|f_2\|_{C^{r+1}} (R(x)+1)^{r}
\left(1+\int_{\|\xi\|\geq 1}\|\xi\|^{-(r+1)}\,d\xi\right).
\end{equation*}
The last integral is finite because $r+1>r=\dim\fa$.
\end{proof}
\begin{lemma}[H\"older norm]
\label{lem:x-x'}
Let $x\in \La_+$, let $x_-,x_-'\in \La_-$, and let $\xi\in \fa$.
\begin{enumerate}
\item For every $H\in \cal S$,
\begin{equation*}
|\widehat{H}(x,x_-,\xi)-\widehat{H}(x,x_-',\xi)|  \ll C(H) e^{\beta C_1 C_2R(x)} (R(x)+1)^r d_{\bar\psi}(x_-,x_-')^{\beta\eta},
\end{equation*}
where $C(F_1):=\|f_1\|_{0,\beta}$ and $
C(\partial_{\fa}^{\alpha}\Fn):=\|f_2\|_{|\alpha|,\beta}$.
\item For $F\in\{F_1,\Fn\}$ and every multi-index $\alpha$,
\begin{equation*}
|\paral\widehat F(x,x_-,\xi)-\paral\widehat F(x,x_-',\xi)| 
\ll \|f\|_{0,\beta} e^{\beta C_1 C_2R(x)} (R(x)+1)^{|\alpha|+r} d_{\bar\psi}(x_-,x_-')^{\beta\eta}.
\end{equation*}
\end{enumerate}
\end{lemma}

\begin{proof}

The cutoff factors in \eqref{eqn:Fi} and \eqref{eqn:F2 sharp} do not
depend on $x_-$. Thus \eqref{eqn:C3 Riem} and the product rule give
$$
|H(x,x_-,z)-H(x,x_-',z)|
\ll C(H)e^{\beta C_2\|z\|}d_{\bar\psi}(x_-,x_-')^{\beta\eta}.
$$

If this difference is nonzero, \eqref{eqn:C1 new} gives
$\|z\|\leq C_1(R(x)+1)$. Integrating over this ball proves (1).
For (2), use \eqref{eqn:partialF1}; the extra factor $z^\alpha$ contributes
at most $O_\alpha((R(x)+1)^{|\alpha|})$.
\end{proof}

\begin{lemma}[Stable projection error]
\label{lem:Fj-H}
Let $k \ge 0$ be an integer and $(x, x_-) \in \La_+ \times \La_-$.
\begin{enumerate}
\item For every $H \in \cal S$ and $\xi \in \fa$,
\begin{align*}
&\bigl| \widehat{H}\bigl(\hat{T}^{k}(x, x_-), \xi\bigr) - \bigl(\boldsymbol{\pi}_{k}\widehat{H}(\cdot, \xi)\bigr)(x) \bigr| \ll C(H) e^{\beta C_1 C_2 R(T^{k}x)}\tau^{-k} (R(T^{k}x) + 1)^r.
\end{align*}

\item For every integer $m > r$ and $\xi \in \fa \smallsetminus \cal O$,
\begin{align*}
&\bigl| \widehat{\Fn}\bigl(\hat{T}^{k}(x, x_-), \xi\bigr) - \bigl(\boldsymbol{\pi}_{k}\widehat{\Fn}(\cdot, \xi)\bigr)(x) \bigr| \ll_m \|\xi\|^{-m} \|f_2\|_{m,\beta} e^{\beta C_1 C_2 R(T^{k}x)} \tau^{-k} (R(T^{k}x) + 1)^r.
\end{align*}

\item For $F\in\{F_1,\Fn\}$, every multi-index $\alpha$ and $\xi \in \fa$,
\begin{align*}
&\bigl| \paral\widehat F\bigl(\hat{T}^{k}(x, x_-), \xi\bigr) - \bigl(\boldsymbol{\pi}_{k}\paral\widehat F(\cdot, \xi)\bigr)(x) \bigr| \ll \|f\|_{0,\beta} e^{\beta C_1 C_2 R(T^{k}x)} \tau^{-k} (R(T^{k}x) + 1)^{|\alpha|+r}.
\end{align*}
\end{enumerate}

\end{lemma}

\begin{proof}
By \Cref{lem:shrinking fibers}, two points of a stable fiber have images
under $\widehat T^{k}$ at stable distance at most
$\lambda^{k}\diam_{d_{\bar\psi}}(\Lambda_-)$.
Integrate the bounds of \Cref{lem:x-x'} against $\nu_{+,x}$, using
\eqref{eqn:def stable projection} and
$\lambda^{\beta\eta k}=\tau^{-k}$ from \eqref{eqn:kappa}.
This proves (1) and (3).
For (2), choose $j$ with $|\xi_j|\geq\|\xi\|/\sqrt r$ and apply
\Cref{lem:Fourierder} with $\alpha=me_j$, followed by (1) and the definition of $C(H)$.
\end{proof}

We record a direct estimate for the branch formula after composing with
inverse branches. For 
$\xi\in\mathfrak a$, a continuous function $F$, and a multi-index $\alpha$,
we have
\begin{equation}
\label{eqn:one-step branch sum Section18}
\|\mathcal L_\xi(K^\alpha F)\|_\infty
+\frac{|\mathcal L_\xi(K^\alpha F)|_{c_1}}{1+\|\xi\|}
\ll_\alpha\sum_{\gamma\in\mathcal H}e^{-R_\gamma}(1+R_\gamma)^{|\alpha|}
\|F\circ\gamma\|_{c_1}
\end{equation}
whenever the sum is finite. Here $K^\alpha$ is the coordinate monomial
in the fixed basis of $\mathfrak a$. To prove this, apply
\Cref{prop.holderlocal}, \eqref{eqn:one branch distortion Section15},
and \eqref{eqn:recall contraction metric} to each summand, using
$h,h^{-1}\in\mathcal F_{c_1}(\Lambda_+)$ and $\|K\|\ll R$ (\Cref{lem.uniformly positive}). The resulting series converges in
$\mathcal F_{c_1}(\Lambda_+,\mathbb C)$.

Using the stable projection of \eqref{eqn:def stable projection}, define
\be\label{eqn:tildeHjt}
\begin{aligned}
H_{1,t}(x,\xi)&:=\boldsymbol\pi_{k_t}(\widehat F_1(\cdot,\xi))(x), \qquad 
H_{2,t}(x,\xi) :=\boldsymbol\pi_{k_t}(\widehat{\Fn}(\cdot,\xi))(x),\\
\widetilde H_{j,t}(x,\xi)&:=e^{i\langle\xi,K_{k_t}(x)\rangle}
H_{j,t}(x,\xi)\qquad(j=1,2).
\end{aligned}
\ee
After cancellation of the exponential factors involving $K_k$, the remaining stable
average may still grow with $R(x)$. One further application of
$\mathcal L_\xi$, using
\eqref{eqn:one-step branch sum Section18}, gives the norm bounds
needed both for the large-frequency contribution and for the normalized
expectations in \Cref{lem:ANF1F2}.

\begin{lemma}
\label{lem:Lkttv} For every integer $k\geq0$ and $\xi\in\mathfrak a$, set
$$
H_{k,\xi}:=
\mathcal L_\xi^{k+1}
\left(
e^{-i\langle\xi,K_k\rangle}
\boldsymbol\pi_k(\widehat{\Fn}(\cdot,-\xi))
\right).
$$
Then $H_{k,\xi}\in\mathcal F_{c_1}(\Lambda_+,\mathbb C)$.
Moreover, the following estimates hold with implied constants
independent of $k$.

\begin{enumerate}
\item
For every $\xi\in\mathfrak a$ with $\|\xi\|\geq b_0$
and $m\in\mathbb N$,
$$
\| H_{k,\xi}
\|_\xi
\ll_m
\|\xi\|^{-m}\|f_2\|_{m,\beta}.
$$

\item
For every $\xi\in\mathcal O$ and multi-index $\alpha$,
$$
\left\|
\partial_\fa^\alpha H_{k,\xi}
\right\|_{c_1}
\ll_\alpha \|f_2\|_{0,\beta}.
$$
\end{enumerate}
For $k=k_t$, the input to $\mathcal L_\xi^{k+1}$ is
$\widetilde H_{2,t}(\cdot,-\xi)$ by \eqref{eqn:tildeHjt}.
\end{lemma}

\begin{proof}
For a Borel function $H$ bounded on each stable fiber,
\eqref{eqn:def stable projection} and
\eqref{eqn:stable kernel recursion} give
$$
\mathcal L_0^k(\boldsymbol\pi_kH)(x)
=
\sum_{\gamma\in\mathcal H^k}p_\gamma^{(k)}(x)
\int_{\Lambda_-}
H(x,\gamma^{-1}x_-)\,d\nu_{+,\gamma x}(x_-)
=
\boldsymbol\pi_0H(x).
$$
The series converges absolutely because its sum of absolute
values is at most
$\sup_{y_-\in\Lambda_-}|H(x,y_-)|$.
The factors $e^{i\langle\xi,K_k\rangle}$ and
$e^{-i\langle\xi,K_k\rangle}$ cancel in the $k$-step
branch formula, giving
\begin{equation}
\label{eqn:Lkt formula}
\begin{aligned}
&\mathcal L_\xi^k
\left(
e^{-i\langle\xi,K_k\rangle} 
\boldsymbol\pi_k(\widehat{\Fn}(\cdot,-\xi))
\right)=
\mathcal L_0^k
\boldsymbol\pi_k(\widehat{\Fn}(\cdot,-\xi))
=
\boldsymbol\pi_0(\widehat{\Fn}(\cdot,-\xi)).
\end{aligned}
\end{equation}
Thus both assertions reduce to estimates for
$\mathcal L_\xi\bigl(
\boldsymbol\pi_0(\widehat{\Fn}(\cdot,-\xi))
\bigr)$, independently of $k$. 
Fix an inverse branch $\gamma\in\mathcal H$. We estimate
$\bigl(\boldsymbol\pi_0\widehat{\Fn}(\cdot,-\xi)\bigr)\circ\gamma$.
 For $x,y$ in a common
one-cylinder, \Cref{prop.holderlocal} gives
$|R(\gamma x)-R(\gamma y)|
\ll d_{\bar\psi}(x,y)^{c_1}$.
Use this bound for the cutoff factors in \eqref{eqn:F2 sharp},
and \eqref{eqn:C3 Riem} with
\eqref{eqn:recall contraction metric} for the derivatives of
$f_2$. Since $c_1\leq\min\{c_0,\beta\eta\}$,
the product rule and \eqref{eqn:C1 new} give, for
$H=\partial_\fa^\delta\Fn$,
$$
\begin{aligned}
&|\widehat H(\gamma x,x_-,\xi)
-\widehat H(\gamma y,x_-,\xi)| \ll_\delta C(H)e^{\beta C_1C_2R_\gamma}
(1+R_\gamma)^r d_{\bar\psi}(x,y)^{c_1}.
\end{aligned}
$$
For the conditional measures,
\eqref{eqn:conditional-kernel-Holder} applied at
$\gamma x,\gamma y$ bounds their contribution by
$$
C\|\widehat H(\gamma y,\cdot,\xi)\|_\infty
d_{\bar\psi}(\gamma x,\gamma y)^{c_1}.
$$
Together with \Cref{lem:hatFsup}(1) and
\eqref{eqn:recall contraction metric}, this gives
\begin{equation}
\label{eqn:projected branch Holder Section18}
\|[\boldsymbol\pi_0(\widehat H(\cdot,\xi))]\circ\gamma\|_{c_1}
\ll_\delta
C(H)e^{\beta C_1C_2R_\gamma}(1+R_\gamma)^r.
\end{equation}
The same argument and \eqref{eqn:partialF1} give
\begin{equation}
\label{eqn:projected branch derivatives Section18}
\begin{aligned}
&\|[\boldsymbol\pi_0(
\partial_\fa^\alpha\widehat{\Fn}(\cdot,-\xi))]
\circ\gamma\|_{c_1} \ll_\alpha
\|f_2\|_{0,\beta}e^{\beta C_1C_2R_\gamma}
(1+R_\gamma)^{r+|\alpha|}.
\end{aligned}
\end{equation}

For (1), choose $j$ with $|\xi_j|\geq\|\xi\|/\sqrt r$
and apply \Cref{lem:Fourierder} with $\delta=me_j$.
By \eqref{eqn:projected branch Holder Section18} and
\eqref{eqn:one-step branch sum Section18},
\begin{align*}
&\left\|
\mathcal L_\xi\bigl(
\boldsymbol\pi_0(\widehat{\Fn}(\cdot,-\xi))
\bigr)
\right\|_\xi\ll_m
\|\xi\|^{-m}\|f_2\|_{m,\beta}
\sum_{\gamma\in\mathcal H}
e^{-(1-\beta C_1C_2)R_\gamma}(1+R_\gamma)^r\ll_m
\|\xi\|^{-m}\|f_2\|_{m,\beta}.
\end{align*}
The last inequality follows from
$\beta C_1C_2<\epsilon_0$ in \eqref{eqn:beta-choiceofe1}
and \eqref{eqn:real branch summability Section15},
absorbing the polynomial factor into the remaining exponential
margin. Together with \eqref{eqn:Lkt formula}, this proves (1).

For (2), differentiating the Fourier transform at $-\xi$
contributes the factor $(-1)^{|\alpha_2|}$. The product rule gives
\begin{align*}
&\partial_\fa^\alpha
\left[
\mathcal L_\xi\bigl(
\boldsymbol\pi_0(\widehat{\Fn}(\cdot,-\xi))
\bigr)
\right]=
\sum_{\alpha_1+\alpha_2=\alpha}
\binom{\alpha}{\alpha_1}
\mathcal L_\xi\left(
(iK)^{\alpha_1}(-1)^{|\alpha_2|}
\boldsymbol\pi_0(
\partial_\fa^{\alpha_2}\widehat{\Fn}(\cdot,-\xi))
\right).
\end{align*}
Apply \eqref{eqn:one-step branch sum Section18} and
\eqref{eqn:projected branch derivatives Section18}.
Since $\mathcal O$ is bounded,
\begin{align*}
&\left\|
\partial_\fa^\alpha
\left[
\mathcal L_\xi\bigl(
\boldsymbol\pi_0(\widehat{\Fn}(\cdot,-\xi))
\bigr)
\right]
\right\|_{c_1} \ll_\alpha
\|f_2\|_{0,\beta}
\sum_{\gamma\in\mathcal H}
e^{-(1-\beta C_1C_2)R_\gamma}
(1+R_\gamma)^{r+|\alpha|}\ll_\alpha\|f_2\|_{0,\beta},
\end{align*}
using \eqref{eqn:real branch summability Section15} as above.
These summable bounds, also at one higher derivative order,
justify differentiation in
$\mathcal F_{c_1}(\Lambda_+,\mathbb C)$.
Equation \eqref{eqn:Lkt formula} now proves (2). The estimates in (1) and in (2) with $\alpha=0$ also show that
$H_{k,\xi}\in\mathcal F_{c_1}(\Lambda_+,\mathbb C)$
for every $\xi\in\mathfrak a$.
\end{proof}

\subsubsection{Reduction to the expanding map}
The following lemma quantifies the error introduced by stable projection,
both outside $\mathcal O$ and for derivatives of the normalized
expectations inside $\mathcal O$.
We write $(N)$ for the $N$th derivative with respect to $\xi$,
viewed as a symmetric $N$-linear form on $\mathfrak a$.
For $\xi\in\mathfrak a$, put
\begin{equation}
\label{eqn:Theta}
\Theta_n(\xi,x):=
e^{-i\langle\xi,t\ut+\ell(t)\ww\rangle}
e^{i\langle\xi,K_n(x)\rangle}.
\end{equation}

\begin{lemma}
\label{lem:reduction of La-}
Let $k\geq0$ be an integer.
\begin{enumerate}
\item For every $n\in \mathbb{N}$ and $t>1$,
\begin{align*}
&\int_{\fa\smallsetminus \cal O} \int_{\La_+\times \La_-}  \Theta_n(\xi,x) \widehat{F}_1(\hat{T}^n(x,x_-),\xi)  \widehat{\Fn}(x,x_-,-\xi)\, d\nu^{\diamond}(x,x_-)\,d\xi\\
&\quad=\int_{\fa\smallsetminus \cal O} \int_{\La_+} \Theta_n(\xi,T^k x) \bigl(\boldsymbol{\pi}_{k}\widehat{F}_1(\cdot,\xi)\bigr)(T^nx) \bigl(\boldsymbol{\pi}_{k}\widehat{\Fn}(\cdot,-\xi)\bigr)(x)\, d\nu_+(x)\,d\xi\\
&\qquad\quad+O\bigl(\bigl(\|f_1\|_{0,\beta} \|f_2\|_{C^{r+1}}+\|f_1\|_{\infty} \|f_2\|_{r+1,\beta}\bigr) \tau^{-k}\bigr).
\end{align*}
\item For every $n\in\mathbb N$, $N\in\mathbb Z_{\geq0}$,
and $\xi\in\mathcal O$,
\begin{align*}
&\left[
\kappa_\xi^{-n}
\int_{\Lambda_+\times\Lambda_-}
e^{i\langle\xi,K_n(x)\rangle}
\widehat F_1(\widehat T^n(x,x_-),\xi)
\widehat{\Fn}(x,x_-,-\xi)\,
d\nu^\diamond(x,x_-)
\right]^{(N)}\\
&\quad=
\left[
\kappa_\xi^{-n}
\int_{\Lambda_+}
e^{i\langle\xi,K_n(T^kx)\rangle}
\bigl(\boldsymbol\pi_k\widehat F_1(\cdot,\xi)\bigr)(T^nx)
\bigl(\boldsymbol\pi_k\widehat{\Fn}(\cdot,-\xi)\bigr)(x)\,
d\nu_+(x)
\right]^{(N)}\\
&\qquad+
O_N\left(
\|f_1\|_{0,\beta}\|f_2\|_{0,\beta}
\tau^{-k}n^N|\kappa_\xi^{-n}|
\right),
\end{align*}
where the derivatives are taken with respect to $\xi$.
\end{enumerate}
\end{lemma}

\begin{proof}
For (1), fix $n\in \mathbb{N}$, $t>1$, and $k\geq0$, and recall
$\Theta_n(\xi,x)$ of \eqref{eqn:Theta}.  By the $\hat{T}$-invariance of
$\nu^{\diamond}$, the left-hand side equals
\begin{equation*}
\int_{\fa\smallsetminus \cal O}\int_{\La_+\times \La_-}\Theta_n(\xi,T^{k}x) \widehat{F}_1(\hat{T}^{n+k}(x,x_-),\xi) \widehat{\Fn}(\hat{T}^{k}(x,x_-),-\xi)\, d\nu^{\diamond}\,d\xi.
\end{equation*}

Replacing each of the two factors by its stable projection, we obtain the
asserted main term together with the two errors
$$
\begin{aligned}
E_1&:=\int_{\fa\smallsetminus \cal O}\int_{\La_+\times \La_-}\Theta_n(\xi,T^{k}x)\Bigl(\widehat{F}_1(\hat{T}^{k+n}(x,x_-),\xi)-\bigl(\boldsymbol{\pi}_{k}\widehat{F}_1(\cdot,\xi)\bigr)(T^nx)\Bigr)\\
&\qquad\qquad\cdot\widehat{\Fn}(\hat{T}^{k}(x,x_-),-\xi)\,d\nu^{\diamond}\,d\xi,\\
E_2&:=\int_{\fa\smallsetminus \cal O}\int_{\La_+\times \La_-}\Theta_n(\xi,T^{k}x)\bigl(\boldsymbol{\pi}_{k}\widehat{F}_1(\cdot,\xi)\bigr)(T^{n}x)\\
&\qquad\qquad\cdot \Bigl(\widehat{\Fn}(\hat{T}^{k}(x,x_-),-\xi)-\bigl(\boldsymbol{\pi}_{k}\widehat{\Fn}(\cdot,-\xi)\bigr)(x)\Bigr)\,d\nu^{\diamond}\,d\xi.
\end{aligned}
$$

It therefore suffices to show that
\begin{equation*}
\label{eqn:E1}
|E_1|\ll \|f_1\|_{0,\beta} \|f_2\|_{C^{r+1}} \tau^{-k},\quad 
|E_2|\ll \|f_1\|_{\infty} \|f_2\|_{r+1,\beta} \tau^{-k}.
\end{equation*}
Since $|\Theta_n|\equiv 1$, \Cref{lem:Fj-H}(1), Fubini's theorem and
\eqref{eqn:L1fa} give
\begin{align*}
|E_1| &\ll \|f_1\|_{0,\beta} \|f_2\|_{C^{r+1}} \tau^{-k}\int_{\La_+} e^{\beta C_1C_2R(T^{k+n}x)} (R(T^{k+n}x)+1)^r (R(T^{k}x)+1)^r\,d\nu_+(x)\\
&\ll \|f_1\|_{0,\beta} \|f_2\|_{C^{r+1}} \tau^{-k} \left(\int_{\La_+} e^{2\beta C_1 C_2 R(x)} (R(x)+1)^{2r} d\nu_+\int_{\La_+} (R(x)+1)^{2r} d\nu_+\right)^{1/2}\\
&\ll \|f_1\|_{0,\beta} \|f_2\|_{C^{r+1}} \tau^{-k},
\end{align*}
where we used Cauchy--Schwarz, $T$-invariance of $\nu_+$, the
exponential moment fixed in \Cref{subsec:constants}, and
$2\beta C_1C_2<\epsilon_0$ from \eqref{eqn:beta-choiceofe1}.  The
estimate for $E_2$ follows by the same argument, applying
\eqref{eqn:inffa} to $\boldsymbol{\pi}_{k}\widehat{F}_1$ and
\Cref{lem:Fj-H}(2) to
$\widehat{\Fn}-\boldsymbol{\pi}_{k}\widehat{\Fn}$ with $m=r+1$, which
produces $\|f_1\|_{\infty}\|f_2\|_{r+1,\beta}$ in place of
$\|f_1\|_{0,\beta}\|f_2\|_{C^{r+1}}$.

For (2), apply the product rule to the two error products
appearing in $E_1$ and $E_2$.
If $j\in\{0,\ldots,N\}$ derivatives with respect to $\xi$
are applied to $e^{i\langle\xi,K_n\circ T^k\rangle}$,
the corresponding integral is bounded by
$\tau^{-k}\int_{\Lambda_+\times\Lambda_-}
|AB|\,\|K_n\circ T^k\|^j\,d\nu^\diamond$,
where $A$ is a derivative of a projection error multiplied
by $\tau^k$, and $B$ is a derivative of the other Fourier
factor or its projection.
By \Cref{lem:hatFsup}(3), \Cref{lem:Fj-H}(3),
$T$-invariance, and the exponential moment of $R$,
$\|A\|_{L^4(\nu^\diamond)}
\|B\|_{L^4(\nu^\diamond)}
\ll_N \|f_1\|_{0,\beta}\|f_2\|_{0,\beta}$,
since $4\beta C_1C_2<\epsilon_0$ by
\eqref{eqn:beta-choiceofe1}.
For $j\geq1$, $\|K\|\ll R$ (\Cref{lem.uniformly positive}) gives
$$
\|K_n\circ T^k\|_{L^{2j}(\nu_+)}
\leq n\|K\|_{L^{2j}(\nu_+)}
\ll n\|R\|_{L^{2j}(\nu_+)}.
$$
H\"older's inequality therefore yields
$$
\begin{aligned}
\tau^{-k}\int |AB|\,\|K_n\circ T^k\|^j\,d\nu^\diamond  &\leq
\tau^{-k}\|A\|_{L^4(\nu^\diamond)}
\|B\|_{L^4(\nu^\diamond)}
\|K_n\circ T^k\|_{L^{2j}(\nu_+)}^j\\ & \ll_N
\|f_1\|_{0,\beta}\|f_2\|_{0,\beta}\tau^{-k}n^j.
\end{aligned}
$$

Finally, \Cref{thm:decomposition of transfer operators}
gives analyticity and nonvanishing of $\kappa_\xi$ on
a neighborhood of $\overline{\mathcal O}$, and hence
$\|D^q(\kappa_\xi^{-n})\|
\ll_q n^q|\kappa_\xi|^{-n}$ for all
$q\geq0$. 
If $q$ derivatives are applied to $\kappa_\xi^{-n}$,
then $j+q\leq N$. The product rule therefore gives
the error bound
$O_N\left(
\|f_1\|_{0,\beta}\|f_2\|_{0,\beta}
\tau^{-k}n^N|\kappa_\xi|^{-n}
\right)$.
This proves (2).
\end{proof}

\subsection{Convergence of the normalized expectations}
\label{subsec:normalized-expectations}
Near zero frequency, we separate the leading eigenvalue from the
Fourier-transformed return correlations. We show that the derivatives
of the resulting normalized expectations converge, providing the
coefficients for the correlation expansion. Their zeroth-order limit
will identify the leading mixing constant.

For $0\leq N\leq J$, $n\geq1$, and $\xi\in\mathcal O$, set
\begin{equation}
\label{eqn:ANn def}
\mathcal A_{N,n}(\widehat F_1,\widehat{\Fn};\xi)
:=
\left[
\kappa_\xi^{-n}
\int_{\Lambda_+\times\Lambda_-}
e^{i\langle\xi,K_n(x)\rangle}
\widehat F_1(\widehat T^n(x,x_-),\xi)
\widehat{\Fn}(x,x_-,-\xi)\,
d\nu^\diamond(x,x_-)
\right]^{(N)}.
\end{equation}

We omit $\xi$ when $\xi=0$ and write
\begin{equation}
\label{eqn:AN def}
\mathcal A_N(\widehat F_1,\widehat{\Fn})
:=\lim_{n\to\infty}\mathcal A_{N,n}(\widehat F_1,\widehat{\Fn})
\end{equation}
for the limits whose existence is established in \Cref{lem:ANF1F2}(1).
We suppress the two function arguments when no confusion can arise.
\begin{lemma}
\label{lem:ANF1F2}
\begin{enumerate}
\item For each $N\in \{0,\dots, J\}$ the limit \eqref{eqn:AN def} exists and
\begin{equation}
\label{eqn:upper bound AN}
\left|\cal A_N(\widehat{F}_1,\,\widehat{\Fn})\right|\ll \|f_1\|_{0,\beta} \|f_2\|_{0,\beta}.
\end{equation}
Moreover, for every $\bar{L}\in \mathbb{N}$ and $n\in \mathbb{N}$,
\begin{equation}
\label{eqn:AN rate}
\left|\cal A_N(\widehat{F}_1,\,\widehat{\Fn})-\cal A_{N,n}(\widehat{F}_1,\,\widehat{\Fn})\right|\ll_{\bar{L}} n^{-\bar{L}} \|f_1\|_{0,\beta} \|f_2\|_{0,\beta}.
\end{equation}

\item For every $\xi\in \cal O$ and $n\in \mathbb{N}$,
\begin{equation*}
|\cal A_{J,n}(\widehat{F}_1,\,\widehat{\Fn};\xi)|
\ll \|f_1\|_{0,\beta} \|f_2\|_{0,\beta} \left(n^{-(\frac{r-1}{2}+L+J)} |\kappa_{\xi}^{-n}|+ n^{2/5}\right).
\end{equation*}

\item We have
\begin{equation*}
\cal A_0(\widehat{F}_1,\,\widehat{\Fn})
=c_R^{4}\mathtt{M}_{\psi}(F_1)\,\mathtt{M}_{\psi}(F_2).
\end{equation*}
\end{enumerate}
\end{lemma}

\begin{proof}
 Fix $N\in\{0,\dots,J\}$ and abbreviate
$\cal A_{N,n}(\xi):=\cal A_{N,n}(\widehat{F}_1,\,\widehat{\Fn};\xi)$.
For an integer $1\leq k\leq n-2$ and $j=1,2$, set
$
G_{j,k}(\cdot,\xi):=e^{i\langle \xi,K_{k}(\cdot)\rangle} \boldsymbol{\pi}_{k}(\widehat{F}_j(\cdot,\xi))$,
where $\widehat{F}_2:=\widehat{\Fn}$, so that $G_{j,k_t}=\widetilde{H}_{j,t}$ by
\eqref{eqn:tildeHjt}.
By \Cref{lem:reduction of La-}(2) and the cocycle and transfer-operator
duality identities, \begin{equation}
\label{eqn:ANnk}
\cal A_{N,n}(\xi)=\widetilde{\cal A}_{N,n,k}(\xi)+O\left(\|f_1\|_{0,\beta} \|f_2\|_{0,\beta} n^N \tau^{-k}|\kappa_{\xi}^{-n}|\right),
\end{equation}
where
$$
\widetilde{\mathcal A}_{N,n,k}(\xi)
:=
\left[
\kappa_\xi^{-n}
\int_{\Lambda_+}
G_{1,k}(x,\xi)\,
\mathcal L_\xi^{n-k-1}
\left(\mathcal L_\xi^{k+1}G_{2,k}(\cdot,-\xi)\right)(x)\,
d\nu_+(x)
\right]^{(N)}.
$$

For every multi-index $\alpha$ with $|\alpha|\leq N$,
\Cref{lem:hatFsup}(3), $\|K\|\ll R$ from the proof of \Cref{Kk2},
and the product rule give
$$
|\partial_\fa^\alpha G_{1,k}(x,\xi)|
\ll_N\|f_1\|_\infty(1+R_k(x))^{|\alpha|}
(1+R(T^kx))^{r+|\alpha|}.
$$
Since $\|R_k\|_{L^p(\nu_+)}\leq k\|R\|_{L^p(\nu_+)}$
for $p\geq1$ by the triangle inequality in $L^p(\nu_+)$
and $T$-invariance, Cauchy--Schwarz and the exponential
moment of $R$ give
 \begin{equation}
\label{eqn:G1k L1 Section18}
\|\partial^\alpha_{\fa}G_{1,k}(\cdot,\xi)\|_{L^1(\nu_+)}
\ll_N\|f_1\|_\infty k^{|\alpha|}.
\end{equation}

We pair this bound with \Cref{lem:Lkttv}(2), which gives
$$
\|\partial_\fa^\alpha(\L_\xi^{k+1}G_{2,k}(\cdot,-\xi))\|_{c_1}
\ll_N\|f_2\|_{0,\beta}
$$ uniformly in $k$. These estimates also justify differentiating
the expectation above.
By \Cref{thm:decomposition of transfer operators},
\begin{equation}
\label{eqn:normalized spectral split Section18}
\kappa_\xi^{-n}\L_\xi^{n-k-1}
=\kappa_\xi^{-(k+1)}\cal P_\xi
+\kappa_\xi^{-(k+1)}(\cal R_\xi/\kappa_\xi)^{n-k-1}.
\end{equation}
We use this identity on $\mathcal F_{c_1}(\Lambda_+,\mathbb C)$.

\smallskip
\emph{Proof of (1).} Fix $\bar L\in\mathbb N$ and take
$k_n:=\lceil(\bar L+N)\log n/\log\tau\rceil$.
For sufficiently large $n$, $1\leq k_n+1\leq n/2$.
Use the same $k_n$ for $n$ and an integer $n'\in[n,2n]$.
The projection term in \eqref{eqn:normalized spectral split Section18}
is independent of $n$, so its contributions to
$\widetilde{\cal A}_{N,n,k_n}(0)$ and
$\widetilde{\cal A}_{N,n',k_n}(0)$ cancel.
Since $\kappa_0=1$, \eqref{eqn:uniform remainder powers Section15}
and the product rule give
$$
\left\|\left.\partial_\fa^\alpha
\bigl(\kappa_\xi^{-(k_n+1)}
(\cal R_\xi/\kappa_\xi)^{n-k_n-1}\bigr)\right|_{\xi=0}\right\|_{c_1\to c_1}
\ll_N k_n^{|\alpha|}\vartheta^{n-k_n-1}
$$
for some $\vartheta\in(0,1)$. Combining this with
\eqref{eqn:G1k L1 Section18} and \Cref{lem:Lkttv}(2), we obtain
$$
|\widetilde{\cal A}_{N,n,k_n}(0)-\widetilde{\cal A}_{N,n',k_n}(0)|
\ll\|f_1\|_\infty\|f_2\|_{0,\beta}k_n^N\vartheta^{n/2}.
$$
Equation \eqref{eqn:ANnk} and the choice of $k_n$ therefore give
$$
\sup_{n\leq n'\leq2n}
|\cal A_{N,n}(0)-\cal A_{N,n'}(0)|
\ll_{\bar L}n^{-\bar L}\|f_1\|_{0,\beta}\|f_2\|_{0,\beta}.
$$

Summing this estimate over successive dyadic intervals proves
the Cauchy property and \eqref{eqn:AN rate} for large $n$.
For each fixed $n$, \eqref{eqn:ANn def}, \Cref{lem:hatFsup}(3),
and the exponential moment of $R$ bound $\cal A_{N,n}(0)$ by
$O_n(\|f_1\|_\infty\|f_2\|_\infty)$.
This covers the remaining finitely many $n$ and gives
\eqref{eqn:upper bound AN}.

\smallskip \emph{Proof of (2).}  Take $N=J$ and
$k:=\left\lceil\frac{(r-1)/2+L+2J}{\log\tau}\log n\right\rceil$.
For sufficiently large $n$, $k+1\leq n/2$ and
$n^J\tau^{-k}\leq n^{-((r-1)/2+L+J)}$.
Thus \eqref{eqn:ANnk} gives the first term of the asserted bound.
By \eqref{eqn:normalized spectral split Section18}, analyticity of
$\kappa_\xi$ and $\cal P_\xi$, and
\eqref{eqn:uniform remainder powers Section15},
$$
\|\partial_\fa^\alpha(\kappa_\xi^{-n}\L_\xi^{n-k-1})\|_{c_1\to c_1}
\ll_J k^{|\alpha|}|\kappa_\xi|^{-(k+1)}
(1+\vartheta^{n-k-1})
\qquad(|\alpha|\leq J).
$$

Using \eqref{eqn:G1k L1 Section18} and \Cref{lem:Lkttv}(2),
we obtain
$$
|\widetilde{\cal A}_{J,n,k}(\xi)|
\ll\|f_1\|_\infty\|f_2\|_{0,\beta}
k^J|\kappa_\xi|^{-(k+1)}
\ll\|f_1\|_{0,\beta}\|f_2\|_{0,\beta}n^{2/5}.
$$

Indeed, \eqref{eqn:kappa} and \eqref{eqn:kappa25-kappafw} give
$|\kappa_\xi|^{-(k+1)}\ll n^{(1-\epsilon_1)/10}$ uniformly on
$\cal O$, and $k^J\ll(\log n)^J\ll n^{1/10}$.
The moment bounds used in (1), together with the nonvanishing of
$\kappa_\xi$ on a neighborhood of $\overline{\cal O}$, cover the
remaining finitely many $n$.

\smallskip
\emph{Proof of (3).} Take $N=0$, $\xi=0$, and $k=\lceil\log n\rceil$
in \eqref{eqn:ANnk} and \eqref{eqn:normalized spectral split Section18}.
The projection error and the spectral remainder tend to zero by
the estimates in (1). By \eqref{eqn:P0}, the remaining term is
$$
\nu_+(\boldsymbol\pi_k\widehat F_1(\cdot,0))\,
\nu_+(\L_0^{k+1}\boldsymbol\pi_k\widehat{\Fn}(\cdot,0))
=\nu^\diamond(\widehat F_1(\cdot,0))\,
\nu^\diamond(\widehat{\Fn}(\cdot,0)),
$$
where we used \eqref{eqn:def stable projection}, $\widehat T$-invariance
of $\nu^\diamond$, and \eqref{contraction}.
Consequently, Fubini's theorem and \eqref{eqn:def M+} give
$\cal A_0(\widehat F_1,\widehat{\Fn})
=c_R^{4}\widetilde{\mathtt M}_\psi(F_1)
\widetilde{\mathtt M}_\psi(\Fn)$. Absolute integrability follows from \eqref{eqn:C1 new} and the
exponential moment of $R$.
Finally, $\widetilde{\mathtt M}_\psi(F_1)=\mathtt M_\psi(F_1)$.
Integrating \eqref{eqn:F2 decomposition} over $\Lambda_\fa$ and
translating $(s,v)$ by $(-R(x),-K^\circ(x))$ in its second term gives
$\widetilde{\mathtt M}_\psi(\Fn)=\mathtt M_\psi(F_2)$.
This proves (3).
\end{proof}

\subsection{Uniform estimates for Gaussian sums}
\label{subsec:uniformrenewalsums}
Fourier inversion of the local spectral expansion produces sums of
Gaussian derivatives indexed by the number of returns. The following
estimates identify the leading transverse Gaussian and bound the
remaining terms uniformly in the transverse displacement. They will
be used in the final step of the proof of \Cref{bmslocal}(2).

\begin{lemma}
\label{lem:uniform renewal sums}
There is $a>0$, depending only on $c_R,r,S$, with the following
properties. 
For every fixed $\epsilon\in(0,1)$, uniformly for $t\geq2$ and
$\ww\in\ker\psi$,
\begin{equation}
\label{eqn:leading renewal sum}
\begin{aligned}
&\sum_{n>t^{1-\epsilon}}n^{-r/2}
\mathcal G\left(\frac{t-c_R^{2}n}{\sqrt n},\frac{\sqrt t\,\ww}{\sqrt n}\right)=c_R^{r-2}t^{-(r-1)/2}
\int_{\mathbb R}\,ds
+O_\epsilon\bigl(t^{-r/2} e^{-a\min\{\|\ww\|^2,\|\ww\|\}} \bigr).
\end{aligned}
\end{equation}
Moreover, for integers $p,k\geq0$ with $p\geq2k$,
\begin{equation}
\label{eqn:weighted Gaussian sum Section18}
\sum_{n\geq1}n^{k-(r+p)/2}
\left\|\mathcal G^{(p)}\left(\frac{t-c_R^{2}n}{\sqrt n},
\frac{\sqrt t\,\ww}{\sqrt n}\right)\right\|
\ll_{p,k}t^{-(r-1+p-2k)/2}e^{-a\min\{\|\ww\|^2,\|\ww\|\}}.
\end{equation}
\end{lemma}
\begin{proof}
For some $a_1>0$ depending only on $S$, Gaussian decay gives
$\|\mathcal G^{(p)}(z)\|\ll_p e^{-a_1\|z\|^2}$ for every $p\geq0$.
Put
$W:=\|\ww\|$ and $
P_n:=((t-c_R^{2}n)/\sqrt n,\sqrt t\,\ww/\sqrt n)$.
On $t/(2c_R^{2})\leq n\leq2t/c_R^{2}$,
$$
\|P_n\|^2\gg(n-t/c_R^{2})^2/t+W^2.
$$
The Gaussian sum over this range is $O(\sqrt t)$, proving
\eqref{eqn:weighted Gaussian sum Section18} there.
For $n<t/(2c_R^{2})$, $\|P_n\|^2\gg t+W^2$; for $n>2t/c_R^{2}$,
$\|P_n\|^2\gg n+\sqrt t\,W$ by the arithmetic--geometric mean
inequality. Both tails satisfy the same bound.

For \eqref{eqn:leading renewal sum}, restrict to
$\mathcal U_t:=\{n:|n-t/c_R^{2}|\leq t^{5/8}\}$.
Splitting the Gaussian exponent in the preceding estimates shows that
the omitted sum is $O(e^{-c t^{1/4}}e^{-a\min\{\|\ww\|^2,\|\ww\|\}})$.
For large $t$, this also permits the cutoff $n>t^{1-\epsilon}$.
For $n\in\mathcal U_t$, put
$y_n:=\frac{t-c_R^{2}n}{\sqrt t}$, $Q_n:=(c_Ry_n,c_R\ww)$ and $
\mu_n:=\sqrt{\frac{t}{c_R^{2}n}}$.
Then $P_n=\mu_nQ_n$ and $|\mu_n-1|\ll|y_n|/\sqrt t$.
The mean value theorem and the Gaussian bounds for
$\mathcal G$ and its first derivative give
\begin{align*}
&\left|n^{-r/2}\mathcal G(P_n)
-(t/c_R^{2})^{-r/2}\mathcal G(Q_n)\right| \ll t^{-(r+1)/2}|y_n|(1+|y_n|+W)e^{-c(y_n^2+W^2)}.
\end{align*}
Summing over the grid of spacing $c_R^{2}/\sqrt t$ gives an error
$O(t^{-r/2}e^{-c'W^2})$.
For an integrable absolutely continuous function $g$ with integrable
derivative,
$$
\left|b\sum_{y\in\alpha+b\mathbb Z}g(y)-\int_{\mathbb R}g(y)\,dy\right|
\leq b\int_{\mathbb R}|g'(y)|\,dy.
$$
Apply this with $g(y)=\mathcal G(c_Ry,c_R\ww)$ and
$b=c_R^{2}/\sqrt t$; the omitted grid points are exponentially small.
It follows that
$$
\sum_{n\in\mathcal U_t}\mathcal G(Q_n)
=\frac{\sqrt t}{c_R^{2}}\int_{\mathbb R}\mathcal G(c_Ry,c_R\ww)\,dy
+O(e^{-cW^2}).
$$
Multiply by $t^{-r/2}c_R^r$ to obtain
\eqref{eqn:leading renewal sum}. The bound
\eqref{eqn:weighted Gaussian sum Section18} covers the remaining bounded
range of $t$, after enlarging the implied constant.
\end{proof}

\subsection{Proof of the local mixing theorem}
\label{subsec:proof-bmslocal}
We now combine the preceding estimates to prove \Cref{bmslocal}.
We follow the proof of \cite[Theorem~2.2]{DNP}, using
\cite[Lemmas~A.1, A.3 and A.4]{DNP} for the Taylor expansions
and Gaussian sums. Starting from the return decomposition
\eqref{eqn:return decomposition Section18}, we discard small return
counts, control frequencies outside $\mathcal O$ using stable projection
and transfer-operator decay, and expand the contribution from
$\mathcal O$ using the normalized expectations.
Fourier inversion and summation of the resulting Gaussian terms then
give the two assertions of the theorem.
We retain the upper cutoff $n<t_+$ until Step~IV,
including after restricting the Fourier integral.

\begin{step}[Return decomposition and Fourier transform]
\end{step}

Split \eqref{eqn:return decomposition Section18} as
$$
C_t^\ell(f_1,f_2;\ww)
=R_{t,1}^\ell(f_1,f_2;\ww)
+C_{t,1}^\ell(f_1,f_2;\ww),
$$
where $R_{t,1}^\ell$ and $C_{t,1}^\ell$ are the sums over
$0\leq n\leq t^{1-\epsilon_1}$ and
$t^{1-\epsilon_1}<n<t_+$, respectively.

\begin{claim}
\label{cla:nt}
We have
$$
C_t^\ell=C_{t,1}^\ell+
O\bigl(\|f_1\|_\infty\|f_2\|_\infty t^{-((r-1)/2+L+1)}\bigr).
$$
\end{claim}
\begin{proof}[Proof of \Cref{cla:nt}] Set
$$
\Lambda_{\mathrm{bad},t}
:=\{x\in\Lambda_+:R_{\lfloor t^{1-\epsilon_1}\rfloor+1}(x)
>t-\frac{\inf R}{10}\}.
$$
Fix an integer $0\leq n\leq t^{1-\epsilon_1}$.
If the integrand of the summand of $R_{t,1}^\ell$
indexed by $n$ is nonzero at
$(x,x_-,s,v)\in\Lambda_{\mathfrak a}$, then
$R_{n+1}(x)>s+t\geq t$.
Hence $x\in\Lambda_{\mathrm{bad},t}$.
Summing over these indices and using \eqref{eqn:C1 new},
we obtain
\begin{equation}
\label{eqn:sumnt}
|R_{t,1}^\ell|
\ll t\|f_1\|_\infty\|f_2\|_\infty
\int_{\Lambda_+}(R+1)^r
\mathbf1_{\Lambda_{\mathrm{bad},t}}\,d\nu_+.
\end{equation}
 By \Cref{lem:N return time}, applied with
$N=\lfloor t^{1-\epsilon_1}\rfloor+1$ and its parameter $t$
replaced by $t-(\inf R)/10$, we have
$\nu_+(\Lambda_{\mathrm{bad},t})\ll t e^{-c t^{\epsilon_1}}$
for some $c>0$. The exponential moment of $R$ implies
$\int_{\Lambda_+}(R+1)^{2r}\,d\nu_+<\infty$.
Hence Cauchy--Schwarz gives
\begin{equation}
\label{eqn:intR}
\begin{aligned}
\int_{\Lambda_+}(R+1)^r
\mathbf1_{\Lambda_{\mathrm{bad},t}}\,d\nu_+\ll t^{1/2}e^{-c t^{\epsilon_1}/2}.
\end{aligned}
\end{equation}
Substituting this into \eqref{eqn:sumnt} proves the claim.
\end{proof}

Next, set
$$
\begin{aligned}
& C_{t,2}^\ell(f_1,f_2;\ww)  :=\sum_{t^{1-\epsilon_1}<n<t_+}
\int_{\widetilde\Lambda_{\mathfrak a}}
(F_1\circ\zeta^n)(x,x_-,s+t,v+\ell(t)\ww)
\cdot\Fn(x,x_-,s,v)\,d\widetilde{\mathtt M}_\psi.
\end{aligned}
$$
\begin{claim}
\label{cla:Ctell12}
We have
$$
C_{t,1}^\ell=C_{t,2}^\ell+
O\bigl(\|f_1\|_\infty\|f_2\|_\infty t^{-((r-1)/2+L+1)}\bigr).
$$
\end{claim}
\begin{proof}[Proof of \Cref{cla:Ctell12}]
Use \eqref{eqn:F2 decomposition}. In the term containing
$\Fn\circ\zeta$, translate $(s,v)$ by $(-R(x),-K^\circ(x))$
and use $\widehat T$-invariance of $\nu^\diamond$. Since
$\Fn$ vanishes for $s\leq-\inf R/10$, this gives
\begin{align*}
&\int_{\Lambda_{\mathfrak a}}
(F_1\circ\zeta^n)(x,x_-,s+t,v+\ell(t)\ww)
(\Fn\circ\zeta)(x,x_-,s,v)\,d\mathtt M_\psi\\
&\quad=\int_{\{-\inf R/10\leq s<0\}}
(F_1\circ\zeta^{n-1})(x,x_-,s+t,v+\ell(t)\ww)
\Fn(x,x_-,s,v)\,d\widetilde{\mathtt M}_\psi.
\end{align*}
Using $F_2=\Fn+\Fn\circ\zeta$ on $\Lambda_{\mathfrak a}$, we sum the preceding identity and reindex the contribution
from $\Fn\circ\zeta$.
The upper boundary term vanishes by the support cutoff.
Thus  $C_{t,1}^\ell=C_{t,2}^\ell+E$, where, with
$N_1:=\lfloor t^{1-\epsilon_1}\rfloor$,
$$
E:=\int_{\{-\inf R/10\leq s<0\}}
(F_1\circ\zeta^{N_1})(x,x_-,s+t,v+\ell(t)\ww)
\Fn(x,x_-,s,v)\,d\widetilde{\mathtt M}_\psi.
$$
Where the integrand defining $E$ is nonzero,
$R_{N_1+1}(x)>t+s\geq t-\inf R/10$.
Thus \eqref{eqn:C1 new} and \eqref{eqn:intR} give the claimed bound.
\end{proof}

For fixed $n$ and $(x,x_-) \in \La_+ \times \La_-$, put
$$
g(z):=\int_{\mathfrak a}
F_1(\widehat T^n(x,x_-),z+z')\Fn(x,x_-,z')\,dz'.
$$
Since $F_1(\widehat T^n(x,x_-),\cdot)$ is bounded and
compactly supported and
$\Fn(x,x_-,\cdot)\in\mathcal C_c^\infty(\mathfrak a)$,
we have $g\in\mathcal C_c^\infty(\mathfrak a)$.
Since
$\widehat g(\xi)
=
\widehat F_1(\widehat T^n(x,x_-),\xi)
\widehat{\Fn}(x,x_-,-\xi)$,  Fourier inversion gives
$$
\begin{aligned}
&g(t\ut+\ell(t)\ww-K_n(x))=\frac1{(2\pi)^r}\int_{\mathfrak a}
e^{-i\langle\xi,t\ut+\ell(t)\ww-K_n(x)\rangle}
\widehat F_1(\widehat T^n(x,x_-),\xi)
\widehat{\Fn}(x,x_-,-\xi)\,d\xi.
\end{aligned}
$$
Here and below $d\xi$ is the dual measure fixed above.
Define $C_{t,3}^\ell$ by summing the integral on the right without
the factor $(2\pi)^{-r}$ over $t^{1-\epsilon_1}<n<t_+$ and integrating
against $d\nu^\diamond (x, x^-)$ on $\La_+\times \La-$. By \Cref{lem:hatFsup}(1),(2), the integrals
are absolutely convergent. Since
$d\widetilde{\mathtt M}_\psi=c_R^{-2}d\nu^\diamond\,dz$,
$$
(2\pi)^r c_R^{2} C_{t,2}^\ell=C_{t,3}^\ell.
$$

\begin{step}[Reduction to integration over a compact domain]
\end{step}
Recall $\mathcal O$ from \Cref{subsec:constants} and the oscillatory factor
$\Theta_n$ from \eqref{eqn:Theta}.
We isolate the contribution from frequencies near zero,
which will produce the Gaussian term and its higher-order
corrections:
$$
C_{t,4}^\ell:=
\sum_{t^{1-\epsilon_1}<n<t_+}
\int_{\mathcal O}\int_{\Lambda_+\times\Lambda_-}
\Theta_n(\xi,x)
\widehat F_1(\widehat T^n(x,x_-),\xi)
\widehat{\Fn}(x,x_-,-\xi)\,d\nu^\diamond\,d\xi,
$$
and set
$$
R_{t,2}^\ell:=C_{t,3}^\ell-C_{t,4}^\ell.
$$
Thus $R_{t,2}^\ell$ is the contribution from frequencies 
outside $\mathcal O$.

\begin{claim}
\label{cla:S1}
There exists $m\in\mathbb N$ such that
$$
C_{t,3}^\ell=C_{t,4}^\ell+
O\bigl(\|f_1\|_{0,\beta}\|f_2\|_{m,\beta}t^{-((r-1)/2+L+1)}\bigr).
$$
\end{claim}

 \Cref{cla:S1} will be a consequence of \Cref{cla:sumint-b_0}.

Let $R_{t,3}^\ell$ denote the expression obtained from $R_{t,2}^\ell$
by composing with $\widehat T^{k_t}$ and replacing both factors
by their stable projections. By \Cref{lem:reduction of La-}(1), this replacement introduces
an error of
$O(\|f_1\|_{0,\beta}\|f_2\|_{r+1,\beta}\tau^{-k_t})$
for each $n$ in the summation range. Summing over $O(t)$ summands and using \eqref{eqn:kt} gives
\begin{equation}
\label{eqn:R34}
R_{t,2}^\ell=R_{t,3}^\ell+
O\bigl(\|f_1\|_{0,\beta}\|f_2\|_{r+1,\beta}t^{-((r-1)/2+L+1)}\bigr).
\end{equation}

The transfer-operator decay obtained from quantitative
non-integrability, combined with the Fourier decay of
$\Fn$, gives the following estimate:
\begin{claim}
\label{cla:sumint-b_0}
For sufficiently large $m\in\mathbb N$,
$$
|R_{t,3}^\ell|
\ll\|f_1\|_\infty\|f_2\|_{m,\beta}t^{-((r-1)/2+L+1)}.
$$
\end{claim}
\begin{proof}[Proof of \Cref{cla:sumint-b_0}]
The cocycle identity
$K_n\circ T^{k_t}=K_n-K_{k_t}+K_{k_t}\circ T^n$
and transfer-operator duality give
\begin{equation}
\label{eqn:intFH}
\begin{aligned}
R_{t,3}^\ell
&=\sum_{t^{1-\epsilon_1}<n<t_+}
\int_{\mathfrak a\smallsetminus\mathcal O}
e^{-i\langle\xi,t\ut+\ell(t)\ww\rangle}\cdot\int_{\Lambda_+}
\widetilde H_{1,t}(x,\xi)
\mathcal L_\xi^{\,n-k_t-1}
\bigl(\mathcal L_\xi^{\,k_t+1}
\widetilde H_{2,t}(\cdot,-\xi)\bigr)(x)
\,d\nu_+(x)\,d\xi.
\end{aligned}
\end{equation}

Let $m>r$ be an integer, to be chosen below.
Set $D_1:=\|f_1\|_\infty\|f_2\|_{m,\beta}$, and let
$\beta_1,\beta_2>0$ be as in \Cref{prop:norm of transfer operator}.
 For $\xi\in\mathfrak a\smallsetminus\mathcal O$, put
$B_{t,\xi}:=\mathcal L_\xi^{k_t+1}
\widetilde H_{2,t}(\cdot,-\xi)$.
Since $k_t=O(\log t)$ and $n>t^{1-\epsilon_1}$, we have
$n-k_t-1\geq n/2$ throughout the summation
range for sufficiently large $t$.
By \Cref{prop:norm of transfer operator} and
\Cref{lem:Lkttv}(1),
\begin{align*}
\|\mathcal L_\xi^{n-k_t-1}B_{t,\xi}\|_\infty
&\ll e^{-\beta_1n/(2\|\xi\|^{\beta_2})}
\|B_{t,\xi}\|_\xi\ll e^{-\beta_1n/(2\|\xi\|^{\beta_2})}
\|\xi\|^{-m}\|f_2\|_{m,\beta}.
\end{align*}
Moreover, \eqref{eqn:inffa}, $T$-invariance of $\nu_+$,
and the exponential moment of $R$ give
$$
\int_{\Lambda_+}|\widetilde H_{1,t}(x,\xi)|\,d\nu_+(x)
\ll \|f_1\|_\infty
\int_{\Lambda_+}(R\circ T^{k_t}+1)^r\,d\nu_+
\ll\|f_1\|_\infty.
$$
Applying these bounds to \eqref{eqn:intFH}, we obtain
$$
|R_{t,3}^\ell|
\ll D_1\sum_{t^{1-\epsilon_1}<n<t_+}
\int_{\mathfrak a\smallsetminus\mathcal O}
e^{-\beta_1 n/(2\|\xi\|^{\beta_2})}\|\xi\|^{-m}\,d\xi.
$$
There are $O(t)$ summands, and
$\mathfrak a\smallsetminus\mathcal O
\subset\{\xi:\|\xi\|\geq b_0\}$.
Using $n>t^{1-\epsilon_1}$ and polar coordinates therefore gives
$$
|R_{t,3}^\ell|
\ll D_1t\int_{b_0}^\infty
e^{-ct^{1-\epsilon_1}s^{-\beta_2}}s^{r-1-m}\,ds.
$$
With the substitution $s=t^{(1-\epsilon_1)/\beta_2}y$,
the integral on the right becomes
$$
t^{-(1-\epsilon_1)(m-r)/\beta_2}
\int_{b_0t^{-(1-\epsilon_1)/\beta_2}}^\infty
e^{-cy^{-\beta_2}}y^{r-1-m}\,dy.
$$
The last integral is bounded uniformly in $t$: the exponential
ensures integrability near zero, while $m>r$ ensures integrability
at infinity. Hence
$
|R_{t,3}^\ell|
\ll D_1t^{1-(1-\epsilon_1)(m-r)/\beta_2}.
$.
Choosing $m$ to be any number larger than $\max\left\{r+1,
\left\lceil r+\frac{\beta_2}{1-\epsilon_1}
\left(\frac{r-1}{2}+L+2\right)\right\rceil\right\}$ gives the required bound.
\end{proof}
Combining \eqref{eqn:R34} and \Cref{cla:sumint-b_0}
proves \Cref{cla:S1}.

\begin{step}[Expansion of the normalized expectations]
\end{step}
Recall the normalized derivatives $\mathcal A_{N,n}$ and their limits
$\mathcal A_N$ from \eqref{eqn:ANn def} and \eqref{eqn:AN def}.
We suppress their two function arguments in this step.
Set $D_2:=\|f_1\|_{0,\beta}\|f_2\|_{0,\beta}$ where $\beta$ is fixed in \eqref{eqn:beta-choiceofe1}.
Taylor's formula at $\xi=0$ gives, for every $n\geq1$
and $\xi\in\mathcal O$,
$$
\begin{aligned}
& \mathcal A_{0,n}(\xi)=\kappa_\xi^{-n}
\int_{\Lambda_+\times\Lambda_-}
e^{i\langle\xi,K_n(x)\rangle}
\widehat F_1(\widehat T^n(x,x_-),\xi)
\widehat{\Fn}(x,x_-,-\xi)\,
d\nu^\diamond(x,x_-)\\
&\quad=
\sum_{N=0}^{J-1}
\frac{\mathcal A_{N,n}(0)*\xi^{\otimes N}}{N!}
+O\left(
\sup_{0\leq s\leq1}|\mathcal A_{J,n}(s\xi)|
\,\|\xi\|^J
\right).
\end{aligned}
$$

By \Cref{lem:ANF1F2}(1),(2) and
\eqref{eqn:radial eigenvalue modulus Section18}, this equals
\begin{equation}
\label{eqn:sumAN}
\sum_{N=0}^{J-1}\frac{\mathcal A_N*\xi^{\otimes N}}{N!}
+O\left(D_2\left(
n^{-((r-1)/2+L+J)}|\kappa_\xi|^{-n}
+n^{-\bar L}+n^{2/5}\|\xi\|^J\right)\right),
\end{equation}
where $\bar L\geq (r-1)/2+L+5/2$ is a fixed integer.
Multiplying \eqref{eqn:sumAN} by $\kappa_\xi^n$ and
substituting the expansion \eqref{eqn:kappawn}, we obtain,
using \eqref{eqn:kappa25-kappafw} and
\eqref{eqn:upper bound AN},
$$
\begin{aligned}
&\int_{\Lambda_+\times\Lambda_-}
e^{i\langle\xi,K_n(x)\rangle}
\widehat F_1(\widehat T^n(x,x_-),\xi)
\widehat{\Fn}(x,x_-,-\xi)\,
d\nu^\diamond(x,x_-)\\
&\quad=
e^{in c_R^2\xi_1}\widehat{\mathcal G}(\sqrt n\,\xi)
\sum_{N=0}^{J-1}\sum_{k=0}^{2L+4}
\sum_{j=3k}^{2L+4+2k}
\frac{n^k}{N!}
(\mathcal A_N\otimes A_{j,k})*\xi^{\otimes(N+j)}
+E_n(\xi),
\end{aligned}
$$
where, by the substitution $\xi\mapsto\xi/\sqrt n$ and $J=2L+5$,
$$
\int_{\mathcal O}|E_n(\xi)|\,d\xi
\ll D_2\left(n^{-((r-1)/2+L+J)}+n^{-((r-1)/2+L+5/2)}\right).
$$
Summing and using \eqref{eqn:beta-choiceofe1}, we obtain
$$
C_{t,4}^\ell=C_{t,5}^\ell+
O\bigl(D_2t^{-((r-1)/2+L+1)}\bigr),
$$
where
$$
C_{t,5}^\ell:=\sum_{N=0}^{J-1}\frac1{N!}
\sum_{k=0}^{2L+4}\sum_{j=3k}^{2L+4+2k}
\sum_{t^{1-\epsilon_1}<n<t_+}\mathcal I_n^{N,k,j}
$$
and
$$
\mathcal I_n^{N,k,j}:=n^k\int_{\mathcal O}
e^{-i\langle\xi,(t-nc_R^{2})\ut+\ell(t)\ww\rangle}
\widehat{\mathcal G}(\sqrt n\,\xi)
(\mathcal A_N\otimes A_{j,k})*\xi^{\otimes(N+j)}\,d\xi.
$$

\begin{step}[Fourier inversion]
\end{step}
After the change of variables $\eta=\sqrt n\,\xi$ in the integral defining $\mathcal I_n^{N,k,j}$,
Gaussian decay and \eqref{eqn:upper bound AN} allow us
to extend the integration domain from $\sqrt n\,\mathcal O$
to $\mathfrak a$, with error $O(D_2e^{-cn})$.
Fourier inversion then gives
$$
\begin{aligned}
\mathcal I_n^{N,k,j}
&=(2\pi)^r i^{N+j}n^{k-(r+N+j)/2}
\mathcal G^{(N+j)}\left(\frac{t-nc_R^{2}}{\sqrt n},
\frac{\ell(t)\ww}{\sqrt n}\right)
*(\mathcal A_N\otimes A_{j,k})+O(D_2e^{-cn}).
\end{aligned}
$$

For $n\geq t_+$, we have $c_R^2n\geq2t$, and hence
$\frac{|t-nc_R^2|}{\sqrt n}
\geq \frac{c_R^2}{2}\sqrt n$.
The Gaussian and its derivatives therefore decay
exponentially in $n$ on this range.
Thus the terms with $n\geq t_+$ contribute
$O(D_2e^{-ct})$, and we may extend the sum to all
$n>t^{1-\epsilon_1}$. Consequently,
$$
C_{t,5}^\ell=C_{t,6}^\ell+
O\bigl(D_2t^{-((r-1)/2+L+1)}\bigr),
$$
where
\begin{equation}
\label{eqn:Ct6}
\begin{aligned}
C_{t,6}^\ell
& := (2\pi)^r\sum_{N=0}^{J-1}\sum_{k=0}^{2L+4}
\sum_{j=3k}^{2L+4+2k}\frac{i^{N+j}}{N!}
\sum_{n>t^{1-\epsilon_1}}n^{k-(r+N+j)/2}\\
&\qquad\cdot\mathcal G^{(N+j)}\left(\frac{t-nc_R^{2}}{\sqrt n},
\frac{\ell(t)\ww}{\sqrt n}\right)
*(\mathcal A_N\otimes A_{j,k}).
\end{aligned}
\end{equation}
Combining Steps~I--IV, for some integer $m\geq r+1$ we have
\begin{equation}
\label{eqn:unweighted approximant error Section18}
\left|t^{(r-1)/2}C_t^\ell-
\frac{t^{(r-1)/2}}{(2\pi)^r c_R^{2}}C_{t,6}^\ell\right|
\ll\|f_1\|_{0,\beta}\|f_2\|_{m,\beta}t^{-(L+1)},
\end{equation}
uniformly in $\ww\in\ker\psi$.

\begin{step}[Summing over $n$] 
\end{step}
\emph{Case 1: $\ell(t)=0$.}
For an integer $p\geq0$, write
$\mathcal G^{(p)}:=D^p\mathcal G$ for the $p$-th total
derivative, with $\mathcal G^{(0)}=\mathcal G$.
For $\sigma\in\mathbb R$, define
$$
h_{p,\sigma}(z_1,z_2)
:=z_2^\sigma D^p\mathcal G
\left(\frac{c_Rz_1}{\sqrt{z_2}}\mathtt u\right),
\qquad
(z_1,z_2)\in\mathbb R\times(0,\infty).
$$

Set
$$
\begin{aligned}
\mathcal V:=\left\{(N,k,j,q)\in\mathbb Z_{\geq0}^4:
\begin{matrix} 
\ N\leq J-1,\ k\leq2L+4, \\
3k\leq j\leq2L+4+2k,
\ N+j-2k+q\leq2L
\end{matrix} \right\}.
\end{aligned}
$$
Apply \cite[Lemma~A.3]{DNP} with $Q=2L+1$ to each Gaussian sum.
The lower cutoff can be changed to any fixed positive multiple of $t$
below $t/c_R^{2}$ at an exponentially small cost. By
\cite[Lemma~A.4]{DNP}, the terms with $N+j+q$ odd vanish.
Thus
\begin{equation}
\label{eqn:CtF1F2 final}
\begin{aligned}
C_{t,6}^{\ell}
={}&(2\pi)^r\sum_{(N,k,j,q)\in\mathcal V}
\frac{i^{N+j}}{N!q!}
\frac{(\sqrt t/c_R)^{-((r-1)+N+j-2k+q)}}{c_R^{q+1}}\\
&\cdot\left(\int_{\mathbb R}
\partial_2^q h_{N+j,k-(N+j+r)/2}(s,1)(-s)^q\,ds\right)
*(\mathcal A_N\otimes A_{j,k}) +O\bigl(D_2t^{-((r-1)/2+L+1)}\bigr).
\end{aligned}
\end{equation}
The remaining omitted terms have $N+j-2k+q\geq2L+2$ and are
absorbed into the error. After dividing by $(2\pi)^r c_R^2$ and multiplying by
$t^{(r-1)/2}$, define, for $0\leq m\leq L$,
$$
\begin{aligned}
\mathfrak c_m(f_1,f_2)
:={}&
\sum_{\substack{(N,k,j,q)\in\mathcal V\\
N+j-2k+q=2m}}
\frac{i^{N+j}c_R^{r-4+N+j-2k}}{N!\,q!}\\
&\quad\cdot
\left(
\int_{\mathbb R}
\partial_2^q h_{N+j,k-(N+j+r)/2}(s,1)
(-s)^q\,ds
\right)
*(\mathcal A_N\otimes A_{j,k}).
\end{aligned}
$$
With these definitions, \eqref{eqn:CtF1F2 final} becomes
$$
\frac{t^{(r-1)/2}}{(2\pi)^r c_R^2}\,C_{t,6}^{\ell}
=
\sum_{m=0}^{L}\mathfrak c_m(f_1,f_2)t^{-m}
+O\bigl(D_2t^{-(L+1)}\bigr).
$$
Combining this with
\eqref{eqn:unweighted approximant error Section18}
gives the correlation expansion in \Cref{bmslocal}(1).
Uniqueness of the coefficients of an asymptotic expansion makes
these definitions compatible as $\mathcal K$ and $L$ vary.
They therefore define bilinear forms on
$\mathcal C_c^\infty(\Gamma_\rho\backslash G)^M
\times\mathcal C_c^\infty(\Gamma_\rho\backslash G)^M$.

Since each defining sum is finite, \eqref{eqn:upper bound AN} gives
$$
|\mathfrak c_m(f_1,f_2)|
\ll_{\mathcal K,L}
\|f_1\|_{0,\beta}\|f_2\|_{0,\beta}
\qquad(0\leq m\leq L)
$$ for $f_1,f_2\in\mathcal C_{\mathcal K}^{\infty}$.
This implies the coefficient bounds in \Cref{bmslocal}(1).

By \Cref{lem:ANF1F2}(3),
$\mathcal A_0=c_R^4\mathtt M_\psi(F_1)\mathtt M_\psi(F_2)$.
The term $N=k=j=q=0$ in \eqref{eqn:CtF1F2 final} is therefore
$$
(2\pi)^r c_R^{r+2}t^{-(r-1)/2}
\left(\int_{\mathbb R}\mathcal G(c_Rs,0)\,ds\right)
\mathtt M_\psi(F_1)\mathtt M_\psi(F_2).
$$
Since $\mathtt M_\psi(F_i)=\mathsf m_\psi(f_i)$ by
\Cref{cor:factor map}, division by $(2\pi)^r c_R^{2}$ proves the
stated leading coefficient and completes part~(1).

\emph{Case 2: $\ell(t)=\sqrt t$.}
For this, carry out Steps~I--IV with
$L=\lceil (r-1)/2+1\rceil$ and put
$$
P_t(\ww):=\frac{t^{(r-1)/2}}{(2\pi)^r c_R^{2}}C_{t,6}^{\ell}(f_1,f_2;\ww).
$$
Since $N+j-2k\geq1$ unless $N=k=j=0$,
\eqref{eqn:weighted Gaussian sum Section18} bounds every nonleading
term by $O\bigl(D_2t^{-1/2}e^{-a_1\min\{\|\ww\|^2,\|\ww\|\}}\bigr)$.
Applying \Cref{lem:uniform renewal sums} to the leading term and
using \Cref{lem:ANF1F2}(3), we obtain
\begin{equation}
\label{eqn:CtF1F2sq}
P_t(\ww)=\mathfrak c_{\ww}\mathsf m_\psi(f_1)\mathsf m_\psi(f_2)
+O\bigl(D_2t^{-1/2}e^{-a_1\min\{\|\ww\|^2,\|\ww\|\}}\bigr),
\end{equation}
where $a_1>0$ is uniform.
In particular, $|P_t(\ww)|\ll D_2e^{-a_1\min\{\|\ww\|^2,\|\ww\|\}}$.
Choose $a_2:=\min\{a_0,a_1\}$. By
\Cref{lem:rough transverse correlation bound}, the triangle inequality
gives
$$
|t^{(r-1)/2}C_t^{\ell}-P_t(\ww)|
\ll D_2t^{(r-1)/2+1}e^{-a_2\min\{\|\ww\|^2,\|\ww\|\}}.
$$
Take the geometric mean of this bound and
\eqref{eqn:unweighted approximant error Section18}. Since
$L\geq (r-1)/2+1$, we obtain
\begin{align*}
|t^{(r-1)/2}C_t^{\ell}-P_t(\ww)|
&\ll\|f_1\|_{0,\beta}\|f_2\|_{m,\beta}
t^{((r-1)/2-L)/2}e^{-(a_2/2)\min\{\|\ww\|^2,\|\ww\|\}}\\
&\ll\|f_1\|_{m,\beta}\|f_2\|_{m,\beta}
t^{-1/2}e^{-(a_2/2)\min\{\|\ww\|^2,\|\ww\|\}}.
\end{align*}
Combining this with \eqref{eqn:CtF1F2sq} proves
\eqref{eqn:global weighted error} with $a=a_2/2$.
This completes the proof of \Cref{bmslocal}. \qed

\section{Haar local mixing and orbital counting}
\label{sec.haar mixing}

We deduce Haar local mixing and orbital counting from BMS local
mixing, using the arguments of Edwards--Lee--Oh
\cite{ELO_anosov}. 
\subsection{Measure normalizations}
\label{subsec.BR}
 Throughout this part,
$\mathsf m_\psi^{\BMS}$ denotes the BMS measure induced by
\eqref{bmsdef}.
Let $m_o$ be the $K$-invariant probability measure on $\mathcal F$,
which is $(G,\varrho)$-conformal, where
$\varrho=\sum_{\alpha\in\Phi^+}\dim(\mathfrak g_\alpha)\alpha$.
Define Haar and Burger--Roblin measures in Hopf coordinates by
$$
\begin{aligned}
&d\widetilde{\mathsf m}^{\op{Haar}}(g)
=e^{\varrho(\beta_{g^+}(e,g))+\varrho(\beta_{g^-}(e,g))}
\,dm_o(g^+)\,dm_o(g^-)\,d\Leb_{\mathfrak a},\\
&d\widetilde{\mathsf m}_\psi^{\BR}(g)
=e^{\varrho(\beta_{g^+}(e,g))+\psi(\beta_{g^-}(e,g))}
\,dm_o(g^+)\,d\nu_\psi(g^-)\,d\Leb_{\mathfrak a},\\
&d\widetilde{\mathsf m}_\psi^{\BR_*}(g)
=e^{\psi(\beta_{g^+}(e,g))+\varrho(\beta_{g^-}(e,g))}
\,d\nu_\psi(g^+)\,dm_o(g^-)\,d\Leb_{\mathfrak a}.
\end{aligned}
$$
We lift these measures from $G/M$ to $G$ using Haar probability
measure on $M$, and denote their descents to
$\Gamma_\rho\backslash G$ by the corresponding untilded symbols.
For $\psi=\psi_u$, write
$m_u^{\BR}:=\mathsf m_{\psi_u}^{\BR}$ and
$m_u^{\BR_*}:=\mathsf m_{\psi_u}^{\BR_*}$.
Uniqueness of tangent forms gives
$\psi_{\i(u)}=\psi_u\circ\i$.
Thus the Haar local mixing limit involves
$m_{\i(u)}^{\BR}$ and $m_u^{\BR_*}$ in the notation
of the introduction.

After rescaling the BMS measure to $\mathsf m_\psi^{\BMS}$,
write the coefficient in \Cref{bmslocal}(2) as
$c_\psi e^{-Q_\psi(w)/2}$, where $c_\psi>0$ and
$w\in\ker\psi$.

\subsection{Local mixing}
\begin{proof}[Proof of \Cref{thm:local-mixing-morse}]
Fix $u\in\operatorname{int}\mathcal L_\rho$.
By \cite[Theorem~10.1]{CZZ_relative}, $\psi_u>0$ on
$\mathcal L_\rho\smallsetminus\{0\}$, so \Cref{bmslocal}
applies with $\psi=\psi_u$.
Since $u=\psi_\rho(u)\ut_{\psi_u}$ by
\Cref{thm.mixingdirectiontangent}, rescaling gives
\Cref{bmslocal_intro}, with
$Q_u=\psi_\rho(u)^{-1}Q_{\psi_u}$.
\Cref{m1} follows from
\cite[Proposition~6.8]{ELO_anosov}.
The convergence is locally uniform in $w$ by the uniform
BMS asymptotic and the horospherical contraction estimates
in the proofs of \cite[Propositions~5.3 and~6.1]{ELO_anosov}.
Indeed, for every compact $E\subset V_u$ and every
simple root $\alpha$,
$\inf_{w\in E}\alpha(tu+\sqrt t\,w)
\geq t\alpha(u)-O_E(\sqrt t)\to \infty$
so these contraction estimates are uniform on $E$.

For the transverse bound, the upper comparisons in
\cite[Propositions~5.3, 6.1, and~6.8]{ELO_anosov}
reduce the estimate to BMS correlations at the same
$A$-parameter whenever $tu+\sqrt t\,w\in\mathfrak a^+$.
For fixed compact supports, domination by fixed smooth
majorants and \Cref{bmslocal_intro}(2) give the asserted
decay, with the stated dependence on
$\|f_1\|_\infty\|f_2\|_\infty$.
Finally, the rescaling and measure normalizations give
$\kappa_u c_u
=\psi_\rho(u)^{-(r-1)/2}c_{\psi_u}$.
\end{proof}

\begin{proof}[Proof of \Cref{cor.ergodicdichotomy}]
The BMS local mixing asymptotic implies $u$-balancedness
in the sense of \cite{BLLO}.
The conclusion follows from \cite[Theorem~1.4]{BLLO}.
\end{proof}

\subsection{Equidistribution of expanding horospheres}

Our local mixing theorem also gives equidistribution of
translates of maximal horospheres.
Let $\check N$ be the maximal horospherical subgroup
satisfying $a_{tu}na_{-tu}\to e$ as $t\to\infty$
for $n\in\check N$ and $u\in\operatorname{int}\mathfrak a^+$. For $x=[g]\in\Gamma_\rho\backslash G$,
the measure $\mu^{\mathrm{PS}}_{x\check N,u}$ on
$\check N$ is defined by
$d\mu^{\mathrm{PS}}_{x\check N,u}(n)
=e^{\psi_u(\beta_{gnP}(e,gn))}\,d\nu_{\psi_u}(gnP)$.
 We normalize Haar measure $dn$ on $\check N$
compatibly with the measures in \Cref{m1}, as in
\cite[Section~4]{ELO_anosov}.
Applying \cite[Proposition~6.1]{ELO_anosov} to
\Cref{bmslocal_intro} gives the following.

\begin{corollary}
\label{cor:horospherical-equidistribution}
For every $u\in\operatorname{int}\mathcal L_\rho$,
$x\in \Ga_\rho\ba G$, $f\in C_c(\Ga_\rho\ba G)^M$, and
$\phi\in C_c(\check N)$, we have
$$
\begin{aligned}
&\lim_{t\to\infty}
t^{\frac{r-1}{2}}e^{t(\varrho-\psi_u)(u)}
\int_{\check N}
f(xn a_{tu})\phi(n)\,dn =
\kappa_u c_u\,
m_{\i(u)}^{\mathrm{BR}}(f)\,
\mu^{\mathrm{PS}}_{x\check N,u}(\phi),
\end{aligned}
$$
where $\kappa_u$ and $c_u$ are as in \Cref{m1}.
\end{corollary}

\subsection{Orbital counting in Riemannian balls}
\label{sec.orbital-counting}

Let $\langle\cdot,\cdot\rangle$ and $\|\cdot\|$ denote
the inner product and norm on $\mathfrak a$ induced by
the Killing form.
By \cite{Quint2002divergence,CZZ_transverse,CZZ_relative},
there is a unique unit vector
$u_\rho\in\operatorname{int}\mathcal L_\rho$ such that
$$
\delta_\rho
=\max_{\|v\|=1}\psi_\rho(v)
=\psi_\rho(u_\rho).
$$
Differentiability of $\psi_\rho$
\cite{CZZ_transverse,CZZ_relative} gives
$\psi_{u_\rho}(v)=\delta_\rho\langle v,u_\rho\rangle$.
In particular,
$\mathbb R u_\rho\perp\ker\psi_{u_\rho}$. For $Q\subset K$, write $Q[P]:=\{qP:q\in Q\}$. Let $\nu_{u_\rho}:=\nu_{\psi_{u_\rho}}$.

\begin{corollary}[Bisector counting in Riemannian balls]
\label{cor.countingbisector}
There exists $c_\rho>0$ such that, for all Borel sets
$Q_1,Q_2\subset K$ satisfying
$\nu_{u_\rho}(\partial(Q_1[P]))
=\nu_{u_\rho}(\partial(Q_2^{-1}w_0[P]))=0$,
$$
\lim_{T\to\infty}e^{-\delta_\rho T}
\#\left\{
\gamma\in\Gamma_\rho\cap Q_1MA^+Q_2:
\|\mu(\gamma)\|<T
\right\}
=
c_\rho\,
\nu_{u_\rho}(Q_1[P])
\nu_{u_\rho}(Q_2^{-1}w_0[P]).
$$
\end{corollary}

\begin{proof}
Apply the counting argument of
\cite[Corollary~9.21 and the proof of Theorem~1.8]{ELO_anosov}
with $H=K$, $\mathfrak b=\mathfrak a$, and the Killing-form
norm, using \Cref{m1} in direction $u_\rho$.
\end{proof}

Taking $Q_1=Q_2=K$ in \Cref{cor.countingbisector} gives
the orbital asymptotic at $o=K$ as stated in \Cref{thm.counting_intro}.
Conjugating $\Gamma_\rho$ gives the result at every
$o\in G/K$.

\end{document}